\documentclass[reqno,11pt]{amsart}
\usepackage{amsmath,amsfonts,amssymb,amsxtra,latexsym,amscd,enumerate,amsthm,verbatim}
\allowdisplaybreaks
\usepackage{hyperref}
\usepackage{graphicx}
\usepackage{bbm}
\usepackage{physics}
\usepackage{bbm}
\usepackage{extarrows}
\usepackage{tikz}
\usetikzlibrary{arrows.meta}

\usepackage{hyperref}
\hypersetup{
    colorlinks=true,
    citecolor= blue,
    linkcolor=black,
    filecolor=magenta,
    urlcolor=cyan
    }

\usepackage[margin=1in]{geometry}
\numberwithin{equation}{section}

\usepackage{subcaption}

\newcommand{\R}{\mathbb{R}}
\newcommand{\E}{\mathcal{E}}
\newcommand{\A}{\mathcal{A}}
\newcommand{\B}{\mathcal{B}}

\newcommand{\T}{\mathbb{T}}

\newcommand{\Z}{\mathbb{Z}}

\newcommand{\eps}{\epsilon}

\numberwithin{equation}{section} %pour numeroter les equations par section

\newtheorem{assumptions}{Assumptions}[section]
\newtheorem{definition}{Definition}[section]
\newtheorem{lemma}{Lemma}[section]
\newtheorem{proposition}{Proposition}[section]
\newtheorem{corollary}{Corollary}[section]
\newtheorem{remark}{Remark}[section]

\newtheorem{theorem}{Theorem}[section]
\newtheorem{question}{Question}[section]
\newtheorem*{theorem*}{Theorem}
\newtheorem*{assumption*}{Assumption}

\newcommand{\p}{\partial}

\newcommand{\cO}{\mathcal O}

\usepackage[shortlabels]{enumitem}

\begin{document}

\title{Global well-posedness and regularity of the dynamical Prandtl Equation}

\author{Hao Jia}
\address{School of Mathematics, University of Minnesota, Minneapolis, Minnesota 55455, USA}

\email{jia@umn.edu}

\author{Zhen Lei}
\address{School of Mathematical Sciences; LMNS and Shanghai Key Laboratory for Contemporary Applied Mathematics, Fudan University, Shanghai 200433, P. R. China.}
\email{zlei@fudan.edu.cn}

\author{Cheng Yuan}
\address{School of Mathematical Sciences; LMNS and Shanghai Key Laboratory for Contemporary Applied Mathematics, Fudan University, Shanghai 200433, P. R. China.}\email{cyuan22@m.fudan.edu.cn}

%\thanks{HJ is supported in part by NSF grant DMS-1945179. }

\begin{abstract}
{\small}
In this paper, we study the dynamical Prandtl equation, which plays an important role in the study of the vanishing-viscosity limit of the Navier--Stokes equation near the boundary. Our focus is on the (Sobolev) well-posedness regime, where the initial and boundary data satisfy a crucial monotonicity condition. In this case, local classical solutions were constructed in the pioneering work of Oleinik \cite{O68}. A global weak solution was subsequently obtained by Xin and Zhang \cite{XZ04};  Xin, Zhang, and Zhao \cite{XZZ24} then proved uniqueness and established interior H\"older estimates in Crocco coordinates.

Using a precise description of the fundamental solution to the Kolmogorov equation in the upper half space, we first obtain the H\"older regularity of local weak solutions up to the boundary. We also provide a detailed proof of higher-order regularity estimates using the H\"older regularity, together with $W^{2,p}$-type Sobolev estimates and $H^{s}$-type hypoelliptic estimates, which are nontrivial. Up-to-boundary smoothness of solutions (in the Crocco coordinates) is important in order to conclude smoothness of the Prandtl solutions in the original physical variables, even in the interior. It is also physically significant for applications to Navier--Stokes equations when constructing globally defined high-precision approximate solutions where the dynamical Prandtl equation is essential.

\begin{comment}
Our first main result extends the H\"older regularity of local weak solutions of Xin-Zhang-Zhao to the boundary using a precise description of the fundamental solution to the Kolmogorov equation in the upper half space, and De Giorgi-Nash-Moser theory generalized to ultraparabolic equations with a boundary following the approach of  \cite{XZZ24} which introduced a key weak type Poincar\'e inequality. We also provide a detailed proof of higher order regularity estimates using the H\"older regularity, together with Sobolev and H\"ormander's hypoelliptic estimates, which are nontrivial. Up-to-boundary smoothness of solutions (in the Crocco coordinates) is important in order to conclude smoothness of the Prandtl solutions in the original physical variables, even in the interior. It is also physically significant for applications to Navier Stokes equations when constructing globally defined high precision approximate solutions where the dynamical Prandtl equation is essential.
\end{comment}

Using the smoothing estimates for local weak solutions, we then prove the global existence and regularity of classical solutions to the dynamical Prandtl equation under monotonicity assumptions, which was listed by Oleinik and Samokhin in \cite{OS99} as one of the open problems for the Prandtl equation. We also develop a self-contained local existence theory using weighted energy estimates and further expand the theory of global weak solutions. The main point is to incorporate all three types (polynomial, exponential and Gaussian) of asymptotic matching of the boundary layer with the outer flow, which is expected to be useful for applications to Navier--Stokes equations.

\end{abstract}

\maketitle
\setcounter{tocdepth}{1}
\pagestyle{plain}

\tableofcontents

\section{Introduction}\label{intro}
In this paper, we consider the unsteady Prandtl boundary-layer system for $(t,x,y)\in (0,T)\times (0,X)\times(0,+\infty)$,
\begin{equation}\label{generalprandtl}
\begin{cases}
\p_{t}u+u\p_{x}u+v\p_{y}u-\p_{y}^{2}u+\partial_xp(t,x)=0,\ \\
\p_{x}u+\p_{y}v=0,
\end{cases}
%\tag{PL}
\end{equation}
with prescribed initial and inflow data as well as boundary conditions
\begin{equation}\label{generalprandtl2}
    \begin{cases}
u|_{y=0}=v|_{y=0}=0,\quad \lim\limits_{y\to\infty}u(t,x,y)=U(t,x),\\
u|_{t=0}=u_{0}(x,y),\quad u|_{x=0}=u_{1}(t,y).
    \end{cases}
\end{equation}

The pressure field $p(t,x)$ and the outer stream speed $U(t,x)$ satisfy  Bernoulli's law
\begin{equation*}
\p_{t}U+U\p_{x}U+\p_{x}p\equiv 0.
\end{equation*}
For simplicity, we assume that $U \equiv 1$, which implies $\p_{x}p \equiv 0$. The case of general (smooth) $U$ may be considered under suitable assumptions such as $U>0$ and the ``favorable pressure" condition $\partial_xp\leq0$. We expect that corresponding versions of our main conclusions still hold in this case but the proof will become more complicated. 

Under this assumption, \eqref{generalprandtl} and \eqref{generalprandtl2} reduce to
\begin{equation}\label{prandtl}
\begin{cases}
\p_{t}u+u\p_{x}u+v\p_{y}u-\p_{y}^{2}u=0,\\
\p_{x}u+\p_{y}v=0,\\
u|_{y=0}=v|_{y=0}=0,\ \lim\limits_{y\to\infty}u(t,x,y)=1,\
u|_{t=0}=u_{0}(x,y),\ u|_{x=0}=u_{1}(t,y),
\end{cases}
\end{equation}
for $(t,x,y)\in(0,T)\times(0,X)\times(0,\infty)$. %Both $T$ and $X$ may be infinite. 

The Prandtl equation \eqref{generalprandtl} was proposed by Prandtl \cite{P1904} in 1904 to describe the dynamics inside a thin layer of  slightly viscous fluids near a physical boundary. The essential observation of Prandtl is that the fluid flow may be separated into two regions: the main flow region where viscosity can be neglected corresponding to the inviscid solution, and a thin boundary layer near the body's surface where viscosity plays a crucial role. The  transition from the boundary layer to the outer inviscid flow is encapsulated mathematically by the Prandtl equation \eqref{generalprandtl}.   The study of boundary layers and the associated Prandtl equations has since played a fundamental role in understanding fluid dynamics at high Reynolds numbers.

The rigorous analysis of \eqref{generalprandtl} was pioneered by Oleinik. Under suitable \textit{monotonicity} and regularity conditions, Oleinik established the first local well-posedness results for classical solutions, see the monograph \cite{OS99} by Oleinik and Samokhin. The monotonicity condition ($\partial_yu>0$ in our case) turns out to be crucial, and strong ill-posedness in Sobolev spaces can occur if the monotonicity condition is violated; see \cite{GD10}. We refer to Section \ref{sec:classical} for more discussions and references.

Given the local well-posedness of \eqref{prandtl} under suitable monotonicity assumptions, a natural question to ask is whether one can also prove global well-posedness. In \cite{OS99}, two types of quasi-global regularity results were obtained, assuming that either $X>0$ or $T>0$ is sufficiently small. A major open problem, raised already by Oleinik and Samokhin (Question 4, page 500) in \cite{OS99}, is whether one can remove the smallness on $T$ or $X$ and prove global classical solutions for arbitrary $T>0$ \emph{and} $X>0$; see for example  \cite{IV16,MaekawaMazzucato2018InviscidLimit,Xin2002SingularSolutions} for recent references to this problem. Besides its significance for \eqref{prandtl}, the global regularity of \eqref{prandtl} is also expected to be important for the study of the validity of global boundary-layer expansions for the Navier--Stokes equations since the dynamical Prandtl equation arises as the limiting equation in the vanishing viscosity limit of the Navier--Stokes equations.

Important progress was made by Xin and Zhang \cite{XZ04}, who proved the existence of global weak solutions to \eqref{prandtl} under general conditions on the initial and boundary data.
\begin{comment}
An important observation used in \cite{XZ04} is that under the Crocco transform, the Prandtl equation \eqref{prandtl} becomes an \textit{ultraparabolic quasilinear system} and exhibits a ``\emph{conservation law}" structure. \end{comment} 
 More recently in a remarkable work, Xin, Zhang and Zhao \cite{XZZ24} proved further that the weak solution is unique and H\"older continuous {\it away from the boundary}, from which higher regularity of the weak solutions in the interior was expected to follow (in the Crocco coordinates). The key tool in \cite{XZZ24} is the extension of De Giorgi-Nash-Moser's H\"older regularity estimates to the ultraparabolic setting, using the approach of Kruzkov \cite{K64} as well as a new weak Poincar\'e type inequality (see Lemma \ref{weakpoin}). 

The important remaining open problem is whether one can establish regularity \textit{up to the boundary $y=0$}, and apply such regularity estimates to construct global in time and space classical solutions to the Prandtl equation \eqref{prandtl}. We remark that up-to-boundary regularity does \textit{not} follow from the interior regularity argument directly, and the De Giorgi-Nash-Moser and higher regularity theory for the ultraparabolic equation with boundary are significantly more complicated than the case without boundary. Up-to-boundary regularity results for the Prandtl system are essential for the study of boundary layer expansions, since one needs to construct very precise approximate solutions to the Navier--Stokes system near the boundary in the vanishing viscosity limit, and these approximate solutions depend crucially on solving the Prandtl system, see e.g. \cite{GI21, gao2023steady,iyer2026global} for the stationary case. Moreover, in order to transform the regularity estimates in the Crocco coordinates back to the physical variables, up-to-boundary regularity is essential even if one is only interested in interior regularity ($t>0, x>0$ and $y>0$), due to the nonlocal nature of the Crocco transform, see \eqref{croccotrans}, \eqref{invertcrocco} and Remark \ref{re:impreg1}.

Our goal in the paper is to address this problem and prove the existence and smoothness of global classical solutions to \eqref{prandtl}. We will also provide an essentially self-contained treatment of the local well-posedness theory and prove existence of globally defined unique weak solutions for quite general initial and boundary conditions. In particular, we allow the asymptotic behavior of $u(t,x,y)\to 1$ as $y\to \infty$ to include three distinct rates: \textit{polynomial}, \textit{exponential} and \textit{Gaussian}. This is important when considering physical regimes not covered by earlier works which require either polynomial or exponential matching rates in our setting (see e.g. \cite{OS99,MW15,AWXY15}), including perturbations of the Blasius flow which has a Gaussian matching rate and is the unique (up to scaling) self-similar solution to the steady Prandtl equation. We expect such general results to be useful when studying stability of boundary layers in the Navier--Stokes setting. 

\subsection{Main Results}
In this section, we present our three main results: (1) up-to-boundary smoothness of local weak solutions; (2) global existence and regularity of classical solutions for sufficiently smooth initial and boundary conditions; and (3) global existence and smoothness of weak solutions under more general assumptions on the initial and boundary conditions. We also provide a new self--contained proof of existence of local classical solutions using novel quasi-linear weighted energy estimates, see Section \ref{local}. Since in our approach, as in \cite{O68,OS99,XZ04,XZZ24}, the Crocco transform plays an essential role, we start with the introduction of the Crocco transform and the transformed system from \eqref{prandtl}.

\subsubsection{The Crocco Transform} A key difficulty in the analysis of the Prandtl system \eqref{prandtl} is the term $v\p_{y}u$. Indeed, notice that by the divergence--free condition,
\begin{equation*}
v(t,x,y)=-\int_{0}^{y}(\p_{x}u)(t,x,s)ds,
\end{equation*}
which costs \emph{one derivative} on $u$ in the $x$ direction and this loss of derivative cannot be easily recovered from the equation due to the lack of dissipation in $x$. %For steady states of \eqref{prandtl}, we can use the von Mises transform to eliminate this term (see \cite{OS99}). However, this method is not applicable in the dynamic case due to inherent structure of the equation is broken. 

To overcome this difficulty, we use the Crocco transform, which was originally introduced by Crocco \cite{C41} in 1941 for the study of laminar gases. Remarkably, the Crocco transform converts the original system \eqref{prandtl} into a quasilinear ultraparabolic equation. 
\begin{definition}[\bf The Crocco Transform]
\label{defcrocco}

Assume that $u(t,x,y)$ defined in $(0,T)\times(0,X)\times[0,\infty)$ is a continuous, strictly monotone (in the $y$ variable) function with
\begin{equation}\label{ybdc}
    u(t,x,0)=0,\quad \lim_{y\to\infty}u(t,x,y)=1.
\end{equation}
Then the following coordinate transform is a well-defined continuous bijection from $(0,T)\times(0,X)\times[0,\infty)$ onto $(0,T)\times(0,X)\times[0,1)$:
\begin{equation}\label{croccotrans}
\tau=t,\quad
\xi=x,\quad
\eta=u(t,x,y).
\end{equation}
Furthermore, under the assumption that $u(t,x,y)$ is continuously differentiable, we define a new unknown
\begin{equation}
    w(\tau,\xi,\eta):=(\p_{y}u)(t,x,y).
\end{equation}
\end{definition}
We have the following basic properties of the Crocco transform.
\begin{lemma}\label{basiccrocco} Assume that $u(t,x,y)$ is continuously differentiable, and satisfies the strict monotonicity condition $\p_{y}u>0$ on $(0,T)\times(0,X)\times[0,\infty)$ and the boundary condition \eqref{ybdc}. Then the following statements hold:
\begin{itemize}
    \item[(1)]$w(\tau,\xi,\eta)$ is continuous and positive in $(0,T)\times(0,X)\times[0,1)$, and we can represent $u(t,x,y)$ on $(0,T)\times(0,X)\times[0,\infty)$ by
    \begin{equation}\label{invertcrocco}
        y=\int_{0}^{u(t,x,y)}\frac{d\eta}{w(t,x,\eta)}.
    \end{equation}
    \item[(2)] Jacobian matrices of the Crocco transform are 
    \begin{equation}
    \frac{\partial(\tau,\xi,\eta)}{\partial(t,x,y)}=
    \begin{pmatrix}
        1&0&0\\
        0&1&0\\
        \p_{t}u&\p_{x}u&\p_{y}u
    \end{pmatrix},\quad
    \frac{\partial(t,x,y)}{\partial(\tau,\xi,\eta)}=
    \begin{pmatrix}
        1&0&0\\
        0&1&0\\
        -\frac{\p_{t}u}{\p_{y}u}&-\frac{\p_{x}u}{\p_{y}u}&\frac{1}{\p_{y}u}
    \end{pmatrix}.
    \end{equation}
In particular, if $u(t,x,y)$ is twice continuously differentiable, then
\begin{equation}
    \p_{\tau}w=\p_{ty}u-\frac{\p_{t}u}{\p_{y}u}\p_{yy}u,\ \p_{\xi}w=\p_{xy}u-\frac{\p_{x}u}{\p_{y}u}\p_{yy}u,\ \p_{\eta}w=\frac{\p_{yy}u}{\p_{y}u}.
\end{equation}
\end{itemize}
\end{lemma}
Under the \emph{strict monotonicity condition} ($\p_{y}u>0$)  and suitable regularity assumptions, the unsteady Prandtl system \eqref{prandtl} is equivalent to the following quasilinear initial--boundary problem for $w=\p_{y}u$ in  Crocco coordinates for $(\tau,\xi,\eta)\in(0,T)\times(0,X)\times (0,1)$,
\begin{equation}\label{prandtlcroccosec1}
\begin{cases}
\p_{\tau}w+\eta\p_{\xi}w-w^{2}\p_{\eta}^{2}w=0,\\
w|_{\tau=0}=w_{0}(\xi,\eta),\quad w|_{\xi=0}=w_{1}(\tau,\eta),\quad w|_{\eta=1}=w\p_{\eta}w|_{\eta=0}=0.
\end{cases}
\end{equation}
In the above, $w_{0},w_{1}$ are defined by $u_{0}, u_{1}$ in \eqref{prandtl} as
\begin{equation}
w_{0}(\xi,\eta)=\p_{y}u_{0}(\xi,y^{\ast}),\,\,
w_{1}(\tau,\eta)=\p_{y}u_{1}(\tau,y^{\ast\ast}),\,\,
u_{0}(\xi,y^{\ast})=u_{1}(\tau,y^{\ast\ast})=\eta.
\end{equation}
We remark that the boundary condition on $\{\eta=0\}$ in \eqref{prandtlcroccosec1} is equivalent to the usual Neumann boundary condition $\p_{\eta}w|_{\eta=0}=0$ whenever we consider positive solutions of \eqref{prandtlcroccosec1}, corresponding to the strict monotonicity assumption.

%In order to understand the origin Prandtl system \eqref{prandtl}, one may turn to research the vorticity equation for $w=\p_{y}u$ in the new coordinate \eqref{prandtlcroccosec1}. We will see that \eqref{prandtlcroccosec1} satisfies a \emph{nonlinear hypoelliptic structure} (see the next subsection).

\begin{remark} The strict monotonicity condition $\partial_yu>0$ is not just a technical condition needed for defining the Crocco transform to eliminate loss of derivatives. In fact, the unsteady Prandtl system may become severely ill-posed due to loss of derivatives in Sobolev spaces without this condition, and the solution may develop a finite-time singularity, see for example \cite{GD10,GN11,LY17,GM15,EE97,IVF17,CGIM22} and references therein for more discussions on this important point.
 
 Heuristically, the monotonicity condition stabilizes \eqref{prandtl} in the following two senses:
 (1) The ``vorticity field" has a definite sign near the boundary of the viscous fluid;
 (2) The shear stress never vanishes near the boundary of the viscous fluid.
 
\end{remark}

%We emphasize that since Prandtl's equation is originally formulated in physical coordinate, it is more nature to impose assumptions on data within the physical coordinate system \eqref{prandtl}, while we need the transformed system \eqref{prandtl} to intermediate.

%For the first main result, we state the up--to--physical boundary smoothness effect  of the weak solution to \eqref{prandtl}, with weak initial regularity assumptions. We point that we can obtain the smoothness effect until the terminal time $t=T$, it is indeed natural since dynamic Prandtl system \eqref{prandtl} is a forward evolution equation.

\subsubsection{Up-to-boundary regularity of local weak solutions}
We start with regularity results for local weak solutions.
\begin{theorem}[\bf Up--to--boundary Smoothing Effect]
\label{mainone}
Fix $T>0$ and $X>0$. Let $u(t,x,y)$ be a function defined in $(0,T]\times(0,X)\times[0,\infty)$ satisfying \eqref{ybdc} and the following assumptions.
\begin{itemize}

  \item[(I)] (Regularity and Monotonicity) $u\in W^{1,1}_{\rm loc}((0,T]\times(0,X)\times[0,\infty))$ and is continuous on $(0,T]\times(0,X)\times[0,\infty)$ with $\ \p_{y}u\in W^{1,1}_{\rm loc}((0,T]\times(0,X)\times[0,\infty))$, and $0<c_D\leq \partial_yu\leq C_D$ on any $D\Subset(0,T]\times(0,X)\times[0,\infty)$ for suitable positive constants $c_D<C_D$.  In addition, assume that
    \begin{equation*}
        \bigg\{\p_{y}\left(\frac{\p_{t}u}{\p_{y}u}\right),\ \p_{y}\left(\frac{\p_{x}u}{\p_{y}u}\right),\ \p_{y}^{3}u\bigg\}\subset L^{1}_{loc}((0,T]\times(0,X)\times[0,\infty)).
        \end{equation*}
    \item[(II)] Set 
    \begin{equation}\label{introv}
        v(t,x,y):=-\int_0^y\partial_xu(t,x,y')dy'. 
    \end{equation}
   Then $u, v$ satisfy the Prandtl equation in the sense of distributions
   \begin{equation}\label{prandtlweak}
\begin{cases}
\p_{t}u+u\p_{x}u+v\p_{y}u-\p_{y}^{2}u=0,\\
\p_{x}u+\p_{y}v=0,\\
u|_{y=0}=v|_{y=0}=0,\
\end{cases}
\end{equation}
for $(t,x,y)\in(0,T)\times(0,X)\times(0,\infty)$.

\end{itemize}
Then we have $u(t,x,y)\in C^{\infty}((0,T]\times(0,X)\times[0,\infty))$. More precisely, for any $D\Subset (0,T]\times(0,X)\times[0,\infty)$, and for any $m\in\mathbb{N}$, we have
\begin{equation}\label{smooeffesti}
    \|(u,v)\|_{C^{m}(D)}\lesssim 1,%j%,\quad \text{for any}\,\, D\Subset (0,T]\times(0,X)\times[0,\infty),
\end{equation}
where the implied constants depend on quantitative versions of the bounds appearing in the regularity assumptions in (I).
\end{theorem}

\begin{remark}The regularity assumptions in (I) are natural under the Crocco transform of \eqref{prandtl}. Indeed, to prove the up--to--boundary regularity in the physical coordinate system \eqref{prandtl}, the crucial step is to obtain the up--to--boundary regularity in the transformed system \eqref{prandtlcroccosec1}. Our assumptions in Theorem \ref{mainone} translate to the assumptions that the Crocco transform is well-defined, and the solution after transformation satisfies
    \begin{equation}\label{regularassume}
        \begin{cases}
            \big\{\p_{\tau}w,\p_{\xi}w,\p_{\eta}w,\p_{\eta}^{2}w\big\}\subseteq L^{1}_{loc},\\
            w(\tau,\xi,\eta)\sim 1\ \text{on every compact subdomain,}\\
            \p_{\eta}w|_{\eta=0}=0,\ \text{in the sense of trace.}
        \end{cases}
    \end{equation}
The conditions in \eqref{regularassume} are the natural regularity assumptions on weak solutions to \eqref{prandtlcroccosec1}.

We remark also that the $L^1$ type estimates on derivatives of $u$ are consequences of bounds obtained from a priori estimates of \eqref{prandtlcroccosec1} when the initial and boundary conditions are just regular enough for the existence of weak solutions (see Assumptions \ref{dataassum3} below for the precise conditions). In this case, the gain of higher--order regularity is more challenging. Even for elliptic equations, similar difficulties of working with $L^1$ (as opposed to more commonly assumed $L^2$) type weak solutions exist, see for example Brezis \cite{brezis2008conjecture} and Serrin \cite{serrin1964pathological}.

\end{remark}

The proof of Theorem \ref{mainone} relies crucially on the corresponding results for the transformed Prandtl equation \eqref{prandtlcroccosec1}, as follows.
\begin{theorem}\label{th:smo1}
    Fix $T>0$ and $X>0$. Let $w(\tau,\xi,\eta)$ be a function defined in $(0,T]\times(0,X)\times[0,1)$ satisfying the following assumptions.
    \begin{itemize}
        \item [(I)] $w\in W^{1,1}_{\rm loc}\cap W^{2,1}_{\eta,\,{\rm loc}}((0,T]\times(0,X)\times[0,1))$ solves the first equation in \eqref{prandtlcroccosec1} in the sense of  distributions. Furthermore, $w$ satisfies the Neumann boundary condition of \eqref{prandtlcroccosec1} on $\{\eta=0\}$ in the trace sense.

        \item[(II)] $0<c_\Omega\leq w\leq C_\Omega$ on any $\Omega\Subset (0,T]\times(0,X)\times[0,1)$ for suitable positive constants $c_\Omega<C_\Omega$.
    \end{itemize}
Then we have $w(\tau,\xi,\eta)\in C^{\infty}((0,T]\times(0,X)\times[0,1))$. More precisely, for any $\Omega\Subset(0,T]\times(0,X)\times[0,1)$, and any $m\in\mathbb{N}$, we have
\begin{equation}
    \|w\|_{C^{m}(\Omega)}\lesssim_{m,\Omega}1,%\,\,\,\, \text{for any}\ \Omega\Subset(0,T]\times(0,X)\times[0,1),
\end{equation}
where the implied constants depend on quantitative versions of the bounds appearing in the above regularity assumptions.
\end{theorem}

\begin{remark}\label{re:impreg1}
    As mentioned in the introduction, the up-to-boundary regularity for weak solutions to \eqref{prandtlcroccosec1} is essential for obtaining regularity of the dynamical Prandtl equation \eqref{prandtl} in the physical variables, even in the interior. This can be seen from the inverse Crocco transform \eqref{invertcrocco}, which is nonlocal and depends crucially on the behavior of $w$ near $\eta=0$. Indeed, with only interior regularity results, the best we can obtain for $w$ (from scaling considerations) would be that $\partial_\xi w= O(\eta^{-3})$ as $\eta \to0+$. Then the transform $(t,x,u)\to y$ given by \eqref{invertcrocco} can be bounded at best by
    \begin{equation}\label{re:impreg2}
        |\partial_xy(t,x,u)|=\bigg|\int_0^u\frac{\partial_\xi w(t,x,\eta)}{w^2(t,x,\eta)}d\eta\bigg|\lesssim \bigg|\int_0^u \eta^{-3}d\eta\bigg|,
    \end{equation}
    which diverges. 
\end{remark}

\begin{comment}
We point that The various regularity theories established in Section \ref{fundamental} to Section \ref{higher} are mainly to solve the regularity issue near the \emph{boundary} $\{\eta=0\}$, and these established tools are adoptable to the interior estimate without extra difficulty. Indeed, one may obtain the interior smoothness effect in $(0,T)\times(0,X)\times(0,1)$ of \eqref{prandtlcroccosec1} by a relative simpler way which is sketched below:
\begin{itemize}
    \item[Step 1:] By \cite{XZZ24}, we have $w(\tau,\xi,\eta)$ is locally H\"older continuous in the interior $(0,T)\times(0,X)\times(0,1)$; Next, combining the continuity of $w(\tau,\xi,\eta)$ with the interior estimate in \cite{BC96b}, we have $\p_{\eta}^{2}w\in L^{p}_{loc}$ in $(0,T)\times(0,X)\times(0,1)$ for all $p<\infty$.
    \item[Step 2:] Bootstrapping the regularity by slight modifying the argument in \cite{HS20}, or applying the method of hypoelliptic regularity improvement stated in Section \ref{higher} (the explicit expression of $\mathcal{L}_{0}$ in whole space simplifies several key estimates).
\end{itemize}
For self--contained, we will prove both boundary regularity and interior regularity rigorously during the proof of Theorem \ref{mainone}.
\end{comment}
Theorem \ref{mainone} provides the main tool to prove the global well-posedness and regularity of classical solutions. The proof of Theorem \ref{mainone}, based on the proof of Theorem \ref{th:smo1}, is given in Section \ref{potmainone}. 

\subsubsection{Global classical solutions}
Our second important goal is to establish global-in-$(t,x)$ well-posedness (in Sobolev spaces) of classical solutions to \eqref{prandtl} for sufficiently regular initial and boundary data that satisfy suitable monotonicity conditions, and thereby provide a positive answer to the fundamental question raised by Oleinik and Samokhin \cite{OS99}.

We first state our main assumptions on boundary and initial data.
\begin{assumptions}\label{dataassum}
Assume the initial data $u_{0}(x,y)$ and $u_{1}(t,y)$ in \eqref{prandtl} satisfy the following hypotheses:
\begin{itemize}
    \item[(H1)] (Asymptotic behavior) $u_{0}, u_{1}$ are continuous on $[0,\infty)\times[0,\infty)$,  strictly monotone in $y$,  and satisfy:
    \begin{equation}\label{matchingrate}
    \begin{split}
        &u_{i}|_{y=0}=0,\,\, \lim_{y\to\infty}u_{i}=1,\quad i\in\{0,1\},\\
        &F(u_{0})\lesssim_{X} \p_{y}u_{0}\lesssim_{X}F(u_{0}),\ \forall\, X>0,\,\,(x,y)\in[0,X]\times[0,\infty),\\
        &F(u_{1})\lesssim_{T} \p_{y}u_{1}\lesssim_{T}F(u_{1}),\ \forall\, T>0,\,\, (t,y)\in[0,T]\times[0,\infty),
    \end{split}
    \end{equation}
    for some matching rate $F(u)\in\Big\{(1-u)\log^{\frac{1}{2}}(\frac{10}{1-u}),(1-u)^{m}\,\,{\rm with}\,\,m\in[1,3/2]\Big\}.$
 %   In addition, we assume the ``{\it finiteness of %displacement}" condition:
%    \begin{equation}\label{displacement}
 %   \begin{split}
%    &\int_{0}^{\infty}(1-u_{0}(x,y))dy\lesssim_{X}1,\ %\forall x\in[0,X],\\
%    &\int_{0}^{\infty}(1-u_{1}(t,y))dy\lesssim_{T}1,\ %\forall
%t\in[0,T].
%\end{split}
%\end{equation}
    \item[(H2)] (Regularity and compatibility condition) There exists a function $u^{(0)}(t,x,y)$ defined on $[0,\infty)^3$ such that (i) $u^{(0)}|_{t=0}\equiv u_0$, $u^{(0)}|_{x=0}=u_1$; (ii) $u^{(0)}$ satisfies the matching rate $$F(u^{(0)})\lesssim_{T,X} \p_{y}u^{(0)}\lesssim_{T,X}F(u^{(0)}),$$ for all $(t,x,y)\in[0,T]\times[0,X]\times[0,\infty)$ and any $T>0, X>0$; (iii) for any integers $\alpha_1, \alpha_2, \alpha_3\ge0$ with $2\alpha_1+2\alpha_2+\alpha_3\leq 9$, 
    \begin{equation}\label{regular}
\partial_t^{\alpha_1}\partial_x^{\alpha_2}\partial_y^{\alpha_3}u^{(0)}(t,x,y) \in C([0,\infty)^{3}),
    \end{equation}
    and the compatibility condition on $\{t=0\}$ and $\{x=0\}$ is verified, in the sense that for any integers $\alpha_1, \alpha_2, \alpha_3\ge0$ with $2\alpha_1+2\alpha_2+\alpha_3\leq9$, $\partial_t^{\alpha_1}\partial_x^{\alpha_2}\partial_y^{\alpha_3}u^{(0)}$ satisfies the differential constraints imposed by the equation \eqref{prandtl} on $\{t=0\}$, $\{x=0\}$, $\{y=0\}$.

    In addition, we assume the following bounds for any $T>0, X>0$,
    \begin{equation}\label{oneorderassume}
            (B1):\bigg|\p_{y}\Big(\frac{\p_{t}u^{(0)}(t,x,y)}{\p_{y}u^{(0)}(t,x,y)}\Big)\bigg|+\bigg|\p_{y}\Big(\frac{\p_{x}u^{(0)}(t,x,y)}{\p_{y}u^{(0)}(t,x,y)}\Big)\bigg|+\left|\p_{y}\left(\frac{\p_{y}^{2}u^{(0)}(t,x,y)}{\p_{y}u^{(0)}(t,x,y)}\right)\right|\lesssim_{T,X}1,\\
    \end{equation}
for all $(t,x,y)\in[0,T]\times[0,X]\times[0,\infty)$.
\begin{equation}\label{higherassume}
\begin{split}
(B2):&\sum_{2\leq\alpha_{1}+\alpha_{2}+\alpha_{3}\leq 4}(1-u^{(0)}(t,x,y))^{\frac{1}{1000}}\\
&\qquad\qquad\times\left|(\p_{y}u^{(0)})^{2\alpha_{3}-1}\cdot (\mathcal{T}_{1}^{[u^{(0)}]})^{\alpha_{1}}(\mathcal{T}_{2}^{[u^{(0)}]})^{\alpha_{2}}(\mathcal{T}_{3}^{[u^{(0)}]})^{2\alpha_{3}}\p_{y}u^{(0)}(t,x,y)\right|\lesssim_{T,X} 1,
\end{split}
\end{equation}
for all $(t,x,y)\in[0,T]\times[0,X]\times[0,\infty)$.

In the above, the directional derivatives $\mathcal{T}_{i}^{[h]}$ are defined as
\begin{equation}\label{eq:goodderivative}
\mathcal{T}_{1}^{[h]}:=\p_{t}-\frac{\p_{t}h}{\p_{y}h}\p_{y},\quad\mathcal{T}_{2}^{[h]}:=\p_{x}-\frac{\p_{x}h}{\p_{y}h}\p_{y},\quad\mathcal{T}_{3}^{[h]}:=\frac{1}{\p_{y}h}\cdot\p_{y}
\end{equation}
\end{itemize}
\end{assumptions}
\begin{remark}
    We note that for solutions $u$ to \eqref{prandtl} and in Crocco coordinates $$\mathcal{T}^{[u]}_{1}=\p_{\tau},\quad\mathcal{T}^{[u]}_{2}=\p_{\xi},\quad\mathcal{T}^{[u]}_{3}=\p_{\eta}.$$ 
\end{remark}
\begin{remark}
As is well known, suitable regularity and compatibility conditions on $u_0, u_1$ are necessary to construct classical solutions to \eqref{prandtl}. In our convention, the differential constraints imposed by \eqref{prandtl} on $u^{(0)}$ are a system of differential identities {\it assuming} that $u^{(0)}$ satisfies \eqref{prandtl}, restricted to the boundary. Using Whitney's extension theory, our assumptions can be made explicit, see Remark \ref{descriptioncompati}.
%See Appendix \ref{compatiappendix} for the more detailed constraints required for $u_0, u_1$ so that such $u^{(0)}$ can be chosen. 
Our formulation of the regularity and compatibility conditions can be relaxed.
We adopt these assumptions for the sake of clarity and simplicity. Our main goal is to obtain global classical solutions (all terms appearing in the equation \eqref{prandtl} are continuous functions in the closure of the domain) and we have not tried to optimize the regularity assumptions on the initial and inflow data. 
\end{remark}

\begin{remark}
We remark that the asymptotic conditions \eqref{matchingrate} allow $u$ to converge to the outer flow $U\equiv 1$ with three distinct and natural rates: algebraic, exponential and Gaussian. Indeed, if $1-u\sim y^{-\alpha}$ for some $\alpha\ge2$ and $\partial_yu\sim y^{-\alpha-1}$, then we have $\partial_yu\sim (1-u)^{\frac{\alpha+1}{\alpha}}$ where $\frac{\alpha+1}{\alpha}\leq 3/2$ and \eqref{matchingrate} are satisfied. Similar computations can be made for the case when $1-u\sim e^{-cy}$, $\partial_yu\sim e^{-cy}$ and the case when $1-u\sim e^{-cy^2}$ and $\partial_yu\sim ye^{-cy^2}$. In particular, crucially, the Blasius profile is included, which is an important self-similar steady state of \eqref{prandtl} and plays an essential role in the study of asymptotics as $x\to\infty$ for the steady Prandtl equation.

The restriction in \eqref{matchingrate} that $m\leq 3/2$ (equivalently $\alpha\ge2$ above) is imposed solely for simplicity and is not essential. We expect that our main conclusions hold provided $1\leq m<2$.% or $m=2$ with an additional ``finiteness of displacement" condition of the type $\int_0^\infty (1-h(s,y))\, dy\lesssim 1$ for $h\in\{u_0, u_1\}$ with the corresponding $s\in\{t,x\}$. 
\end{remark}

\begin{remark}
 
    The assumption (B2) is a natural extension of (B1) for higher order derivatives, except for the factor $(1-u^{(0)})^{\frac{1}{1000}}$. This factor is crucial to incorporate the Blasius profile.  We note also that the specific power $\frac{1}{1000}$ is not important and is fixed here for concreteness. 
\end{remark}

The assumptions for data $(u_{0},u_{1})$ in Assumptions \ref{dataassum} have concise formulations for the transformed Prandtl system
\eqref{prandtlcroccosec1}, which we present as follows.
\begin{assumptions}\label{dataassumecrocco}
    Assume the initial data $w_{0}(\xi,\eta)$ and $w_{1}(\tau,\eta)$ in \eqref{prandtlcroccosec1} satisfy the following hypotheses:
\begin{itemize}
    \item[(H1)] (Vanishing rate) $w_{0},w_{1}$ are positive on $[0,\infty)\times[0,1)$, and $w_{i}|_{\eta=1}=0$ with the vanishing rate:
    \begin{equation}\label{matchingrate2}
    \begin{split}
        &F(\eta)\lesssim_{X} w_{0}\lesssim_{X}F(\eta),\ \forall\, X>0,\,\,(\xi,\eta)\in[0,X]\times[0,1),\\
        &F(\eta)\lesssim_{T} w_{1}\lesssim_{T}F(\eta),\ \forall\, T>0,\,\, (\tau,\eta)\in[0,T]\times[0,1),
    \end{split}
    \end{equation}
    for some vanishing rate $F(\eta)\in\big\{(1-\eta)\log^{\frac{1}{2}}\left(\frac{10}{1-\eta}\right),(1-\eta)^{m}\,\,{\rm with}\,\,m\in[1,3/2]\big\}$.
    %In addition, we assume the ``{\it finiteness of %displacement}" condition:
    %\begin{equation}\label{displacement2}
    %\begin{split}
    %&\int_{0}^{1}\frac{1-\eta}{w_{0}}d\eta\lesssim_{X}1,\ %\forall \xi\in[0,X],\\
    %&\int_{0}^{1}\frac{1-\eta}{w_{1}}d\eta\lesssim_{T}1,\ %\forall \tau\in[0,T].
%\end{split}
%\end{equation}
    \item[(H2)] (Regularity and compatibility condition) There exists a function $w^{(0)}(\tau,\xi,\eta)$ defined on $[0,\infty)^2\times[0,1)$ such that (i) $w^{(0)}|_{\tau=0}\equiv w_0$, $w^{(0)}|_{\xi=0}=w_1$; (ii) $w^{(0)}$ satisfies the vanishing rate $$F(\eta)\lesssim_{T,X} w^{(0)}\lesssim_{T,X}F(\eta),$$ for all $(\tau,\xi,\eta)\in[0,T]\times[0,X]\times[0,1)$ and any $T>0, X>0$; and (iii)  for any integers $\alpha_1, \alpha_2, \alpha_3\ge0$ with $2\alpha_1+2\alpha_2+\alpha_3\leq 8$, 
    \begin{equation}\label{regular2}
\partial_\tau^{\alpha_1}\partial_\xi^{\alpha_2}\partial_\eta^{\alpha_3}w^{(0)}(\tau,\xi,\eta) \in C([0,\infty)^{2}\times[0,1)),
    \end{equation}
    and the compatibility condition on $\{\tau=0\}$ and $\{\xi=0\}$ is verified, in the sense that  for any integers $\alpha_1, \alpha_2, \alpha_3\ge0$ with $2\alpha_1+2\alpha_2+\alpha_3\leq8$, $\partial_\tau^{\alpha_1}\partial_\xi^{\alpha_2}\partial_\eta^{\alpha_3}w^{(0)}$ satisfies the differential constraints imposed by the equation \eqref{prandtlcroccosec1} at $\xi=0$, $\tau=0$, and $\eta=0$.  
    
    In addition, we have the following bounds for any $T>0, X>0$,
    \begin{equation}\label{oneorderassume2}
            (B1):\left|\frac{\p_{\tau}w^{(0)}(\tau,\xi,\eta)}{w^{(0)}(\tau,\xi,\eta)}\right| + \left|\frac{\p_{\xi}w^{(0)}(\tau,\xi,\eta)}{w^{(0)}(\tau,\xi,\eta)}\right|+\left|w^{(0)}(\tau,\xi,\eta)\p_{\eta}^{2}w^{(0)}(\tau,\xi,\eta)\right|   \lesssim_{T,X}1,
    \end{equation}
for all $(\tau,\xi,\eta)\in[0,T]\times[0,X]\times[0,1)$;
\begin{equation}\label{higherassume2}
(B2):\sum_{2\leq\alpha_{1}+\alpha_{2}+\alpha_{3}\leq 4}(1-\eta)^{\frac{1}{1000}}\left|(w^{(0)})^{2\alpha_{3}-1}\p_{\tau}^{\alpha_{1}}\p_{\xi}^{\alpha_{2}}\p_{\eta}^{2\alpha_{3}}w^{(0)}(\tau,\xi,\eta)\right|\lesssim_{T,X}1,\\
\end{equation}
for all $(\tau,\xi,\eta)\in[0,T]\times[0,X]\times[0,1)$.
\end{itemize}
\end{assumptions}
With the above assumptions, we are now ready to present our main result on the global well-posedness of the dynamical Prandtl equation \eqref{prandtl}.

\begin{theorem}[\bf Global Well-posedness for Classical Solutions]
\label{maintwo}
Under Assumption \ref{dataassum}, the dynamical Prandtl equation \eqref{prandtl} admits a unique \emph{global} classical solution $u(t,x,y)$ in $[0,+\infty)^{3}$ with the following properties:
\begin{itemize}
    \item [(P1)] $u$ is continuously differentiable in $[0,+\infty)^{3}$, and satisfies the bounds \eqref{globalbound}, and the initial--boundary conditions in \eqref{prandtl} in the classical sense. Furthermore, $\partial_tu, \partial_xu$, $\p_{y}u, \partial_y^2u$ are continuous in $[0,+\infty)^{3}$, and together with the associated $v$ as in \eqref{introv}, $u$ satisfies the Prandtl equation \eqref{prandtl} in the classical sense. 
    \item[(P2)] $u(t,x,y)$ is strictly monotone in $y$ and satisfies for any $T>0, X>0$ and suitable $c_{T,X}\in(0, 1), \,C_{T,X}\in(1,\infty)$ the pointwise bounds
    \begin{equation}\label{solutionmatching}
       c_{T,X}F\big(u(t,x,y)\big)\leq \p_{y}u(t,x,y)\leq C_{T,X} F\big(u(t,x,y)\big),\quad \forall (t,x,y)\in [0,T]\times[0,X]\times[0,\infty).\\
    \end{equation}
    \item[(P3)] $u(t,x,y)$ is smooth up to the boundary $y=0$ for $t>0$ and $x>0$, that is
    \begin{equation*}
            u(t,x,y)\in C^{\infty}\left((0,\infty)\times(0,\infty)\times[0,+\infty)\right).
    \end{equation*}
    \begin{comment}
    and satisfying the uniform estimates up to $t=0$ and $x=0$:
    $$\p_{t}u,\p_{x}u,\p_{y}u,\p_{t}\p_{y}u,\p_{x}\p_{y}u,\p_{y}^{2}u,\p_{y}^{3}u\in C^{k-2,\alpha}(\overline{D})$$
    for some $\alpha>0$, satisfying Prandtl equation\eqref{prandtl} in the classical way. 
    \item[(P3)] For any  $0<T,X<\infty$ we have
    \begin{equation}\label{physicalbv}
        \p_{t}u,\p_{x}u\in L^{\infty}_{t}L^{1}_{x,y}(\Omega_{T,X}).
    \end{equation}
\end{comment}
\end{itemize}
\end{theorem}

The proof of Theorem \ref{maintwo} is the main contribution of our paper. The proof relies crucially on the up-to-boundary regularity of local weak solutions (Theorem \ref{mainone}). To prove Theorem \ref{maintwo}, we also establish a self-contained proof of local-in-time existence of classical solutions to \eqref{prandtl} and to the associated approximate system used to construct the solution. This framework seems efficient in obtaining quantitative weighted estimates which are useful in capturing the behavior of solutions in the limit $y\to\infty$. 

%By virtue of an illustrating analysis in the previous subsection (see Figure \ref{globalsolfigure}), it is suggested to start with establishing the global weak solution theory and local classical solution theory on \eqref{prandtl} in suitable functional frameworks, respectively. 

We now present the corresponding global well-posedness result for the transformed Prandtl equation \eqref{prandtlcroccosec1}.
\begin{theorem}\label{thm:gcs}
  Under Assumption \ref{dataassumecrocco}, the transformed dynamical Prandtl equation \eqref{prandtlcroccosec1} admits a unique \emph{global} classical solution $w(\tau,\xi,\eta)$ in $[0,+\infty)\times[0,+\infty)\times[0,1)$ with the following properties.
\begin{itemize}
    \item [(P1)] $w(\tau,\xi,\eta)$ is  continuous on $[0,+\infty)\times[0,+\infty)\times[0,1)$, satisfies the bounds \eqref{globalboundcrocco}, and assumes the initial--boundary conditions in \eqref{prandtlcroccosec1} in the classical sense. Furthermore, $\partial_\tau w, \partial_\xi w, w^{2}\partial_\eta^2w$ are continuous in $[0,+\infty)\times[0,+\infty)\times[0,1)$, and the transformed Prandtl equation \eqref{prandtlcroccosec1} is satisfied in the classical sense. 
    \item[(P2)] $w(\tau,\xi,\eta)$ is strictly positive in $[0,\infty)\times[0,\infty)\times[0,1)$, and satisfies for any $T>0, X>0$ and suitable $c_{T,X}\in(0, 1), \,C_{T,X}\in(1,\infty)$ the pointwise estimate
    \begin{equation}\label{solutionmatchingcrocco}
      c_{T,X}F(\eta)\leq  w(\tau,\xi,\eta)\leq C_{T,X} F(\eta),\ \forall (\tau,\xi,\eta)\in [0,T]\times[0,X]\times[0,1).\\
    \end{equation}
    \item[(P3)] $w(\tau,\xi,\eta)$ is smooth up to the boundary $\eta=0$ for $\tau>0$ and $\xi>0$, that is
    \begin{equation*}
            w(\tau,\xi,\eta)\in C^{\infty}\left((0,\infty)\times(0,\infty)\times[0,1)\right).
    \end{equation*} 
\end{itemize}
\end{theorem}

\begin{remark}
We record the following weighted bounds that both the global weak solutions and classical solutions are required to satisfy, in order to guarantee uniqueness. In the Crocco coordinates for any $T>0, X>0$, we have
\begin{equation}\label{globalboundcrocco}
\begin{cases}\displaystyle
    (1-\eta)^{-\frac{1}{2}}\log^{-\frac{3}{2}}\left(\frac{10}{1-\eta}\right)\cdot\p_{\eta}w\in L^{2}([0,T]\times[0,X]\times[0,1)),\\
    \displaystyle\frac{(\p_{\tau},\p_{\xi})w}{w^{2}}\cdot(1-\eta)\in L^{\infty}_{\tau}L^{1}_{\xi,\eta}([0,T]\times[0,X]\times[0,1)),\\
    \p_{\eta\eta}w\cdot(1-\eta)\in L^{\infty}_{\tau}L^{1}_{\xi,\eta}([0,T]\times[0,X]\times[0,1)).
\end{cases}
\end{equation}
In the physical coordinates (when $y\to \infty$) for any $T>0, X>0$ the corresponding bounds are
\begin{equation}\label{globalbound}
    \begin{cases}\displaystyle
&\p_{t}u,\p_{x}u\in L^{\infty}_{t}L^{1}_{x,y}([0,T]\times[0,X]\times[0,\infty)),\\
&\displaystyle\frac{(\mathcal{T}^{[u]}_{1},\mathcal{T}^{[u]}_{2})\p_{y}u}{\p_{y}u}\cdot(1-u)\in L^{\infty}_{t}L^{1}_{x,y}([0,T]\times[0,X]\times[0,\infty)),\\
        &\displaystyle\frac{\p_{y}^{2}u}{(\p_{y}u)^{\frac{1}{2}}}\cdot(1-u)^{-\frac{1}{2}}\log^{-\frac{3}{2}}\left(\frac{10}{1-u}\right)\in L^{2}_{t,x,y}([0,T]\times[0,X]\times[0,\infty)),\\
       &\displaystyle\frac{\p_{y}^{2}u}{\p_{y}u}\cdot(1-u)^{\frac{1}{2}}\log^{-1}\left(\frac{10}{1-u}\right)\in L^{\infty}_{t}L^{2}_{x,y}([0,T]\times[0,X]\times[0,\infty)),\\
        &\displaystyle\frac{\p_{y}^{3}u}{\p_{y}u}\cdot (1-u)\log^{-2}\left(\frac{10}{1-u}\right)\in L^{\infty}_{t}L^{1}_{x,y}([0,T]\times[0,X]\times[0,\infty)).
    \end{cases}
    \end{equation}
\end{remark}

\begin{remark}It is an interesting question to investigate more refined behavior of $\partial_t\partial_yu, \partial_x\partial_yu$ (and correspondingly $\partial_\tau w, \partial_\xi w$) as $y\to\infty$ (and correspondingly as $\eta\to1-$). Under our general assumptions \eqref{matchingrate}, this seems to be a difficult question. See however Lemma \ref{bootstrap} below which shows that, under natural conditions, we can expect $\partial_\tau w, \partial_\xi w$ to be controlled by $w$ at least for a short time. If we impose the stronger assumption that $w\sim (1-\eta)$, then similar results can be proved quite straightforwardly using the smoothing estimates (see Theorem \ref{th:smo1} and its interior version) and rescaling arguments. The case of a general polynomial matching rate seems also manageable but the analysis is more complicated. The case of a Gaussian matching rate is more challenging. We plan to address this problem in a future work. 
    
\end{remark}

\subsubsection{Global existence, uniqueness and regularity of weak solutions}
Finally, we turn to the global weak solution of \eqref{prandtl}. The theory of weak solutions is of independent interest, since the required regularity and compatibility conditions on the initial and boundary data $u_0, u_1$ are much weaker than those for the classical solutions and since it provides the natural setting for the smoothing estimates to apply. 

We first state our main assumptions on boundary and initial data needed for the global well-posedness theory for weak solutions.
\begin{assumptions}\label{dataassum3}
Assume the initial data $u_{0}(x,y)$ and inflow data $u_{1}(t,y)$ in \eqref{prandtl} satisfy the following hypotheses:
\begin{itemize}
    \item[(H1)] (Asymptotic behavior) $u_{0}, u_{1}$ are continuous on $[0,\infty)\times[0,\infty)$,  strictly monotone in $y$,  and satisfy:
    \begin{equation}\label{matchingrate3}
    \begin{split}
        &u_{i}|_{y=0}=0,\,\, \lim_{y\to\infty}u_{i}=1,\quad i\in\{0,1\},\\
        &(1-u_{0})^{3/2}\lesssim_{X} \p_{y}u_{0}\lesssim_{X}(1-u_{0})\log^{\frac{1}{2}}\left(\frac{10}{1-u_{0}}\right),\ \forall\, X>0,\,\,(x,y)\in[0,X]\times[0,\infty),\\
        &(1-u_{1})^{3/2}\lesssim_{T} \p_{y}u_{1}\lesssim_{T}(1-u_{1})\log^{\frac{1}{2}}\left(\frac{10}{1-u_{1}}\right),\ \forall\, T>0,\,\, (t,y)\in[0,T]\times[0,\infty).
    \end{split}
    \end{equation}
 %   In addition, we assume the ``{\it finiteness of %displacement}" condition:
%    \begin{equation}\label{displacement}
 %   \begin{split}
%    &\int_{0}^{\infty}(1-u_{0}(x,y))dy\lesssim_{X}1,\ %\forall x\in[0,X],\\
%    &\int_{0}^{\infty}(1-u_{1}(t,y))dy\lesssim_{T}1,\ %\forall
%t\in[0,T].
%\end{split}
%\end{equation}
    \item[(H2)] (Regularity and compatibility condition) There exists a function $u^{(0)}(t,x,y)$ defined on $[0,\infty)^3$ such that (i) $u^{(0)}|_{t=0}\equiv u_0$, $u^{(0)}|_{x=0}=u_1$; (ii) $u^{(0)}$ satisfies the matching rate $$(1-u^{(0)})^{3/2}\lesssim_{T,X} \p_{y}u^{(0)}\lesssim_{T,X}(1-u^{(0)})\log^{\frac{1}{2}}\left(\frac{10}{1-u^{(0)}}\right),$$ for all $(t,x,y)\in[0,T]\times[0,X]\times[0,\infty)$ and any $T>0, X>0$; and (iii) for any integers $\alpha_1, \alpha_2, \alpha_3\ge0$ with $2\alpha_1+2\alpha_2+\alpha_3\leq 3$, 
    \begin{equation}\label{regular3}
\partial_t^{\alpha_1}\partial_x^{\alpha_2}\partial_y^{\alpha_3}u^{(0)}(t,x,y) \in C([0,\infty)^{3}),
    \end{equation}
    and the compatibility condition on $\{t=0\}$ and $\{x=0\}$ is verified, in the sense that  for any integers $\alpha_1, \alpha_2, \alpha_3\ge0$ with $2\alpha_1+2\alpha_2+\alpha_3\leq3$, $\partial_t^{\alpha_1}\partial_x^{\alpha_2}\partial_y^{\alpha_3}u^{(0)}$ satisfies the differential constraints imposed by the equation \eqref{prandtl} on $\{t=0\}$, $\{x=0\}$, $\{y=0\}$.
    In addition, we have the following bounds for any $T>0, X>0$,
\begin{equation}\label{higherassume3}
\begin{split}
&\sup_{t\in[0,T]}\sum_{\alpha+\beta+\gamma\leq 1}\int_{0}^{X}\int_{0}^{\infty}(\p_{y}u^{(0)})^{2\gamma -1}|(\mathcal{T}_{1}^{[u^{(0)}]})^{\alpha}(\mathcal{T}_{2}^{[u^{(0)}]})^{\beta}(\mathcal{T}_{3}^{[u^{(0)}]})^{2\gamma}\p_{y}u^{(0)}(t,x,y)|\cdot(1-u^{(0)})dxdy\\
&\qquad\qquad\lesssim_{T,X}1,\\
&\sup_{x\in[0,X]}\sum_{\alpha+\beta+\gamma\leq 1}\int_{0}^{T}\int_{0}^{\infty}(\p_{y}u^{(0)})^{2\gamma -1}|(\mathcal{T}_{1}^{[u^{(0)}]})^{\alpha}(\mathcal{T}_{2}^{[u^{(0)}]})^{\beta}(\mathcal{T}_{3}^{[u^{(0)}]})^{2\gamma}\p_{y}u^{(0)}(t,x,y)|\cdot(1-u^{(0)})dtdy\\
&\qquad\qquad\lesssim_{T,X}1,
\end{split}
\end{equation}
In the above the indices $\alpha, \beta,\gamma$ are nonnegative integers. Recall that the directional derivatives $\mathcal{T}_{i}^{[h]}$ are defined in \eqref{eq:goodderivative}.
\end{itemize}
\end{assumptions}

\begin{remark}
    We note that the conditions \eqref{higherassume3} are considerably weaker than \eqref{oneorderassume}-\eqref{higherassume}, and hence apply to a much wider class of initial and inflow data. 
    
\end{remark}

The assumptions for data $(u_{0},u_{1})$ in Assumptions \ref{dataassum3} have concise formulations for the transformed Prandtl system
\eqref{prandtlcroccosec1}, which we present as follows.
\begin{assumptions}\label{dataassumecrocco4}
    Assume the initial data $w_{0}(\xi,\eta)$ and $w_{1}(\tau,\eta)$ in \eqref{prandtlcroccosec1} satisfy the following hypotheses:
\begin{itemize}
    \item[(H1)] (Vanishing rate) $w_{0},w_{1}$ are positive on $[0,\infty)\times[0,1)$, and $w_{i}|_{\eta=1}=0$ with the vanishing rate:
    \begin{equation}\label{matchingrate4}
    \begin{split}
        &(1-\eta)^{3/2}\lesssim_{X} w_{0}\lesssim_{X}(1-\eta)\log^{\frac{1}{2}}\left(\frac{10}{1-\eta}\right),\ \forall\, X>0,\,\,(\xi,\eta)\in[0,X]\times[0,1),\\
        &(1-\eta)^{3/2}\lesssim_{T} w_{1}\lesssim_{T}(1-\eta)\log^{\frac{1}{2}}\left(\frac{10}{1-\eta}\right),\ \forall\, T>0,\,\, (\tau,\eta)\in[0,T]\times[0,1).
    \end{split}
    \end{equation}
    %In addition, we assume the ``{\it finiteness of %displacement}" condition:
    %\begin{equation}\label{displacement2}
    %\begin{split}
    %&\int_{0}^{1}\frac{1-\eta}{w_{0}}d\eta\lesssim_{X}1,\ %\forall \xi\in[0,X],\\
    %&\int_{0}^{1}\frac{1-\eta}{w_{1}}d\eta\lesssim_{T}1,\ %\forall \tau\in[0,T].
%\end{split}
%\end{equation}
    \item[(H2)] (Regularity and compatibility condition) There exists a function $w^{(0)}(\tau,\xi,\eta)$ defined on $[0,\infty)^2\times[0,1)$ such that (i) $w^{(0)}|_{\tau=0}\equiv w_0$, $w^{(0)}|_{\xi=0}=w_1$; (ii) $w^{(0)}$ satisfies the vanishing rate $$(1-\eta)^{3/2}\lesssim_{T,X} w^{(0)}\lesssim_{T,X}(1-\eta)\log^{\frac{1}{2}}\left(\frac{10}{1-\eta}\right),$$ for all $(\tau,\xi,\eta)\in[0,T]\times[0,X]\times[0,1)$ and any $T>0, X>0$; and (iii) for any integers $\alpha_1, \alpha_2, \alpha_3\ge0$ with $2\alpha_1+2\alpha_2+\alpha_3\leq 2$, 
    \begin{equation}\label{regular4}
\partial_\tau^{\alpha_1}\partial_\xi^{\alpha_2}\partial_\eta^{\alpha_3}w^{(0)}(\tau,\xi,\eta) \in C([0,\infty)^{2}\times[0,1)),
    \end{equation}
    and the compatibility condition on $\{\tau=0\}$ and $\{\xi=0\}$ is verified, in the sense that  for any integers $\alpha_1, \alpha_2, \alpha_3\ge0$ with $2\alpha_1+2\alpha_2+\alpha_3\leq2$, $\partial_\tau^{\alpha_1}\partial_\xi^{\alpha_2}\partial_\eta^{\alpha_3}w^{(0)}$ satisfies the differential constraints imposed by the equation \eqref{prandtlcroccosec1} at $\xi=0$, $\tau=0$, and $\eta=0$.  
    
    In addition, we have the following bounds for any $T>0, X>0$,
\begin{equation}\label{higherassume4}
\begin{split}
&\sup_{\tau\in[0,T]}\sum_{\alpha+\beta+\gamma \leq 1}\int_{0}^{X}\int_{0}^{1}(w^{(0)})^{2\gamma-2}|\p_{\tau}^{\alpha}\p_{\xi}^{\beta}\p_{\eta}^{2\gamma}w^{(0)}(\tau,\xi,\eta)|(1-\eta)d\xi d\eta\lesssim_{T,X} 1,\\
&\sup_{\xi\in[0,X]}\sum_{\alpha+\beta+\gamma \leq 1}\int_{0}^{T}\int_{0}^{1}(w^{(0)})^{2\gamma-2}|\p_{\tau}^{\alpha}\p_{\xi}^{\beta}\p_{\eta}^{2\gamma}w^{(0)}(\tau,\xi,\eta)|(1-\eta)d\tau d\eta\lesssim_{T,X} 1.
\end{split}
\end{equation}
In the above the indices $\alpha, \beta,\gamma$ are nonnegative integers.
\end{itemize}
\end{assumptions}

Our main results for the global well-posedness of weak solutions are as follows.
\begin{theorem}[\bf Global Well-posedness for Weak Solutions]
\label{mainthree}
Assume that Assumption \ref{dataassum3} holds. Then the Prandtl equation \eqref{prandtl} has a unique \emph{global} weak solution $u(t,x,y)$ with the following properties:
\begin{enumerate}
    \item $u\in C^{\infty}((0,\infty)\times(0,\infty)\times[0,\infty))$, and together with the associated $v$ as in \eqref{introv} satisfies the first two equations of \eqref{prandtl} and the no--slip condition $u|_{y=0}=0$ in the classical sense;
    
    \item $u(t,x,y)$ is strictly monotone in $y$, and we have the pointwise bounds for some $0<c<C$ (depending on $T,X$)
    \begin{equation}\label{solutionmatching2}
    \begin{split}
        &c(1-u)^{3/2}\leq\p_{y}u\leq C(1-u)\log^{\frac{1}{2}}\left(\frac{10}{1-u}\right),\ \forall (t,x,y)\in [0,T]\times[0,X]\times[0,\infty),\\
        & c\cdot e^{-Cy^{2}}\leq 1-u\leq C\cdot(y+1)^{-2},\ \forall(t,x,y)\in [0,T]\times[0,X]\times[0,\infty);\\
        \end{split}
    \end{equation}
    
    \item For any $T>0,X>0$, $u$ satisfies \eqref{globalbound};
    
    \item $u(t,x,y)$ satisfies the initial and boundary conditions on $\{t=0\}$ and $\{x=0\}$ in the sense of the boundary trace for $W^{1,1}$ functions.
\end{enumerate}
\end{theorem}

We present the corresponding version of Theorem \ref{mainthree} for the transformed Prandtl equation \eqref{prandtlcroccosec1} as follows.
\begin{theorem}\label{thm:gws}
Assume that Assumption \ref{dataassumecrocco4} holds. Then the transformed Prandtl equation \eqref{prandtlcroccosec1} has a unique \emph{global} weak solution $w(\tau,\xi,\eta)$ with the following properties:
\begin{enumerate}
    \item $w\in C^{\infty}((0,\infty)\times(0,\infty)\times[0,1))$, satisfies the first equation in \eqref{prandtlcroccosec1} and the Neumann condition $\p_{\eta}w|_{\eta=0}=0$ in the classical sense;
    
    \item $w(\tau,\xi,\eta)$ is strictly positive in $[0,\infty)\times[0,\infty)\times[0,1)$, and $w|_{\eta=1}=0$ with the vanishing rates 
    \begin{equation}\label{solutionmatchingcrocco4}
        (1-\eta)^{3/2}\lesssim_{T,X}w(\tau,\xi,\eta)\lesssim_{T,X}(1-\eta)\log^{\frac{1}{2}}\left(\frac{10}{1-\eta}\right),\ \forall (\tau,\xi,\eta)\in [0,T]\times[0,X]\times[0,1);\\
    \end{equation}
    
    \item For any $T>0,X>0$, $w$ satisfies \eqref{globalboundcrocco};
    
    \item $w(\tau,\xi,\eta)$ satisfies the initial and boundary conditions on $\{\tau=0\}$ and $\{\xi=0\}$ in the sense of the boundary trace for $W^{1,1}$ functions.
\end{enumerate}
\end{theorem}

\begin{remark} As discussed in the introduction, the well-posedness theory for weak solutions of \eqref{prandtl} was first established in \cite{XZ04} and \cite{XZZ24} under the assumption that the convergence rate to the outer--flow is exactly exponential.  More precisely, it is required that $\p_{y}u\sim 1-u$ as $y\to+\infty$. In Theorem \ref{mainthree}, we extend the weak solution theory to more general settings. 

\end{remark}
\begin{remark}
   The (unique) global weak solution to \eqref{prandtl} in Theorem \ref{mainthree} is constructed through the transformed Prandtl equation \eqref{prandtlcroccosec1} as the first step, using crucially the conservation law type structures of the equation.
   
   The solution above is weak only in the sense of its behavior near $t=0$ or $x=0$, but otherwise is globally smooth up to the boundary $y=0$. Conversely, given any weak solution of \eqref{prandtl} with relatively low regularity requirements, we can obtain the global up--to--boundary $y=0$ smoothness using Theorem \ref{mainone}. 
\end{remark}

\begin{remark}
As in our theory of classical solutions,  the compatibility conditions which can be made explicit (see Remark \ref{descriptioncompati}) as well as regularity assumptions on the initial and inflow boundary data are adopted for the sake of conciseness, and are not sharp. However, the specific form of our assumptions using extended functions turns out to be convenient for the construction of compatible data for the approximation systems \eqref{approsys} and \eqref{approsyss}. This construction is quite nontrivial and plays an essential role in both the global weak solution theory and the local classical solution theory.  We refer to Appendix \ref{compatiappendix} for more details.
\end{remark}

\subsection{Related work on the Prandtl System}
The well-posedness theory for the Prandtl system \eqref{prandtl} plays an important role in the study of boundary layers. In this section we briefly review some earlier results on well-posedness (and related ill-posedness results when the monotonicity condition does not hold), and due to the vast literature, we focus only on those results most relevant to our work. 

\subsubsection{The steady Prandtl equation.} We start with the steady Prandtl equation (i.e., $u$ is time independent in \eqref{prandtl}), which remarkably can be viewed as an evolution equation in the $x$ direction. Using the von Mises transform, the uniqueness of the steady Prandtl equation was proved by Nickel \cite{Ni58}, and the global-in-$x$ well-posedness result on steady solutions to \eqref{prandtl} was proved by Oleinik \cite{OS99}. More recently, the propagation of higher regularity was established in Guo and Iyer \cite{GI21} and instant smoothing of steady solutions for positive $x$ was proved in Wang and Zhang \cite{WZ21}. When pressure is adverse, global existence may fail; in general only local existence is guaranteed. The breakdown of solutions is known as \emph{boundary layer separation}, a phenomenon justified mathematically in \cite{DM19} and \cite{SWZ21}. See also the formation of reversal flow in the steady case \cite{IMreversal}.
 
In \cite{Blas1908}, Blasius introduced an important special solution to the steady Prandtl equation known as the \textit{Blasius profile}, which is unique up to rescaling. More precisely, let $f$ be the unique solution to the differential equation
\begin{equation}\label{blasiusode}
\begin{cases}
f''' + \frac{1}{2}ff'' = 0, \quad \text{in} \ \mathbb{R}^{+},\\
f(0) = f'(0) = 0, \quad f'(\infty) = 1.
\end{cases}
\end{equation}
Define the vector field $[\bar{u}, \bar{v}]$ on $\mathbb{R}^{+} \times \mathbb{R}^{+}$ as
\begin{equation}\label{blasius}
[\bar{u}, \bar{v}] = \left[f'(\eta),\, \frac{1}{2\sqrt{x+1}}\left(\eta f'(\eta) - f(\eta)\right)\right], \quad \eta = \frac{y}{\sqrt{x+1}}.
\end{equation}
Then the pair $[\bar{u},\bar{v}]$ is a steady solution of the Prandtl equation \eqref{prandtl} with a self--similar structure, which is called the Blasius profile.

The Blasius profile plays an important role in the steady Prandtl equation since it is the global attractor of steady solutions in the limit $x\to+\infty$. This was first proved by Serrin in the seminal work \cite{Se67}, which showed that under natural assumptions any solution $u$ of the steady Prandtl equation converges to the Blasius profile
\begin{equation*}
    \lim_{x\to\infty}\|u(x,y)-\bar{u}(x,y)\|_{L^{\infty}(y\in\R^+)}=0.
\end{equation*}
After Serrin's work, Peletier \cite{P72} proved an $L^{1}$-type convergence result from the solution $u$ to the Blasius profile $\bar{u}$ as $x\to\infty$. More quantitative convergence rates were obtained in Iyer \cite{Iy20} and Wang and Zhang \cite{WZ23}, and the optimal convergence rate was recently established by the authors \cite{JLY25} (see also the refinement in Gao and Zhao \cite{GZ25}). Let us also mention that under suitable  structural and regularity assumptions, the stability of steady solutions under the flow of the unsteady Prandtl system was proved in Guo, Wang and Zhang \cite{GWZ23}.

In the steady case, the validity of boundary layer expansions has been demonstrated in recent years, see \cite{GI21}, \cite{gao2023steady}, \cite{GX26}, and \cite{iyer2026global}. 

\subsubsection{Well-posedness theory for the dynamical Prandtl system}\label{sec:classical} The first well-posedness result for classical solutions to the dynamical Prandtl system \eqref{prandtl} was established by Oleinik. We refer to the monograph \cite{OS99} by Oleinik and Samokhin for details. Under suitable monotonicity and regularity conditions, local--in--time or local--in--$x$ well-posedness of \eqref{prandtl} was proved. In other words, either $X>0$ or $T>0$ has to be taken sufficiently small to ensure global existence in the other variable ($T$ or $X$).  

\begin{comment}
More precisely, the following statements hold.
\begin{proposition}[\bf Oleinik]
\label{oleinik}
Oleinik's assumptions:
\begin{itemize}
    \item[(1)] The initial data and inflow data $(u_{0},u_{1})$ in continuously differentiable \eqref{prandtl} is continuously differentiable and strictly monotone, satisfying the following condition on the matching rate to the outer flow $U\equiv 1$:
    \begin{equation}\label{matchrate}
    (1-u_{i})^{k}\lesssim \p_{y}u_{i}\lesssim 1-u^{i},\ i=0,1
    \end{equation}
    for some $k\geq 1$.
    \item[(2)] In Crocco coordinate, the transformed initial data $w_{0}(\xi,\eta)$ and the inflow data $w_{1}(\tau,\eta)$ belong to $C^{10}$, and satisfy the compatibility conditions on the system \eqref{prandtlcroccosec1} (corresponding the compatibility conditions on the original system \eqref{prandtl}).
\end{itemize}

Then, the Prandtl system in Crocco coordinate \eqref{prandtlcroccosec1} exists an unique classical solution either $X$ being arbitrary and $T$ depending on the data $(w_{0},w_{1})$, or $T$ being arbitrary and $X$ depending on the data $(w_{0},w_{1})$. 
\end{proposition}
\begin{remark}By representing $u$ via \eqref{invertcrocco} and defining the divergence--free counterpart $v$ via
\begin{equation*}
    v(t,x,y)=\frac{-\p_{t}u-u\p_{x}u+\p_{y}^{2}u}{\p_{y}u},
\end{equation*}
the local--in--$t$ or local--in--$x$ theory of classical solutions also holds for Prandtl equation in original physical coordinate \eqref{prandtl}.
\end{remark} \end{comment}
In Oleinik's work, the Crocco transform plays an important role. For the purposes of applications to the Navier--Stokes system, it is also desirable to have proofs without the use of the Crocco transform. 
Under the strict monotonicity  condition, Alexandre-Wang-Xu-Yang \cite{AWXY15} and Masmoudi-Wong \cite{MW15} independently proved local well-posedness results in Sobolev spaces in the original physical coordinates, by introducing new methods that capture subtle cancellation mechanisms in the original Prandtl system \eqref{prandtl} to overcome the loss of derivative. See also \cite{XZ17,LWX15} on various improvements in the lifespan of existence and smoothing estimates in Sobolev spaces under additional assumptions.

The strict monotonicity condition $\p_{y}u>0$ turns out to be crucial and not merely a technical assumption to ensure well-posedness. Indeed, without this assumption, the Prandtl system \eqref{prandtl} may become severely ill-posed in Sobolev spaces due to high frequency instabilities in $x$. The instability, which already occurs at the linearized level, was discovered by Gerard-Varet and Dormy \cite{GD10} (see also refinements in \cite{LY17}). The instability result was extended to the nonlinear setting by Guo and Nguyen \cite{GN11}.

Due to the strong ill-posedness in the absence of the monotonicity condition, the best one can expect is to prove local well-posedness of \eqref{prandtl} in the analytic or Gevrey setting. In the celebrated work \cite{SC98a,SC98b}, using the Cauchy--Kowalewski theorem, Sammartino and Caflisch proved for analytic solutions the local well-posedness of \eqref{prandtl} and established the validity of the vanishing viscosity limit for the Navier--Stokes equation. The analyticity requirement can be somewhat relaxed, see e.g.  \cite{LCS03,GM15}. Under a convexity assumption,  Gerard-Varet and Masmoudi \cite{GM15} proved the local well-posedness of \eqref{prandtl} in Gevrey space. The Gevrey regularity index was subsequently improved by Li and Yang \cite{LY20}. Later, the structural convexity assumption was removed by Dietert and Gerard-Varet \cite{DG19}. In the analytic or Gevrey setting, the longer time well-posedness of \eqref{prandtl} was also studied  in the literature. We refer to Zhang and Zhang \cite{ZZ16}, Ignatova and Vicol \cite{IV16}, Paicu and Zhang \cite{PZ21}, Wang, Wang and Zhang \cite{WWZ24} for the precise results and references therein for more discussions. 

Global well-posedness for the Prandtl equation \eqref{prandtl} without any monotonicity or pressure condition may not hold, even for analytic data. The finite-time singularity formation for the dynamical Prandtl equation was first found using numerical simulations by van Dommelen and Shen \cite{VS80}, and studied mathematically by E and Engquist \cite{EE97}, Kukavica, Vicol and Wang \cite{IVF17}, Collot, Ghoul, Ibrahim and Masmoudi \cite{CGIM22}.

In the unsteady case, the validity of boundary layer expansions (locally in time) has been studied in recent years, beginning with Sammartino--Caflisch \cite{SC98b} (see also Wang--Wang--Zhang \cite{WWZ17}) in the analytic setting, and Maekawa \cite{M14} in the Sobolev setting under the assumption that the vorticity is supported away from the boundary. Gerard-Varet, Maekawa and Masmoudi \cite{DMM18} proved the Gevrey stability near monotonic and concave Prandtl expansion solutions of shear type. Near unstable shear flows, the $L^{\infty}$ instability of unsteady boundary layer expansion was proved by Grenier and Nguyen \cite{GN19}.

\subsubsection{The theory of weak solutions to the Prandtl equation}
Under the strict monotonicity condition, the transformed Prandtl system \eqref{prandtlcroccosec1} exhibits a \emph{conservation law} type structure. Assuming that $w_0, w_1\sim 1-\eta$, Xin and Zhang \cite{XZ04} demonstrated the global--in--$(\tau,\xi)$ existence of a $BV\cap L^{\infty}$ weak solution to \eqref{prandtlcroccosec1}, which then yields a global weak solution to \eqref{prandtl}. More recently in \cite{XZZ24}, Xin, Zhang and Zhao provided a direct $BV$--estimate to give a new proof of the global existence of weak solutions to the Prandtl equation in the Crocco coordinates \eqref{prandtlcroccosec1}, and established the uniqueness and H\"older regularity of such weak solutions {\it away from the boundary}. In principle at least, in Crocco coordinates, higher regularity and smoothness of the weak solution in the interior are expected to follow from the H\"older estimates using hypoelliptic smoothing arguments, although the detailed argument turns out to be somewhat nontrivial (see Section \ref{ingredient} below for more discussion on this point).

\subsection{Main ideas of the proof}\label{ingredient}
%\subsubsection{Structure of presented paper}
Our proof was inspired in part by the work \cite{XZ04, XZZ24}. However, extending the regularity results up to the boundary presents several new difficulties.
In this section, we briefly explain these difficulties and our main ideas in the proofs of Theorem \ref{mainone}--Theorem \ref{mainthree}.

\subsubsection{Regularity of the fundamental solution to the Kolmogorov equation in the half space}
The first major difficulty is that the fundamental solution to the Kolmogorov equation (which can be viewed as the ``linearized" version of \eqref{prandtlcroccosec1}) is no longer explicitly given and lacks key symmetries in the half space (with zero Neumann boundary condition), in contrast to the whole space case (see \eqref{wholespace}). The formula for the fundamental solution plays an important role in the analysis of weak solutions to \eqref{prandtlcroccosec1} since it encodes the hypoelliptic smoothing effect of the Kolmogorov equation in a simple form, and the positivity property which is essential for implementing the De Giorgi-Nash-Moser theory (using the approach of Kruzkov \cite{K64}) is transparent. 

To explain our approach to overcome this difficulty, let us take the fundamental solution  $\Gamma(z;\zeta)=\Gamma(t,x,y;t_{0},x_{0},y_{0})$ to the Kolmogorov operator in the half--space, where for ease of notation we denote $z=(t,x,y), \,\zeta=(t_0,x_0,y_0)$. Hence $\Gamma$ satisfies
\begin{equation}\label{eqfundamentalsec1}
\begin{cases}
\p_{t}\Gamma(z;\zeta)+y\p_{x}\Gamma(z;\zeta)-\p_{y}^{2}\Gamma(z;\zeta)=0,\,\, t>t_{0}, \,x\in\R,\,y\in\R^{+},\\
\p_{y}\Gamma|_{y=0}=0,\quad
\Gamma|_{t=t_{0}}=\delta(x-x_{0},y-y_{0}).
\end{cases}
\end{equation}

Thanks to the translation invariance of \eqref{eqfundamentalsec1}, we can assume that $t_0=0$ and $x_0=0$. To prove regularity of $\Gamma(z;\zeta)$ for $z\neq \zeta$, as expected the most dangerous situation occurs when the boundary effect is significant, i.e., when $z$ and $w$ are close to the boundary. After rescaling, we focus on the behavior of the fundamental solution $\Gamma(z;\zeta)$ in the region
{\small \begin{equation*}
    \ \ \ \ (t,x,y)\in\Omega:=\Big\{\frac{1}{2}\leq t \leq 1,\ 0\leq y\leq 1,\ |x|\leq 1\Big\}\bigcup \Big\{\frac{1}{2}\leq|x|\leq 1,\ 0\leq y\leq 1, 0\leq t\leq 1\Big\},
\end{equation*}}
where $\zeta = (0,0,y_{0})$ with $y_{0}\ll 1$.

In this ``near boundary" region, the H\"ormander--type hypoelliptic interior estimates are not directly applicable. Instead, we take the Fourier transform in $x$ for \eqref{eqfundamentalsec1} and obtain the following heat equation with an oscillation term on the Fourier side
\begin{equation}\label{ftxsection1}
\begin{cases}
\p_{t}\widehat{\Gamma}+2\pi \mathrm{i} y\xi\widehat{\Gamma}-\p_{y}^{2}\widehat{\Gamma}=0,\\
\widehat{\Gamma}|_{t=0}=\delta(y-y_{0}),\quad \p_{y}\hat{\Gamma}|_{y=0}=0.
\end{cases}
\end{equation}
After suitable localization in the physical space, the regularity bounds of $\Gamma$ follow from quantitative decay in $|\xi|$ of $\widehat{\Gamma}(t,\xi,y)$.

Notice that large $|\xi|$ corresponds to strong oscillations in \eqref{ftxsection1}, which leads to ``enhanced dissipation" of $\widehat{\Gamma}$. Using the now well-known \emph{enhanced dissipation} theory (see Proposition \ref{enhanced}), we get that
\begin{equation}\label{enhancedsection1}
    \|\widehat{\Gamma}(t,\xi,y;\zeta)\|_{L^{2}(y\in\R^+)}\lesssim e^{-c|\xi|^{\frac{2}{3}}(t-s)}\|\widehat{\Gamma}(s,\xi,y;\zeta)\|_{L^{2}(y\in\R^+)},\,\,\forall t>s>0.
\end{equation}
Precise quantitative decay in $|\xi|$ of $\widehat{\Gamma}(t,\xi,y)$ for $1/2\leq t\leq1$ then follows by combining \eqref{enhancedsection1} with a Nash--type inequality on the $L^{2}_y$ norm of $\widehat{\Gamma}(t,\xi,y)$ for \eqref{ftxsection1} (from the a priori $L^1_y$ bound).

The case of $|x|\gtrsim 1$ is more complicated and we handle this case using the following observation. Let $\chi(x)$ be a given cut--off function located in $|x|\gtrsim 1$. Then $\Gamma_{1}(z;\zeta):=\Gamma(z;\zeta)\chi(x)$ satisfies the bound
\begin{equation}\label{gammaonebdsection1}
\int_{0}^{10}\int_{0}^{\infty}|\widehat{\Gamma}_{1}(t,\xi,y;\zeta)|^{2}\phi^{2}(y)\,dydt\lesssim (1+|\xi|)^{-\frac{4}{3}},
\end{equation}
where $\phi$ is a suitable explicit weight (see \eqref{eqphi}). The estimate \eqref{gammaonebdsection1} follows from applying \eqref{enhancedsection1} to the equation for $\widehat{\Gamma}_{1}$, with suitable energy estimates to handle the growth--in--$y$ in the oscillation term. As a consequence, the regularity in $|x|\gtrsim 1$ (corresponding to the decay in $|\xi|$ after suitable localization) can be obtained from an iteration argument. For more details, we refer to subsection \ref{subsec2.2}.
    
  \subsubsection{Positivity property of the fundamental solution.} As remarked above, besides regularity, to prove the H\"older continuity of the weak solution, the \emph{positivity property of the fundamental solution} is also essential in the approach of Kruzkov \cite{K64} and \cite{XZZ24}, to expand the positivity of the weak solution.

    Due to the one--sided transport effect of the transport operator $\p_{t}+y\p_{x}$ since $y\geq 0$, the fundamental solution $\Gamma(z;\zeta)$ vanishes identically if $x<x_{0}$. To apply the crucial ``expansion of positivity" estimate (see Lemma \ref{oscillation}), it is crucial to show that $\Gamma(z,\zeta)$ is indeed positive whenever $x>x_{0}$. This turns out to be quite nontrivial, despite strong evidence for its validity from physical intuition.

    In order to show the positivity rigorously, a natural idea is to compare $\Gamma$ with the fundamental solution in the whole space $\widetilde{\Gamma}$ with the same source. However, these two fundamental solutions are substantially different near the boundary. For example, $\widetilde{\Gamma}$ may be positive for all $x$ while $\Gamma$ vanishes for $x<x_0$, and as we shall see below $\Gamma(z;\zeta)$ contains an additional type of singularity besides that for $\widetilde{\Gamma}$ as $z\to \zeta$.

To resolve this issue, our key new observation is that positivity can be obtained through the anisotropic scaling property (where $x\sim L^3$, $y\sim L$ and $t\sim L^2$ for a suitable length scale $L$) and the reproducing formula of the fundamental solution, using also a form of the maximum principle due to Nirenberg \cite{N53}.  
    Let us briefly go over the key ideas. Preliminary pointwise bounds imply that for some $M\gg1$ we have the following local positivity
    \begin{equation}
    \Gamma(t,x,y;0,0,y_{0})>0,\quad \forall \,0<t\leq 1,\ x\ge M t^{\frac{3}{2}},\ y,\,y_{0}>0.
    \end{equation}
    In view of the reproducing formula (see \eqref{eqreproduce}), we have the following positivity for $t_{N}=2^{-N}$ ($N\in\Z\cap[1,\infty)$) and $x\gg 2^{-\frac{3}{2}N}$, 
    \begin{equation}\label{baseineq0}
    \begin{split}
    \Gamma(t_{N-1},x,y;0,0,y_{0})&=\int_{0}^\infty\int_{0}^x\Gamma(t_{N},x-\xi,y;0,0,\eta)\Gamma(t_{N},\xi,\eta;0,0,y_{0})\,d\xi d\eta\\
    &\geq \int_{0}^\infty\int_{ 2^{-3N/2}M}^{x- 2^{-3N/2}M}\Gamma(t_{N},x-\xi,y;0,0,\eta)\Gamma(t_{N},\xi,\eta;0,0,y_{0})\,d\xi d\eta>0.
    \end{split}
    \end{equation}

Iterating \eqref{baseineq0} $N$ times, we obtain that
\begin{equation}\label{posisec1}
\Gamma(1,x,y;0,0,y_{0})>0,\quad \forall\, x\gtrsim 2^{-\frac{1}{2}N},\,\, y>0,\,\,y_{0}>0,
\end{equation}
which implies the positivity of $\Gamma(z;\zeta)$ if $x>x_{0}$ since $N\gg 1$ is arbitrary. For more details, we refer to subsection \ref{subsec2.3}.
    
\subsubsection{ H\"older estimates and higher regularity of weak solutions} To demonstrate the up--to--boundary smoothing effect of the Prandtl system (Theorem \ref{mainone}), it suffices to show the up--to--boundary smoothness for weak solutions to the transformed Prandtl system \eqref{prandtlcroccosec1} in the Crocco coordinates. The main task is to show that for a positive weak solution $w$ to \eqref{prandtlcroccosec1} (which can be reformulated in divergence form by considering the equation for $\frac{1}{w}$, see \eqref{vdivereq}) in $(0,T)\times(0,X)\times(0,1)$ with the regularity
    \begin{equation}\label{eqreg0}
        w\sim_\epsilon 1\quad \text{and}\quad  w\in W^{1,1}\cap W^{2,1}_{y},\quad \text{in }(\epsilon,T)\times(\epsilon,X)\times[0,1-\epsilon)
    \end{equation}
    for all $0<\epsilon\ll1$, $w$ is smooth up to the boundary. More precisely,
    \begin{equation}
        w\in C^{\infty}((0,T]\times(0,X)\times[0,1)).
    \end{equation}
The regularity assumption \eqref{eqreg0} is motivated by the regularity property of the global weak solutions constructed below. We establish the smoothness of weak solutions in two main steps. In the first key step, we prove H\"older estimates using De Giorgi-Nash-Moser theory adapted to the ultraparabolic setting. This has been done in the interior case without boundary by Xin, Zhang and Zhao \cite{XZZ24} using the approach of Kruzkov \cite{K64} (see also Golse, Imbert, Mouhot and Vasseur \cite{GIMV19} using the  De Giorgi approach \cite{D57}). A key original idea in \cite{XZZ24} is to use a weak type Poincar\'e  inequality, see Lemma \ref{weakpoin}, that overcomes the lack of dissipation in the $x$ direction. Our proof is inspired by \cite{XZZ24}, and we also benefited from the use of the interior Moser iterations in \cite{PP04} by Pascucci and Polidoro and the presentation of the weak Poincar\'e inequality in \cite{GI23} by Guerand and Imbert. 

The main new difficulty, as discussed above, is to deal with the boundary effects in the half space (with zero Neumann boundary condition). Due to the one-sided positivity (in $x$) of the fundamental solution, we need to adapt in a nontrivial way several steps in the argument, as the propagation of positivity works only in one direction in $x$. As a consequence, for example in the key lemma \ref{oscillation} the lower bound is established in an ultraparabolic ball only after a significant shift in $x$. We refer to Section \ref{continuity} for the detailed proof.  

%For instance, due to the effect of one--side transport whenever boundary is present, one may only obtain the improvement of the oscillation in the shrinking cube with \emph{translated center}. 
    
%   Considering interior smoothing effect for kinetic equation \cite{HS20} and more early the Hilbert's 19th problem for elliptic/parabolic equation \cite{V16}, one may \emph{expect} the smoothing effect for \eqref{prandtlcroccosec1} may rely on two key steps: (i)H\"older continuity of the weak solution; (ii) Improving the regularity by suitable Schauder--type estimate like \eqref{schauder_whole_space}.

 To obtain higher regularity of the solution, a natural idea is to apply the Schauder theory based on the H\"older estimates we just obtained. This, however, is not straightforward. The first main difficulty is that our equations are quasilinear and it seems difficult to ``regularize" the equation to directly apply the a priori bounds from Schauder's theory, at least for the weak solutions. To resolve this, one would need to develop a comprehensive well-posedness theory of \eqref{prandtlcroccosec1} in H\"older spaces, which is interesting but nontrivial. A second difficulty is that our hypoelliptic operator does not gain one full derivative of regularity in $(\tau,\xi)$, which makes it harder to gain higher regularity. Moreover, the fundamental solution to the Kolmogorov equation in the half space lacks key symmetries that are available in the whole space. As a consequence, the H\"older space that we use needs to be very specifically designed, see e.g. \eqref{w2pdesire2}--\eqref{newholder}, which makes it technically complicated to use Schauder's theory.

%However, various crucial difficulties come in inevitably .First and foremost, in the techniqune level, the Schauder--type estimate for Kolmogorov operator $\mathcal{L}_{0}$ \emph{does not improve one full derivative} for $t,x$ variable, such that we can not bootstrap the regularity directly. More seriously, since we aim to obtain the regularity improvement up--to--boundary, several technique obstacles also involve:\begin{itemize}\item To state the Schauder theory for Kolmogorov equation $\mathcal{L}_{0}$, the underlying functional (H\"older) spaces are necessary to match some key symmetric properties of Kolmogorov operator. \item For existed interior Schauder--type estimates for $\mathcal{L}_{0}$ in literature(see \cite{},\cite{}), the explicit expression for the fundamental solution in whole space is useful. Moreover, the symmetry property $\widetilde{\Gamma}(z;\zeta)=\widetilde{\Gamma}(\zeta^{-1}\circ z;0)$ is quite essential.\item However, in half--space, the Galilean invariance for Kolmogorov equation is \emph{not} valid. Also, both the explicit expression and the symmetry under Galilean coordinate transform of fundamental solution (in half--space) do not hold. Hence, Schauder--type estimate for $\mathcal{L}_{0}$ seems to be not valid any more in half--space.\end{itemize}}

  Instead, we use the divergence structure of the equation \eqref{prandtlcroccosec1}, and for simplicity consider the following model equation in the half--space with zero Neumann boundary condition, 
     \begin{equation}\label{divergencesec1}
         \begin{cases}
             \p_{t}g+y\p_{x}g-\p_{y}(a(t,x,y)\cdot\p_{y}g) = h,\ (t,x,y)\in\R^{2}\times\R^{+},\\
             \p_{y}g|_{y=0}=0.
         \end{cases}
     \end{equation}
In the above, we assume that $a(t,x,y)$ is ``elliptic", i.e., $c\leq a\leq \frac{1}{c}$ for some $c\in(0,1)$. Note that by energy estimates we have
\begin{equation}\label{hypoe1}
    \|\p_{y}g\|_{L^{2}_{t,x,y}}\lesssim \|g\|_{L^{2}_{t,x,y}}+\|h\|_{L^{2}_{t,x}H^{-1}_{y}}.
\end{equation}
Inspired by H\"ormander's seminal work on hypoellipticity \cite{H67}, we can expect that regularity in $y$ obtained by an energy estimate may be \emph{transferred} to (fractional) regularity in the $x$ variable. Indeed using regularity bounds for the fundamental solution $\Gamma$ in Section \ref{fundamental}, we can effectively exploit this hypoelliptic smoothing mechanism in the half--space to obtain roughly the bounds
\begin{equation}\label{hypoe2}
\||D_{x}|^{\frac{1}{3}}g\|_{L^{2}}+\|\langle y\rangle ^{-1}|D_{t}|^{\frac{1}{3}}g\|_{L^{2}}\lesssim \|g\|_{L^{2}_{t,x}H^{1}_{y}}+\|h\|_{L^{2}_{t,x}H^{-1}_{y}}.
\end{equation}
See \cite{B02} for related work. 

In view of the bounds \eqref{hypoe1}--\eqref{hypoe2} for solutions to \eqref{divergencesec1}, it is natural to use a bootstrap argument to obtain higher regularity and ultimately smoothness, by taking fractional-order derivatives in \eqref{divergencesec1}. This approach indeed works, and the main technical part is to quantify precisely the required regularity on the coefficient $a$ so that various commutators that appear when taking derivatives may be bounded, and simultaneously that these regularity assumptions match what can be obtained from \eqref{prandtlcroccosec1} since $a$ depends on the solution.  Crucially, to initiate the bootstrap argument, besides the standard bootstrap assumptions, we need also a bound on $\|\p_{y}w\|_{L^{\infty}}$ which appears in controlling one of the commutator terms. We refer to Section \ref{higher} for details.

  \subsubsection{Closing the crucial $\|\partial_yw\|_{L^\infty}$ bounds and completion of smoothing estimates for weak solutions}  To obtain the crucial bound on $\|\partial_yw\|_{L^\infty}$, using an anisotropic interpolation lemma (see Lemma \ref{aniinterpolation}), it suffices to obtain $W^{2,p}_{y}$ estimates for the solution of \eqref{prandtlcroccosec1} for arbitrary $1<p<\infty$.

            The interior $W^{2,p}_y$ estimate for a variable-coefficient Kolmogorov operator was established in \cite{BC96b}, using the theory of singular integral operators (SIOs) on homogeneous spaces as well as an  explicit expression and key symmetries of the fundamental solution in the whole--space. The singularity property of the fundamental solution in the half space is more complicated, and proving the corresponding $L^p$ estimates is technically more difficult. In particular, we need to use the enhanced dissipation estimates (see Proposition \ref{enhanced} for the details) when establishing $L^2$ boundedness of certain SIOs. Another  difficulty is that, to begin with, $\p_{y}^{2}w$ belongs to $L^1$ only, and as a consequence the standard perturbation method (for instance that in \cite{BC96b}) does not directly apply due to the lack of $L^1$ boundedness of SIOs. To treat this difficulty, our main strategy is to (i) view the weak solution to \eqref{divergencesec1} as the solution to a constant-coefficient Kolmogorov type equation with a right hand side by suitable perturbation and localization arguments; (ii) design a contraction mapping scheme in a higher integrability space; and (iii) prove a uniqueness result. The combination of (i)-(iii) finally yields higher integrability of the weak solution. To implement this idea, we need to adapt our regularity assumptions on the coefficient $a$ to strike the right balance between the minimum required regularity for steps (i)-(iii) to work, and the need to obtain them from the relation between $a$ and the solution itself in our applications when $a$ is given in terms of $w$. We refer to Section \ref{sobolev} for the full details.

\subsubsection{Local well-posedness (LWP) for classical solutions.} To prove the global-in-$(t,x)$ well-posedness for classical solutions of the Prandtl boundary layer equation \eqref{prandtl}, local well-posedness (LWP) for classical solutions is a necessary first step. The LWP for classical solutions has been addressed in \cite{O68,OS99}. We revisit this theory using a novel \emph{quasilinear weighted energy estimate} on the transformed Prandtl equation \eqref{prandtlcroccosec1}, with the aim of obtaining both local-in-$t$ and local-in-$x$ well-posedness results, as well as (crucially) allowing a wide range of matching rates of the boundary layer flow to the outer Euler flow. %To avoid confusion, we note that here local-in-$t$ but global-in-$x$ wellposedness refers to the wellposedness result that for any $X>0$, there exists a sufficiently small $T>0$ depending on $X$ such that \eqref{prandtlcroccosec1} is wellposed for $t\in[0,T]$ and $x\in[0,X]$. 

More precisely, we use the weight $w^{-2}$ in our energy functionals, which could be quite singular as $\eta\to1-$. Using the weighted energy estimate, we obtain a Beale-Kato-Majda type regularity criterion: the energy functional can be bounded as long as $$\mathcal{N}(w):=\sup_{\tau\leq T,\,\xi\le X,\eta\in[0,1)}\left |\frac{\p_{\tau,\xi} w(\tau,\xi,\eta)}{w(\tau,\xi,\eta)}\right|$$ stays bounded, see Lemma \ref{highenergyesti}. 
%Formally, we obtain the energy bounds for the tangential derivatives of $w$:\begin{equation}\label{localenergysec1}\sup_{\tau\in[0,T]}\bigg\|\frac{\p^{\leq k}w}{w}\bigg\|_{L^{2}_{\xi,\eta}}^{2}+\sup_{\xi\in[0,X]}\bigg\|\sqrt{\eta}\frac{\p^{\leq k}w}{w}\bigg\|_{L^{2}_{\tau,\eta}}^{2}+\|\p^{\leq k}\p_{\eta}w\|^{2}_{L^{2}_{\tau,\xi,\eta}}\leq \text{Nonlinear Terms},\end{equation}where $\partial^{\leq k}:=\sum_{1\leq \alpha+\beta\leq k}\p^{\alpha}_{\tau}\p^{\beta}_{\xi}$. To proceed, by careful calculations for non--linear terms, more precisely we have\begin{equation}\label{localenergysec11}\begin{split}&\sup_{\tau\in[0,T]}\bigg\|\frac{\p^{\leq k}w}{w}\bigg\|_{L^{2}_{\xi,\eta}}^{2}+\sup_{\xi\in[0,X]}\bigg\|\sqrt{\eta}\frac{\p^{\leq k}w}{w}\bigg\|_{L^{2}_{\tau,\eta}}^{2}+\|\p^{\leq k}\p_{\eta}w\|^{2}_{L^{2}_{\tau,\xi,\eta}}\\&\quad\quad\lesssim \left\|\frac{\p w}{w}\right\|_{L^{\infty}_{\tau,\xi,\eta}}\times \int_{0}^{T}\int_{0}^{X}\int_{0}^{1}\left|\frac{\p^{\leq k}w}{w}\right|^{2}d\tau d\xi d\eta.\end{split}\end{equation}
 To close the energy estimate, it is therefore sufficient to control $\mathcal{N}(w)$. %However, since we allow rather general behavior of $w$ near $\eta=1$ (determined by the matching rates in physical coordinates), hence we do not have certain regularity information from the energy norm. Furthermore, due to the degeneracy of the equation \eqref{prandtlcroccosec1},  we may not have  strong controls to the normal derivative from the energy norm with respect to tangential derivatives. 
Due to the strongly singular weight in $\mathcal{N}(w)$ as $\eta\to1-$ and the lack of control on derivatives $\partial_\eta\partial_\tau w$ and $\partial_\eta\partial_\xi w$ from the energy functional, this turns out to be quite nontrivial. Our main idea is to distinguish two regions $\eta\leq 1/2$ and $\eta\ge1/2$ and apply different strategies in each region. 

In the region where $\eta\leq\frac{1}{2}$, we estimate $\mathcal{N}(w)$ using the \emph{energy functional} by carefully adapted anisotropic Sobolev embeddings, since in this region the weight $w\sim 1$ and therefore is not singular. 

In the region where $\eta\geq\frac{1}{2}$, the weight is very singular and the energy functional alone is not sufficient to bound $\mathcal{N}(w)$. Instead, we estimate $\mathcal{N}(w)$ using also the equation \eqref{prandtlcroccosec1} and the \emph{comparison principle} between $\partial_\tau w$ (or similarly $\partial_\xi w$) and a suitable modification of $w$ involving a large constant $M>1$. 

The key technical complication is that $M$ depends also on the energy functional to be controlled, which creates a nontrivial set of constraints. Fortunately, by assuming suitable smallness of $X$ (or $T$), we can select parameters that satisfy all the constraints and prove the local well-posedness result. We refer to Section \ref{local} for more details. 

\subsubsection{Global weak solutions}
In Section \ref{basic} we revisit the theory of global weak solutions to \eqref{prandtl} developed in \cite{XZ04, XZZ24}. The global theory of weak solutions is still useful in applications even though we have global classical solutions, since we need less regularity and compatibility conditions on the initial and boundary data for weak solutions.  Moreover, the key a priori bounds which are obtained based on the conservation structure of \eqref{prandtlcroccosec1} provide additional control even for classical solutions. In particular, we emphasize that the spacetime bounds \eqref{twdecay0.1} on $\partial_\tau w$ when the boundary condition is time independent suggest convergence of solutions to \eqref{prandtl} to steady states, under more general conditions than those considered in \cite{GWZ23}, and are expected to be useful for studying asymptotic stability of \eqref{prandtl}. 

While our arguments are broadly similar to those in \cite{XZ04,XZZ24}, our main results significantly generalize the class of initial and boundary data that are allowed for proving existence and uniqueness of weak solutions, to incorporate all three natural rates: polynomial, exponential and Gaussian convergence to outer Euler flow as $y\to\infty$. Another aspect that is worth mentioning is that we use the smoothness results for weak solutions (Theorem \ref{mainone}) to work with weak solutions with $\partial_y^2w\in L^1_{\rm loc}$ rather than $BV$ regularity in \cite{XZ04,XZZ24}. This simplifies some of our arguments such as the proof of uniqueness. 

\subsubsection{Global well-posedness of the approximation system} In Appendix \ref{proofapproxia}, we provide a self-contained proof of the global well-posedness of classical solutions to an approximation system to \eqref{prandtlcroccosec1}, see \eqref{approsys} (and \eqref{equationappendix}). %which was also used in the earlier work \cite{XZZ24} as a key tool to construct global weak solutions. 
\eqref{equationappendix} is simpler than \eqref{prandtlcroccosec1} in the sense that the degeneracy of the diffusion coefficient is removed. As a consequence, we no longer need to use sophisticated weighted energy estimates to prove the local well-posedness. The key advantage of using \eqref{approsys} to approximate \eqref{prandtlcroccosec1} (compared with adding a diffusion term, for example) is that the conservation law structure is preserved. However, it turns out that proving the full regularity of the approximation system is not as straightforward as one might expect, since the approximation system \eqref{approsys} is still a quasilinear ultraparabolic system despite the fact that it is no longer degenerate. To establish global well-posedness, we have to come up with a new energy functional, see \eqref{energynormappendix}-\eqref{dissi}, to deal with boundary effects for higher order regularity estimates. We also need to use the smoothing estimates from Theorem \ref{mainone} to close the key regularity controlling quantity.  

Another important issue is to construct approximate data satisfying compatibility conditions of relatively high order, from the original initial and inflow data. Notice that due to the change of equations, the original data in general do not satisfy the compatibility conditions of \eqref{approsys}. Instead, we must carefully modify the data to ensure compatibility and simultaneously preserve important bounds such as \eqref{oneorderassume2}-\eqref{higherassume2}. The construction turns out to be surprisingly subtle, due to degeneracies of the data as $\eta\to 1-$ and to the presence of constraints from three distinct boundary intersections $\{\tau=0, \eta=0\}$, $\{\xi=0, \eta=0\}$ and $\{\tau=0, \xi=0\}$. We present the detailed construction in Appendix \ref{compatiappendix}, which is of independent interest.

%This is based on two key observations: (i) the value of $|\p w|$ on the $\{\eta=\frac{1}{2}\}$ can be estimated by energy norms; (ii) The constant $M$ can be chosen relatively explicitly depending on the energy functional using anisotropic Sobolev embeddings. By a delicate analysis, we are able to ensure that all involved estimates can be only relied on one of $T$ or $X$ by a series of careful analysis, therefore we can conclude both local--in--$t$ and local--in--$x$ well-posedness. For complete details, kindly consult Section \ref{local}.

\subsection{Organization of the paper}

The rest of the paper is structured as follows.

In Section \ref{fundamental}, we study properties of the fundamental solution to the Kolmogorov equation in the half--space. %For our applications, we also prove the optimal positivity of the fundamental solution by virtue of novel iterations.

In Section \ref{continuity}, we prove up--to--boundary H\"older continuity for weak solutions to a linear ultraparabolic equation (see \eqref{diverform}) with a bounded coefficient. 
%of $\mathcal{L}_{0}^{a}=0$ with $a$ merely bounded. We adopt the approach in the interior estimate case (without boundary) with subtle modifications due to the effect of \emph{one--sided transport}.

In Section \ref{sobolev}, we develop $W^{1,p}_y$ and $W^{2,p}_y$ estimates for weak solutions to a linear ultraparabolic equation (see \eqref{nondiver}) with a H\"older (and specifically adapted higher order regularity) coefficient.
%we first show the $W^{-1,p}_{y}\to W^{1,p}_{y}$ improvement and then $L^{p}\to W^{2,p}_{y}$ improvement for the forced equation $\mathcal{L}_{0}^{a}w=h$ for H\"older continuous coefficient $a(t,x,y)$. The crucial step is to obtain the continuity of certain singular integral operators via delicate estimates of integral kernels.

In Section \ref{higher}, we establish optimal higher-order regularity estimates for a forced ultraparabolic equation (see \eqref{diverformsectionSIX}) with the coefficient and the force in suitable Sobolev spaces. %Novel energy estimates and hypoelliptic regularity improvement estimates associated with fractional regularity are involved.
As an application, we complete the proof of Theorem \ref{mainone} (together with that of Theorem \ref{th:smo1}). %Starting at low initial regularity requirement, we bootstrap the regularity of the solution by employing the quasilinear structure and various applicable regularity estimates stated in previous sections.

In Section \ref{basic}, we revisit and generalize the theory of global weak solutions to \eqref{prandtl} and complete the proof of Theorem \ref{mainthree} (together with that of Theorem \ref{thm:gws}).

%). We adopt the argument in \cite{XZZ24} but allowing more general matching rates. By virtue of BdRT in our first part, the constructed (unique) weak solution is indeed smooth expect near $t=0$ and $x=0$. A uniqueness result for larger class of solutions is also obtained, which is fluently used in our argument.

In Section \ref{local}, we give a new proof of the existence of local classical solutions to \eqref{prandtl} based on quasi--linear weighted energy estimates. %We derive a priori estimates by so called quasilinear weighted energy estimates, choosing the weight depends on the solution itself, which is new in the local theory of Prandtl equation. Various anisotropic embeddings are involved to close the estimates and gain the compactness.
We then prove the global-in-$(t,x)$ well-posedness for classical solutions in Sobolev spaces (Theorem \ref{maintwo} together with the corresponding Theorem \ref{thm:gcs} in the Crocco coordinates) by combining the global weak solution theory and local classical solution theory.

In Appendix \ref{compatiappendix}, we formulate the compatibility conditions for the approximation system and present the construction of approximating data in detail.

In Appendix \ref{proofapproxia}, we prove the global well-posedness of the approximation system \eqref{approsys}, which is the fundamental step to construct the global weak solution and the local classical solution. %Our BdRT is employed to prevent the finite time singularity of the approximation system.

In Appendix \ref{pre}, we present various analysis tools that were used throughout the article, including Schur's test, Littlewood--Paley dyadic decomposition, interpolations, and enhanced dissipation.

\subsection{Notations} In this article, we adopt the following notational conventions:

\begin{itemize}
\item[(1)] We use $A\lesssim$ ($\gtrsim$) $B$ to denote $A\leq$ ($\geq$) $cB$ for some absolute constant $c$. The notation $A\sim B$ means $A\lesssim B$ and $A\gtrsim B$. Additionally, $A\lesssim_{\star}$ ($\gtrsim_{\star}$) $B$ implies $A\leq$ ($\geq$) $cB$ with $c$ being continuously dependent on a specific quantity $\star$, i.e., $c=c(\star)$. 

\item[(2)] The notation $\p_{l}w$ denotes the partial derivative of $w$ with respect to the variable $l$, and $w_{l}:=\p_{l}w$. 

\item[(3)] $|D_{x}|^{s}$ and $\langle D_{x}\rangle^{s}$ denote the non--local differential operators with respect to $x$, defined using the Fourier transform as
\begin{equation}
    |D_{x}|^{s}f:=\mathcal{F}_{x}^{-1}(|\xi|^{s}\mathcal{F}_{x}(f)),\ \langle D_{x}\rangle^{s} f:=\mathcal{F}_{x}^{-1}(\langle \xi\rangle^{s}\mathcal{F}_{x}(f)),
\end{equation}
where $\langle\cdot\rangle$ is the classical Japanese bracket: $\langle a \rangle=(|a|^{2}+1)^{\frac{1}{2}}$.

\item[(4)] $\mathrm{sgn}(z)$ denotes the standard sign-function:
\begin{equation}
    \mathrm{sgn}(z)=
    \begin{cases}
        1,\ z>0,\\
        0,\ z=0,\\
        -1,\ z<0.
    \end{cases}
\end{equation}

\item[(5)] Notations of domains:
\begin{itemize}
\item [(5.a)] We denote $\mathbb{H}:=\R_{t}\times \R_{x}\times \R^{+}_{y}$ and $\mathbb{H}^-:=\R^{-}_{t}\times\R_{x}\times\R^{+}_{y}$; 

\item[(5.b)] For $z_{0}=(t_{0},x_{0},0)\in \p\mathbb{H}$ and $r_{0}>0$, denote
$$\mathcal{B}_{r_{0}}^{+}(z_{0}):= (t_{0}-r_{0}^{2},t_{0}]\times (x_{0}-r_{0}^{3},x_{0}+r_{0}^{3})\times [0,r_{0}),\ \mathcal{B}_{r_{0}}^{+}:=\mathcal{B}_{r_{0}}^{+}(0,0,0);$$ 

\item[(5.c)] For fixed $y_{0}>0$, denote the strip $$\mathcal{S}_{y_{0}}:=(-\infty,0]\times\R\times[0,y_{0}).$$
\end{itemize}

\end{itemize}

%section 3

\section{The fundamental solution to the Kolmogorov operator in the half space}\label{fundamental}
In this section, we study the basic properties of the fundamental solution for the linear Kolmogorov  partial differential operator
\begin{equation}
\mathcal{L}_{0}=\p_{t}+y\p_{x}-\p_{y}^{2}
\end{equation}
in the upper half space with the Neumann boundary condition.
Clearly, $\mathcal{L}_{0}$ is hypoelliptic (see e.g. \cite{H67,LP93}). As a consequence, all distributional solutions $u$ of $\mathcal{L}_{0}u=0$ are smooth in the interior of the domain. 

The fundamental solution $\Gamma(z;w)=\Gamma(t,x,y;t_{0},x_{0},y_{0})$ for $\mathcal{L}_{0}$ with pole at $w=(t_{0},x_{0},y_{0})\in \R\times\R\times(0,\infty)$ and zero Neumann boundary condition satisfies
\begin{equation}\label{eqfundamental}
\begin{cases}
\p_{t}\Gamma(z;w)+y\p_{x}\Gamma(z;w)-\p_{y}^{2}\Gamma(z;w)=0,\ t>t_{0},x\in\R,y\in\R^{+},\\
\p_{y}\Gamma|_{y=0}=0,\,\,\,
\Gamma|_{t=t_{0}}=\delta(x-x_{0},y-y_{0}).
\end{cases}
\end{equation}

%Recall in the whole space case, the fundamental solution satisfies an ``invariance property" under the full Galilean transformations (see e.g. \cite{BC96b,PP04}). This property implies a single--variable structure of the fundamental solution and an explicit analytic expression for the fundamental solution. In our setting with nontrivial boundary conditions, it seems to be difficult to obtain an explicit formula or even a simplified representation such as certain single--variable structure for the fundamental solution, which is one of the  primary difficulties of analyzing the fundamental solution quantitatively.

 For $z,w$ in $ \mathbb{H}:=\R_{t}\times \R_{x}\times \R_{y}^{+}$, we introduce the following geometric notations, adapted to the scaling property of $\mathcal{L}_0$. 
\begin{itemize}
\item[(1)] 
Dilation: for $z=(t,x,y)\in\mathbb{H}$, and $\lambda>0$, define
\begin{equation}
D(\lambda)z:=\left(\frac{t}{\lambda^{2}},\frac{x}{\lambda^{3}},\frac{y}{\lambda}\right)\in\mathbb{H}.
\end{equation}
\item[(2)] 
Distance: for $z=(t,x,y)\in\mathbb{H}, w=(t_{0},x_{0},y_{0})\in\mathbb{H}$ and $z\neq w$, let $\|z-w\|:=\rho$ be the unique positive root of 
\begin{equation}
\frac{(t-t_{0})^{2}}{\rho^{4}}+\frac{(x-x_{0})^{2}}{\rho^{6}}+\frac{(y-y_{0})^{2}}{\rho^{2}}=1.
\end{equation}

\end{itemize}

\begin{remark} One can check that the metric norm $\|\cdot\|$ satisfies the triangle inequality, using the inequality that for positive numbers $a, b, \alpha, \beta$ and $\gamma\ge2$,
\begin{equation}\label{eq:dist}
    \frac{(\alpha+\beta)^2}{(a+b)^\gamma}\leq \frac{a}{a+b}\frac{\alpha^2}{a^\gamma}+\frac{b}{a+b}\frac{\beta^2}{b^\gamma}. 
\end{equation}
Furthermore, the Lebesgue measure of the ball with radius $r$ is given by 
\begin{equation}\label{volume}
\left|\left\{z\in\mathbb{H}, \|z\|<r\right\}\right|\sim r^{6}.
\end{equation}
In particular, we can see that $\mathbb{H}$ equipped with the metric $d(z,w)=\|z-w\|$ and the standard Lebesgue measure is a \emph{homogeneous space}; see Section 1 in \cite{BC96a}.
\end{remark}

Note that $\Gamma$ is non--negative for all $t>t_{0}$ by the maximum principle. We shall also use the following conservation of mass property
\begin{equation}\label{consermass}
\iint_{\R\times\R^{+}}\Gamma(t,x,y;t_{0},x_{0},y_{0})\,dx dy\equiv 1,\ \forall t>t_{0},\ x_{0}\in\R,\ y_{0}>0.
\end{equation}

\begin{comment}
Here, the initial condition in \eqref{eqfundamental} means
\begin{equation}
\Gamma(t,x,y;t_{0},x_{0},y_{0})\rightharpoonup_{\text{weak--star}} \delta(x-x_{0},y-y_{0}),\ \text{as}\ t\to t_{0}+.
\end{equation}
\end{comment}

\subsection{General properties of the fundamental solution}
We begin with some basic symmetry properties of the fundamental solution $\Gamma(z;w)$, which are consequences of the invariance of $\mathcal{L}_{0}$ under the dilation $D(\lambda)$ and translation in $(t,x)$.
\begin{proposition}\label{proper}
For $\lambda>0$ and $\tau,\xi\in\mathbb{R}$, the following identities hold:
\begin{align}
&\Gamma(z;w)=\lambda^{-4}\Gamma(D(\lambda)z, D(\lambda)w).\\
&\Gamma(z;w)=\Gamma\left(z-(\tau,\xi,0);w-(\tau,\xi,0)\right).
\end{align}
\end{proposition}

We record the following duality property for the fundamental solution $\Gamma(z;w)$. 

\begin{lemma}\label{lm:h2}
Assume that $z=(t,x,y)\in\mathbb{H},\, w=(t_0,x_0,y_0)\in \mathbb{H}$, $z\neq w$. Then we have the following conservation of mass in the dual variable
\begin{equation}\label{dualmassconser}
\iint_{\R\times\R^{+}}\Gamma(t,x,y;t_{0},x_{0},y_{0})\,dx_{0}dy_{0}\equiv 1,\ \forall\ t>t_{0},x\in\R,y\in\R^{+}.
\end{equation}

Moreover, in the dual variable, $\Gamma(z;w)$ satisfies the equation
\begin{equation}\label{dualeq}
\mathcal{L}_{0}^{*}\Gamma(z;w):=-(\p_{t_{0}}+y_{0}\p_{x_{0}}+\p_{y_{0}}^{2})\Gamma(t,x,y;t_{0},x_{0},y_{0})\equiv 0,\ \text{for}\ t_{0}<t.
\end{equation}
\end{lemma}

\begin{proof}
  We present the (well-known) argument for completeness. For any $\varphi \in C_c^\infty(\overline{\mathbb{H}})$ with $\partial_y\varphi|_{y=0}=0$, we have 
  \begin{equation}\label{dualeq1}
      \begin{split}
          \varphi(t_0,x_0,y_0)=&\int_{\mathbb{H}}(\partial_t+y\partial_x-\partial_y^2)\Gamma(z;w) \varphi(t,x,y) dz\\
          =& \int_{\mathbb{H}} \Gamma(z;w)(-\partial_t-y\partial_x-\partial_y^2)\varphi(t,x,y)dz.
      \end{split}
  \end{equation}
  Define for $(t,x,y)\in\mathbb{H}$,
  \begin{equation}\label{dualeq2}
    h(t,x,y):=(-\partial_t-y\partial_x-\partial_y^2)\varphi(t,x,y).
  \end{equation}
  Define $\Gamma^\ast(t_0,x_0,y_0;t,x,y)$ as the fundamental solution to $-\partial_{t_0}-y_0\partial_{x_0}-\partial_{y_0}^2$ on $\mathbb{H}$ with zero Neumann boundary condition, i.e.,
  \begin{equation}\label{dualeq3}
      \begin{split}
          (-\partial_{t_0}-y_0\partial_{x_0}-\partial_{y_0}^2)\Gamma^\ast(t_0,x_0,y_0;t,x,y)=&\,\delta(t_0-t,x_0-x,y_0-y),\\
          \partial_{y_0}\Gamma^\ast(t_0,x_0,y_0;t,x,y)|_{y_0=0}=&\,0.
      \end{split}
  \end{equation}
  It follows from the identity \eqref{dualeq2} that 
  \begin{equation}\label{dualeq3.1}
      \varphi(t_0,x_0,y_0)=\int_{\mathbb{H}} \Gamma^\ast(t_0,x_0,y_0;t,x,y)h(t,x,y)\,dtdxdy.
  \end{equation}
  Comparing \eqref{dualeq1} and \eqref{dualeq3.1}, we conclude that
  \begin{equation}\label{dualeq4}
      \Gamma(t,x,y;t_0,x_0,y_0)=\Gamma^\ast(t_0,x_0,y_0;t,x,y). 
  \end{equation}
  The desired conclusions follow from \eqref{dualeq4} and \eqref{dualeq3}.
\end{proof}

\subsection{Regularity estimates for the fundamental solution}\label{subsec2.2}
We now turn to the pointwise estimates of $\Gamma$. Assume that $z=(t,x,y)\in \mathbb{H}, \,\,w=(t_{0},x_{0},y_{0})\in\mathbb{H}$ and $z\neq w$. Let $\widetilde{\Gamma}(z,w)$ be the fundamental solution to $\mathcal{L
}_0$ in the whole space. $\widetilde{\Gamma}(z,w)$ admits the following explicit formula
\begin{equation}\label{wholespace}
\widetilde{\Gamma}(z,w)=
\begin{cases}
\frac{\sqrt{3}}{2\pi(t-t_{0})^{2}}\exp\left\{-\frac{(y-y_{0})^{2}}{4(t-t_{0})}-\frac{3}{4(t-t_{0})^{3}}(2(x-x_{0})-(t-t_{0})(y+y_{0}))^{2}\right\},\ &t> t_{0},\\
0,\quad &t<t_{0}.
\end{cases}
\end{equation}
\begin{proposition}[\bf Pointwise Bounds]
\label{ptwise}
Assume that $z=(t,x,y)\in \mathbb{H}, \,\,w=(t_{0},x_{0},y_{0})\in\mathbb{H}$ and $z\neq w$. Then $\Gamma(z;w)$ can be decomposed as
\begin{equation}\label{greendecom}
\Gamma(z;w)=\widetilde{\Gamma}(z,w)+y_{0}^{-4}V\Big(\frac{t-t_{0}}{y_{0}^{2}},\frac{x-x_{0}}{y_{0}^{3}},\frac{y}{y_{0}}\Big)
\end{equation}
for some smooth function $V(t,x,y)\in C^\infty\big(\overline{\mathbb{H}}\big)$. Moreover, we have the large scale pointwise bounds
\begin{equation}\label{potwise}
|\Gamma(z;w)|\lesssim \|z-w\|^{-4}, \quad {\rm for}\,\, y_{0}\leq 10\|z-w\|.
\end{equation}
\end{proposition}
\begin{proof} By translation invariance (see Proposition \ref{proper}), we can assume without loss of generality that $z=(t,x,y)$ and $w=(0,0,y_{0})$.

\begin{comment}
Indeed, we can view $\Gamma$ as a \emph{Green function}. i.e.
\begin{equation}
\begin{split}
\Gamma(z;w)=&\text{Fundamental solution in the whole space}\ \R_{t}\times \R_{x}\times \R_{y}\\
 &\quad\quad + \text{correction due to the boundary condition}.
 \end{split}
\end{equation}
\end{comment}

We can write
\begin{equation}\label{decom}
\begin{split}
\Gamma(z;w)&=y_{0}^{-4}\Gamma(D(y_{0})z; 0,0,1)\\
&=y_{0}^{-4}\big(\widetilde{\Gamma}(D(y_{0})z; 0,0,1)+(\Gamma-\widetilde{\Gamma})(D(y_{0})z; 0,0,1)\big)\\
&=\widetilde{\Gamma}(z,w)+y_{0}^{-4}(\Gamma-\widetilde{\Gamma})(D(y_{0})z; 0,0,1).
\end{split}
\end{equation}

Let 
$$V(t,x,y)=(\Gamma-\widetilde{\Gamma})(t,x,y;0,0,1).$$ 
By \eqref{eqfundamental} and \eqref{wholespace}, $V$ satisfies the  equation 
\begin{equation}\label{Vequa}
\begin{cases}
\p_{t}V+y\p_{x}V-\p_{y}^{2}V=0,\quad t>0,x\in\mathbb{R},y>0;\\
\p_{y}V|_{y=0}=\displaystyle\frac{\sqrt{3}}{2\pi t^{2}}\Big(\frac{1}{t}-\frac{3x}{t^{2}}\Big)\exp\left\{-\frac{1}{4t}-\frac{3}{t^{3}}(x-\frac{t}{2})^{2}\right\}=:g(t,x),\\
V|_{t=0}=0.
\end{cases}
\end{equation}

\textbf{Step 1: smoothness of $V$.}
We first prove that the solution $V(t,x,y)$ in \eqref{Vequa} and all its derivatives are uniformly bounded for $t\ge0, x\in\R, y\ge0$. Decompose $V(t,x,y)$ as
\begin{equation}
V(t,x,y)=\mathcal{V}(t,x,y)+g(t,x)\sqrt{1+t}\cdot\chi\Big(\frac{y}{\sqrt{1+t}}\Big),
\end{equation}
where $\chi(z)\in C_{c}^{\infty}([0,1))$ and $\chi'(0)=1$. By the expression of $g(t,x)$ in \eqref{Vequa} and simple calculations, we have
\begin{equation*}
\left|\nabla_{(t,x,y)}^{k}\Big(g(t,x)\sqrt{1+t}\cdot\chi\Big(\frac{y}{\sqrt{1+t}}\Big)\Big)\right|\lesssim_{k} t^{-2}e^{-\frac{1}{10t}},\ \forall k\geq 0.
\end{equation*}
Thus we only need to show $\mathcal{V}$ and its derivatives are bounded.

Thanks to \eqref{Vequa} and $g|_{t=0}\equiv 0$, we see that $\mathcal{V}(t,x,y)$ satisfies the following equation
\begin{equation}\label{Vequ}
\begin{cases}
\p_{t}\mathcal{V}+y\p_{x}\mathcal{V}-\p_{y}^{2}\mathcal{V}=\mathcal{G}(t,x,y),\quad t>0,x\in\mathbb{R},y>0;\\
\p_{y}\mathcal{V}|_{y=0}=0,\mathcal{V}|_{t=0}=0,\\
\end{cases}
\end{equation}
where
\begin{equation}
\begin{split}
\mathcal{G}(t,x,y)=&\,\frac{1}{2}g(t,x)\frac{y}{1+t}\chi'\Big(\frac{y}{\sqrt{1+t}}\Big)+g(t,x)(1+t)^{-\frac{1}{2}}\chi''\Big(\frac{y}{\sqrt{1+t}}\Big)\\
&-\Big[g_{t}(t,x)\sqrt{1+t}+\frac{1}{2}g(t,x)(1+t)^{-\frac{1}{2}}+yg_{x}(t,x)\sqrt{1+t}\Big]\chi\Big(\frac{y}{\sqrt{1+t}}\Big).
\end{split}
\end{equation}
By the standard estimate for the transport--diffusion equation, we have
\begin{equation}\label{boundnessofv}
|\mathcal{V}(t,x,y)|\leq \int_{0}^{t}\|\mathcal{G}(s)\|_{L^{\infty}_{x,y}}ds\lesssim \min\{t,1\}.
\end{equation}
We can also obtain the bounds for the higher--order derivatives of $V$ using the same approach, by taking $t,x$ derivatives in \eqref{Vequ} and using \eqref{boundnessofv} to control the $t,x$ derivatives, and then using equation \eqref{Vequ} to control the $y$ derivatives.

\textbf{Step 2: Proof of \eqref{potwise}}.  By the translation invariance of $\Gamma$, we can assume $w=(0,0,y_{0})$ for some $y_{0}>0$; furthermore by the dilation invariance of $\Gamma$, we have
\begin{equation*}
\Gamma(z;w)=\|z-w\|^{-4}\Gamma(D(\|z-w\|)z, D(\|z-w\|)w).
\end{equation*}
Hence, we only need to show that
\begin{equation}\label{goal}
\sup_{\|z-w\|=1}|\Gamma(z;w)|\lesssim 1,\ \text{uniformly in}\  y_{0}\in(0,10].
\end{equation}

We consider two separate cases.

\textit{Case 1: $10\geq y_{0}>\frac{1}{100}$.}  \eqref{decom} implies that
\begin{equation}
\Gamma(t,x,y;0,0,y_{0})=\widetilde{\Gamma}(t,x,y;0,0,y_{0})+y_{0}^{-4}V\left(D(y_{0})z\right).
\end{equation}

By \eqref{wholespace} and \eqref{boundnessofv}, we obtain that
\begin{equation}\label{wholespacebd}
\sup_{\|z-w\|=1, y_{0}\leq 10}|\widetilde{\Gamma}(z,w)|\lesssim 1.
\end{equation}

Hence, \eqref{goal} follows from \eqref{boundnessofv} and \eqref{wholespacebd}.

\textit{Case 2: $y_{0}\leq \frac{1}{100}$.} This case is more complicated since the source is close to the boundary and the bounds on $V$ from the last step are not sufficient for our purposes. Notice that if $y_{0}\leq \frac{1}{100}$ and $\|z-(0,0,y_{0})\|=1$, then at least one of the following three cases must happen: 
(i) $(t,x,y)\in\Omega_{1}:=\left\{\frac{1}{2}\leq y\leq 2,\ 0\leq t\leq 1,\ |x|\leq 1\right\}$; (ii) $(t,x,y)\in\Omega_{2}:=\left\{\frac{1}{2}\leq t \leq 1,\ 0\leq y\leq 1,\ |x|\leq 1\right\}$;
(iii) $(t,x,y)\in\Omega_{3}:=\left\{\frac{1}{2}\leq|x|\leq 1,\ 0\leq y\leq 1, 0\leq t\leq 1\right\}$.

We distinguish three subcases depending on (i) - (iii). 
 
 \textit{Subcase 2.1: $(t,x,y)\in\Omega_{1}.$}
By the non-negativity and conservation of mass, we have  
\begin{equation}\label{l1conser}
\int_{\R\times\R^{+}}\Gamma(z;w)\,dxdy\equiv 1,\quad \text{for}\,\, t>0,y_{0}>0,
\end{equation}
Hence, in this case, the desired bound  \eqref{goal} follows from the uniform bound \eqref{l1conser} and the interior estimate for the hypoelliptic operator (Proposition 3.2 in \cite{H67}) since $y\ge1/2$.

\textit{Subcase 2.2: $(t,x,y)\in\Omega_{2}$.} Taking the  Fourier transform in the $x$--variable, and denoting
\begin{equation*}
\widehat{\Gamma}(t,\xi,y;0,0,y_{0})=\int_{-\infty}^{\infty}\Gamma(t,x,y;0,0,y_{0})e^{-2\pi i x\xi}dx,
\end{equation*}
 $\widehat{\Gamma}$ satisfies the equation for $y\ge0, t\ge0$,
\begin{equation}\label{ftx}
\begin{cases}
\p_{t}\widehat{\Gamma}+2\pi \mathrm{i} y\xi\widehat{\Gamma}-\p_{y}^{2}\hat{\Gamma}=0,\\
\widehat{\Gamma}|_{t=0}=\delta(y-y_{0}),\ \p_{y}\hat{\Gamma}|_{y=0}=0.
\end{cases}
\end{equation}
By \eqref{l1conser}, we have
\begin{equation}\label{l1conserft}
\sup_{t>0,\,\xi\in\mathbb{R}}\int_{0}^{\infty}|\widehat{\Gamma}(t,\xi,y;0,0,y_0)|\,dy\lesssim 1.
\end{equation}

Multiplying \eqref{ftx} by the complex conjugate of $\widehat{\Gamma}$, integrating in $y$, and taking the real part, we obtain the following energy inequality:
\begin{equation}\label{diffineq}
\frac{1}{2}\frac{d}{dt}\|\widehat{\Gamma}(t,\xi,y;w)\|^2_{L^{2}_{y}}+\|\p_{y}\widehat{\Gamma}(t,\xi,y;w)\|^2_{L^{2}_{y}}\leq 0.
\end{equation}
By Nash's inequality (see \cite{N58}) 
\begin{equation}
\int_{0}^{\infty}|F(y)|^{2}dy\lesssim \left(\int_{0}^{\infty}|\p_{y}F|^{2}dy\right)^{\frac{1}{3}}\left(\int_{0}^{\infty}|F(y)|dy\right)^{\frac{4}{3}},
\end{equation}
and \eqref{l1conserft}, we have
\begin{equation}\label{nash}
\int_{0}^{\infty}|\widehat{\Gamma}(t,\xi,y;w)|^{2}dy\lesssim \left(\int_{0}^{\infty}|\p_{y}\widehat{\Gamma}(t,\xi,y;w)|^{2}dy\right)^{\frac{1}{3}}.
\end{equation}

Combining \eqref{diffineq} and \eqref{nash}, we conclude that there exists a constant $C$ such that for all $\xi\in\mathbb{R}$ and $y_{0}>0$,
\begin{equation}\label{ftl2bd}
\|\widehat{\Gamma}(t,\xi,y;w)\|_{L^{2}_{y}}\leq  Ct^{-\frac{1}{4}}.
\end{equation}

\begin{comment}
Next, let $\widetilde{\Gamma}(t,\xi,y)=\hat{\Gamma}(t+\frac{1}{3},\xi,y)$, then we have
\begin{equation}\label{ftx2}
\begin{cases}
\p_{t}\widetilde{\Gamma}+2\pi \mathrm{i} y\xi\widetilde{\Gamma}-\p_{y}^{2}\widetilde{\Gamma}=0,\\
\widetilde{\Gamma}|_{t=0}=\hat{\Gamma}|_{t=\frac{1}{3}},\ \p_{y}\hat{\Gamma}|_{y=0}=0.
\end{cases}
\end{equation}
\end{comment}

By \eqref{ftl2bd} and enhanced dissipation estimates in Proposition \ref{enhanced}, we have for $t\ge 1/3,\, y_0>0$ and some $c>0$,
\begin{equation}\label{gerveyesti1}
\|\widehat{\Gamma}(t,\xi,y;w)\|_{L^{2}_{y}}\lesssim e^{-c|\xi|^{\frac{2}{3}}t}.
\end{equation}
By \eqref{gerveyesti1} and the standard energy estimate of the heat equation, denoting $D:=\{t\ge1/2, y\ge0\}$, we also have 
\begin{equation}\label{gerveyesti2}
\|\p_{t}\widehat{\Gamma}(t,\xi,y;w)(1+y)^{-1}\|_{L^2(D)}+\|\p_{y}^{2}\widehat{\Gamma}(t,\xi,y;w)(1+y)^{-1}\|_{L^2(D)}\lesssim e^{-c|\xi|^{\frac{2}{3}}}.
\end{equation}
The desired bound \eqref{goal} in this case then follows from \eqref{gerveyesti1}, \eqref{gerveyesti2}, Fourier inversion and Sobolev inequality. 

\textit{Subcase 2.3: $(t,x,y)\in\Omega_{3}$.} This is perhaps the most complicated case since the bound \eqref{ftl2bd} is not very effective for small $t>0$. To overcome this difficulty, let $\chi(x)$ be a smooth cutoff function in the $x$ variable such that
\begin{equation*}
\mathrm{supp}\,\chi\subset\{x:1/10\leq |x|\leq 3/2\},\ \chi(x)\equiv 1\ \text{in}\ \{x: 1/5\leq |x|\leq 5/4\}.
\end{equation*}
Let $\Gamma_{1}(t,x,y;0,0,y_{0})=\Gamma(t,x,y;0,0,y_{0})\chi(x)$. Then $\Gamma_{1}$ satisfies the following equation
\begin{equation}\label{gamma1eq}
\begin{cases}
\p_{t}\Gamma_{1}+y\p_{x}\Gamma_{1}-\p_{y}^{2}\Gamma_{1}=y\chi'(x)\Gamma,\\
\Gamma_{1}|_{t=0}=0,\ \p_{y}\Gamma_{1}|_{y=0}=0.
\end{cases}
\end{equation}
Taking the Fourier transform in $x$, we get that
\begin{equation}\label{gamma1fteq}
\begin{cases}
\p_{t}\widehat{\Gamma}_{1}+2\pi\mathrm{i}y\xi\widehat{\Gamma}_{1}-\p_{y}^{2}\widehat{\Gamma}_{1}=y\,\widehat{\chi'\Gamma},\\
\widehat{\Gamma}_{1}|_{t=0}=0,\ \p_{y}\widehat{\Gamma}_{1}|_{y=0}=0.
\end{cases}
\end{equation}
Let
\begin{equation}\label{eqphi}
\phi(y)=\frac{1}{1+y}-\frac{1}{2(1+y)^{2}}.
\end{equation}
Clearly, $\phi'(0)=0$ and $\phi(y)\sim (1+y)^{-1}$.

Multiplying \eqref{gamma1fteq} by the complex conjugate of 
$\widehat{\Gamma}_{1}\phi^{2}(y)$, integrating and taking the real part, we get the following energy identity
\begin{equation}\label{ftenergyid}
\begin{split}
&\frac{1}{2}\frac{d}{dt}\left(\int_{0}^{\infty}|\widehat{\Gamma}_{1}(t,\xi,y;w)|^{2}\phi^{2}(y)dy\right)+\int_{0}^{\infty}|\p_{y}\widehat{\Gamma}_{1}(t,\xi,y;w)|^{2}\phi^{2}(y)dy\\
&+2\mathrm{Re}\left(\int_{0}^{\infty}\p_{y}\widehat{\Gamma}_{1}(t,\xi,y;w)\cdot\overline{\widehat{\Gamma}_{1}}(t,\xi,y;w)\phi(y)\phi'(y)dy\right)\\
&=\mathrm{Re}\left(\int_{0}^{\infty}y\,\widehat{\chi'\Gamma}(t,\xi,y;w)\overline{\widehat{\Gamma}_{1}}(t,\xi,y;w)\phi^{2}(y)dy\right).
\end{split}
\end{equation}
Notice from \eqref{ftl2bd} that
\begin{equation}
\begin{split}
&\left|\int_{0}^{\infty}y\,\widehat{\chi'\Gamma}(t,\xi,y;w)\overline{\widehat{\Gamma}_{1}}(t,\xi,y;w)\phi^{2}(y)dy\right|\lesssim t^{-\frac{1}{4}} \|\widehat{\Gamma}_{1}(t,\xi,y;w)\phi(y)\|_{L^{2}_{y}}.
\end{split}
\end{equation}
%\\&\lesssim  \|\widehat{\Gamma}_{1}(t,\xi,y;w)\phi(y)\|_{L^{2}_{y}}\|\widehat{\Gamma}(t,\xi,y;w)\|_{L^{2}_{y}}
Therefore, by \eqref{ftenergyid} and Gronwall's inequality, for any $T>0$ we have
\begin{equation}\label{ftenergybd}
\sup_{t\in[0,T]}\left(\int_{0}^{\infty}|\widehat{\Gamma}_{1}(t,\xi,y;w)|^{2}\phi^{2}(y)\,dy\right)+\int_{0}^{T}\int_{0}^{\infty}|\p_{y}\widehat{\Gamma}_{1}(t,\xi,y;w)|^{2}\phi^{2}(y)\,dydt\lesssim_T1.
\end{equation}

Set $\Gamma^{*}=\widehat{\Gamma}_{1}(t,\xi,y;w)\phi(y)$. Then we have 
\begin{equation}\label{gammastareq}
\begin{cases}
\p_{t}\Gamma^{*}+2\pi \mathrm{i}y\xi\Gamma^{*}-\p_{y}^{2}\Gamma^{*}=y\phi(y)\widehat{\chi'\Gamma}-2\p_{y}\widehat{\Gamma}_{1}\phi'(y)-\widehat{\Gamma}_{1}\phi''(y),\\
\Gamma^{*}|_{t=0}=0,\ \p_{y}\Gamma^{*}|_{y=0}=0.
\end{cases}
\end{equation}
If $|\xi|\geq 1$, by Duhamel's principle and the enhanced dissipation estimate (Proposition \ref{enhanced}), we get that
\begin{equation}\label{eq:que1}
\|\Gamma^{*}(t,\xi,y;w)\|_{L^{2}_{y}}\lesssim \int_{0}^{t}e^{-c|\xi|^{\frac{2}{3}}(t-s)}\left\|y\phi(y)\widehat{\chi'\Gamma}-2\p_{y}\widehat{\Gamma}_{1}\phi'(y)-\widehat{\Gamma}_{1}\phi''(y)\right\|_{L^{2}_{y}}ds,
\end{equation}
for some $c>0$. 

Using the bounds \eqref{ftl2bd} and \eqref{ftenergybd}, together with the pointwise inequality $|\phi'|+|\phi''|\lesssim|\phi|$, we obtain that
\begin{equation}\label{gammaonebd}
\int_{0}^{T}\int_{0}^{\infty}|\widehat{\Gamma}_{1}(t,\xi,y;w)|^{2}\phi^{2}(y)\,dydt\lesssim_{T} (1+|\xi|)^{-\frac{4}{3}}.
\end{equation}

We next set $\Gamma_{2}=\Gamma(t,x,y;0,0,y_{0})\chi_{2}(x)$, where $\chi_{2}$ is a standard smooth cutoff function with
\begin{equation*}
\mathrm{supp}\,\chi_{2}\subset\{x:1/5\leq|x|\leq 5/4\},\ \chi_{2}(x)\equiv 1\ \text{in}\ \{1/4\leq|x|\leq 6/5\}.
\end{equation*}
Then $\Gamma_{2}$ satisfies
\begin{equation}\label{gamma2eq}
\begin{cases}
\p_{t}\Gamma_{2}+y\p_{x}\Gamma_{2}-\p_{y}^{2}\Gamma_{2}=y\chi_{2}'(x)\Gamma_{1},\\
\Gamma_{2}|_{t=0}=0,\ \p_{y}\Gamma_{2}|_{y=0}=0,
\end{cases}
\end{equation}
and correspondingly,
\begin{equation}\label{gamma2fteq}
\begin{cases}
\p_{t}\widehat{\Gamma}_{2}+2\pi\mathrm{i}y\xi\,\widehat{\Gamma}_{2}-\p_{y}^{2}\widehat{\Gamma}_{2}=y\,\widehat{\chi_{2}'\Gamma_{1}},\\
\widehat{\Gamma}_{2}|_{t=0}=0,\ \p_{y}\widehat{\Gamma}_{2}|_{y=0}=0.
\end{cases}
\end{equation}
By \eqref{gammaonebd} and the fact that $\widehat{\chi_{2}}\in\mathcal{S}_{\xi}(\R)$, we have
\begin{equation}\label{gammaonebddd}
\begin{split}
        &\int_{0}^{T}\int_{0}^{\infty}|\widehat{\chi_{2}'\Gamma_{1}}(t,\xi,y)|^{2}\phi^{2}(y)dydt\\
        &\quad\lesssim\int_{0}^{T}\int_{0}^{\infty}\left|\int_{-\infty}^{\infty}\lambda\cdot\widehat{\chi_{2}}(\lambda)\widehat{\Gamma_{1}}(t,\xi-\lambda,y)d\lambda\right|^{2}\phi^{2}(y)dydt\\
        &\quad\lesssim \|\lambda^{2}\widehat{\chi_{2}}(\lambda)\|_{L^{1}_{\lambda}}\cdot\int_{0}^{T}\int_{0}^{\infty}\int_{-\infty}^{\infty}|\widehat{\chi_{2}}(\lambda)|\cdot|\widehat{\Gamma_{1}}(t,\xi-\lambda,y)|^{2}\phi^{2}(y)d\lambda dydt\\
        &\quad\lesssim\int_{-\infty}^{\infty}|\widehat{\chi_{2}}(\lambda)|\cdot(1+|\xi-\lambda|)^{-\frac{4}{3}}d\lambda\lesssim (1+|\xi|)^{-\frac{4}{3}}.
\end{split}
\end{equation}
By the same argument as in the derivation of  \eqref{gammaonebd} and  using \eqref{gammaonebddd}, we obtain 
\begin{equation}\label{gammatwobd}
\int_{0}^{T}\int_{0}^{\infty}|\widehat{\Gamma}_{2}(t,\xi,y;w)|^{2}\phi^{4}(y)\,dydt\lesssim _{T}(1+|\xi|)^{-\frac{8}{3}}.
\end{equation}

By induction, for any $k\in\mathbb{N}^{+}$, we obtain that
\begin{equation}\label{gammakbd}
\int_{0}^{T}\int_{0}^{\infty}|\widehat{\Gamma}_{k}(t,\xi,y;w)|^{2}\phi^{2k}(y)\,dydt\lesssim _{k,T}(1+|\xi|)^{-\frac{4k}{3}},
\end{equation}
 where $\Gamma_{k}=\Gamma(t,x,y;0,0,y_{0})\chi_{k}(x)$, with
 \begin{equation}
 \mathrm{supp}\,\chi_{k}\subset\{x:\chi_{k-1}(x)=1\},\ \text{and}\ \chi_{k}(x)\equiv 1\ \text{in}\ \{2/5\leq |x|\leq 11/10\}\ \text{for all}\ k\in\mathbb{N}^{+},\,k\ge2.
 \end{equation}
 As a consequence, we have
 \begin{equation}\label{gammaphybd}
 \int_{0}^{T}\int_{0}^{\infty}|(\p_{x}^{k}\Gamma)(t,x,y;0,0,y_{0})|^{2}(1+|y|)^{-8k}dydt\lesssim_{k,T}1,\ \forall k\in\mathbb{N},\ \frac{1}{2}\leq|x|\leq 1.
 \end{equation}
 
 Finally, for fixed $\frac{1}{2}\leq |x|\leq 1$, $\Gamma=\Gamma(t,x,y;0,0,y_{0})$ satisfies the following one-dimensional inhomogeneous heat equation
 \begin{equation}
 \begin{cases}
 \p_{t}\Gamma-\p_{y}^{2}\Gamma=-y\p_{x}\Gamma,\ t> 0, y>0,\\
 \Gamma|_{t=0}=0,\ \p_{y}\Gamma|_{y=0}=0.
 \end{cases}
 \end{equation}
 By \eqref{gammaphybd} and properties of the heat equation, for any $\frac{1}{2}\leq|x|\leq 1$, we have
 \begin{equation}\label{gammaphysbd}
 \int_{0}^{T}\int_{0}^{1}\left(|\Gamma|^{2}+|\p_{t}\Gamma|^{2}+|\p_{y}^{2}\Gamma|^{2}\right)dydt\lesssim_{T} 1.
 \end{equation}
Therefore, the desired bound \eqref{goal} in this case follows from \eqref{gammaphybd}, \eqref{gammaphysbd} and Sobolev inequalities. This completes the proof of Proposition \ref{ptwise}.

\end{proof}

Similar arguments also imply the following lemma. 
\begin{lemma}\label{lm:h1}
Assume that $z=(t,x,y)\in\mathbb{H},\, w=(t_0,x_0,y_0)\in \mathbb{H}$, $z\neq w$. The following higher regularity estimates for $\Gamma$ hold for any $\alpha, \beta, \gamma\in \Z\cap[0,\infty)$ and $y_{0}\leq10 \|z-w\|$,
\begin{equation}\label{higherordertxy}
|\p_{t}^{\alpha}\p_{x}^{\beta}\p_{y}^{\gamma}\Gamma(z;w)|\lesssim_{\alpha,\beta,\gamma} \|z-w\|^{-4-(2\alpha+3\beta+\gamma)}.
\end{equation}
\end{lemma}

We can extend the regularity bounds in Lemma \ref{lm:h1} to all variables $(z,w)$, as a consequence of the translation and dilation invariance of $\Gamma(z;w)$, the regularity bounds for $(t,x,y)$ variables \eqref{higherordertxy}, and the dual equation \eqref{dualeq}. %The proof is direct, hence we omit it.

\begin{proposition}[\bf Higher--order Pointwise Bounds]\label{highptwise}
Assume $z=(t,x,y)$, $w=(t_{0},x_{0},y_{0})$ and $z\neq w$. Then
\begin{equation}\label{highpotwise}
|\p_{t}^{\alpha_{1}}\p_{t_{0}}^{\alpha_{2}}\p_{x}^{\beta_{1}}\p_{x_{0}}^{\beta_{2}}\p_{y}^{\gamma_{1}}\p_{y_{0}}^{\gamma_{2}}\Gamma(z;w)|\lesssim \|z-w\|^{-4-\left[2(\alpha_{1}+\alpha_{2})+3(\beta_{1}+\beta_{2})+(\gamma_{1}+\gamma_{2})\right]},
\end{equation}
for $y_{0}\leq 10\|z-w\|$.
\end{proposition}

%\begin{comment}
\begin{proof}The proof is similar to that of Proposition \ref{ptwise}. For completeness we outline the proof below.

By the argument in Proposition \ref{ptwise}, we need to show that for $y_{0}\leq 10\|z-w\|$,
\begin{equation}\label{highgoal}
\sup_{\|z-w\|=1}\left|\p_{t}^{\alpha_{1}}\p_{t_{0}}^{\alpha_{2}}\p_{x}^{\beta_{1}}\p_{x_{0}}^{\beta_{2}}\p_{y}^{\gamma_{1}}\p_{y_{0}}^{\gamma_{2}}\Gamma(z;w)\right|\lesssim 1.
\end{equation}
By Proposition \ref{proper}, we only need to consider the case $\alpha_{2}=\beta_{2}=0$. Since $\mathcal{L}_{0}$ commutes with $\p_{t}$ and $\p_{x}$, we can obtain \eqref{highgoal} by similar arguments as in Proposition \ref{ptwise} when $\gamma_{1}=\gamma_{2}=0$.

Next, notice that
\begin{equation}
(\p_{t}+y\p_{x}-\p_{y}^{2})\Gamma(z;w)=(\p_{t_{0}}+y_{0}\p_{x_{0}}+\p_{y_{0}}^{2})\Gamma(z;w)=0,\quad z\neq w.
\end{equation}
Hence \eqref{highgoal} holds when $\gamma_{1}$ and $\gamma_{2}$ are even. 

When one of $\gamma_{i} (i=1,2)$ is odd, we can interpolate in the $y$--direction and apply \eqref{highgoal} when $\gamma_{i} (i=1,2)$ are both even. Finally, we can interpolate again to handle the case when $\gamma_{i} (i=1,2)$ are both odd. The proof of Proposition \ref{highptwise} is then complete.
\end{proof}
%\end{comment}

\begin{remark} By Proposition \ref{highptwise}, if $t\neq t_{0}$ or $x\neq x_{0}$, the trace
\begin{equation}
\Gamma(t,x,y;t_0,x_0,y_{0})|_{y_{0}=0}
\end{equation}
is well defined.
Moreover, using \eqref{consermass} and \eqref{highpotwise}, we obtain that
\begin{equation}
\iint_{x\in\R,\,y\in\R^{+}}\Gamma(t,x,y;t_0,x_0,0)\,dxdy\equiv 1,\ \forall t>t_0,x_0\in\R,
\end{equation}
which is useful below in demonstrating the (strict) positivity of $\Gamma$.
\end{remark}

For applications below, we establish the following potential estimates as a consequence of the pointwise bound on $\Gamma(z,w)$.
\begin{proposition}[\bf Potential Estimate]\label{operatorbd} Assume $f(w)\in L^{2}(\mathbb{H})$. Then we have
\begin{align}
\left\|\int \Gamma(z;w)f(w)dw\right\|_{L^{6}_{z}}+\left\|\int (\p_{y},\p_{y_{0}})\Gamma(z;w)f(w)dw\right\|_{L^{3}_{z}}\lesssim \|f\|_{L^{2}(\mathbb{H})}.
\end{align}
\end{proposition}
\begin{proof}By Proposition \ref{ptwise} and Proposition \ref{highptwise} we have
\begin{equation}
\begin{split}
&0\leq \Gamma(z;w)\leq \widetilde{\Gamma}(z;w)+C_{1}\|z-w\|^{-4},\\
&|(\p_{y},\p_{y_{0}})\Gamma(z;w)|\leq |(\p_{y},\p_{y_{0}})\widetilde{\Gamma}(z;w)|+C_{2}\|z-w\|^{-5}.
\end{split}
\end{equation}

Recalling \eqref{volume}, we have
\begin{equation*}
\|\cdot\|^{-4}\in L^{\frac{3}{2},\infty},\ \|\cdot\|^{-5}\in L^{\frac{6}{5},\infty},
\end{equation*}
where $L^{p,q}$ is the classical Lorentz space with respect to the Lebesgue measure.

Therefore, Proposition \ref{operatorbd} follows from the potential estimate in the whole space (see e.g. \cite{PP04}, Corollary 2.2), as well as Young's inequality. 
\end{proof}

\subsection{Positivity property of the fundamental solution}\label{subsec2.3}
In this section we prove the positivity of the fundamental solution in suitable regions. In contrast to the fundamental solution $\widetilde{\Gamma}(t,x,y;t_0,x_0,y_0)$ in the whole space, $\Gamma(t,x,y;t_0,x_0,y_0)$ is not positive everywhere even for $t>t_0$. Indeed, by an elementary $L^{1}$ estimate, we can obtain that
\begin{equation}
\int_0^\infty\int_{-\infty}^{x_{0}-\epsilon}\Gamma(t,x,y;t_0,x_0,y_{0})\,dxdy \equiv 0,\ \forall \epsilon>0,\ y_{0}\in\R^{+}, \,\,t>t_0.
\end{equation}
Hence, $\Gamma(z;w)$ vanishes for all $x<x_{0}$. This is of course consistent with our physical intuition, as in the half space the term $\p_{t}+y\p_{x}$ transports quantities to the right but not to the left, resulting in the propagation of positivity only to the right. 

\begin{comment}
(See figure \ref{fig:one-side-transport}). 
\begin{figure}[htbp]
    \centering % 使图片居中
    
    \begin{tikzpicture}[>=Stealth]
        % 只画第一象限坐标轴（无刻度）
        \draw[->] (0,0) -- (5,0) node[right] {$x$};
        \draw[->] (0,0) -- (0,5) node[above] {$t$};

        % 右上的箭头
        \draw[->, thick] (2,2) -- (4,4);
        % 文字标注放在箭头正上方
        \node[align=left, font=\small, above] at (3,4) {transport direction: $(y,1)$};
    \end{tikzpicture}
    
    \caption{one-side transport} % 添加标题
    \label{fig:one-side-transport} % 可选的标签，便于引用
\end{figure}
\end{comment}
Nevertheless, after taking this effect into consideration, we are still able to demonstrate that $\Gamma(z;w)$ is positive provided $t> t_{0}$ and $x>x_{0}$. This positivity property is essential in proving the H\"older regularity of weak solutions below, using De Giorgi-Nash-Moser techniques.

\begin{proposition}[\bf Positivity]
\label{posi}
Assume that $z=(t,x,y)\in\overline{\mathbb{H}}$, $w=(t_{0},x_{0},y_{0})\in\overline{\mathbb{H}}$ with $z\neq w$. For any $t>t_{0}, x>x_{0}, y\ge 0, y_{0}\ge0$, we have $\Gamma(z;w)>0$. 
\end{proposition}

\begin{proof}
By translation invariance, we can assume that $t_{0}=x_{0}=0$. We first note that $\Gamma(z;w)$ is smooth away from the initial time by Proposition \ref{ptwise} and Proposition \ref{highptwise}, and $\Gamma(z;w)$ is non--negative as proven previously. Next we will exclude the following four cases one by one:
\begin{equation*}
\begin{split}
&\mathrm{Case}\, (\alpha): \quad\Gamma(\bar{t},\bar{x},\bar{y};0,0,\bar{y}_{0})=0,\ \text{for some}\ \bar{t}>0,\bar{x}>0,\bar{y}>0,\bar{y}_{0}>0,\\
&\mathrm{Case}\, (\beta):\quad \Gamma(\bar{t},\bar{x},\bar{y};0,0,0)=0,\ \text{for some}\ \bar{t}>0,\bar{x}>0,\bar{y}>0,\\
&\mathrm{Case}\, (\gamma): \quad\Gamma(\bar{t},\bar{x},0;0,0,\bar{y}_{0})=0,\ \text{for some}\ \bar{t}>0,\bar{x}>0,\bar{y}_{0}>0,\\
&\mathrm{Case}\, (\theta): \quad \Gamma(\bar{t},\bar{x},0;0,0,0)=0,\ \text{for some}\ \bar{t}>0,\bar{x}>0.
\end{split}
\end{equation*}

We divide the proof into several steps.

\textit{Step 1: excluding case ($\alpha$).} By the dilation invariance of $\Gamma(\cdot,\cdot)$ (Proposition \ref{proper}), we can assume that $\bar{t}=1$. Our first important observation is that there exists a universal constant $M>0$ such that
\begin{equation}\label{eq:po1}
\Gamma(t,x,xt^{-1};0,0,xt^{-1})>0,
\end{equation}
for all $0<t\leq 1$ and $x\geq Mt^{\frac{3}{2}}$. 

To see \eqref{eq:po1}, we recall the decomposition \eqref{greendecom}. By \eqref{wholespace}, we have
\begin{equation*}
\begin{split}
\Gamma(t,x,xt^{-1};0,0,xt^{-1})&\geq\widetilde{\Gamma}(t,x,xt^{-1};0,0,xt^{-1})-\|V\|_{L^{\infty}}\cdot x^{-4}t^{4}\\
    &=\frac{\sqrt{3}}{2\pi t^{2}}-\|V\|_{L^{\infty}}\cdot x^{-4}t^{4}
\end{split}
\end{equation*}
Hence \eqref{eq:po1} holds for all $0<t\leq 1$ and $x\geq Mt^{\frac{3}{2}}$. 

Applying the strong maximum principle due to Nirenberg \cite{N53} in the $y$--direction as well as in the $y_{0}$--direction (by considering the dual equation of $\Gamma(\cdot,w)$ in the $w$--variable), we have
\begin{equation}\label{positivedomain}
\Gamma(t,x,y;0,0,y_{0})>0,\ \forall\ 0<t\leq 1,\ x>Mt^{\frac{3}{2}},\ y>0,\ y_{0}>0.
\end{equation}

Next, define the dyadic time sequence
\begin{equation*}
t_{k}=2^{-k},\ \ k\geq 0.
\end{equation*}
By \eqref{positivedomain}, taking an integer $N\gg1$, we have
\begin{equation}
\Gamma(t_{N},x,y;0,0,y_{0})>0,\ \forall\ x>M\cdot 2^{-\frac{3}{2}N},\ y>0,\ y_{0}>0.
\end{equation}
By the reproducing formula and the translation invariance in time of the fundamental solution, for any $x>2M\cdot 2^{-\frac{3}{2}N}$, and $y>0,y_{0}>0$, we have
\begin{equation}\label{eqreproduce}
\begin{split}
\Gamma(t_{N-1},x,y;0,0,y_{0})&=\int_{0}^\infty\int_{0}^x\Gamma(t_{N},x-\xi,y;0,0,\eta)\Gamma(t_{N},\xi,\eta;0,0,y_{0})\,d\xi d\eta\\
&\geq \int_{0}^\infty\int_{ 2^{-3N/2}M}^{x- 2^{-3N/2}M}\Gamma(t_{N},x-\xi,y;0,0,\eta)\Gamma(t_{N},\xi,\eta;0,0,y_{0})\,d\xi d\eta\\
&>0.
\end{split}\end{equation}

Repeating the iteration procedure for $N$ times, we have
\begin{equation}
\Gamma(1,x,y;0,0,y_{0})>0,\ \forall x>M\cdot 2^{-\frac{1}{2}N},\ y,y_{0}>0,
\end{equation}
contradiction with case ($\alpha$), provided that $N\gg 1$ is chosen sufficiently large.

\textit{Step 2: excluding case ($\beta$).} If case ($\beta$) happens, then by the reproducing formula of the fundamental solution, we obtain that for all $0<\tau<\bar{t}$,
\begin{equation*}
0=\Gamma(\bar{t},\bar{x},\bar{y};0,0,0)=\int_{0<\xi<\bar{x}}\int_{\eta>0}\Gamma(\bar{t}-\tau,\bar{x}-\xi,\bar{y};0,0,\eta)\Gamma(\tau,\xi,\eta;0,0,0)\,d\xi d\eta.
\end{equation*}
Since the case ($\alpha$) is already excluded, we have for all $0<\tau<\bar{t},0<\xi<\bar{x}$ and $\eta>0$ that $\Gamma(\tau,\xi,\eta;0,0,0)=0$.

By the dilation invariance of $\Gamma$, it follows that
\begin{equation*}
\Gamma(t,x,y;0,0,0)\equiv 0,\ \forall t>0,x>0,y>0,
\end{equation*}
which contradicts the fact that
\begin{equation*}
\int_{0}^{\infty}\int_{0}^{\infty}\Gamma(t,x,y;0,0,0)\,dxdy=\int_{-\infty}^{\infty}\int_{0}^{\infty}\Gamma(t,x,y;0,0,0)\,dxdy\equiv 1,\ \forall t>0.
\end{equation*}

\textit{Step 3: excluding case ($\gamma$).} Notice that
\begin{equation*}
\p_{y}\Gamma|_{y=0}\equiv 0.
\end{equation*}
Case $\gamma$ can then be ruled out using the classical Hopf's boundary-point lemma for a degenerate elliptic operator on non-characteristic boundary points, see Lemma 7.1.7 in \cite{Ta88} and also the reference therein. 

\textit{Step 4: excluding case ($\theta$).} It follows from Step 2 and Step 3 that  
\begin{equation*}
\Gamma(\bar{t},\bar{x},0;0,0,0)=\int_{0<\xi<\bar{x}}\int_{\eta>0}\Gamma(\frac{\bar{t}}{2},\bar{x}-\xi,0;0,0,\eta)\Gamma(\frac{\bar{t}}{2},\xi,\eta;0,0,0)d\xi d\eta>0.
\end{equation*}
Combining Steps 1 -- 4, the proof of Proposition \ref{posi} is then complete.
\end{proof}

\begin{remark}
An alternative approach to prove Proposition \ref{posi} may be obtained by applying a more sophisticated and powerful maximum principle for hypoelliptic and ultraparabolic operators due to Bony \cite{Bo69} (see also \cite{Ta88} and \cite{AE04}). Our proof has the advantage of being more direct and elementary in this specific case.
\end{remark}

The following localized quantitative positivity estimate is a consequence of Proposition \ref{posi}, and is important for proving regularity of weak solutions to a hypoelliptic equation in the half space in Section \ref{continuity}.
\begin{corollary}\label{positive}
Assume $z=(t,x,y), w=(t_{0},x_{0},y_{0})\in\overline{\mathbb{H}}$, and $r>0$. Let $0<\alpha<\beta,\ \delta\geq\gamma>\kappa>0, \iota>0$ and $M>0$. If $\|z\|,\|w\|\leq Mr$, and 
\begin{equation*}
t\geq -\alpha r^{2},\ t_{0}\leq -\beta r^{2},\ \gamma r^{3}\leq x\leq \delta r^{3}, |x_{0}|\leq\kappa r^{3},\ 0\leq y\leq \iota r,\ 0\leq y_{0}\leq \iota r,
\end{equation*}
then
\begin{equation}
\Gamma(z;w)\gtrsim_{\alpha,\beta,\gamma,\delta,\kappa,\iota,M}r^{-4}.
\end{equation}
\end{corollary}
\begin{proof}
If $(z,w)$ satisfies the assumptions, then by dilation invariance (see Proposition \ref{proper}),
\begin{equation}
\Gamma(z;w)\geq r^{-4}\min_{(Z,W)\in\mathcal{Q}}\Gamma(Z;W),
\end{equation}
where $\mathcal{Q}$ is the compact subset defined by
\begin{equation*}
\begin{split}
\mathcal{Q}=\Big\{Z=(t,x,y),\ W=(t_{0},x_{0},y_{0}):\,\, &\|Z\|\leq M, \|W\|\leq M, \\
&t\geq-\alpha,t_{0}\leq-\beta,\gamma\leq x\leq \delta, |x_{0}|\leq \kappa, 0\leq y,y_{0}\leq \iota \Big\}.
\end{split}
\end{equation*}
By Proposition \ref{posi}, $\Gamma(Z;W)$ is positive in $\mathcal{Q}$. By Proposition \ref{ptwise} and Proposition \ref{highptwise}, $\Gamma(Z;W)$ is continuous in $Q$. Hence, the continuity of $\Gamma(Z;W)$ on the compact set $Q$ implies that the positive minimum is attained. The proof of Corollary \ref{positive} then follows.
\end{proof}

%section 4

\section{H\"older regularity of weak solutions to ultraparabolic equations in the half space}\label{continuity}
In this section, we prove that weak solutions that satisfy an ultraparabolic equation in divergence form are H\"older continuous \emph{up to the boundary}. The corresponding H\"older regularity in the interior of such solutions was proved in \cite{XZZ24}. Up-to-boundary H\"older estimates are crucial to establish the boundary regularity of weak solutions of the dynamical Prandtl equation (in the coordinate system introduced by Crocco). 

For $z_{0}=(t_{0},x_{0},0)\in \p\mathbb{H}$ and $r_{0}>0$, define the up--to--boundary ultraparabolic ball
\begin{equation}\label{eq:paball}
\mathcal{B}_{r_{0}}^{+}(z_{0}):= (t_{0}-r_{0}^{2},t_{0}]\times (x_{0}-r_{0}^{3},x_{0}+r_{0}^{3})\times [0,r_{0})
\end{equation}
and we denote $\mathcal{B}_{r_{0}}^{+}:=\mathcal{B}_{r_{0}}^{+}(0,0,0)$ for brevity.

Our main result of this section is the following proposition.
\begin{proposition}\label{holder}
Let $z_{0}=(t_{0},x_{0},0)\in \p\mathbb{H}$ and $r_{0}>0$. Assume that $w(t,x,y)\in L^{\infty}\cap W^{1,1}\cap W^{2,1}_{y}(\mathcal{B}_{r_{0}}^{+}(z_{0}))$ is a weak solution (in the sense of distributions) to
\begin{equation}\label{diverform}
\begin{cases}
\p_{t}w+y\p_{x}w-\p_{y}(a\p_{y}w)=0,\ \text{ in}\ \mathcal{B}_{r_{0}}^{+}(z_{0}),\\
\p_{y}w|_{y=0}=0,
\end{cases}
\end{equation}
where $a(t,x,y)\in L^{\infty}(\mathcal{B}_{r_{0}}^{+}(z_{0}))$ satisfies the uniform ellipticity condition for $(t,x,y)\in \mathcal{B}_{r_{0}}^{+}(z_{0})$,
\begin{equation}\label{elliptic}
c\leq a(t,x,y)\leq c^{-1}
\end{equation}
for some $c\in(0,\frac{1}{2}]$. Then we have $w(t,x,y)\in C^{\alpha}(\mathcal{B}_{\frac{r_{0}}{2}}^{+}(z_{0}))$ and the bounds
\begin{equation}\label{holderesti}
 \|w\|_{C^{\alpha}(\mathcal{B}_{\frac{r_{0}}{2}}^{+}(z_{0}))}\lesssim_{c,r_{0}}\|w\|_{L^{\infty}(\mathcal{B}_{r_{0}}^{+}(z_{0}))},
\end{equation}
for some $\alpha>0$ independent of the solution $w$.
\end{proposition}

\begin{comment}
\begin{definition}\label{defofweaksol}
We say that $w\in L^{\infty}\cap W^{1,1}\cap W^{2,1}_{y}(\mathcal{B}_{r_{0}}^{+}(z_{0}))$ is the weak solution of \eqref{diverform}, provided
\begin{itemize}
    \item[(i)] For any test function $\phi(t,x,y)\in C_{c}^{1}(\mathcal{B}_{r_{0}}^{+}(z_{0}))$, we have
\begin{equation}
    \iiint_{\mathcal{B}_{r_{0}}^{+}(z_{0})}(\p_{t}w+y\p_{x}w)\cdot\phi dtdxdy+\iiint_{\mathcal{B}_{r_{0}}^{+}(z_{0})}a\cdot\p_{y}w\cdot\p_{y}\phi dtdxdy=0.
\end{equation}
\item[(ii)] $w$ satisfies the Neumann boundary condition $\p_{y}w|_{y=0}$ in the sense of boundary trace.
\end{itemize}
\end{definition}
\end{comment}
We note that the Neumann boundary condition $\p_{y}w|_{y=0}$ is well defined in the sense of boundary trace, since $w\in W_{y}^{2,1}$. We remark also that the right hand side of \eqref{holderesti} can be replaced by $\|w\|_{L^{1}(\mathcal{B}_{r_{0}}^{+}(z_{0}))}$ in view of \eqref{onetoinfty}, although this is not important for our purposes.

The rest of the section is devoted to the proof of Proposition \ref{holder}. 

The proof of Proposition \ref{holder} consists of two main steps. In the first step (subsection \ref{localboundness}), we prove a Moser--type local pointwise bound, see Lemma \ref{localbd}, which shows that a nonnegative subsolution of \eqref{diverform} can be controlled by the $L^{2}$ norm of the subsolution in a larger domain. In the second main step (subsection \ref{ExpansionofPositivity}), we prove an ``expansion of positivity" estimate (Lemma \ref{oscillation}) for the non--negative weak solution of \eqref{diverform}. This estimate implies that the ``local oscillations" of the weak solution shrink quantitatively in smaller ultraparabolic balls, and by translation and scaling invariance of these estimates, ultimately leads to H\"older continuity of weak solutions to \eqref{diverform}.

In contrast to the uniformly parabolic case, the lack of diffusion in the $x$--direction precludes the use of Poincar\'e inequalities for fixed $t$. To overcome this difficulty, we follow the approach  in \cite{XZZ24}, which is based on studying the log--transform of the weak solution (introduced in \cite{K64}), as well as the simplified presentation in \cite{CC21}. A key original idea in \cite{XZZ24} is the introduction of a weak form of the Poincar\'e inequality (Lemma \ref{weakpoin}) for nonnegative subsolutions and a careful analysis of a new dynamic average of the solution. 

In our case of boundary estimates, the transport operator $\p_{t}+y\p_{x}$ only produces the transport effect to the right. As a result, the positivity of the solution may only expand along the positive $x$ direction. Thus, we need significant modifications of the standard scheme in the proof of Lemma \ref{oscillation}.

We now turn to the detailed proof. First we note that if $w(t,x,y)$ satisfies \eqref{diverform}, then 
\begin{equation*}
\widetilde{w}(t,x,y)=w(t_{0}+r_{0}^{2}t,x_{0}+r_{0}^{3}x,r_{0}y), \quad \widetilde{a}(t,x,y)=a(t_{0}+r_{0}^{2}t,x_{0}+r_{0}^{3}x,r_{0}y)
\end{equation*}
satisfy
\begin{equation}
\begin{cases}
\p_{t}\widetilde{w}+y\p_{x}\widetilde{w}-\p_{y}(\widetilde{a}\p_{y}\widetilde{w})=0,\ \text{in}\ \mathcal{B}_{1}^{+}.\\
\p_{y}\widetilde{w}|_{y=0}=0.
\end{cases}
\end{equation}
Therefore, without loss of generality, we assume that $t_{0}=x_{0}=0$ and $r_{0}=1$. We allow the implied constants for the inequalities in this section to depend on the ellipticity constant $c$. 

\subsection{Local Boundedness}\label{localboundness}
In this subsection, we prove a Moser--type local boundedness result for the equation \eqref{diverform}.

\begin{lemma}\label{localbd}
Assume that $w$ is a nonnegative subsolution of \eqref{diverform} in $\mathcal{B}_{1}^{+}$. Then there exists an $l>0$ such that for any $0<r_{2}<r_{1}\leq 1$, we have
\begin{equation}\label{localbdesti}
{\rm ess}\sup_{\mathcal{B}_{r_{2}}^{+}}w\lesssim (r_{1}-r_{2})^{-l}\|w\|_{L^{2}(\mathcal{B}_{r_{1}}^{+})}.
\end{equation}
\eqref{localbdesti} also holds for a weak solution of \eqref{diverform}.
\end{lemma}

\begin{proof}
We remark that the proof of Lemma \ref{localbd} is similar to \cite{PP04} (the interior estimate case). For completeness, we give a self--contained proof below.

\textit{Step 1: energy estimates.} For $0<r_{2}<r_{1}\leq 1$, choose a smooth cut--off function $\chi(t,x,y)=\chi_{1}(t)\chi_{2}(x)\chi_{3}(y)$ such that
\begin{equation}
\begin{split}
&\mathrm{supp}(\chi_{1})\subset (-r_{1}^{2},0],\ \chi_{1}(t)\equiv 1\ \text{for}\ t\in[-r_{2}^{2},0],\\
&\mathrm{supp}(\chi_{2})\subset (-r_{1}^{3},r_{1}^{3}),\ \chi_{2}(x)\equiv 1\ \text{for}\ x\in[-r_{2}^{3},r_{2}^{3}],\\
&\mathrm{supp}(\chi_{3})\subset [0,r_{1}),\ \chi_{3}(y)\equiv 1\ \text{for}\ y\in[0,r_{2}].
\end{split}
\end{equation}
Multiplying \eqref{diverform} by $w\chi^{2}$ and integrating on $\mathcal{B}_{r_{1}}^{+}$, we obtain that
\begin{equation}
\begin{split}
&\iint_{[-r_{1}^{3},r_{1}^{3}]\times[0,r_{1}]}|w|^{2}(0,x,y)\chi^{2}dxdy+2\iiint_{\mathcal{B}_{r_{1}}^{+}}a(t,x,y)|\p_{y}w|^{2}\chi^{2}dtdxdy\\
&\leq 2\iiint_{\mathcal{B}_{r_{1}}^{+}}|w|^{2}\chi(\p_{t}\chi+y\p_{x}\chi)\,dtdxdy-4\iiint_{\mathcal{B}_{r_{1}}^{+}}a(t,x,y)w\p_{y}w\cdot \chi\p_{y}\chi \,dtdxdy.
\end{split}
\end{equation}
Direct calculations show that
\begin{equation}\label{caccioppoli}
\iiint_{\mathcal{B}_{r_{2}}^{+}}|\p_{y}w|^{2}\chi^{2}dtdxdy\lesssim(r_{1}-r_{2})^{-3}\iiint_{\mathcal{B}_{r_{1}}^{+}}|w|^{2}dtdxdy.
\end{equation}

\textit{Step 2: Sobolev--type inequality.} Recall the fundamental solution $\Gamma(\cdot,\cdot)$ for the operator $\mathcal{L}_{0}=\p_{t}+y\p_{x}-\p_{y}^{2}$ with the Neumann boundary condition (see Section \ref{fundamental}). For almost every $(t,x,y)\in \mathcal{B}_{r_{2}}^{+}$, we have
\begin{equation}
\begin{split}
w(t,x,y)=(w\chi)(t,x,y)&=\iiint_{\mathcal{B}_{r_{1}}^{+}}\Gamma(t,x,y;\tau,\xi,\eta)\left(\p_{\tau}+\eta\p_{\xi}-\p_{\eta}^{2}\right)(w\chi)\,d\tau d\xi d\eta\\
&=\iiint_{\mathcal{B}_{r_{1}}^{+}}\Gamma(t,x,y;\tau,\xi,\eta)\left(\p_{\tau}+\eta\p_{\xi}\right)(w\chi)\,d\tau d\xi d\eta\\
&\quad+\iiint_{\mathcal{B}_{r_{1}}^{+}}\p_{\eta}\Gamma(t,x,y;\tau,\xi,\eta)\p_{\eta}(w\chi)\,d\tau d\xi d\eta:=\mathcal{T}+\mathcal{D}.
\end{split}
\end{equation}
Moreover, we have by integration by parts,
\begin{equation}
\begin{split}
&\mathcal{T}\leq \iiint_{\mathcal{B}_{r_{1}}^{+}}\Gamma(t,x,y;\tau,\xi,\eta)w(\tau,\xi,\eta)\left(\p_{\tau}+\eta\p_{\xi}\right)\chi\,d\tau d\xi d\eta\\
&\quad -\iiint_{\mathcal{B}_{r_{1}}^{+}}a(\tau,\xi,\eta)\p_{\eta}\Gamma(t,x,y;\tau,\xi,\eta)\cdot \p_{\eta}w(\tau,\xi,\eta)\cdot\chi(\tau,\xi,\eta)d\tau d\xi d\eta\\
&\quad-\iiint_{\mathcal{B}_{r_{1}}^{+}}a(\tau,\xi,\eta)\Gamma(t,x,y;\tau,\xi,\eta)\cdot \p_{\eta}w(\tau,\xi,\eta)\cdot\p_{\eta}\chi(\tau,\xi,\eta)d\tau d\xi d\eta=:I_1+I_2+I_3,
\end{split}
\end{equation}
and
\begin{equation}
\begin{split}
\mathcal{D}&=\iiint_{\mathcal{B}_{r_{1}}^{+}}\p_{\eta}\Gamma(t,x,y;\tau,\xi,\eta)\p_{\eta}w(\tau,\xi,\eta)\cdot \chi(\tau,\xi,\eta)d\tau d\xi d\eta\\
&\qquad+\iiint_{\mathcal{B}_{r_{1}}^{+}}\p_{\eta}\Gamma(t,x,y;\tau,\xi,\eta)w(\tau,\xi,\eta)\cdot \p_{\eta}\chi(\tau,\xi,\eta)d\tau d\xi d\eta=:I_4+I_5.
\end{split}
\end{equation}
In the above calculation, we have used the boundary condition of $w$ and the assumption that $w$ is a subsolution of \eqref{diverform}.

By Proposition \ref{operatorbd}, we have the following bounds 
\begin{equation}\label{eq:hIs}
\begin{split}
&\|I_{1}\|_{L^{6}(\mathbb{H})}\lesssim (r_{1}-r_{2})^{-3}\|w\|_{L^{2}(\mathcal{B}_{r_{1}}^{+})},\quad
\|I_{2}\|_{L^{3}(\mathbb{H})}\lesssim \|\p_{y}w\|_{L^{2}(\mathcal{B}_{r_{1}}^{+})},\\
&\|I_{3}\|_{L^{6}(\mathbb{H})}\lesssim (r_{1}-r_{2})^{-1}\|\p_{y}w\|_{L^{2}(\mathcal{B}_{r_{1}}^{+})},\quad
\|I_{4}\|_{L^{3}(\mathbb{H})}\lesssim \|\p_{y}w\|_{L^{2}(\mathcal{B}_{r_{1}}^{+})},\\
&\|I_{5}\|_{L^{3}(\mathbb{H})}\lesssim (r_{1}-r_{2})^{-1}\|w\|_{L^{2}(\mathcal{B}_{r_{1}}^{+})}.
\end{split}
\end{equation}
It follows from \eqref{eq:hIs} that 
\begin{equation}
\|w\|_{L^{3}(\mathcal{B}_{r_{2}}^{+})}\leq\|w\chi\|_{L^{3}(\mathbb{H})}\lesssim (r_{1}-r_{2})^{-3}\big(\|w\|_{L^{2}(\mathcal{B}_{r_{1}}^{+})}+\|\p_{y}w\|_{L^{2}(\mathcal{B}_{r_{1}}^{+})}\big).
\end{equation}

\textit{Step 3: A reverse H\"older's inequality.} For fixed $0<r_{2}<r_{1}\leq 1$, by Step 1 and Step 2, we obtain that
\begin{equation}\label{reverseholder}
\begin{split}
\|w\|_{L^{3}(\mathcal{B}_{r_{2}}^{+})}&\lesssim (r_{1}-r_{2})^{-3}\Big(\|w\|_{L^{2}(\mathcal{B}_{\frac{r_{1}+r_{2}}{2}}^{+})}+\|\p_{y}w\|_{L^{2}(\mathcal{B}_{\frac{r_{1}+r_{2}}{2}}^{+})}\Big)\lesssim (r_{1}-r_{2})^{-\frac{9}{2}}\|w\|_{L^{2}(\mathcal{B}_{r_{1}}^{+})}.
\end{split}
\end{equation}

\textit{Step 4: Moser's Iteration.} Set
\begin{equation}
r^{(n)}=r_{2}+2^{-n}(r_{1}-r_{2}),\ \sigma_{n}=\left(3/2\right)^{n},\ w_{n}=w^{\sigma_{n}}.
\end{equation}
One can check that $w_{n}$ are also nonnegative subsolutions of \eqref{diverform}. By \eqref{reverseholder}, we have
\begin{equation}
\|w_{n}\|_{L^{3}(\mathcal{B}_{r^{(n+1)}}^{+})}\lesssim (r^{(n)}-r^{(n+1)})^{-\frac{9}{2}}\|w_{n}\|_{L^{2}(\mathcal{B}_{r^{(n)}}^{+})}.
\end{equation}
Equivalently,
\begin{equation}
\|w\|_{L^{2\sigma_{n+1}}(\mathcal{B}_{r^{(n+1)}}^{+})}\lesssim (r^{(n)}-r^{(n+1)})^{-\frac{9}{2\sigma_{n}}}\|w\|_{L^{2\sigma_{n}}(\mathcal{B}_{r^{(n)}}^{+})}.
\end{equation}
Sending $n\to\infty$, we get
\begin{equation}
\begin{split}
{\rm ess}\sup_{\mathcal{B}_{r_{2}}^{+}}w&\lesssim\left(\prod_{k=0}^{\infty}(r^{(k)}-r^{(k+1)})^{-\frac{9}{2\sigma_{k}}}\right)\|w\|_{L^{2}(\mathcal{B}_{r_{1}}^{+})}\lesssim (r_{1}-r_{2})^{-l}\|w\|_{L^{2}(\mathcal{B}_{r_{1}}^{+})},
\end{split}
\end{equation}
for some $l>0$. The proof of Lemma \ref{localbd} is now complete.
\end{proof}

\begin{remark}By Lemma \ref{localbd}, we obtain that
\begin{equation}
\begin{split}
{\rm ess}\sup_{\mathcal{B}_{r_{2}}^{+}}w&\lesssim (r_{1}-r_{2})^{-l}\|w\|_{L^{2}(\mathcal{B}_{r_{1}}^{+})}\leq  (r_{1}-r_{2})^{-l}\|w\|_{L^{1}(\mathcal{B}_{r_{1}}^{+})}^{1/2}\cdot\big({\rm ess}\sup_{\mathcal{B}_{r_{1}}^{+}}w\big)^{1/2}\\
&\leq  \frac{1}{2}\big({\rm ess}\sup_{\mathcal{B}_{r_{1}}^{+}}w\big)+(r_{1}-r_{2})^{-2l}\|w\|_{L^{1}(\mathcal{B}_{r_{1}}^{+})}.
\end{split}
\end{equation}
By applying a standard iteration, we can also obtain the following $L^{1}-L^{\infty}$ estimates:
\begin{equation}\label{onetoinfty}
{\rm ess}\sup_{\mathcal{B}_{r_{2}}^{+}}w\lesssim (r_{1}-r_{2})^{-\widetilde{l}}\|w\|_{L^{1}(\mathcal{B}_{r_{1}}^{+})},
\end{equation}
for some $\widetilde{l}>0$ independent of $0<r_2<r_1$.
\end{remark}
\subsection{Expansion of Positivity}\label{ExpansionofPositivity}
In this subsection, we establish the following ``expansion of positivity estimate" for weak solutions to equation \eqref{diverform}.
\begin{lemma}\label{oscillation}
Assume that $w$ is a nonnegative weak solution to \eqref{diverform} in $\mathcal{B}_{1}^{+}$. Then there exist two constants $0<\theta<\frac{1}{100}$, and $h^{*}>0$, such that if
\begin{equation}\label{measure}
\left|\left\{(t,x,y)\in\mathcal{B}_{\theta}^{+}: w(t,x,y)\geq 1\right\}\right|\geq \frac{1}{2}|\mathcal{B}_{\theta}^{+}|,
\end{equation}
then we have
\begin{equation}\label{lowerbd}
w(t,x,y)\geq h^{*}, \quad{\rm for\,\, a.e.}\,\, (t,x,y)\in\mathcal{B}_{\frac{\theta^{2}}{2}}^{+}\left(0,99\theta^{3}/100,0\right). 
\end{equation}
\end{lemma}
\begin{remark}
We note that in Lemma \ref{oscillation}, there is a significant shift in the $x$ direction in the domain for the bounds in \eqref{lowerbd}. This shift is needed since the operator $\p_{t}+y\p_{x}$ only provides a one--sided transport effect, in sharp contrast to the case when there is no boundary. 
\end{remark}

\begin{comment}
\begin{figure}[htbp]
    \centering
    
    \begin{tikzpicture}[>=Stealth, scale=1.5]
        % 坐标轴（代替大正方形的左边和下边）
        \draw[->, thick] (0,0) -- (4.1,0) node[right] {$x$};  % x轴
        \draw[->, thick] (0,0) -- (0,4.1) node[above] {$t$};  % t轴
        
        % 大正方形的右边和上边（完成正方形）
        \draw[thick] (0,4) -- (4,4);  % 上边
        \draw[thick] (4,0) -- (4,4);  % 右边
        
        % 大正方形内部右上角的小正方形（贴紧上边界，向左移动）
        \fill[gray!60] (2.8,3.2) rectangle (3.8,4);
        \draw[thick] (2.8,3.2) rectangle (3.8,4);
       % \node[font=\small, white] at (3.3,3.6);
        
        % 大正方形内部下方的灰色不规则图形（去掉右下角的）
        % 第一个不规则图形（左下）
        \fill[gray!60] (0.3,0.3) -- (0.5,1.2) -- (1.5,1.5) -- (1.8,0.8) -- (1.2,0.4) -- cycle;
        
        % 第二个不规则图形（中下）
        \fill[gray!60] (1.8,0.5) -- (2.2,1.3) -- (3.0,1.1) -- (2.8,0.4) -- (2.3,0.3) -- cycle;
        
        % 从彩色区域指向小正方形的箭头（黑色）
        \draw[->, thick, black] (1.5,1.2) to[out=60, in=-120] (3.1,3.3);
        
        % 文字标注（与箭头分开，黑色）
        \node[align=center, font=\small, black] at (1.5,2.8) {regularity expansion};
        
        % 右下角的颜色说明（调整位置）
        \fill[gray!60] (4.5,0) rectangle (5,0.5);
        \draw (4.5,0) rectangle (5,0.5);
        \node[align=left, right] at (5.1,0.25) {: positivity};
    \end{tikzpicture}
    
    \caption{Positivity Expansion}
    \label{fig:expansion-positivity}
\end{figure}
\end{comment}

Lemma \ref{oscillation} follows from the following series of lemmas.

\begin{lemma}[\bf Weak Poincar\'e Inequality]\label{weakpoin}
Assume that $w$ is a nonnegative subsolution of \eqref{diverform} in $\mathcal{B}_{1}^{+}$. Let $0<\theta<1/4$ and $\psi(t,x,y)$ be a cut--off function compactly supported in $\mathcal{B}_{\frac{1}{2}}^{+}(0,0,0)$ with $\psi|_{\mathcal{B}_{\theta^{2}}^{+}(0,\frac{99}{100}\theta^{3},0)}\equiv 1$ and $\p_{y}\psi|_{y=0} = 0$. Define $H$ as the solution of
\begin{equation}\label{auxieq}
\begin{cases}
\p_{t}H+y\p_{x}H-\p_{y}^{2}H=w\cdot(\p_{t}\psi+y\p_{x}\psi-\p_{y}^{2}\psi),\ t>-1,x\in\mathbb{R},y\in\mathbb{R}^{+},\\
H|_{t=-1}=0,\ \p_{y}H|_{y=0}=0.
\end{cases}
\end{equation}
Then we have the bounds
\begin{equation}\label{weakpo}
\iiint_{\mathcal{B}_{\theta^{2}}^{+}(0,\frac{99}{100}\theta^{3},0)}|(w-H)^{+}|^{2}\,dtdxdy\leq C\iiint_{\mathcal{B}_{\frac{1}{2}}^{+}}|\p_{y}w|^{2}dtdxdy,
\end{equation}
where $C=C(\theta,\psi)$.
\end{lemma}

\begin{remark}Notice that
\begin{equation}
H(t,x,y)=\int_{-\frac{1}{4}}^{t}\int_{-\infty}^{\infty}\int_{0}^{\infty}\Gamma(t,x,y;\tau,\xi,\eta)w(\tau,\xi,\eta)(\p_{\tau}\psi+\eta\p_{\xi}\psi-\p_{\eta}^{2}\psi)d\tau d\xi d\eta.
\end{equation}
 For each $z\in\mathcal{B}_{\theta^{2}}^{+}(0,\frac{99}{100}\theta^{3},0)$, we have
\begin{equation}
\int_{-\frac{1}{4}}^{t}\int_{-\infty}^{\infty}\int_{0}^{\infty}\Gamma(t,x,y;\tau,\xi,\eta)(\p_{\tau}\psi+\eta\p_{\xi}\psi-\p_{\eta}^{2}\psi)d\tau d\xi d\eta=\psi(z)=1.
\end{equation}
Therefore, $H(z)$ can be viewed as a ``dynamic average" of $w$. %It is worth comparing with the classical Poincar\'e inequality.
\end{remark}

\begin{proof}[Proof of Lemma \ref{weakpoin}]One can check that
\begin{equation}\label{subsoleq}
\p_{t}(w\psi-H)+y\p_{x}(w\psi-H)-\p_{y}^{2}(w\psi-H)\leq \p_{y}\left((a-1)\p_{y}w\right)\psi-2\p_{y}w\p_{y}\psi.
\end{equation}
Then \eqref{weakpo} follows from testing \eqref{subsoleq} by $(w\psi-H)^{+}$ and integrating over $t\in[-\frac{1}{4},0], x\in\mathbb{R}, y\in\mathbb{R}^{+}$. For more details we refer to Lemma 3.1 in \cite{GI23}.
\end{proof}

The next lemma shows that the upper bound on $H$ can be explicitly controlled by the upper bound of $w$ by choosing a suitable cut--off function $\psi$.

\begin{lemma}\label{hcontrol}There exist a constant $0<\theta<1/4$ sufficiently small and a cut--off function $\psi$ satisfying the assumptions in Lemma \ref{weakpoin}, such that for each nonnegative weak subsolution $w$ of \eqref{diverform} with the property that
\begin{equation}\label{slidevanish}
\left|\left\{(x,y):\,w(t,x,y)=0,\ |x|<\frac{49}{50}\theta^{3},\ 0<y<\theta\right\}\right|\geq\frac{1}{100}\theta^{4},\ \text{for a.e.}\ -\frac{1}{100}\theta^{2}<t<0,
\end{equation}
we have
\begin{equation}
{\rm ess\ sup}_{(t,x,y)\in\mathcal{B}_{\theta^{2}}^{+}(0,\frac{99}{100}\theta^{3},0)}H(t,x,y)\leq \lambda_{0}\,{\rm ess}\sup_{\mathcal{B}_{1}^{+}}w.
\end{equation}
In the above, $H(t,x,y)$ is as defined in \eqref{auxieq} and the constant $0<\lambda_{0}<1$ is independent of $w$ and $\theta$.
\end{lemma}

\begin{comment}
\begin{remark}For the classical Poincar\'e inequality, $H$ is the standard average. i.e.
\begin{equation*}
H=\int_{X}u(\bm{x})\,d\mu(x),
\end{equation*}
where $\mu$ is a probability measure on $X$. Clearly we have
\begin{equation*}
H\leq \lambda_{0}{\rm ess}\sup_{\bm{x}\in X}u(\bm{x}),\ \text{for some}\ 0<\lambda_{0}<1,
\end{equation*}
provided
\begin{equation*}
\mu\left(\{\bm{x}:u(\bm{x})=0\}\right)\geq (1-\lambda_{0})\mu(X),
\end{equation*}
which is sightly weaker than the temporal--slide vanishing condition \eqref{slidevanish}.
\end{remark}
\end{comment}

\begin{proof} Let $\psi(t,x,y)$ be a cut--off function in $\mathcal{B}^+_{\frac{1}{2}}(0,0,0)$ such that
\begin{itemize}
\item[(i)] $\psi$ is compactly supported in $\mathcal{B}_{\frac{1}{2}}^{+}(0,0,0)$, and $\psi|_{\mathcal{B}_{\theta^{2}}^{+}(0,\frac{99}{100}\theta^{3},0)}\equiv 1$,$\p_{y}\psi|_{y=0}\equiv 0$;

\item[(ii)] $\psi(t,x,y)\equiv 0$ if $t\leq-\theta^{2}$;

\item[(iii)] $|\p_{y}\psi,\p_{y}^{2}\psi|\lesssim 1$, $\p_{t}\psi+y\p_{x}\psi$ is non--negative, and

\begin{equation}
\p_{t}\psi+y\p_{x}\psi\gtrsim \frac{1}{\theta^{2}},\ \forall t\in (-\frac{1}{100}\theta^{2},-\frac{1}{200}\theta^{2}),\ |x|<\theta^{3}, 0\leq y<\theta.
\end{equation}
\end{itemize}
As a concrete example, we can choose
\begin{equation}
\psi(t,x,y)=\chi\left(\frac{x^{2}}{\theta}-\frac{t}{\theta^{2}}\right)\chi(y),\ \mathrm{supp}(\chi)\subset [0,\frac{1}{10}), \,\chi|_{[0, \frac{1}{1000}]}\equiv 1,\,-\chi'\gtrsim 1 \,{\rm for}\,y\in[\frac{1}{500},\frac{1}{50}],
\end{equation}
with $\theta\in(0,1)$ sufficiently small.

Without loss of generality, we can assume that ${\rm ess}\sup_{\mathcal{B}_{1}^{+}}w=1$. By the superposition principle, we can decompose
\begin{equation}
H(t,x,y)=\psi(t,x,y)-P(t,x,y)+E(t,x,y),
\end{equation}
where $P$ satisfies
\begin{equation}
\begin{cases}
\p_{t}P+y\p_{x}P-\p_{y}^{2}P=(\p_{t}\psi+y\p_{x}\psi)(1-w),t>-\frac{1}{4},\ x\in\mathbb{R},\ y\in\mathbb{R}^{+},\\
P|_{t=-\frac{1}{4}}=0,\ \,\p_{y}P|_{y=0} = 0,
\end{cases}
\end{equation}
and $E$ satisfies
\begin{equation}
\begin{cases}
\p_{t}E+y\p_{x}E-\p_{y}^{2}E=(1-w)\p_{y}^{2}\psi, t>-\frac{1}{4},\ x\in\mathbb{R},\ y\in\mathbb{R}^{+},\\
E|_{t=-\frac{1}{4}}=0,\ \p_{y}E|_{y=0} = 0.
\end{cases}
\end{equation}

We first consider the $P$ term. Write
\begin{equation}
P(t,x,y)=\int_{-\frac{1}{4}}^{t}\int_{-\infty}^{\infty}\int_{0}^{\infty}\Gamma(t,x,y;\tau,\xi,\eta)(\p_{\tau}\psi+\eta\p_{\xi}\psi)(1-w) (\tau,\xi,\eta)\,d\eta d\xi d\tau.
\end{equation}
By the properties of $\psi$, denoting $z=(t,x,y), \,z^\ast=(\tau,\xi,\eta)$, for each $(t,x,y)\in\mathcal{B}_{\theta^{2}}^{+}(0,\frac{99}{100}\theta^{3},0)$, we have
\begin{equation}
P(t,x,y)\gtrsim \int_{-\frac{1}{100}\theta^{2}}^{-\frac{1}{200}\theta^{2}}\iint_{\left\{w(\tau,\xi,\eta)=0,\ |\xi|<\frac{49}{50}\theta^{3},\ 0<\eta<\theta\right\}}\Gamma(z;z^\ast)\theta^{-2}d\tau d\xi d\eta.
\end{equation}
By Corollary \ref{positive} and \eqref{slidevanish}, for each  $(t,x,y)\in\mathcal{B}_{\theta^{2}}^{+}(0,\frac{99}{100}\theta^{3},0)$, we have
\begin{equation}
P(t,x,y)\gtrsim \int_{-\frac{1}{100}\theta^{2}}^{-\frac{1}{200}\theta^{2}}\iint_{\left\{w(\tau,\xi,\eta)=0,\ |\xi|<\frac{49}{50}\theta^{3},\ 0<\eta<\theta\right\}}\theta^{-6}d\tau d\xi d\eta\geq c_{0}>0.
\end{equation}

We now turn to the bounds on $E$. By the property of the fundamental solution \eqref{dualmassconser},
\begin{equation}
\begin{split}
E&\leq 2\int_{-\theta^{2}}^{0}\iint_{\R\times\R^{+}}\Gamma(t,x,y;\tau,\xi,\eta)|\p_{\eta}^{2}\psi|d\tau d\xi d\eta\\
&\lesssim \int_{-\theta^{2}}^{0}\iint_{\R\times\R^{+}}\Gamma(t,x,y;\tau,\xi,\eta)d\tau d\xi d\eta\leq c_{1}\theta^{2}.
\end{split}
\end{equation}
Hence, for each $(t,x,y)\in \mathcal{B}_{\theta^{2}}^{+}(0,\frac{99}{100}\theta^{3},0)$, we have
\begin{equation}
H(t,x,y)\leq 1-c_{0}+c_{1}\theta^{2}\leq \lambda_{0}<1,
\end{equation}
provided that $\theta\in(0,1)$ is chosen sufficiently small.
\end{proof}

We also need the following measure estimate. This type of estimate can be traced back to the early work \cite{K64} for uniformly parabolic equations, where the classical Poincar\'e inequality applies.
\begin{lemma}\label{measureesti}Let $0<\theta<1$. Assume that $w$ is a nonnegative weak solution of \eqref{diverform} in $\mathcal{B}_{1}^{+}$ with
\begin{equation}
\left|\left\{(t,x,y)\in\mathcal{B}_{\theta}^{+}:\,w(t,x,y)\geq 1\right\}\right|\geq \frac{1}{2}|\mathcal{B}_{\theta}^{+}|.
\end{equation}
Then there exists a constant $h_{0}:=h_0(c,\theta)\in (0,1)$ such that 
\begin{equation}
\big|\big\{(x,y):\,w(t,x,y)\geq h_{0},\ |x|<\beta^{3}\theta^{3},\ 0<y<\beta\theta\big\}\big|\geq\frac{1}{100}\theta^{4},\ \text{for a.e.}\ -\frac{1}{100}\theta^{2}<t<0,
\end{equation}
where $\beta:=(\frac{49}{50})^{\frac{1}{3}}$.
\end{lemma}
\begin{proof} The proof is essentially the same as that for Proposition 4.2 in \cite{XZZ24} and we present the details here for the reader's convenience. 

For $h\in(0,\frac{1}{2})$ and $r\in(0,\theta)$, define 
\begin{equation}
    \begin{split}
        &C_{r}:=\left\{(x,y):|x|<r^{3},\ 0<y<r\right\},\\
        &C_{r,h}(t):=\left\{(x,y)\in C_{r},\ w(t,x,y)\geq h\right\},\\
        &\mu_{r}(t):=\left|\left\{(x,y)\in C_{r},\ w(t,x,y)\geq 1\right\}\right|.
    \end{split}
\end{equation}
We introduce the auxiliary function
\begin{equation}
     \widetilde{w}(t,x,y):=\log^{+}\frac{1}{w+h^{\frac{101}{100}}}.
\end{equation}
One can check that
\begin{equation}\label{tildew}
    \begin{cases}
        \p_{t}\widetilde{w}+y\p_{x}\widetilde{w}-\p_{y}(a\p_{y}\widetilde{w})+a(\p_{y}\widetilde{w})^{2}\leq 0,\\
        \p_{y}\widetilde{w}|_{y=0}=0,\\
        0\leq\widetilde{w}\leq\log h^{-\frac{101}{100}}.
    \end{cases}
\end{equation}
Let $\chi(y)$ be a smooth cut--off function such that
\begin{equation*}
\chi(y)\equiv 1,\ 0\leq y\leq\beta\theta,\quad
\chi(y)\equiv 0,\ y\geq \theta,\quad
|\chi'(y)|\leq \frac{2}{(1-\beta)\theta}.
\end{equation*}

Multiplying \eqref{tildew} by $\chi^{2}(y)$ and integrating in $(\tau,t)\times C_{\theta}$ where $-\theta^{2}<\tau<t<0$, we obtain the inequality
\begin{equation}\label{massesti}
    \begin{split}
        &\int_{C_{\theta}}\chi^{2}(y)\widetilde{w}(t,x,y)dxdy+\frac{c}{2}\int_{\tau}^{t}\int_{C_{\theta}}\chi^{2}|\p_{y}\widetilde{w}|^{2}dsdxdy\\
        &\leq\int_{C_{\theta}}\chi^{2}(y)\widetilde{w}(\tau,x,y)dxdy+\frac{16}{c\beta^{4}(1-\beta)^{2}}|C_{\beta\theta}|-\int_{\tau}^{t}\int_{0}^{\theta}\chi^{2}y\widetilde{w}\big|_{x=-\theta^{3}}^{x=\theta^{3}}\,dyds.
    \end{split}
\end{equation}
We first estimate the last term on the right hand side. Thanks to $0\leq\widetilde{w}\leq\log h^{-\frac{101}{100}}$,
\begin{equation*}
    -\int_{\tau}^{t}\int_{0}^{\theta}\chi^{2}y\widetilde{w}\big|_{x=-\theta^{3}}^{x=\theta^{3}}dyds\leq\big(\log h^{-\frac{101}{100}}\big)(t-\tau)\int_{0}^{\theta}ydy\leq \frac{1}{4}\beta^{-4}\big(\log h^{-\frac{101}{100}}\big)|C_{\beta\theta}|.
\end{equation*}

On the other hand, notice that
\begin{equation*}
    \begin{split}
        \int_{-\theta^{2}}^{-\frac{1}{100}\theta^{2}}\mu_\theta(t)dt&=\int_{-\theta^{2}}^{0}\mu_\theta(t)dt-\int_{-\frac{1}{100}\theta^{2}}^{0}\mu_\theta(t)dt\\
        &\geq \frac{1}{2}|\mathcal{B}_{\theta}^{+}|-\frac{1}{100}\theta^{2}|C_{\theta}|=\frac{49}{100}\theta^{2}|C_{\theta}|.
    \end{split}
\end{equation*}
By the mean-value theorem, there exists $\tau^{\ast}\in(-\theta^{2},-\frac{1}{100}\theta^{2})$, such that
\begin{equation*}
    \mu_\theta(\tau^{\ast})\geq \frac{49}{99}|C_{\theta}|.
\end{equation*}
As a consequence,
\begin{equation*}
\int_{C_{\theta}}\widetilde{w}(\tau^\ast,x,y)dxdy=\int_{C_{\theta}\cap\{w<1\}}\widetilde{w}(\tau^\ast,x,y)dxdy\leq\frac{50}{99}\big(\log h^{-\frac{101}{100}}\big)|C_{\theta}|.
\end{equation*}
By virtue of the mass estimate \eqref{massesti}, we have
\begin{equation}
    \begin{split}
        &\big(\log\frac{1}{h+h^{\frac{101}{100}}}\big)|C_{\beta\theta}\setminus C_{\beta\theta,h}(t)|\leq \int_{C_{\beta\theta}}\widetilde{w}(t,x,y)dxdy\leq\int_{C_{\theta}}\chi^{2}\widetilde{w}dxdy\\
        &\leq\Big(\frac{50}{99}\beta^{-4}+\frac{1}{4}\beta^{-4}+\frac{16}{c\beta^{4}(1-\beta)^{2}\log h^{-\frac{101}{100}}}\Big)\big(\log h^{-\frac{101}{100}}\big)|C_{\beta\theta}|.
    \end{split}
\end{equation}
Note that
\begin{equation*}
    \lim_{h\to 0+}\Big(\frac{50}{99}\beta^{-4}+\frac{1}{4}\beta^{-4}+\frac{16}{c\beta^{4}(1-\beta)^{2}\log h^{-\frac{101}{100}}}\Big)\cdot\frac{\log h^{-\frac{101}{100}}}{\log(h+h^{\frac{101}{100}})^{-1}}<\frac{99}{100},
\end{equation*}
hence the conclusion holds by choosing $h\in(0,\frac{1}{2})$ sufficiently small.
\end{proof}

Now we are ready to prove Lemma \ref{oscillation}.
\begin{proof}[Proof of Lemma \ref{oscillation}] Define
%\footnote{We can also replace $V$ by the $G$--function introduced in \cite{K64} (See also \cite{GI23}, \cite{WZ24}).}
\begin{equation}
V(t,x,y)=\log^{+}\left(\frac{h}{w+h^{2}}\right),
\end{equation}
where $0<h<h_{0}(c,\theta)$ will be determined later. By Lemma \ref{measureesti}, we have
\begin{equation}\label{slidevanishing}
\Big|\Big\{V(t,x,y)=0,\ |x|<\frac{49}{50}\theta^{3},\ 0<y<\theta\Big\}\Big|\geq\frac{1}{100}\theta^{4},\ \text{for a.e.}\ -\frac{1}{100}\theta^{2}<t<0.
\end{equation}
Notice that
\begin{equation}\label{logeq}
\begin{cases}
\p_{t}V+y\p_{x}V-\p_{y}(a(t,x,y)\p_{y}V)+a(t,x,y)(\p_{y}V)^{2}\leq0,\ \text{in}\  \mathcal{B}_{1}^{+},\\
\p_{y}V|_{y=0}=0.
\end{cases}
\end{equation}
Thus $V$ is a nonnegative subsolution of \eqref{diverform}. By Lemma \ref{weakpoin} and Lemma \ref{hcontrol}, for a suitable $0<\lambda_{0}<1$, we have
\begin{equation}\label{Vestimat}
\iiint_{\mathcal{B}_{\theta^{2}}^{+}(0,\frac{99}{100}\theta^{3},0)}\big|(V-\lambda_{0}\log\frac{1}{h})^{+}\big|^{2}dtdxdy\lesssim \iiint_{\mathcal{B}_{\frac{1}{2}}^{+}}|\p_{y}V|^{2}dtdxdy.
\end{equation}
To estimate the right hand side of \eqref{Vestimat}, let $\chi(\cdot)$ be a standard cut--off function supported in $(-1,1)$ with $\chi|_{[-\frac{1}{2},\frac{1}{2}]}\equiv 1$. Testing \eqref{logeq} by $\chi(t)\chi(x)\chi^{2}(y)$, we have for some $c\in(0,1)$,
\begin{equation}
\begin{split}
&c\iiint_{\mathcal{B}_{1}^{+}}|\p_{y}V|^{2}\chi(t)\chi(x)\chi^{2}(y)\,dtdxdy\leq \iiint_{\mathcal{B}_{1}^{+}}a(t,x,y)|\p_{y}V|^{2}\chi(t)\chi(x)\chi^{2}(y)dtdxdy\\
&\leq\left|\iiint_{\mathcal{B}_{1}^{+}}V(t,x,y)\left(\chi'(t)\chi(x)\chi^{2}(y)+y\chi(t)\chi'(x)\chi^{2}(y)\right)dtdxdy\right|\\
&\quad\quad\quad\ +2\iiint_{\mathcal{B}_{1}^{+}}a(t,x,y)\p_{y}V(t,x,y)\chi(t)\chi(x)\chi(y)\chi'(y)dtdxdy\\
&\leq \left|\iiint_{\mathcal{B}_{1}^{+}}V(t,x,y)\left(\chi'(t)\chi(x)\chi^{2}(y)+y\chi(t)\chi'(x)\chi^{2}(y)\right)dtdxdy\right|\\
&\quad\quad\quad\ +\frac{c}{2}\iiint_{\mathcal{B}_{1}^{+}}|\p_{y}V|^{2}\chi(t)\chi(x)\chi^{2}(y)dtdxdy+c^{-3}\iiint_{\mathcal{B}_{1}^{+}}\chi(t)\chi(x)|\chi'(y)|^{2}dtdxdy.
\end{split}
\end{equation}
Hence,
\begin{equation}\label{vybound}
\iiint_{\mathcal{B}_{1}^{+}}|\p_{y}V|^{2}\chi(t)\chi(x)\chi^{2}(y)dtdxdy\leq C(c)\log\left(\frac{1}{h}\right),
\end{equation}
provided that $h\in(0,1)$ is sufficiently small. 

Combining \eqref{Vestimat} and \eqref{vybound}, and applying Lemma \ref{localbd} (with a suitable rescaling), 
%(or applying \eqref{onetoinfty} for $[(V-\lambda_{0}\log\frac{1}{h})^{+}]^{2}$) , 
we can conclude that there exists a positive constant $C(\theta)$ such that
\begin{equation}\label{eq:hol3}
{\rm ess\ sup}_{\mathcal{B}_{\frac{\theta^{2}}{2}}^{+}(0,\frac{99}{100}\theta^{3},0)}V(t,x,y)\leq \lambda_{0}\log\frac{1}{h}+C(\theta)\left(\log\frac{1}{h}\right)^{\frac{1}{2}},
\end{equation}
for any $h\leq h(\theta)$. \eqref{eq:hol3} implies that
\begin{equation}
{\rm ess\ inf}_{\mathcal{B}_{\frac{\theta^{2}}{2}}^{+}(0,\frac{99}{100}\theta^{3},0)}w(t,x,y)\geq h^{*},
\end{equation}
for some $h^{*}=h^{*}(c,\theta)$. This completes the proof of Lemma \ref{oscillation}.
\end{proof}

The following corollary is a direct consequence of Lemma \ref{oscillation}.
\begin{corollary}\label{osciesti}
Assume that $w$ is a weak solution to \eqref{diverform} in $\mathcal{B}_{2}^{+}(0,0,0)$. Then there exist $\bar{\alpha}\in(0,1)$ and $\delta\in(0,1)$ independent of $w$ such that
\begin{equation}
\mathrm{osc}_{\mathcal{B}_{\delta}^{+}(0,0,0)}w\leq\bar{\alpha}\,(\mathrm{osc}_{\mathcal{B}_{2}^{+}(0,0,0)}w).
\end{equation}
\end{corollary}
\begin{proof}
For $0<r\leq 1$, $x_0\in \R$, we denote
\begin{equation}\label{shiftball}
    \mathcal{B}^S_r(0,x_0,0):= \mathcal{B}^+_r(0,x_0-\frac{99}{100}\theta^3,0),\quad \mathcal{B}^S_r:=\mathcal{B}^S_r(0,0,0).
\end{equation}
Define
\begin{equation*}
M={\rm ess}\sup_{\B_{\theta}^{S}}w,\quad m={\rm ess}\inf_{\mathcal{B}_{\theta}^{S}}w,
\end{equation*}
where $\theta>0, h^{\ast}>0$ are the same constants as in Lemma \ref{oscillation}. We assume that $M\neq m$, since otherwise the desired conclusion is trivial since $\mathcal{B}_{\delta}^{+}(0,0,0)\subset \B_{\theta}^{S}$. Define then
\begin{equation*}
    \ w^{+}=2\frac{w-m}{M-m}+\frac{h^{\ast}}{2},\quad w^{-}=2\frac{M-w}{M-m}+\frac{h^{\ast}}{2}.
\end{equation*}

\textit{Case 1: At least one of ${\rm ess}\inf_{\mathcal{B}_{1}^{S}} w^{+}$ and ${\rm ess}\inf_{\mathcal{B}_{1}^{S}}w^{-}$ is negative.} Assume without loss of generality that ${\rm ess}\inf_{\mathcal{B}_{1}^{S}} w^{+}$ is negative. It follows that
\begin{equation*}
{\rm ess}\inf_{\mathcal{B}_{1}^{S}}w\leq m-\frac{h^{\ast}}{4}(M-m).
\end{equation*}
Hence for sufficiently small $\theta>0$,
\begin{equation*}
\mathrm{osc}_{\mathcal{B}_{2}^{+}}w\geq\mathrm{osc}_{\mathcal{B}_{1}^{S}}w\geq \big(1+\frac{h^{\ast}}{4}\big)\cdot\mathrm{osc}_{\mathcal{B}_{\theta}^{S}}w\geq \big(1+\frac{h^{\ast}}{4}\big)\cdot\mathrm{osc}_{\mathcal{B}_{\theta^2}^{+}}w.
\end{equation*}

\textit{Case 2: Both $w^{+}$ and $w^{-}$ are non--negative almost everywhere in $\mathcal{B}_{1}^{S}$.} Clearly, at least one of the two satisfies \eqref{measure}. Without loss of generality, assume that $w^{+}$ satisfies \eqref{measure}. Then Lemma \ref{oscillation} implies that
\begin{equation}
w\geq \frac{h^{\ast}}{4}(M-m)+m,\quad a.e. \ \text{in}\ \mathcal{B}_{\frac{\theta^{2}}{2}}^{S}\big(0,\frac{99}{100}\theta^{3},0\big)=\mathcal{B}_{\frac{\theta^{2}}{2}}^{+}.
\end{equation}
It follows that
\begin{equation}\label{oscesti1}
\mathrm{osc}_{\mathcal{B}_{\frac{\theta^{2}}{2}}^{+}}w\leq M-\big(\frac{h^{\ast}}{4}(M-m)+m\big)=\big(1-\frac{h^{\ast}}{4}\big)(M-m)\leq\big(1-\frac{h^{\ast}}{4}\big)\mathrm{osc}_{\B_{2}^{+}}w.
\end{equation}
Hence combining both cases above, the conclusion of Corollary \ref{osciesti} follows.
\end{proof}

\subsection{Proof of Proposition \ref{holder}}\label{pfholder}

The proof of Proposition \ref{holder} is based on combining the interior estimates for system \eqref{diverform} without boundary (see \cite{XZZ24}) and Corollary \ref{osciesti}. Let $z=(t,x,y),\zeta=(\tau,\xi,\eta)\in\mathcal{B}_{\frac{1}{2}}^{+}$, and set $\kappa=\|z-\zeta\|\ll 1$. Without loss of generality we assume that $t\geq\tau$.

Let $z_{0}=(t,x,0)$ be the projection of $z$ on $\{y=0\}$. Notice that
\begin{equation}
|w(z)-w(\zeta)|\leq\mathrm{osc}_{\mathcal{B}_{y+\kappa}^{+}(z_{0})}w.
\end{equation}
By Corollary \ref{osciesti}, we have for some constant $c_0>0$,
\begin{equation}\label{bdosci}
\mathrm{osc}_{\mathcal{B}_{y+\kappa}^{+}(z_{0})}w\leq \bar{\alpha}^{M}\mathrm{osc}_{\mathcal{B}_{1/4}^{+}(z_{0})}w,\,\ {\rm where}\,\,M=\max \big\{k\in\Z: k\leq c_0\log(y+\kappa)/(8\log(\delta/2))\big\}.
\end{equation}

On the other hand, if $y\gg \kappa>0$, by the interior estimate for the equation \eqref{diverform} without boundary stated in \cite{XZZ24} (see also \cite{Z08}, \cite{WZ11}, \cite{WZ24}), as well as standard translation and rescaling arguments, we have
\begin{equation}
|w(z)-w(\zeta)|\lesssim \kappa^{\gamma}y^{-\gamma}\|w\|_{L^{\infty}(\mathcal{B}_{1}^{+})},
\end{equation}
for some $0<\gamma<1$.

Therefore, when $\kappa\leq y^{2}$, we have 
\begin{equation}
|w(z)-w(\zeta)|\lesssim \kappa^{\frac{\gamma}{2}}\|w\|_{L^{\infty}(\mathcal{B}_{1}^{+})}.
\end{equation}

In the opposite case $\kappa\geq y^{2}$, by \eqref{bdosci}, we have 
\begin{equation}
\begin{split}
|w(z)-w(\zeta)|&\leq\mathrm{osc}_{\mathcal{B}_{y+\kappa}^{+}(z_{0})}w \lesssim\bar{\alpha}^{M}\|w\|_{L^{\infty}(\mathcal{B}_{1}^{+})}\lesssim \kappa^{c_0\frac{\log{\overline{\alpha}}}{\log(\delta/2)}}\|w\|_{L^{\infty}(\mathcal{B}_{1}^{+})}.
\end{split}
\end{equation}
This completes the proof of Proposition \ref{holder}.

\begin{comment}
\begin{remark} One may generalize the conclusions stated in this section to weak solution of \eqref{diverform} with slightly weaker regularity. For example, we can generalize to the weak solution of \eqref{diverform} with
\begin{equation*}
w\in L^{\infty}\cap BV,\ \p_{y}w\in BV_{y}
\end{equation*}
by modifying the procedures in this section but need a few tools for bounded variation functions (such as approximation, integration by parts, an so on; we refer to \cite{E92} or \cite{V67}). 

What's more, by making full use of the regularity effect of the operator $\mathcal{L}_{0}$, one may also generalize the regularity assumption for the weak solution of \eqref{diverform} in other types of functional spaces.
\end{remark}

\end{comment}
Finally, we formulate the analogue of Proposition \ref{holder} in the interior estimate case (without boundary), which was initially proved in \cite{Zhang2005} (see also \cite{GIMV19}). For $z_{0}=(t_{0},x_{0},y_{0})\in\R^{3}$ we adopt the notation 
\begin{equation*}
    \mathcal{B}_{r_{0}}(z_{0}):=(t_{0}-r_{0}^{2},t_{0}]\times(x_{0}-r_{0}^{3},x_{0}+r_{0}^{3})\times (y_{0}-r_{0},y_{0}+r_{0}).
\end{equation*}
Using a Galilean transform, it suffices to consider the case $y_0=0$.
\begin{proposition}\label{interiorholder}
Let $z_{0}=(t_{0},x_{0},0)\in \R^3$ and $r_{0}>0$. Assume that $w(t,x,y)\in L^{\infty}\cap W^{1,1}\cap W^{2,1}_{y}(\mathcal{B}_{r_{0}}(z_{0}))$ is a weak solution to
\begin{equation}\label{interiordiverform}
\p_{t}w+y\p_{x}w-\p_{y}(a\p_{y}w)=0,\ \text{in}\ \mathcal{B}_{r_{0}}(z_{0}),
\end{equation}
where $a(t,x,y)\in L^{\infty}(\mathcal{B}_{r_{0}}(z_{0}))$ satisfies the uniform ellipticity condition for $(t,x,y)\in \mathcal{B}_{r_{0}}(z_{0})$,
\begin{equation}\label{interiorelliptic}
c\leq a(t,x,y)\leq c^{-1},
\end{equation}
for some $c\in(0,\frac{1}{2}]$. Then we have $w(t,x,y)\in C^{\alpha}(\mathcal{B}_{\frac{r_{0}}{2}}(z_{0}))$ and the bounds
\begin{equation}\label{interiorholderesti}
 \|w\|_{C^{\alpha}(\mathcal{B}_{\frac{r_{0}}{2}}(z_{0}))}\lesssim_{c,r_{0}}\|w\|_{L^{\infty}(\mathcal{B}_{r_{0}}(z_{0}))},
\end{equation}
for some $\alpha>0$ independent of the solution $w$. 
\end{proposition}

% section 5

\section{Sobolev estimates for weak solutions to ultraparabolic equations in the half space}\label{sobolev}
In this section, we establish $L^{p}$--type Sobolev regularity estimates for the following linear inhomogeneous ultraparabolic equation with rough diffusion coefficients:
\begin{equation}\label{nondiver}
\begin{cases}
\p_{t}w+y\p_{x}w-\partial_y(a(t,x,y)\p_{y}w)=\partial_yf(t,x,y)+g(t,x,y),\ \text{in}\ \B_{r_{0}}^{+},\\
\p_{y}w|_{y=0}=0,
\end{cases}
\end{equation}
where $a(t,x,y)$ is assumed to be H\"older continuous with exponent $\alpha>0$, satisfying $c\leq a(t,x,y)\leq c^{-1}$ for some $c\in(0,1/2)$ on $\overline{\mathcal{B}_{r_{0}}^{+}}$, throughout the section.

Our goal is to prove a $W^{2,p}_y$ estimate \emph{up to the boundary $y=0$} for weak solutions $w\in L^{\infty}\cap W^{1,1}\cap W^{2,1}_{y}(\B_{r_{0}}^{+})$. We note that the interior $W^{2,p}_{y}$ estimates for system \eqref{nondiver} were established in \cite{BC96b}. The main new difficulties in our case are (i) the lack of an explicit formula for the fundamental solution to the Kolmogorov equation in the upper half space, and (ii) the very weak regularity assumption on the solution $w$ for obtaining $W^{2,p}_y$ bounds. Due to the limited regularity on $w$, we need to work with somewhat nonstandard regularity assumptions on the coefficient $a$.

Our first main result in this section is the following $W^{1,p}_{y}$ regularity estimate.
\begin{proposition}\label{sobolevesti}
Fix $r_0\in(0,1]$. Assume that $w\in L^{\infty}\cap W^{1,1}\cap W^{2,1}_{y}(\B_{r_{0}}^{+})$ satisfies \eqref{nondiver} with $g\in L^p(\B_{r_{0}}^{+})$ and $f\in W_y^{1,1}\cap L^{p}(\B_{r_{0}}^{+})$ for some $2\leq p<\infty$, and $f|_{y=0}=0$. Then there exists $r^{\ast}\in(0,\frac{r_{0}}{4}]$ depending only on the diffusion coefficient $a(t,x,y)$ such that
\begin{equation}\label{w2pdesire}
\|\p_{y}w\|_{L^{p}(\B_{r^\ast}^{+})}\lesssim _{p,r_0,c,[a]_{C^{\alpha}}} \|w\|_{L^{p}(\B_{r_0}^{+})}+\|f\|_{L^{p}(\B_{r_0}^{+})}+\|g\|_{L^{p}(\B_{r_0}^{+})}.
\end{equation}

Furthermore, when $p>6$, we also have
\begin{equation}\label{w2pdesire2}
    [w]_{C^{\beta}_{\ast}(\mathcal{B}_{r}^{+})}\lesssim_{p,r,c,[a]_{C^{\alpha}}}\|w\|_{L^{p}(\B_{2r}^{+})}+\|f\|_{L^{p}(\B_{2r}^{+})}+\|g\|_{L^{p}(\B_{2r}^{+})},\ \forall r\leq r^{\ast}.
\end{equation}
    In the above, $\beta=1-\frac{6}{p}$, and $C^{\beta}_{\ast}$ is the anisotropic H\"older space associated with $\mathcal{L}_{0}$ with the semi--norm in a domain $\Sigma\subset\mathbb{H}^{-}$:
\begin{equation}\label{newholder}
\begin{split}
    [w]_{C^{\beta}_{\ast}(\Sigma)}&:=\sup_{\delta\in(0,\frac{1}{2}],\,z\in\Sigma,\,\zeta\in\Sigma_{z,\delta}}\Big|\frac{w(\zeta)-w(z)}{\delta^{\beta}}\Big|.
    \end{split}
\end{equation}
where $z=(t_{0},x_{0},y_{0})$, and $\Sigma_{z,\delta}:=\big\{(t_{0},x_0+\delta^3,y_0), (t_0,x_0,y_0+\delta),(t_{0}+\delta^2,x_{0}+y_{0}\delta^2,y_{0})\big\}\cap \Sigma$.
\end{proposition}

\begin{comment}
\begin{remark}In particular, if $f\equiv 0$, then \eqref{w2pdesire} holds for all $1<p<\infty$.
\end{remark}
\end{comment}

For our applications below, we also need the following $W^{2,p}_y$ estimates.
\begin{proposition}\label{sob1}
Fix $r_0\in(0,1]$. Let $a\in C^{\beta}_{\ast}(\mathcal{B}_{r_{0}}^{+})$ for all $0<\beta<1$ and $\partial_ya\in L^q(\mathcal{B}_{r_{0}}^{+})$ for any $1<q<\infty$. Assume that $w\in L^{\infty}\cap W^{1,1}\cap W^{2,1}_{y}(\B_{r_{0}}^{+})$ satisfies \eqref{nondiver} with $g\in L^p(\B_{r_{0}}^{+})$ for some $3< p<\infty$, and $f\equiv0$. Then there exists $r^{\ast}\in(0,\frac{r_{0}}{4}]$ such that
\begin{equation}\label{sob2}
\|\p_{y}^2w\|_{L^{p}(\B_{r^\ast}^{+})}\lesssim _{p,r_0,c,a} \|w\|_{L^{\infty}(\B_{r_0}^{+})}+\|\partial_yw\|_{L^{p}(\B_{r_0}^{+})}+\|g\|_{L^{p}(\B_{r_0}^{+})}.
\end{equation}
In the above, $r^\ast$ and the implied constant depend on $a$ through the H\"older norm $\|a\|_{C^{\beta}_{\ast}}$ and $\|\partial_ya\|_{L^q(\mathcal{B}_{r_0}^+)}$ for $\beta$ close to $1$ and sufficiently large $q>1$ determined by $p$.
\end{proposition}
\begin{remark} The formulation of Propositions \ref{sobolevesti}--\ref{sob1} is adapted specifically to our applications to  the dynamical Prandtl equation below. More precisely, we will first apply Proposition \ref{sobolevesti} to get bounds on the first derivative in $y$ and improved H\"older estimates of the solution, which then allow us to apply Proposition \ref{sob1} to obtain control on the second derivative in $y$ of the solution, since in \eqref{prandtlcroccosec1}, the diffusion coefficient depends on the solution itself.%, it is natural to assume more regularity of $a(t,x,y)$ in Proposition \ref{sob1}.

The restriction to $p>3$ is for technical simplicity and is sufficient for our applications below. The case $p\leq 3$ can also be considered, but some of our arguments need suitable adjustments, see e.g. \eqref{sobadd3.3}.     
\end{remark}

\subsection{Estimates on volume potentials}

The following regularity estimates for \emph{volume potentials} associated with $\Gamma$ play an essential role in the proof of Propositions \ref{sobolevesti} and \ref{sob1}.

Let $g\in C_{0}^{\infty}(\mathbb{H}^{-})$ and define the following volume potentials
\begin{equation}\label{sob3}
w(t,x,y)=\int_{-\infty}^{t}\int_{-\infty}^{\infty}\int_{0}^{\infty}\Gamma(t,x,y;\tau,\xi,\eta)g(\tau,\xi,\eta)d\tau d\xi d\eta,
\end{equation}
\begin{equation}\label{sob4}
w^\ast(t,x,y)=\int_{-\infty}^{t}\int_{-\infty}^{\infty}\int_{0}^{\infty}\partial_\eta\Gamma(t,x,y;\tau,\xi,\eta)g(\tau,\xi,\eta)d\tau d\xi d\eta.
\end{equation}
\begin{lemma}[\bf H\"older estimates]
\label{potentialembed}
For $w$ and $w^{\ast}$ defined above, we have
\begin{equation}
\begin{cases}
\|w\|_{C^{\alpha}_{\ast}(\mathbb{H}^{-})}\lesssim \|g\|_{L^{p}(\mathbb{H}^{-})},\ \text{for}\ p\in(3,6),\ \alpha = 2-\frac{6}{p},\\
\|w^{\ast}\|_{C^{\tilde{\alpha}}_{\ast}(\mathbb{H}^{-})}\lesssim \|g\|_{L^{p}(\mathbb{H}^{-})},\ \text{for}\ p>6,\ \tilde{\alpha} = 1-\frac{6}{p}.
\end{cases}
\end{equation}
\end{lemma}
\begin{lemma}[\bf Sobolev estimates]
\label{constantcoeff}  
For $w$ and $w^{\ast}$ defined above, we have
\begin{equation}\label{sob5}
\|\p_{y}^{2}w\|_{L^{p}(\mathbb{H}^{-})}\lesssim_{p} \|g\|_{L^{p}(\mathbb{H}^{-})},\ \forall\ 1<p<\infty,
\end{equation}
\begin{equation}\label{sob6}
  \|\p_{y}w^\ast\|_{L^{p}(\mathbb{H}^{-})}\lesssim_{p} \|g\|_{L^{p}(\mathbb{H}^{-})},\ \forall\ 1<p<\infty.
\end{equation}
\end{lemma}

\begin{comment}
By a density argument, the above inequality still holds if we only assume that $g\in L^{p}(\mathbb{H}),\, 1<p<\infty$. {\color{red}Warning: For $g\in L^{p}$, the expression above may not well-defined? Unless $g$ has compacted supported with $t$.}
\end{comment}
\begin{remark}
Lemma \ref{potentialembed} and Lemma \ref{constantcoeff} provide global bounds for the following linear equation (where we assume $g_2|_{y=0}\equiv0$)
\begin{equation}\label{constanteq}
\begin{cases}
\p_{t}w+y\p_{x}w-\p_{y}^{2}w=g_1+\partial_yg_2,\ \text{in}\ \mathbb{H},\\
\p_{y}w|_{y=0}=0.
\end{cases}
\end{equation}
\end{remark}

We assume Lemma \ref{potentialembed} and Lemma \ref{constantcoeff} momentarily and first present the proof of Proposition \ref{sobolevesti} and Proposition \ref{sob1}. The proof of Lemma \ref{potentialembed} is similar to that of \eqref{differentialquo} in the next section; hence we omit it. The details of Lemma \ref{constantcoeff} will be given in subsections \ref{polc1} and \ref{polc2} below. Without loss of generality we shall assume that $a(0,0,0)=1$. The general case follows by a suitable rescaling.

\subsection{Proof of Proposition \ref{sobolevesti}}

We first give the proof of Proposition \ref{sobolevesti}. We can normalize 
$$\|w\|_{L^{p}(\B_{r_0}^{+})}+\|f\|_{L^{p}(\B_{r_0}^{+})}+\|g\|_{L^{p}(\B_{r_0}^{+})}=1.$$
Assume that $w\in L^{\infty}\cap W^{1,1}_{t,x}\cap W^{2,1}_{y}$ (hence $\p_{y}w\in L^{2}$ by interpolation) is a weak solution to
\eqref{nondiver}. Standard energy estimates imply that 
\begin{equation}\label{sobadd0.1}
    \|\partial_yw\|_{L^2(\mathcal{B}^+_{7r_0/8})}\lesssim_{r_0,c} 1.
\end{equation}

Let $\chi(t,x,y)$ be a smooth cut--off function such that $\chi|_{\mathcal{B}_r^+}\equiv1$, $\partial_y\chi|_{y=0}=0$ and $\chi\in C_c^\infty(\mathcal{B}_{2r}^+)$, for some $r\in(0,7r_0/16)$ to be determined below. We also choose a smooth cut--off function $\widetilde{\chi}\in C_c^\infty(\mathcal{B}_{4r}^+)$ with $\widetilde{\chi}\equiv 1$ on $\mathcal{B}_{2r}^+$. We have
\begin{equation}\label{sobadd1}
\begin{cases}
\p_{t}(w\chi)+y\p_{x}(w\chi)-\p_{y}^{2}(w\chi)=\p_{y}\big((a-1)\widetilde{\chi}\p_{y}(w\chi)\big)+h(t,x,y),\ in\ \mathbb{H}^{-},\\
\p_{y}(w\chi)|_{y=0}=0.
\end{cases}
\end{equation}

In the above, $h$ is given by
\begin{equation}\label{sobadd2}
h(t,x,y)=w(\p_{t}\chi+y\p_{x}\chi)-a\p_{y}w\p_{y}\chi-\p_{y}(aw\p_{y}\chi)+\partial_y(\chi f)+\chi g-\partial_y\chi f.
\end{equation} 
Define the mapping $\mathcal{T}$ as
\begin{equation}\label{sobadd2.9}
    \mathcal{T}[v]:=-\int_{-\infty}^{t}\iint_{\R\times\R^{+}}\p_{\eta}\Gamma(t,x,y;\tau,\xi,\eta)(a-1)\widetilde{\chi} \p_{\eta}v\, d\tau d\xi d\eta,
\end{equation}
and the inhomogeneous term $H$ as
\begin{equation}\label{sobadd3}
\begin{split}
H(t,x,y):=&\int_{-\infty}^{t}\iint_{\R\times\R^{+}}\p_{\eta}\Gamma(t,x,y;\tau,\xi,\eta)\big[-\chi f+aw\p_{\eta}\chi\big]\, d\tau d\xi d\eta\\
&+\int_{-\infty}^{t}\iint_{\R\times\R^{+}}\Gamma(t,x,y;\tau,\xi,\eta) \big[w(\p_{\tau}\chi+\eta\p_{\xi}\chi)-a\p_{\eta}w\p_{\eta}\chi+\chi g-\partial_\eta\chi f\big]d\tau d\xi d\eta.
\end{split}
\end{equation}
Set $p_1:=\min\{p, 3\}$.
By Lemma \ref{constantcoeff} and \eqref{sobadd0.1}, we have
\begin{equation}\label{sobadd7}
    \|\mathcal{T}[v]\|_{L^{p_1}\cap L^2(\mathbb{H}^-)}+\|\partial_y\mathcal{T}[v]\|_{L^{p_1}\cap L^2(\mathbb{H}^-)}\lesssim \|(a-1)\widetilde{\chi}\|_{L^\infty(\mathbb{H}^-)} \|\partial_\eta v\|_{L^{p_1}\cap L^2(\mathbb{H}^-)}.
\end{equation}
The inhomogeneous term $H$ satisfies the bound
\begin{equation}\label{sobadd8}
    \|\partial_yH\|_{L^2\cap L^{p_1}(\mathbb{H}^-)}+\|H\|_{L^2\cap L^{p_1}(\mathbb{H}^-)}\lesssim_{r} 1.
\end{equation}
In the above, for the contribution of $a\,\partial_\eta w\,\partial_\eta\chi$ to $H$, we also use Proposition \ref{operatorbd}: by \eqref{sobadd0.1}, its source belongs to $L^2$, and hence the corresponding potential and its $y$-derivative belong to $L^6$ and $L^3$, respectively; interpolation with the corresponding $L^2$ bounds yields \eqref{sobadd8}, since $p_1\leq 3$.

It follows from \eqref{sobadd1} that $W:=w\chi$ satisfies on $\mathbb{H}^-$ the integral equation
\begin{equation}
\begin{split}
W&=\mathcal{T}[W]+H.
\end{split}
\end{equation}

Let
\begin{equation}\label{x2}
X=\left\{u(t,x,y)\in L^{2}\cap L^{p_1}(\mathbb{H}^{-}),\ \p_{y}u\in L^{2}\cap L^{p_1}(\mathbb{H}^{-})\right\},
\end{equation}
with the norm
\begin{equation}\label{X1}
\|u\|_X:=\|u\|_{L^{2}\cap L^{p_1}(\mathbb{H}^{-})}+\|\partial_yu\|_{L^{2}\cap L^{p_1}(\mathbb{H}^{-})}.
\end{equation}
If $r<7r_0/8$ is chosen suitably small, then $\|(a-1)\widetilde{\chi}\|_{L^\infty(\mathbb{H}^-)}$ can be made small enough such that \eqref{sobadd7} implies 
\begin{equation}
\|\mathcal{T}[u_{1}]-\mathcal{T}[u_{2}]\|_{X}\leq \lambda_{0}\|u_{1}-u_{2}\|_{X},\ \text{for\ some}\ \lambda_{0}<1.
\end{equation}
By the fixed-point theorem, there exists a unique $V\in X$ such that $V=\mathcal{T}[V]+H$. Recall that $W=\mathcal{T}[W]+H$. 
Hence by Lemma \ref{constantcoeff},
\begin{equation}
\|\p_{y}(W-V)\|_{L^{2}}=\|\p_{y}(\mathcal{T}[W]-\mathcal{T}[V])\|_{L^{2}}\leq \lambda_0\|\p_{y}(W-V)\|_{L^{2}}.
\end{equation}
It follows that $\p_{y}W=\p_{y}V\in L^{p_1}$, and therefore $\partial_yw\,\chi\in L^{p_1}(\mathbb{H}^-)$. 

We may now repeat the argument with a cutoff supported in a smaller half-ball. If \(p_j<6\), set
\[
p_{j+1}:=\min\left\{p,\frac{6p_j}{6-p_j}\right\}.
\]
If the iteration reaches \(p_j=6\), then \(L^6\subset L^{6-\delta}\) on a smaller half-ball for every \(\delta>0\), and one further subcritical iteration yields
\[
\partial_y w\in L^{\frac{6(6-\delta)}{\delta}}.
\]
Choosing \(\delta>0\) sufficiently small so that \(6(6-\delta)/\delta\ge p\) completes the bootstrap, and we conclude the first part of Proposition \ref{sobolevesti}. 

Next we show the extra H\"older continuity estimate whenever $p>6$. Recall that $W=\mathcal{T}[W]+H$. By Lemma \ref{potentialembed},
\begin{equation}\label{holderW}
    \begin{split}
        &\|\mathcal{T}[W]\|_{C^{1-\frac{6}{p}}_{\ast}(\mathbb{H}^{-})}\lesssim_{p} \|(a-1)\widetilde{\chi} \p_{\eta}W\|_{L^{p}(\mathbb{H}^{-})},\\
        &\|H\|_{C^{1-\frac{6}{p}}_{\ast}(\mathbb{H}^{-})}\lesssim_{p}\left(\|-\chi f+aw\p_{\eta}\chi\|_{L^{p}}+\|w(\p_{\tau}\chi+\eta\p_{\xi}\chi)-a\p_{\eta}w\p_{\eta}\chi+\chi g-\partial_\eta\chi f\|_{L^{\frac{6p}{6+p}}}\right).
    \end{split}
\end{equation}
Hence, \eqref{w2pdesire2} follows from \eqref{w2pdesire} and \eqref{holderW}.
\subsection{Proof of Proposition \ref{sob1}}

We now turn to the proof of Proposition \ref{sob1}. Our idea is similar to that in the proof of Proposition \ref{sobolevesti}, using localization, well-posedness of the localized equation in a higher regularity space, and uniqueness to improve the regularity of $w$. We again normalize
$$\|w\|_{L^{\infty}(\B_{r_0}^{+})}+\|\partial_yw\|_{L^{p}(\B_{r_0}^{+})}+\|g\|_{L^{p}(\B_{r_0}^{+})}=1.$$
%For clarity, we present the argument only for $p>3$. When $2\leq p\leq 3$, the same contraction argument applies after dropping the $L^\infty$ component from the space $Y$ in \eqref{sobadd3.3}; the pointwise estimate \eqref{eqW2p3} is then unnecessary.

Let $\chi,\, \widetilde{\chi}$ be the same smooth cutoff functions as in \eqref{sobadd1}. Then
\begin{equation}\label{sobadd10}
\begin{cases}
\p_{t}(w\chi)+y\p_{x}(w\chi)-\p_{y}^{2}(w\chi)=\partial_y((a-1)\widetilde{\chi}\p_{y}(w\chi))+h^\ast(t,x,y),\ in\ \mathbb{H}^{-},\\
\p_{y}(w\chi)|_{y=0}=0.
\end{cases}
\end{equation}
In the above, $h^\ast$ is given by (since $f=0$)
\begin{equation}\label{sobadd11}
h^\ast(t,x,y)=w(\p_{t}\chi+y\p_{x}\chi)-a\p_{y}w\p_{y}\chi-\p_{y}(aw\p_{y}\chi)+\chi g.
\end{equation} 
\begin{comment}
We can construct a H\"older continuous coefficient $a^\ast$ defined on $\mathbb{H}^-$ such that $a^\ast\equiv a-1$ on $\rm{supp}\,\chi$ and satisfies for any $q\in(1,\infty)$,
\begin{equation}\label{sobadd10.1}
    \|a^\ast\|_{L^\infty(\mathbb{H}^-)}\lesssim \|(a-1)\widetilde{\chi}\|_{L^\infty(\mathbb{H}^-)},\quad \|\partial_ya^\ast\|_{L^q(\mathbb{H}^-)}\lesssim_q \|\partial_ya\|_{L^q(\mathcal{B}_{r_0}^+)}+\|a\|_{L^\infty(\mathcal{B}_{r_0}^+)}.
\end{equation}
\end{comment}

We then define the mapping $\mathcal{T}^\ast$ as
\begin{equation}\label{sobadd2.91}
    \mathcal{T}^\ast[v]:=\int\limits_{-\infty}^{t}\iint_{\R\times\R^{+}}\Gamma(t,x,y;\tau,\xi,\eta)\p_\eta\big((a-1)\widetilde{\chi} \p_{\eta}v\big)\, d\tau d\xi d\eta,
\end{equation}
and the inhomogeneous term $H^\ast$ as (with $z=(t,x,y)$, $\zeta=(\tau,\xi,\eta)$ and $f\equiv0$)
\begin{equation}\label{sobadd3.00}
\begin{split}
H^\ast(t,x,y):=\int\limits_{-\infty}^{t}\iint_{\R\times\R^{+}}\Gamma(z;\zeta)h^\ast(\tau,\xi,\eta)\, d\tau d\xi d\eta.
\end{split}
\end{equation}

\begin{comment}
We decompose the mapping $\mathcal{T}^\ast=\mathcal{T}^\ast_1+\mathcal{T}^\ast_2$ where
\begin{equation}\label{sobadd3.01}
    \begin{split}
        &\mathcal{T}^\ast_1[v]:=\int\limits_{-\infty}^{t}\iint_{\R\times\R^{+}}\Gamma(t,x,y;\tau,\xi,\eta)(a-1)\widetilde{\chi} \p^2_{\eta}v\, d\tau d\xi d\eta,\\
        &\mathcal{T}^\ast_2[v]:=\int\limits_{-\infty}^{t}\iint_{\R\times\R^{+}}\Gamma(t,x,y;\tau,\xi,\eta)\p_\eta\big((a-1)\widetilde{\chi}\big) \p_{\eta}v\, d\tau d\xi d\eta.
    \end{split}
\end{equation}
\end{comment}

Define the space 
\begin{equation}\label{sobadd3.2}
    Y:= W^{2,2}_{y}\cap W^{2,p}_{y}\cap L^{\infty}(\mathbb{H}^{-}), 
\end{equation}
with the norm
\begin{equation}\label{sobadd3.3}
\|v\|_Y:=\sum_{\alpha\in\{0,1,2\}}\|\partial_y^\alpha v\|_{L^{2}\cap L^{p}(\mathbb{H}^-)}+\|v\|_{L^{\infty}(\mathbb{H}^{-})}.
\end{equation}
Notice that the standard interpolation inequality shows that
\begin{equation*}
    \|\p_{y}v\|_{L^{2p}(\mathbb{H}^{-})}\lesssim \|v\|_{Y}.
\end{equation*}
Applying Lemma \ref{constantcoeff}, we have
\begin{equation}\label{eqW2p1}
    \begin{split}
        \|\p_{y}^{2}\mathcal{T}^{\ast}[v]\|_{L^{2}\cap L^{p}}&\lesssim_{p}\left\|(a-1)\widetilde{\chi}\p_{y}^{2}v+\p_{y}\big((a-1)\widetilde{\chi}\big) \p_{y}v\right\|_{L^{2}\cap L^{p}}\\
        &\lesssim\left(\|\p_{y}a\|_{L^{2p}(\mathrm{supp}(\widetilde{\chi}))}+r^{\frac{3}{p}-1}\|a-1\|_{L^{\infty}(\mathrm{supp}(\widetilde{\chi}))}\right)\cdot\|v\|_{Y}\\
        &\leq \left(\|\p_{y}a\|_{L^{2p}(\mathrm{supp}(\widetilde{\chi}))}+r^{\frac{3}{p}+\beta-1}[a]_{C^{\beta}_{\ast}(\mathcal{B}_{r_{0}}^{+})}\right)\cdot\|v\|_{Y} \leq \frac{1}{10}\|v\|_{Y},
    \end{split}
\end{equation}
by choosing $r\leq \frac{r_{0}}{4}$ suitably small.

Applying Schur's test (Lemma \ref{schur}), we have
\begin{equation}\label{eqW2p2}
        \|\mathcal{T}^{\ast}[v]\|_{W^{1,2}_{y}\cap W^{1,p}_{y}}\lesssim \left\|(a-1)\widetilde{\chi}\p_{y}^{2}v+\p_{y}\big((a-1)\widetilde{\chi}\big) \p_{y}v\right\|_{L^{2}\cap L^{p}}\leq \frac{1}{10}\|v\|_{Y}.
\end{equation}

By H\"older's inequality and Proposition \ref{ptwise}, we also obtain that
\begin{equation}\label{eqW2p3}
        |\mathcal{T}^{\ast}[v](z)|\leq \|\Gamma(z;\cdot)\|_{L^{p'}(\mathrm{supp}(\widetilde{\chi})\cap \{\tau<t\})}\cdot \left\|(a-1)\widetilde{\chi}\p_{y}^{2}v+\p_{y}\big((a-1)\widetilde{\chi}\big) \p_{y}v\right\|_{L^{p}}\leq \frac{1}{10}\|v\|_{Y}.
\end{equation}

For the non--homogeneous term $H^{\ast}$, using Lemma \ref{constantcoeff} and the compact support property of $\chi$, we also have
\begin{equation}\label{sobadd3.5}
    \|H^\ast\|_{Y}\lesssim_r1. 
\end{equation}

In view of \eqref{eqW2p1}-\eqref{sobadd3.5} and using the contraction mapping theorem, we can conclude that there is a unique $V\in Y$ such that $V=\mathcal{T}^{\ast}[V]+H^{\ast}$. It follows from equations \eqref{sobadd10} and \eqref{sobadd11} that $W=w\chi$ also satisfies $W=\mathcal{T}^\ast[W]+H^\ast$. Choosing $r\in(0,\frac{r_{0}}{4}]$ sufficiently small, by \eqref{sob6} and the fact that $\p_{y}W|_{y=0}=\p_{y}V|_{y=0}=0$, we have
\begin{equation}
\|\p_{y}(W-V)\|_{L^{2}}\lesssim \|(a-1)\|_{L^{\infty}(\mathrm{supp}(\widetilde{\chi}))}\cdot \|\p_{y}(W-V)\|_{L^{2}}\leq\frac{1}{2}\|\p_{y}(W-V)\|_{L^{2}}.
\end{equation}
Hence $\p_{y}W=\p_{y}V$. We conclude \eqref{sob2} since $V\in Y$.

\subsection{Proof of Lemma \ref{constantcoeff},\ $p=2$}\label{polc1}
We turn to the proof of Lemma \ref{constantcoeff}. By integration by parts, it follows from \eqref{sob4} that $w^\ast$ satisfies \eqref{constanteq} with $g_1=0, \,\,g_2=-g$. Then \eqref{sob6} follows by multiplying \eqref{constanteq} with $w^{\ast}$ and integrating by parts. 

It remains to prove the estimate \eqref{sob5} for $p=2$, which is implied by the following Lemma:

\begin{lemma}\label{h2esti}
Assume that $w$ is the ancient mild solution to
\begin{equation}\label{constantequ}
\begin{cases}
\p_{t}w+y\p_{x}w-\p_{y}^{2}w=g\in C_{0}^{\infty}(\mathbb{H}),\\
\p_{y}w|_{y=0}=0,
\end{cases}
\end{equation}
with $\p_{y}w,\p_{y}^{2}w\in L^{2}(\mathbb{H})$. Then we have
\begin{equation}
\|\p_{y}^{2}w\|_{L^{2}(\mathbb{H})}\lesssim \|g\|_{L^{2}(\mathbb{H})}.
\end{equation}
In particular, Lemma \ref{constantcoeff} holds for $p=2$.
\end{lemma}

\begin{proof}
Taking the Fourier transform with respect to the $x$--variable, we get
\begin{equation}\label{fteq}
\begin{cases}
\p_{t}\widehat{w}(t,\xi,y)+i2\pi y\xi\widehat{w}(t,\xi,y)-\p_{y}^{2}\widehat{w}(t,\xi,y)=\widehat{g\,}(t,\xi,y),\\
\p_{y}\widehat{w}|_{y=0}=0.
\end{cases}
\end{equation}

We need to prove that for any $\xi\in\R$,
\begin{equation}\label{fourieresti}
\|\p_{y}^{2}\widehat{w}\|_{L^{2}_{t,y}}\lesssim \|\widehat{g\,}\|_{L^{2}_{t,y}}.
\end{equation}
 Lemma \ref{h2esti} follows from \eqref{fourieresti} and Plancherel's identity.

Notice that \eqref{fourieresti} is obvious when $\xi=0$ by the standard energy estimates for the one-dimensional heat equation. 

We assume without loss of generality $\xi>0$, and apply the enhanced dissipation estimate (Proposition \ref{enhanced}) and Duhamel's principle. It follows that for some $\delta_0>0$,
\begin{equation}
\|\widehat{w}(t,\xi,\cdot)\|_{L^{2}_{y}}\lesssim \int_{-\infty}^{t}e^{-\delta_0\xi^{\frac{2}{3}}(t-s)}\|\widehat{g\,}(s,\xi,\cdot)\|_{L^{2}_{y}}ds.
\end{equation}
By Young's inequality, we conclude that
\begin{equation}\label{enhancebound}
\|\widehat{w}(\cdot,\xi,\cdot)\|_{L^{2}_{t,y}}\lesssim \xi^{-\frac{2}{3}}\|\widehat{g\,}(\cdot,\xi,\cdot)\|_{L^{2}_{t,y}}.
\end{equation}

Next we turn to the $H^{2}_{y}$ estimate. Testing \eqref{fteq} by $-\overline{\p_{y}^{2}\widehat{w}}$, integrating by parts and taking the real part, we have (with $D:=\R\times[0,\infty)$)
\begin{equation}
\iint_{D} |\p_{y}^{2}\widehat{w}|^{2}dtdy\lesssim\xi\iint_D |\widehat{w}||\p_{y}\widehat{w}|\,dtdy+\iint_D |\widehat{g\,}|^{2}dtdy.
\end{equation}
Thanks to $\p_{y}\widehat{w}|_{y=0}=0$, applying the interpolation inequality in $y$--direction and combining with the enhanced dissipation estimate \eqref{enhancebound}, we obtain that
\begin{equation}\label{enhancedl2added}
\begin{split}
\iint_D |\p_{y}^{2}\widehat{w}|^{2}dtdy&\lesssim \xi \|\widehat{w}\|_{L^{2}_{t,y}}\|\p_{y}\widehat{w}\|_{L^{2}_{t,y}}+\|\widehat{g\,}\|_{L^{2}_{t,y}}^{2}\lesssim \xi \|\widehat{w}\|_{L^{2}_{t,y}}^{\frac{3}{2}}\|\p_{y}^{2}\widehat{w}\|_{L^{2}_{t,y}}^{\frac{1}{2}}+\|\widehat{g\,}\|_{L^{2}_{t,y}}^{2}\\
&\leq \frac{1}{2}\|\p_{y}^{2}\widehat{w}\|_{L^{2}_{t,y}}^{2}+C\big(\xi^{\frac{4}{3}}\|\widehat{w}\|_{L^{2}_{t,y}}^{2}+\|\widehat{g\,}\|_{L^{2}_{t,y}}^{2}\big)\leq\frac{1}{2}\|\p_{y}^{2}\widehat{w}\|_{L^{2}_{t,y}}^{2}+C\|\widehat{g\,}\|_{L^{2}_{t,y}}^{2}.
\end{split}
\end{equation}
The proof of Lemma \ref{h2esti} follows from \eqref{enhancedl2added} by integrating in $\xi$.
\end{proof}

\subsection{Proof of Lemma \ref{constantcoeff},\ $1<p<\infty$.}\label{polc2}
In this subsection we consider the general case when $1<p<\infty$ and complete the proof of Lemma \ref{constantcoeff}. We present the detailed proof only for the bound \eqref{sob5}, since the proof for \eqref{sob6} follows similarly. 
\begin{proof}[Proof of  Lemma \ref{constantcoeff}]
Let $p'$ be the dual exponent of $p$, that is $\frac{1}{p}+\frac{1}{p'}=1$. By duality, we need to prove that, for any $\phi(t,x,y)\in C_{0}^{\infty}(\mathbb{H})$,
\begin{equation}\label{dualgoal}
\left|\int_{\mathbb{H}}(\p_{y}^{2}w)(t,x,y)\phi(t,x,y)\,dxdydt\right|\lesssim \|g\|_{L^{p}(\mathbb{H})}\|\phi\|_{L^{p'}(\mathbb{H})},
\end{equation}
where with $z=(t,x,y), \,\zeta=(\tau,\xi,\eta)$,
\begin{equation}\label{expression}
w(z) = \int_{-\infty}^{t}\int_{-\infty}^{\infty}\int_{0}^{\infty}\Gamma(z;\zeta)g(\zeta)\,d\eta d\xi d\tau.
\end{equation}

Due to the singularity of the fundamental solution near the source, we cannot take $\p_{y}^{2}$ into the integrand in \eqref{expression} directly. To overcome this technical difficulty, for any $\epsilon>0$, we decompose
\begin{equation}\label{singudecom}
\begin{split}
w(z)&=\int_{-\infty}^{t}\int_{-\infty}^{\infty}\int_{0}^{\infty}\chi\left(\frac{t-\tau}{\epsilon}\right)\Gamma(z;\zeta)g(\zeta)\, d\eta d\xi d\tau\\
&\quad+\int_{-\infty}^{t}\int_{-\infty}^{\infty}\int_{0}^{\infty}\left[1-\chi\left(\frac{t-\tau}{\epsilon}\right)\right]\Gamma(z;\zeta)g(\zeta)\,d\eta d\xi d\tau\\
&=:w_{1}(z)+w_{2}(z),
\end{split}
\end{equation}
where $\chi(\cdot)$ is a standard smooth cut--off function on the line (supported in the far field), satisfying
$
\chi(\sigma)=
0\,\,{\rm for}\,\, \sigma\leq 5\,\,{\rm and}\,\, \chi(\sigma)=
1\,\,{\rm for}\,\, \sigma\geq 10.
$

Hence,
\begin{equation}
\begin{split}
\left|\int_{\mathbb{H}}(\p_{y}^{2}w)(t,x,y)\phi(t,x,y)\,dxdydt\right|&\leq \left|\int_{\mathbb{H}}(\p_{y}^{2}w_{1})(t,x,y)\phi(t,x,y)\,dxdydt\right|\\
&\quad+\left|\int_{\mathbb{H}}(\p_{y}^{2}w_{2})(t,x,y)\phi(t,x,y)\,dxdydt\right|\\
&=:A_{g,\phi}+B_{g,\phi}.
\end{split}
\end{equation}

\textit{Estimates of $B_{g,\phi}$.} For the term $B_{g,\phi}$, integrating by parts and applying Proposition \ref{highptwise} and Schur's test (Lemma \ref{schur}), we get that
\begin{equation}\label{bgphibd}
B_{g,\phi}\leq \|\p_{y}w_{2}\|_{L^{p}}\|\p_{y}\phi\|_{L^{p'}}\lesssim \epsilon^{\frac{1}{2}}\|g\|_{L^{p}}\|\p_{y}\phi\|_{L^{p'}}.
\end{equation}

\textit{Estimates of $A_{g,\phi}$.} Recall that the fundamental solution $\Gamma(z;\zeta)=\Gamma(t,x,y;\tau,\xi,\eta)$ satisfies different estimates for $\eta\gtrsim \|z-\zeta\|$ and $\eta\lesssim \|z-\zeta\|$. To estimate $A_{g,\phi}$, using \eqref{greendecom} we decompose $w_{1}(z)$ into three parts,
\begin{equation}\label{decomw1}
\begin{split}
w_{1}(z)=&\int_{-\infty}^{t}\int_{-\infty}^{\infty}\int_{0}^{\infty}\chi\left(\frac{t-\tau}{\epsilon}\right)\left[1-\chi\left(\frac{\eta}{\|z-\zeta\|}\right)\right]\Gamma(z;\zeta)g(\zeta) \,d\eta d\xi d\tau\\
&+\int_{-\infty}^{t}\int_{-\infty}^{\infty}\int_{0}^{\infty}\chi\left(\frac{t-\tau}{\epsilon}\right)\chi\left(\frac{\eta}{\|z-\zeta\|}\right)\widetilde{\Gamma}(z;\zeta)g(\zeta)\,d\eta d\xi d\tau\\
&+\int_{-\infty}^{t}\int_{-\infty}^{\infty}\int_{0}^{\infty}\chi\left(\frac{t-\tau}{\epsilon}\right)\chi\left(\frac{\eta}{\|z-\zeta\|}\right)\eta^{-4}V\Big(\frac{t-\tau}{\eta^{2}},\frac{x-\xi}{\eta^{3}},\frac{y}{\eta}\Big)g(\zeta)\,d\eta d\xi d\tau\\
=&w_{11}(z)+w_{12}(z)+w_{13}(z).
\end{split}
\end{equation}

Notice from the definition that $w_{12}(z)$ can be extended from $\mathbb{H}$ to $\R^3$. 
By direct computation, $w_{12}(z)$ satisfies the equation
\begin{equation}\label{w12eq}
\begin{split}
&\p_{t}w_{12}+y\p_{x}w_{12}-\p_{y}^{2}w_{12}=g_{12}(z),\ (t,x,y)\in\R^{3},\\
&g_{12}(z)=\int_{-\infty}^{t}\int_{-\infty}^{\infty}\int_{0}^{\infty}\mathcal{K}_{12}(z,\zeta)g(\zeta)d\zeta.
\end{split}
\end{equation}
In the above we have used the notation that
\begin{equation}\label{kernelk12}
\begin{split}
\mathcal{K}_{12}=&\frac{1}{\epsilon}\chi'\big(\frac{t-\tau}{\epsilon}\big)\chi\big(\frac{\eta}{\|z-\zeta\|}\big)\widetilde{\Gamma}(z;\zeta)-\chi\big(\frac{t-\tau}{\epsilon}\big)\chi'\big(\frac{\eta}{\|z-\zeta\|}\big)\widetilde{\Gamma}(z;\zeta)\frac{\eta (\p_{t}+y\p_{x})\|z-\zeta\|}{\|z-\zeta\|^{2}}\\
&+2\chi\big(\frac{t-\tau}{\epsilon}\big)\chi'\big(\frac{\eta}{\|z-\zeta\|}\big)\p_{y}\widetilde{\Gamma}(z;\zeta)\frac{\eta}{\|z-\zeta\|^{2}}\p_{y}\|z-\zeta\|\\
&-\chi\big(\frac{t-\tau}{\epsilon}\big)\widetilde{\Gamma}(z;\zeta)\bigg[\chi''\big(\frac{\eta}{\|z-\zeta\|}\big)\eta^{2}\|z-\zeta\|^{-4}(\p_{y}\|z-\zeta\|)^{2}\\
&+2\chi'\big(\frac{\eta}{\|z-\zeta\|}\big)\eta\|z-\zeta\|^{-3}(\p_{y}\|z-\zeta\|)^{2}-\chi'\big(\frac{\eta}{\|z-\zeta\|}\big)\eta\|z-\zeta\|^{-2}(\p_{y}^{2}\|z-\zeta\|)\bigg].
\end{split}
\end{equation}
By the $W^{2,p}_{y}$ estimate for the operator $\mathcal{L}_{0}=\p_{t}+y\p_{x}-\p_{y}^{2}$ in the whole space \cite{BC96b}, we have
\begin{equation}\label{wholespacew2p}
\|\p_{y}^{2}w_{12}(z)\|_{L^{p}}\lesssim_{p} \|g_{12}(z)\|_{L^{p}}\lesssim \|g\|_{L^{p}}.
\end{equation}
We have employed Schur's test in the last inequality, as well as the explicit expression of $\widetilde{\Gamma}(z;\zeta)$.

Next, we need to estimate $w_{11}(z)$ and $w_{13}(z)$. Let
\begin{equation}
\begin{split}
&T_{11}(g)(z):=\p_{y}^{2}w_{11}(z)=\int_{-\infty}^{t}\int_{-\infty}^{\infty}\int_{0}^{\infty}\mathcal{K}_{11}(z,\zeta)g(\zeta)\,d\zeta,\\
&\mathcal{K}_{11}(z,\zeta)=\chi\big(\frac{t-\tau}{\epsilon}\big)\p_{y}^{2}\left\{\left[1-\chi\left(\frac{\eta}{\|z-\zeta\|}\right)\right]\Gamma(z;\zeta)\right\}.
\end{split}
\end{equation}
We need to deduce that $T_{11}$ is bounded on $L^{p}(\mathbb{H})$.

First we show that $T_{11}$ is bounded on $L^{2}(\mathbb{H})$. To this end, notice that $w_{11}(z)$ satisfies the following equation
\begin{equation}\label{w11eq}
\begin{cases}
\p_{t}w_{11}+y\p_{x}w_{11}-\p_{y}^{2}w_{11}=:\widetilde{g}(t,x,y),\\
\p_{y}w_{11}|_{y=0}=0,
\end{cases}
\end{equation}
where $\widetilde{g}$ is given by
\begin{equation}
\widetilde{g}(t,x,y)=\int_{-\infty}^{t}\int_{-\infty}^{\infty}\int_{0}^{\infty}\widetilde{\mathcal{K}}_{11}(z,\zeta)g(\zeta)d\zeta,
\end{equation}
with
\begin{equation}
\begin{split}
\widetilde{\mathcal{K}}_{11}(z,\zeta)=&\frac{1}{\epsilon}\chi'\big(\frac{t-\tau}{\epsilon}\big)\left[1-\chi\big(\frac{\eta}{\|z-\zeta\|}\big)\right]\Gamma(z;\zeta)\\
&+\chi\big(\frac{t-\tau}{\epsilon}\big)\chi'\big(\frac{\eta}{\|z-\zeta\|}\big)\Gamma(z;\zeta)\frac{\eta (\p_{t}+y\p_{x})\|z-\zeta\|}{\|z-\zeta\|^{2}}\\
&-2\chi\big(\frac{t-\tau}{\epsilon}\big)\chi'\big(\frac{\eta}{\|z-\zeta\|}\big)\p_{y}\Gamma(z;\zeta)\frac{\eta}{\|z-\zeta\|^{2}}(\p_{y}\|z-\zeta\|)\\
&+\chi\big(\frac{t-\tau}{\epsilon}\big)\Gamma(z;\zeta)\bigg[\chi''\big(\frac{\eta}{\|z-\zeta\|}\big)\eta^{2}\|z-\zeta\|^{-4}(\p_{y}\|z-\zeta\|)^{2}\\
&\quad\quad+2\chi'\big(\frac{\eta}{\|z-\zeta\|}\big)\eta\|z-\zeta\|^{-3}(\p_{y}\|z-\zeta\|)^{2}\\
&\quad\quad-\chi'\big(\frac{\eta}{\|z-\zeta\|}\big)\eta\|z-\zeta\|^{-2}(\p_{y}^{2}\|z-\zeta\|)\bigg].
\end{split}
\end{equation}
By Lemma \ref{h2esti}, equation \eqref{w11eq}, and Schur's test, we have
\begin{equation}
\|T_{11}(g)\|_{L^{2}(\mathbb{H})}\lesssim \|\widetilde{g}\|_{L^{2}(\mathbb{H})}\lesssim \|g\|_{L^{2}(\mathbb{H})},
\end{equation}
which means $T_{11}$ is bounded from $L^{2}(\mathbb{H})$ to $L^{2}(\mathbb{H})$.

On the other hand, by Proposition \ref{ptwise} and Proposition \ref{highptwise}, we have the following pointwise bound and continuity bounds of the kernel $\mathcal{K}_{11}(z,\zeta)$, which means the kernel satisfies H\"ormander's condition of the non--convolution type of SIO in the homogeneous metric space $\mathbb{H}$ equipped with the metric $d(z,\zeta)=\|z-\zeta\|$ and the Lebesgue measure.
\begin{equation}\label{bdkernel}
\begin{split}
&|\mathcal{K}_{11}(z,\zeta)|\lesssim \frac{1}{\|z-\zeta\|^{6}},\ z\neq \zeta,\\
&|\mathcal{K}_{11}(z,\zeta)-\mathcal{K}_{11}(z_{0},\zeta)|\lesssim \frac{\|z_{0}-z\|}{\|z_{0}-\zeta\|^{7}},
\end{split}
\end{equation}
provided  $ \|z-z_{0}\|\leq c\|\zeta-z_{0}\|$; and 
\begin{equation}\label{bdkernell}
|\mathcal{K}_{11}(z,\zeta)-\mathcal{K}_{11}(z,\zeta_{0})|\lesssim \frac{\|\zeta_{0}-\zeta\|}{\|z-\zeta_{0}\|^{7}},
\end{equation}
provided   $ \|\zeta-\zeta_{0}\|\leq c\|z-\zeta_{0}\|,$
for some $0<c<1$.

Notice that $T_{11}$ is bounded from $L^{2}(\mathbb{H})$ to $L^{2}(\mathbb{H})$. Hence, by \eqref{bdkernel}, \eqref{bdkernell}, and Theorem 4.7 in \cite{BC96a} (see also \cite{CG70}, \cite{CW71}, \cite{St70}, \cite{St93}), we have that $T_{11}$ is bounded from $L^{p}(\mathbb{H})$ to $L^{p}(\mathbb{H})$ for any $p\in(1,\infty)$. 

Finally, we estimate $w_{13}(z)$. As before, define
\begin{equation}\label{4.60}
 \begin{split}
&T_{13}(g)(z):=\p_{y}^{2}w_{13}(z)=\int_{-\infty}^{t}\int_{-\infty}^{\infty}\int_{0}^{\infty}\mathcal{K}_{13}(z,\zeta)g(\zeta)\,d\zeta,\\
&\mathcal{K}_{13}(z,\zeta)=\chi\left(\frac{t-\tau}{\epsilon}\right)\eta^{-4}\p_{y}^{2}\left[\chi\left(\frac{\eta}{\|z-\zeta\|}\right)V\left(\frac{t-\tau}{\eta^{2}},\frac{x-\xi}{\eta^{3}},\frac{y}{\eta}\right)\right].
\end{split}
\end{equation}
First, notice with $T_{1}(g)(z)=\p_{y}^{2}w_{1}(z),\, T_{12}(g)(z)=\partial_y^2w_{12}(z)$ that
\begin{equation*}
\|T_{13}\|_{L^{2}\to L^{2}}\leq\|T_{1}\|_{L^{2}\to L^{2}}+\|T_{11}\|_{L^{2}\to L^{2}}+\|T_{12}\|_{L^{2}\to L^{2}}.
\end{equation*}

Similarly to the $L^{2}$ boundedness of $T_{11}$, we can deduce that $T_{1}$ is bounded on $L^{2}(\mathbb{H})$ by employing Lemma \ref{h2esti} again. The $L^{2}$ boundedness of $T_{12}$ follows from \eqref{wholespacew2p}.

Thus, $T_{13}$ is bounded on $L^{2}(\mathbb{H})$. 
%Furthermore, by direct calculation, the kernel $\mathcal{K}_{13}$ also satisfies \eqref{bdkernel}. Therefore, $T_{13}$ is bounded on $L^{p}(\mathbb{H})$ for any $p\in(1,\infty)$.
Moreover, differentiating the expression for $K_{13}$ in \eqref{4.60} with respect to either kernel variable and using the preceding estimates for $V$, we obtain
\[
\left|K_{13}(z,\zeta)-K_{13}(z,\zeta_0)\right|
 \lesssim
 \frac{\|\zeta-\zeta_0\|}{\|z-\zeta_0\|^{7}},
 \qquad
 \|\zeta-\zeta_0\|\leq \frac12\|z-\zeta_0\|,
\]
and, analogously,
\[
\left|K_{13}(z,\zeta)-K_{13}(z_0,\zeta)\right|
 \lesssim
 \frac{\|z-z_0\|}{\|z_0-\zeta\|^{7}},
 \qquad
 \|z-z_0\|\leq \frac12\|z_0-\zeta\|.
\]
Thus, together with the size estimate as in \eqref{bdkernel}, $K_{13}$ satisfies the standard kernel conditions for a nonconvolution Calderón--Zygmund operator, and hence $T_{13}$ is bounded on $L^p$ for every $1<p<\infty$.

In summary, we conclude that the operator $T_{1}$ is bounded on $L^{p}(\mathbb{H})$ for any $p\in(1,\infty)$. By H\"older's inequality, we have
\begin{equation}\label{agphibound}
A_{g,\phi}\leq \|T_{1}g\|_{L^{p}(\mathbb{H})}\|\phi\|_{L^{p'}(\mathbb{H})}\lesssim \|g\|_{L^{p}(\mathbb{H})}\|\phi\|_{L^{p'}(\mathbb{H})}.
\end{equation}
Combining with \eqref{bgphibd}, we obtain \eqref{dualgoal} by the arbitrariness of $\epsilon>0$. Therefore we complete the proof of Lemma \ref{constantcoeff}.
\end{proof}
We end this section by stating statements analogous to Propositions \ref{sobolevesti} and \ref{sob1} for interior estimates. Since we have the explicit expression \eqref{wholespace} for the fundamental solution, the proofs are similar to those above and indeed simpler; hence we omit them. Here we adopt the notation
\begin{equation*}
    \mathcal{B}_{r_{0}}:=(-r_{0}^{2},0]\times(-r_{0}^{3},r_{0}^{3})\times (-r_{0},r_{0}),
\end{equation*}
and the diffusion coefficient $a(t,x,y)$ is assumed to be H\"older continuous with exponent $\alpha>0$, satisfying $c\leq a(t,x,y)\leq c^{-1}$ for some $c\in(0,1/2)$ on $\overline{\mathcal{B}_{r_{0}}}$.
\begin{comment}
\begin{remark} We note that Lemma \ref{constantcoeff} can be easily generalized to the following version, which provides bounds on ancient solutions: For any $g\in C_{0}^{\infty}(\mathbb{H})$, and $\tau_{0}\in\R$, we have
\begin{equation}\label{ancientw2p}
\|\p_{y}^{2}w\|_{L^{p}((-\infty,\tau_{0}]\times\R\times\R^{+})}\lesssim_{p} \|g\|_{L^{p}((-\infty,\tau_{0}]\times\R\times\R^{+})},\ \forall 1<p<\infty,
\end{equation}
where $w$ is the volume potential for $g$.

\textcolor{red}{I am confused by the following remark. It looks like our proposition is for $\mathcal{B}_{r}^{+}$.}
With a sight modification of our argument in this section, we can also prove the parabolic version of Proposition \ref{sobolevesti}, i.e. replace $\mathcal{B}_{r}$ by $\mathcal{B}_{r}^{+}$ (Recall the definition of $\mathcal{B}_{r}^{+}$ in Section \ref{continuity}).
\end{remark}
\end{comment}
\begin{proposition}\label{interiorsobolevesti}
Fix $r_0\in(0,1]$. Assume that $w\in L^{\infty}\cap W^{1,1}\cap W^{2,1}_{y}(\B_{r_{0}})$ satisfies the inhomogeneous ultraparabolic equation in $\eqref{nondiver}$ with $g\in L^p(\B_{r_{0}})$ and $f\in W_y^{1,1}\cap L^{p}(\B_{r_{0}})$ for some $3< p<\infty$. Then there exists $r^{\ast}\in(0,\frac{r_{0}}{4}]$ depending only on the diffusion coefficient $a(t,x,y)$ such that
\begin{equation}\label{interiorw2pdesire}
\|\p_{y}w\|_{L^{p}(\B_{r^\ast})}\lesssim _{p,r_0,c,[a]_{C^{\alpha}}} \|w\|_{L^{p}(\B_{r_0})}+\|f\|_{L^{p}(\B_{r_0})}+\|g\|_{L^{p}(\B_{r_0})}.
\end{equation}

Furthermore, when $p>6$, we also have
\begin{equation}\label{interiorw2pdesire2}
    [w]_{C^{\beta}_{\ast}(\mathcal{B}_{r})}\lesssim_{p,r,c,[a]_{C^{\alpha}}}\|w\|_{L^{p}(\B_{2r})}+\|f\|_{L^{p}(\B_{2r})}+\|g\|_{L^{p}(\B_{2r})},\ \forall r\leq r^{\ast}.
\end{equation}
    In the above, $\beta=1-\frac{6}{p}$, and $C^{\beta}_{\ast}$ is the anisotropic H\"older space associated with $\mathcal{L}_{0}$ with the seminorm in a domain $\Sigma\subset(-\infty,0)\times\R\times\R$:
\begin{equation}
\begin{split}
    [w]_{C^{\beta}_{\ast}(\Sigma)}&:=\sup_{\delta\in(0,\frac{1}{2}],\,q_1\in\Sigma,\,q_2\in\Sigma_{q_{1},\delta}}\Big|\frac{w(q_2)-w(q_1)}{\delta^{\beta}}\Big|,
    \end{split}
\end{equation}
where $\Sigma_{q_{1},\delta}$ is defined in the same way as \eqref{newholder}.
\end{proposition}
\begin{comment}
\begin{remark}In particular, if $f\equiv 0$, then \eqref{w2pdesire} holds for all $1<p<\infty$.
\end{remark}
\end{comment}

\begin{proposition}\label{interiorsob1}
Fix $r_0\in(0,1]$. Let $a\in C^{\beta}_{\ast}(\mathcal{B}_{r_{0}})$ for all $0<\beta<1$ and $\partial_ya\in L^q(\mathcal{B}_{r_{0}})$ for any $1<q<\infty$. Assume that $w\in L^{\infty}\cap W^{1,1}\cap W^{2,1}_{y}(\B_{r_{0}})$ satisfies the first equation in $\eqref{nondiver}$ with $g\in L^p(\B_{r_{0}})$ for some $3< p<\infty$, and $f\equiv0$. Then there exists $r^{\ast}\in(0,\frac{r_{0}}{4}]$ such that
\begin{equation}\label{interiorsob2}
\|\p_{y}^2w\|_{L^{p}(\B_{r^\ast})}\lesssim _{p,r_0,c,a} \|w\|_{L^{\infty}(\B_{r_0})}+\|\partial_yw\|_{L^{p}(\B_{r_0})}+\|g\|_{L^{p}(\B_{r_0})}.
\end{equation}
In the above, $r^\ast$ and the implied constant depend on $a$ through the H\"older norm $\|a\|_{C^{\beta}_{\ast}}$ and $\|\partial_ya\|_{L^q(\mathcal{B}_{r_0})}$ for $\beta$ close to $1$ and sufficiently large $q>1$ determined by $p$.
\end{proposition}

%section 6
\section{Higher regularity for inhomogeneous ultraparabolic equations}\label{higher}
In this section, we revisit the following inhomogeneous ultraparabolic equation in divergence form
\begin{equation}\label{diverformsectionSIX}
\begin{cases}
\p_{t}w+y\p_{x}w-\p_{y}(a(t,x,y)\p_{y}w)=f(t,x,y),\ \text{in}\ \mathcal{B}^+_{1}:=(-1,0]\times(-1,1)\times[0,1)\\
\p_{y}w|_{y=0}=0.
\end{cases}
\end{equation}
Recall that in Section \ref{continuity}, we proved that when $f\equiv 0$ and $a(t,x,y)$ is uniformly elliptic, weak solutions to \eqref{diverformsectionSIX} are H\"older continuous in $\mathcal{B}^+_{r_{0}},\ r_{0}<1$. In this section, we explore the higher regularity properties of the solution to \eqref{diverformsectionSIX}, provided that the diffusion coefficient $a(t,x,y)$ is uniformly elliptic, and $a, f$ satisfy suitable regularity assumptions.

Denote $$\mathcal{S}_{y_0}:=(-\infty,0]\times\R\times[0,y_0)$$ for $y_0\in(0,1)$.  After localization in $t$ and $x$, it is convenient to work on the following equation in a strip 
\begin{equation}\label{maineqsection6}
\begin{cases}
\p_{t}w+y\p_{x}w-\p_{y}(a(t,x,y)\partial_yw)=f(t,x,y),\ (t,x,y)\in \mathcal{S}_{1},\\
\p_{y}w|_{y=0}=0.
\end{cases}
\end{equation}

This section is divided into two parts:
\begin{itemize}
\item[(1)] \underline{Zeroth--order hypoelliptic estimate:} In this part, we establish a key regularity result in the spirit of hypoelliptic smoothing for weak solutions to \eqref{maineqsection6} for all $0<y_{0}<1$. More precisely, we shall prove that 
\begin{equation}\label{l2hypo}
\||D_{x}|^{\frac{1}{3}}w\|_{L^{2}(\mathcal{S}_{y_{0}})}+\|\p_{y}w\|_{L^{2}(\mathcal{S}_{y_{0}})}\lesssim_{c,y_{0}}\|w\|_{L^{2}(\mathcal{S}_{1})}+\|f\|_{L^{2}(\mathcal{S}_{1})}.
\end{equation}
At this stage, we \emph{only} assume that $a(t,x,y)$ is uniformly elliptic (i.e., $a$ has positive finite lower and upper bounds). 
\item[(2)] \underline{Higher--order hypoelliptic estimate:} In the second part, we  show (optimal) higher regularity estimates for the solution to \eqref{maineqsection6}, provided that $a(t,x,y)$ and $f(t,x,y)$ satisfy precisely quantified regularity assumptions adapted to the desired regularity for the solution that we obtain. In particular, the weak solution of \eqref{maineqsection6} is shown to be smooth \emph{up to the boundary}, if $a(t,x,y)$ and $f(t,x,y)$ are smooth up to the boundary.
\end{itemize}

\begin{remark}
It is worth noting that for a smooth diffusion coefficient and a smooth source term in the equation \eqref{maineqsection6}, the interior smoothness (away from the boundary) of the solution was already established in the seminal work of H\"ormander \cite{H67} (see also certain extensions in \cite{B02}). The key ingredient in \cite{H67} is the interior regularity estimate \eqref{l2hypo}, which is proven when the coefficient is smooth. 

However, in order to generalize the regularity theory of the linear equation to the quasilinear equation (i.e. the diffusion coefficient may depend on the solution itself), it is important to weaken the regularity assumptions on the coefficient and source term as much as possible, in order to bootstrap the regularity of the solution step by step. Such considerations are crucial in our applications below in the study of smoothness properties of  solutions to the dynamical Prandtl equation.
\end{remark}
In this section, we adopt the notations that $$\mathbb{H}^-:=(-\infty,0)\times\R\times\R^{+}$$ and that for $s>0$, 
\begin{equation}
\|v\|_{\widetilde{H}^{s}(\mathcal{S}_{y_{0}})}^{2}:=\int_{0}^{y_{0}}\|\langle D_x\rangle^{s}v(\cdot,\cdot,y)\|_{L^2(\R^-_t\times \R_x)}^{2}dy.
\end{equation}

\subsection{Zeroth--order hypoelliptic estimate} Our main result in this section consists of the following bounds for equation \eqref{maineqsection6}.

\begin{proposition}\label{zerohypo}
Assume that $c\leq a(t,x,y)\leq c^{-1}$ for some $c\in(0,\frac{1}{2}]$, $f(t,x,y)\in L^{2}(\mathcal{S}_{1})$, and 
\begin{equation*}
w(t,x,y)\in L^{\infty}(\mathcal{S}_{1})\cap W^{1,1}(\mathcal{S}_{1})\cap W^{2,1}_{y}(\mathcal{S}_{1})
\end{equation*}
is a weak solution to \eqref{maineqsection6}. Then for any $y_{0}<1$, we have
\begin{equation}
\|w\|_{\widetilde{H}^{\frac{1}{3}}(\mathcal{S}_{y_{0}})}+\|\p_{y}w\|_{L^{2}(\mathcal{S}_{y_{0}})}\lesssim_{c,y_{0}} \|w\|_{L^{2}(\mathcal{S}_{1})}+\|f\|_{L^{2}(\mathcal{S}_{1})}.
\end{equation}
\end{proposition}

\begin{remark}
As in the interior estimate, the regularity index $s=\frac{1}{3}$ is optimal.
\end{remark}

Proposition \ref{zerohypo} follows from the following two lemmas.

\begin{lemma}
Under the assumptions in Proposition \ref{zerohypo}, for any $y_{0}<1$ we have
\begin{equation}\label{high1}
\|\p_{y}w\|_{L^{2}(\mathcal{S}_{y_{0}})}\lesssim_{c,y_{0}} \|w\|_{L^{2}(\mathcal{S}_{1})}+\|f\|_{L^{2}(\mathcal{S}_{1})}.
\end{equation}
\end{lemma}
\begin{proof}
 \eqref{high1} follows from standard energy estimates, by choosing a smooth cut--off function $\eta(y)$ such that $\eta(y)|_{[0,y_{0}]}\equiv 1$ and $\eta(y)$ is supported in $[0,1)$ and then testing \eqref{maineqsection6} against
$w(t,x,y)\eta^{2}(y)$, using integration by parts.
\end{proof}

\begin{lemma}\label{tangentialimprove1}
Under the assumptions in Proposition \ref{zerohypo}, for any $y_{0}<y_1<1$ we have
\begin{equation}
\|w\|_{\widetilde{H}^{\frac{1}{3}}(\mathcal{S}_{y_{0}})}\lesssim_{c,y_{0},y_1} \|w\|_{L^{2}(\mathcal{S}_{y_1})}+\|\p_{y}w\|_{L^{2}(\mathcal{S}_{y_1})}+\|f\|_{L^{2}(\mathcal{S}_{y_1})}.
\end{equation}
\end{lemma}

\begin{proof}
Fix a smooth cut--off function $\eta(y)$ such that $\eta(y)|_{[0,y_{0}]}\equiv 1$ and $\eta(y)$ is supported in $[0,y_1)$. Let $W=w\eta^{2}$. It is easy to check that $W$ satisfies the following equation in $\mathbb{H}^-$:
\begin{equation}\label{Weq}
\begin{split}
&\p_{t}W+y\p_{x}W-\p_{y}^{2}W=\eta^{2}f+\eta^{2}\p_{y}\left((a-1)\p_{y}w\right)-4\eta\eta'\p_{y}w-\left(2\eta\eta''+2(\eta')^{2}\right)w,\\
&\p_{y}W|_{y=0}=0.
\end{split}
\end{equation}
By the superposition principle, $W=W_{1}+W_{2}$, where on $\mathbb{H}^-$, $W_1$ is the solution to 
\begin{equation}\label{W1eq}
\begin{split}
&\p_{t}W_{1}+y\p_{x}W_{1}-\p_{y}^{2}W_1=\eta^{2}f-4\eta\eta'\p_{y}w-\left(2\eta\eta''+2(\eta')^{2}\right)w-2\eta\eta'(a-1)\p_{y}w,\\
&\p_{y}W_{1}|_{y=0}=0,\\
\end{split}
\end{equation}
and $W_2$ is the solution to
\begin{equation}\label{W2eq}
\begin{split}
&\p_{t}W_{2}+y\p_{x}W_{2}-\p_{y}^{2}W_{2}=\p_{y}\left((a(t,x,y)-1)\eta^{2}\p_{y}w\right)\\
&\p_{y}W_{2}|_{y=0}=0.
\end{split}
\end{equation}
It suffices to prove that for $h\in\{W_1, W_2\}$ (where $S_j$ is the standard Littlewood-Paley low-frequency cutoff, see Appendix \ref{littlewood} and \eqref{besovesti} below),
\begin{equation}\label{high2}
\|(Id-S_{10})h\|_{\widetilde{H}^{\frac{1}{3}}(\mathcal{S}_{y_{0}})}\lesssim_{c,y_{0},y_1} \|w\|_{L^{2}(\mathcal{S}_{y_1})}+\|\p_{y}w\|_{L^{2}(\mathcal{S}_{y_1})}+\|f\|_{L^{2}(\mathcal{S}_{y_1})}.
\end{equation}

\textit{Step 1: Proof of \eqref{high2} for $h=W_{1}$.} By Lemma \ref{h2esti} and \eqref{enhancebound}, we have
\begin{equation}\label{estiw1}
\begin{split}
&\left\|\p_{y}^{2}W_{1}\right\|_{L^{2}(\mathbb{H}^-)}+\left\||D_{x}|^{\frac{2}{3}}W_{1}\right\|_{L^{2}(\mathbb{H}^-)}\\
&\lesssim\left\|\eta^{2}f(t,x,y)-4\eta(y)\eta'(y)\p_{y}w-\left(2\eta\eta''+2(\eta')^{2}\right)w-2\eta(y)\eta'(y)(a(t,x,y)-1)\p_{y}w\right\|_{L^{2}(\mathbb{H}^-)}\\
&\lesssim_{c,y_{0},y_1} \|w\|_{L^{2}(\mathcal{S}_{y_1})}+\|\p_{y}w\|_{L^{2}(\mathcal{S}_{y_1})}+\|f\|_{L^{2}(\mathcal{S}_{y_1})}.
\end{split}
\end{equation}
We then complete the proof of \eqref{high2} in the case $h=W_1$. 

\medskip

\textit{Step 2: Proof of \eqref{high2} for $h=W_{2}$.} The analysis of $W_{2}$ is more difficult due to the extra derivative in the force term. 

Write
\begin{equation}\label{specificF}
F(t,x,y)=\left(a(t,x,y)-1\right)\p_{y}w\cdot\eta^{2}.
\end{equation}
Then $W_{2}$ satisfies the following linear system with the force term in divergence form
\begin{equation}\label{toysystem}
\begin{cases}
    \p_{t}W_{2}+y\p_{x}W_2-\p_{y}^{2}W_{2}=\p_{y}F,\\
    \p_{y}W_{2}|_{y=0}=0.
\end{cases}
\end{equation}
We claim that for any $0<|\delta|<1$, we have
\begin{equation}\label{differentialquo}
\begin{split}
&\left\|W_{2}(T,X+\delta,Y)-W_{2}(T,X,Y)\right\|_{L^{2}(\mathbb{H}^-)}\lesssim |\delta|^{\frac{1}{3}}\|F\|_{L^{2}(\mathbb{H}^-)}.
\end{split}
\end{equation}
For \eqref{differentialquo} we do not need $F$ to be given in the specific form \eqref{specificF}. In fact, we only need suitable decay conditions at infinity so that $W_2$ admits the integral formula \eqref{Wexpress} below. 
\begin{comment}
By the support property of $w$, we also have the following bound
\begin{equation}\label{differentialquo0.1}
\|W_2(T,X,Y)\|_{L^2(\mathbb{H}^-)}\lesssim \|F\|_{L^2(\mathbb{H}^-)}.
\end{equation}
\end{comment}
We assume \eqref{differentialquo} momentarily and complete the proof of \eqref{high2} for $h=W_2$. Recall that $\Delta_j$ denotes the standard Littlewood-Paley projection in $x$, localized to the frequency  $|\xi|\sim 2^j$ (for more details, see Appendix \ref{pre}). As a consequence of \eqref{differentialquo}, we obtain the following estimate for the linear system \eqref{toysystem} for all $j\geq 10$,
\begin{equation}\label{besovesti}
\sup_{j\geq 10}2^{\frac{1}{3}j}\|\Delta_{j}W_{2}\|_{L^{2}((-\infty,0]\times\R\times[0,10])}\lesssim\|F\|_{L^{2}(\mathbb{H}^-)}.
\end{equation}
Since
\begin{equation*}
\Delta_{j}=\Delta_{j}\widetilde{\Delta}_{j},\qquad\widetilde{\Delta}_{j}=\sum_{|l|\leq 2}\Delta_{j+l},
\end{equation*}
applying $\widetilde{\Delta}_{j}$ to \eqref{toysystem} and using the linear estimate \eqref{besovesti} once again, we obtain
\begin{equation}\label{refinebesovesti}
    2^{\frac{1}{3}j}\|\Delta_{j}W_{2}\|_{L^{2}((-\infty,0]\times\R\times[0,10])}\lesssim\sum_{|l|\leq 2}\|\Delta_{j+l}F\|_{L^{2}(\mathbb{H}^{-})}.
\end{equation}
Therefore,
\begin{equation}\label{zerow2}
    \sum_{j\geq 10}2^{\frac{2}{3}j}\|\Delta_{j}W_{2}\|_{L^{2}((-\infty,0]\times\R\times[0,10])}^{2}\lesssim\sum_{j}\|\Delta_{j}F\|_{L^{2}(\mathbb{H}^-)}^{2}\sim \|F\|_{L^{2}(\mathbb{H}^-)}^{2}.
\end{equation}
\eqref{high2} then follows from \eqref{estiw1} and \eqref{zerow2}. 

To complete the proof, it remains to prove the  bounds \eqref{differentialquo}. Without loss of generality we assume $\delta>0$ for convenience.

In view of \eqref{W2eq} and $F|_{y=0}=0$, for $(T,X,Y)\in\mathbb{H}^-$ we can write
\begin{equation}\label{Wexpress}
W_{2}(T,X,Y)=-\int_{\mathbb{H}^-}(\p_{y}\Gamma)(T,X,Y;t,x,y)F(t,x,y)dtdxdy,
\end{equation}
where $\Gamma$ is the fundamental solution of the operator $\mathcal{L}_{0}$ on the upper half space with the zero Neumann boundary condition from Section \ref{fundamental}.

We have
\begin{equation}\label{decomw2}
\begin{split}
&|W_{2}(T,X+\delta,Y)-W_{2}(T,X,Y)|\\
&\leq \int_{\mathbb{H}^-}|(\p_{y}\Gamma)(T,X+\delta,Y;t,x,y)-(\p_{y}\Gamma)(T,X,Y;t,x,y)|\cdot |F(t,x,y)|\,dtdxdy\\
&\leq W_{21}+W_{22}+W_{23}+W_{24}.
\end{split}
\end{equation}
In the above, for $j\in\{1,2,3,4\}$,
\begin{equation}\label{decomw2''}
W_{2j}:=\int_{D_{2j}}|(\p_{y}\Gamma)(T,X+\delta,Y;t,x,y)-(\p_{y}\Gamma)(T,X,Y;t,x,y)|\cdot |F(t,x,y)|\,dtdxdy,
\end{equation}
\begin{equation*}
D_{21}:= \big\{(t,x,y)\in\mathbb{H}^-:\,\|(T-t,X-x,Y-y)\|\geq 5\delta^{\frac{1}{3}}, \|(T-t,X-x-(T-t)Y,Y-y)\|\geq 5\delta^{\frac{1}{3}}\big\},
\end{equation*}
\begin{equation*}
D_{22}:= \big\{(t,x,y)\in\mathbb{H}^-:\,\|(T-t,X-x,Y-y)\|\leq 20\delta^{\frac{1}{3}}, \|(T-t,X-x-(T-t)Y,Y-y)\|\leq 20\delta^{\frac{1}{3}}\big\},
\end{equation*}
\begin{equation*}
D_{23}:= \big\{(t,x,y)\in\mathbb{H}^-:\,\|(T-t,X-x,Y-y)\|\ge 20\delta^{\frac{1}{3}}, \|(T-t,X-x-(T-t)Y,Y-y)\|\leq 5\delta^{\frac{1}{3}}\big\},
\end{equation*}
\begin{equation*}
D_{24}:= \big\{(t,x,y)\in\mathbb{H}^-:\,\|(T-t,X-x,Y-y)\|\leq 5\delta^{\frac{1}{3}}, \|(T-t,X-x-(T-t)Y,Y-y)\|\ge 20\delta^{\frac{1}{3}}\big\}.
\end{equation*}

We need to show that for $j\in\{1,2,3,4\}$,
\begin{equation}\label{high3}
    \left\|W_{2j}\right\|_{L^{2}(\mathbb{H}^-)}\lesssim |\delta|^{\frac{1}{3}}\|F\|_{L^{2}(\mathbb{H}^-)}.
\end{equation}
We consider the terms $W_{2j}$ case by case. 

\medskip

\textit{Case (I): the bound \eqref{high3} for $W_{21}$.} By the mean-value theorem and Proposition \ref{highptwise}, for suitable $\theta\in(0,1)$, we have
\begin{equation}
\begin{split}
&|(\p_{y}\Gamma)(T,X+\delta,Y;t,x,y)-(\p_{y}\Gamma)(T,X,Y;t,x,y)|\lesssim \delta|(\p_{x}\p_{y}\Gamma)(T,X+\theta\delta,Y;t,x,y)|\\
&\lesssim  \delta\left(\|(T-t,X+\theta\delta-x,Y-y)\|^{-8}+\|(T-t,X+\theta\delta-x-(T-t)Y,Y-y)\|^{-8}\right)\\
&\lesssim \delta\left(\|(T-t,X-x,Y-y)\|^{-8}+\|(T-t,X-x-(T-t)Y,Y-y)\|^{-8}\right),
\end{split}
\end{equation}
provided that
\begin{equation*}
\|(T-t,X-x,Y-y)\|\geq 5\delta^{\frac{1}{3}},\quad \|(T-t,X-x-(T-t)Y,Y-y)\|\geq 5\delta^{\frac{1}{3}}.
\end{equation*}
Therefore, by Young's inequality,
\begin{equation}
\|W_{21}\|_{L^{2}_{(T,X,Y)}(\mathbb{H}^-)}\lesssim\delta\left(\int_{\|z\|\geq 5\delta^{\frac{1}{3}}}\|z\|^{-8}dtdxdy\right)\times\|F\|_{L^{2}(\mathbb{H}^-)}\lesssim\delta^{\frac{1}{3}}\|F\|_{L^{2}(\mathbb{H}^-)}.
\end{equation}

\medskip

\textit{Case (II): the bound \eqref{high3} for $W_{22}$.} By \eqref{greendecom}, \eqref{wholespace}, and Proposition \ref{highptwise}, we have
\begin{equation}
\begin{split}
&|(\p_{y}\Gamma)(T,X+\delta,Y;t,x,y)-(\p_{y}\Gamma)(T,X,Y;t,x,y)|\\
&\lesssim \|(T-t,X+\delta-x,Y-y)\|^{-5}+\|(T-t,X-x,Y-y)\|^{-5}\\
&\quad+\|(T-t,X+\delta-x-(T-t)Y,Y-y)\|^{-5}+\|(T-t,X-x-(T-t)Y,Y-y)\|^{-5}.
\end{split}
\end{equation}
Hence, applying Young's inequality again, we have
\begin{equation}
\|W_{22}\|_{L^{2}_{(T,X,Y)}(\mathbb{H}^-)}\lesssim\Big(\int_{\|z\|\leq 25\delta^{\frac{1}{3}}}\|z\|^{-5}dtdxdy\Big)\times\|F\|_{L^{2}(\mathbb{H}^-)}\lesssim\delta^{\frac{1}{3}}\|F\|_{L^{2}(\mathbb{H}^-)}.
\end{equation}

\medskip

\textit{Case (III): the bound \eqref{high3} for $W_{23}$.} In $D_{23}$, we use the decomposition \eqref{greendecom}. Set
\begin{equation}\label{high4}
  p:= \Big(\frac{T-t}{y^{2}},\frac{X+\delta-x}{y^{3}},\frac{Y}{y}\Big), \qquad q:= \Big(\frac{T-t}{y^{2}},\frac{X-x}{y^{3}},\frac{Y}{y}\Big).
\end{equation}

By the Green decomposition \eqref{greendecom}, we get that
\begin{equation}
\begin{split}
W_{23}&\lesssim\int_{D_{23}}|(\p_{y}\widetilde{\Gamma})(T,X+\delta,Y;t,x,y)-(\p_{y}\widetilde{\Gamma})(T,X,Y;t,x,y)|\cdot |F(t,x,y)|\,dtdxdy\\
&+\int_{D_{23}}y^{-5}\left|V(p)-V(q)\right|\cdot |F(t,x,y)|\,dtdxdy\\
&+\int_{D_{23}}\frac{|T-t|}{y^{7}}\left|(\p_{1}V)(p)-(\p_{1}V)(q)\right|\cdot |F(t,x,y)|\,dtdxdy\\
&+\int_{D_{23}}\frac{|X-x|}{y^{8}}\left|(\p_{2}V)(p)-(\p_{2}V)(q)\right|\cdot |F(t,x,y)|\,dtdxdy\\
&+\int_{D_{23}}\frac{Y}{y^{6}}\left|(\p_{3}V)(p)-(\p_{3}V)(q)\right|\cdot |F(t,x,y)|\,dtdxdy+3\delta\int_{D_{23}}y^{-8}
|(\partial_2V)(p)|\,|F|\, dtdxdy\\
&=:W_{23,1}+W_{23,2}+W_{23,3}+W_{23,4}+W_{23,5}+W_{23,6}.
\end{split}
\end{equation}
\begin{comment}
where
\begin{equation*}
\mathcal{A}_{\delta}=\left\{(t,x,y):\|(T-t,X-x,Y-y)\|\geq 20\delta^{\frac{1}{3}},\\ \|(T-t,X-x-(T-t)Y,Y-y)\|\leq 5\delta^{\frac{1}{3}}\right\}.
\end{equation*}
\end{comment}
By \eqref{wholespace}, we have
\begin{equation*}
\begin{split}
&|(\p_{y}\widetilde{\Gamma})(T,X+\delta,Y;t,x,y)-(\p_{y}\widetilde{\Gamma})(T,X,Y;t,x,y)|\\
&\lesssim\|(T-t,X-x-(T-t)Y,Y-y)\|^{-5}+\|(T-t,X+\delta-x-(T-t)Y,Y-y)\|^{-5}.
\end{split}
\end{equation*}
Therefore, $W_{23,1}$ can be estimated via Young's inequality
\begin{equation}
\|W_{23,1}\|_{L^{2}_{(T,X,Y)}(\mathbb{H}^-)}\lesssim\left(\int_{\|z\|\leq 10\delta^{\frac{1}{3}}}\|z\|^{-5}dtdxdy\right)\times\|F\|_{L^{2}(\mathbb{H}^-)}\lesssim\delta^{\frac{1}{3}}\|F\|_{L^{2}(\mathbb{H}^-)}.
\end{equation}
To estimate $W_{23,2}$, using the mean--value theorem we have for some $\theta\in(0,1)$,
\begin{equation*}
\begin{split}
&\left|V\Big(\frac{T-t}{y^{2}},\frac{X+\delta-x}{y^{3}},\frac{Y}{y}\Big)-V\Big(\frac{T-t}{y^{2}},\frac{X-x}{y^{3}},\frac{Y}{y}\Big)\right|\\
&= y^{-3}\delta\left|(\p_{2}V)\Big(\frac{T-t}{y^{2}},\frac{X+\theta\delta-x}{y^{3}},\frac{Y}{y}\Big)\right|.
\end{split}
\end{equation*}
Notice that in $D_{23}$, we have
\begin{equation*}
y\gtrsim \|(T-t,X-x,Y-y)\|.
\end{equation*}
Hence, by Young's inequality and the fact that $|\p_{2}V|\lesssim 1$, we have
\begin{equation}
\|W_{23,2}\|_{L^{2}_{(T,X,Y)}(\mathbb{H}^-)}\lesssim\delta\left(\int_{\|z\|\geq 20\delta^{\frac{1}{3}}}\|z\|^{-8}dtdxdy\right)\times\|F\|_{L^{2}(\mathbb{H}^-)}\lesssim\delta^{\frac{1}{3}}\|F\|_{L^{2}(\mathbb{H}^-)}.
\end{equation}
The remaining terms $W_{23,3}-W_{23,6}$ can be estimated as follows. 
For $i=1,2,3$, the mean-value theorem and the uniform bounds on the
derivatives of $V$ give
\[
 |(\partial_iV)(p)-(\partial_iV)(q)|\lesssim \delta y^{-3}.
\]
Writing $s=(T-t,X-x,Y-y)$, on $D_{23}$ we have
$y\gtrsim\|s\|$ and $Y\lesssim y$. Since
$|T-t|\leq\|s\|^2$ and $|X-x|\leq\|s\|^3$, the kernels in
$W_{23,3}$, $W_{23,4}$, $W_{23,5}$ and $W_{23,6}$ are bounded by
$C\delta\|s\|^{-8}$. Therefore, by Young's inequality and the
homogeneous dimension six,
\[
 \sum_{k=3}^6\|W_{23,k}\|_{L^2(\mathbb H^-)}
 \lesssim
 \delta\!\int_{\|s\|\geq20\delta^{1/3}}\!\|s\|^{-8}\,ds\,
 \|F\|_{L^2(\mathbb H^-)}
 \lesssim\delta^{1/3}\|F\|_{L^2(\mathbb H^-)}.
\]

Thus we obtain that
\begin{equation}
\|W_{23}\|_{L^{2}_{(T,X,Y)}(\mathbb{H}^-)}\lesssim\delta^{\frac{1}{3}}\|F\|_{L^{2}(\mathbb{H}^-)}.
\end{equation}
\textit{Case (IV): the bound \eqref{high3} for $W_{24}$.} The treatment of $W_{24}$ is similar to that of $W_{23}$, and we will be somewhat brief. As before, we decompose
\begin{equation}\label{w24}
\begin{split}
W_{24}&\lesssim\int_{D_{24}}|(\p_{y}\widetilde{\Gamma})(T,X+\delta,Y;t,x,y)-(\p_{y}\widetilde{\Gamma})(T,X,Y;t,x,y)|\cdot |F(t,x,y)|\,dtdxdy\\
&+\int_{D_{24}}y^{-5}\left|V(p)-V(q)\right|\cdot |F(t,x,y)|\,dtdxdy\\
&+\int_{D_{24}}\frac{|T-t|}{y^{7}}\left|(\p_{1}V)(p)-(\p_{1}V)(q)\right|\cdot |F(t,x,y)|\,dtdxdy\\
&+\int_{D_{24}}\frac{|X-x|}{y^{8}}\left|(\p_{2}V)(p)-(\p_{2}V)(q)\right|\cdot |F(t,x,y)|\,dtdxdy\\
&+\int_{D_{24}}\frac{Y}{y^{6}}\left|(\p_{3}V)(p)-(\p_{3}V)(q)\right|\cdot |F(t,x,y)|\,dtdxdy+3\delta\int_{D_{24}}y^{-8}
|(\partial_2V)(p)|\,|F|dtdxdy\\
&=:W_{24,1}+W_{24,2}+W_{24,3}+W_{24,4}+W_{24,5}+W_{24,6}.
\end{split}
\end{equation}
\begin{comment}
where
\begin{equation*}
\mathcal{B}_{\delta}=\left\{(t,x,y):\|(T-t,X-x,Y-y)\|\leq 5\delta^{\frac{1}{3}},\\ \|(T-t,X-x-(T-t)Y,Y-y)\|\geq 20\delta^{\frac{1}{3}}\right\}.
\end{equation*}
\end{comment}
In view of \eqref{wholespace}, as well as Young's inequality, denoting $r:=(T-t,X-x-(T-t)Y,Y-y)$ for brevity of notation, we have
\begin{equation}
\begin{split}
&\|W_{24,1}\|_{L^{2}(\mathbb{H}^-)}\lesssim \delta\Big\|\int_{\|r\|\geq10\delta^{\frac{1}{3}}}\|r\|^{-8}\cdot |F(t,x,y)|\,dtdxdy\Big\|_{L^{2}(\mathbb{H}^-)}\lesssim \delta^{\frac{1}{3}}\|F\|_{L^{2}(\mathbb{H}^-)}.
\end{split}
\end{equation}
Notice that in $D_{24}$, we have
\begin{equation*}
y\gtrsim \|(T-t,X-x,Y-y)\|.
\end{equation*}
Moreover, on \(D_{24}\), the triangle inequality gives
\(\bigl|(T-t)Y\bigr|^{1/3}\ge 15\delta^{1/3}\); since
\(\lvert T-t\rvert^{1/2},\lvert Y-y\rvert\le 5\delta^{1/3}\),
it follows that \(y\gtrsim\delta^{1/3}\).
Hence denoting $s:=(T-t,X-x,Y-y)$ for brevity of notation, we have
\begin{equation}
\begin{split}
&\|W_{24,2}\|_{L^{2}(\mathbb{H}^-)}\lesssim \Big\|\int_{\|s\|\leq5\delta^{\frac{1}{3}}}\|s\|^{-5}\cdot |F(t,x,y)|\,dtdxdy\Big\|_{L^{2}(\mathbb{H}^-)}\lesssim \delta^{\frac{1}{3}}\|F\|_{L^{2}(\mathbb{H}^-)}.
\end{split}
\end{equation}
The remaining terms in \eqref{w24} can be handled in the same way as $W_{24,2}$. The proof of \eqref{differentialquo} is now complete.
\end{proof}
\begin{comment}
We also record the following bounds which seem to be of independent interest.
\begin{corollary}
     For $(T,X,Y)\in\mathbb{H}^-$, let
\begin{equation}\label{WexpressC1}
v(T,X,Y)=-\int_{\mathbb{H}^-}(\p_{y}\Gamma)(T,X,Y;t,x,y)F(t,x,y)dtdxdy.
\end{equation}
Then  
for any $0<|\delta|<1$, we have
\begin{equation}\label{boundC2}
\begin{split}
&\left\|v(T,X+\delta,Y)-v(T,X,Y)\right\|_{L^{2}(\mathbb{H}^-)}\lesssim |\delta|^{\frac{1}{3}}\|F\|_{L^{2}(\mathbb{H}^-)},\\
&\left\|v(T+\delta,X+Y\delta,Y)-v(T,X,Y)\right\|_{L^{2}(\mathbb{H}^-)}\lesssim |\delta|^{\frac{1}{2}}\|F\|_{L^{2}(\mathbb{H}^-)},
\end{split}
\end{equation}

\begin{proof}
The first inequality in \eqref{boundC2} is the same as the bound on $W_2$. The second inequality follows from a similar proof. We omit the details.\end{proof}

\end{corollary}
\end{comment}
\subsection{Higher--order hypoelliptic estimates}
In this section we establish higher--order hypoelliptic estimates for equation \eqref{maineqsection6}, provided that $a(t,x,y)$ and $f(t,x,y)$ satisfy suitable quantified regularity assumptions. 

\begin{proposition}\label{highhypo} 
Fix $k\in\Z\cap[1,\infty)$, $\alpha\in(0,1)$, $c\in(0,1/2]$ and a positive $s\leq\min\{1/3,\alpha\}$. Assume that $c\leq a(t,x,y)\leq c^{-1}$, $f(t,x,y)\in \widetilde{H}^{ks}(\mathcal{S}_{1})$ and $[a]_{C^{\alpha}(\mathcal{S}_{1})}\lesssim 1$. Let 
\begin{equation*}
w(t,x,y)\in L^{\infty}(\mathcal{S}_{1})\cap W^{1,1}(\mathcal{S}_{1})\cap W^{2,1}_{y}(\mathcal{S}_{1})
\end{equation*}
be a weak solution to \eqref{maineqsection6}. Assume in addition that 
\begin{equation}
M_k:=\|\p_{y}a\|_{L^{\infty}(\mathcal{S}_{1})}+\|(Id-\Delta_{-1})a\|_{\widetilde{H}^{ks}(\mathcal{S}_{1})}+[a]_{C^{\alpha}(\mathcal{S}_{1})}<\infty.
\end{equation}
Then for $0<y_{0}<1$, 
we have
\begin{equation}\label{hohe1}
\begin{split}
&\|w\|_{\widetilde{H}^{(k+1)s}(\mathcal{S}_{y_{0}})}+\|\p_{y}w\|_{\widetilde{H}^{ks}(\mathcal{S}_{y_{0}})}\lesssim_{c,k,s,y_{0},\alpha, M_k} \|w\|_{L^2(\mathcal{S}_1)}+\|\partial_yw\|_{L^\infty(\mathcal{S}_1)}+\|f\|_{\widetilde{H}^{ks}(\mathcal{S}_{1})}.
\end{split}
\end{equation}
\end{proposition}

By linearity, we can assume that $\|w\|_{L^2(\mathcal{S}_1)}+\|\partial_yw\|_{L^\infty(\mathcal{S}_1)}+\|f\|_{\widetilde{H}^{ks}(\mathcal{S}_{1})}\leq 1$.

The conclusion of Proposition \ref{highhypo} follows from the following two lemmas and a standard induction argument. Notice that \eqref{hohe1} when $k=0$ follows from Proposition \ref{zerohypo}.

\begin{lemma}\label{nonlocalenergyesti}
Let $k$ be an integer with $k\ge1$. For $0<y_0<y_1<1$, denote
\begin{equation}\label{M_k}
 N_k:=\|w\|_{\widetilde{H}^{ks}(\mathcal{S}_{y_1})}+\|\p_{y}w\|_{\widetilde{H}^{(k-1)s}(\mathcal{S}_{y_1})}+M_k.
\end{equation}
Under the assumptions of Proposition \ref{highhypo}, we have
\begin{equation}
\begin{split}
\|\p_{y}w\|_{\widetilde{H}^{ks}(\mathcal{S}_{y_{0}})}&\lesssim C_{c,k,s,y_{0},y_1,\alpha,N_k}.
\end{split}
\end{equation}
\end{lemma}

\begin{lemma}\label{nonlocaltangential}
Let $k$ be an integer with $k\ge1$. For $0<y_0<y_1<1$, denote
\begin{equation}\label{M_k00}
 P_k:=\|w\|_{\widetilde{H}^{ks}(\mathcal{S}_{y_1})}+\|\p_{y}w\|_{\widetilde{H}^{ks}(\mathcal{S}_{y_1})}+M_k.
\end{equation}
Under the assumption of Proposition \ref{highhypo}, we have
\begin{equation}
\begin{split}
\|w\|_{\widetilde{H}^{ks+\frac{1}{3}}(\mathcal{S}_{y_{0}})}&\lesssim C_{c,k,s,y_{0},y_1,\alpha, P_k}.
\end{split}
\end{equation}
\end{lemma}

It remains to give the proof of the above key lemmas.
\begin{proof}[Proof of Lemma \ref{nonlocalenergyesti}] Let $\Delta_{j}$ be the standard Littlewood--Paley projection operator in $x$. By \eqref{maineqsection6}, we have for $(t,x,y)\in\mathcal{S}_1$,
\begin{equation}\label{Vjeq}
\begin{cases}
\p_{t}\Delta_{j}w+y\p_{x}\Delta_{j}w-\p_{y}\Delta_{j}\left(a(t,x,y)\p_{y}w\right)=\Delta_{j}f,\\
\p_{y}\Delta_{j}w|_{y=0}=0.
\end{cases}
\end{equation}

Fix $K\gg1$ to be determined below and a standard smooth cutoff function $\eta(y)\in C_0^\infty([0,y_1))$ with $\eta\equiv1$ on $[0,y_0]$. Testing \eqref{Vjeq} against
$2^{2jks}\Delta_{j}w\,\eta^{2}(y)$ with $j\geq K,$
and integrating by parts, we get
\begin{equation}\label{energyide}
\begin{split}
I_{j}:&=\int_{\mathcal{S}_{1}}2^{2jks}\Delta_{j}(a\p_{y}w)\p_{y}\Delta_{j}w\,\eta^{2}(y)\,dtdxdy\\
&\leq\int_{\mathcal{S}_{1}}2^{2jks}\Delta_{j}f\Delta_{j}w\,\eta^{2}(y)\,dtdxdy -2\int_{\mathcal{S}_{1}}2^{2jks}\Delta_{j}(a\p_{y}w)\Delta_{j}w\,\eta(y)\eta'(y)\,dtdxdy\\
&=:J_{j,1}-2J_{j,2}.
\end{split}
\end{equation}
By Bony's paraproduct decomposition \eqref{parapro}, we have
\begin{equation}\label{paraproduct}
\begin{split}
\Delta_{j}(a\p_{y}w)&=\sum_{|l|\leq 5}\Delta_{j}\left(S_{j+l-1}a\,\Delta_{j+l}\p_{y}w\right)+\sum_{|l|\leq 5}\Delta_{j}\left(S_{j+l-1}\p_{y}w\,\Delta_{j+l}a\right)\\
&\quad +\sum_{l> j-5,\,|q-l|\leq 1}\Delta_{j}\left(\Delta_{l}a\,\Delta_{q}\p_{y}w\right)
\end{split}
\end{equation}
Hence, the commutator term $[\Delta_{j},a]\p_{y}w$ can be expressed as 
\begin{equation}\label{commutator}
\begin{split}
\Delta_{j}(a\p_{y}w)-a\Delta_{j}\p_{y}w&=(S_{j-6}-\mathrm{Id})a\,\Delta_{j}\p_{y}w+\sum_{|l|\leq 5}[\Delta_{j}, S_{j-6}a]\Delta_{j+l}\p_{y}w\\
&\quad+\sum_{|l|\leq 5}\Delta_{j}\left(S_{j+l-1}\p_{y}w\,\Delta_{j+l}a\right)+\sum_{l> j-5,|q-l|\leq 1}\Delta_{j}\left(\Delta_{l}a\,\Delta_{q}\p_{y}w\right)\\
&\quad+\sum_{-5<l\leq5}\sum_{q=-6}^{l-2}\Delta_{j}\left(\Delta_{j+q}a\,\Delta_{j+l}\p_{y}w\right).
\end{split}
\end{equation}
Therefore, by \eqref{commutator},
\begin{equation}\label{ijformula}
\begin{split}
I_{j}=I_{j,1}+I_{j,2}+I_{j,3}+I_{j,4}+I_{j,5}+I_{j,6},
\end{split}
\end{equation}
where the terms $I_{j,l},\,l\in\Z\cap[1,6]$ are defined as
\begin{equation*}
\begin{split}
&I_{j,1}:=\int_{\mathcal{S}_{1}}2^{2jks}a(t,x,y)|\p_{y}\Delta_{j}w|^{2}\eta^{2}(y)\,dtdxdy,\\
&I_{j,2}:=\int_{\mathcal{S}_{1}}2^{2jks}\big[(S_{j-6}-\mathrm{Id})a\,\Delta_{j}\p_{y}w\big]\p_{y}\Delta_{j}w\,\eta^{2}(y)\,dtdxdy,\\
&I_{j,3}:=\int_{\mathcal{S}_{1}}2^{2jks}\Big(\sum_{|l|\leq 5}[\Delta_{j}, S_{j-6}a]\Delta_{j+l}\p_{y}w\Big)\p_{y}\Delta_{j}w\,\eta^{2}(y)\,dtdxdy,\\
&I_{j,4}:=\int_{\mathcal{S}_{1}}2^{2jks}\Big(\sum_{|l|\leq 5}\Delta_{j}(S_{j+l-1}\p_{y}w\Delta_{j+l}a)\Big)\p_{y}\Delta_{j}w\,\eta^{2}(y)\,dtdxdy,\\
&I_{j,5}:=\int_{\mathcal{S}_{1}}2^{2jks}\Big(\sum_{-5<l\leq5}\sum_{q=-6}^{l-2}\Delta_{j}\left(\Delta_{j+q}a\Delta_{j+l}\p_{y}w\right)\Big)\p_{y}\Delta_{j}w\,\eta^{2}(y)\,dtdxdy,\\
&I_{j,6}:=\int_{\mathcal{S}_{1}}2^{2jks}\Big(\sum_{l> j-5,|q-l|\leq 1}\Delta_{j}\left(\Delta_{l}a\,\Delta_{q}\p_{y}w\right)\Big)\p_{y}\Delta_{j}w\,\eta^{2}(y)\,dtdxdy.
\end{split}
\end{equation*}

We bound the terms $I_{j,l}, \,l\in \Z\cap[1,6]$ as follows.

\textit{(I): Estimate of $I_{j,1}$.} By the property of $a(t,x,y)$, we have
\begin{equation}
I_{j,1}\geq c\int_{\mathcal{S}_{1}}2^{2jks}|\p_{y}\Delta_{j}w|^{2}\eta^{2}dtdxdy.
\end{equation}

\textit{(II): Estimate of $I_{j,2}$.} Since $j\geq K$, by Bernstein's inequality,
\begin{equation}
\begin{split}
|I_{j,2}|&\leq \Big(\sum_{q\geq K-6}\|\Delta_{q}a\|_{L^{\infty}}\Big)\int_{\mathcal{S}_{1}}2^{2jks}|\p_{y}\Delta_{j}w|^{2}\eta^{2}dtdxdy\\
&\lesssim 2^{-\alpha K}\|a\|_{C^{\alpha}(\mathcal{S}_{1})}\int_{\mathcal{S}_{1}}2^{2jks}|\p_{y}\Delta_{j}w|^{2}\eta^{2}dtdxdy\leq \frac{1}{10}I_{j,1},
\end{split}
\end{equation}
if  $K\gg1$ is chosen sufficiently large, depending on $[a]_{C^{\alpha}(\mathcal{S}_{1})}$ and the ellipticity constant $c$.

\textit{(III): Estimate of $I_{j,3}$.} Recall the standard commutator estimate (see e.g. Lemma 2.97 in \cite{BCD11}),
\begin{equation}\label{standardcomm}
\|[\Delta_{j},b]f\|_{L^{2}(\R)}\lesssim 2^{-j}\|\nabla_x b\|_{L^{\infty}(\R)}\|f\|_{L^{2}(\R)}.
\end{equation}
Hence, by \eqref{standardcomm} and Bernstein's inequality, for $y\in(0,1)$ and $D:=(-\infty,0]\times\R$,
\begin{equation}
\begin{split}
&\left\|[\Delta_{j},S_{j-6}a]\Delta_{j+l}\p_{y}w\right\|_{L^{2}_{t,x}(D)}\lesssim 2^{-j}\|\nabla_x S_{j-6}a\|_{L^{\infty}}\left\|\Delta_{j+l}\p_{y}w\right\|_{L^{2}_{t,x}(D)}\\
&\leq 2^{-j}\Big(\sum_{q<j-6}\|\nabla_x\Delta_{q}a\|_{L^{\infty}}\Big)\left\|\Delta_{j+l}\p_{y}w\right\|_{L^{2}_{t,x}(D)} %\lesssim 2^{-j}\Big(\sum_{q<j-6}2^{q(1-\alpha)}\|a\|_{C^{\alpha}(D)}\Big)\left\|\Delta_{j+l}\p_{y}w\right\|_{L^{2}_{t,x}(D)}\\
\lesssim 2^{-j\alpha}\|a\|_{C^{\alpha}(D)}\left\|\Delta_{j+l}\p_{y}w\right\|_{L^{2}_{t,x}(D)}.
\end{split}
\end{equation}
By the Cauchy--Schwarz inequality,
\begin{equation}
|I_{j,3}|\leq \frac{1}{10}I_{j,1}+C2^{-2j(\alpha-s)}\|a\|_{C^{\alpha}}^{2}\Big(\sum_{|l|\leq 5}\int_{\mathcal{S}_{1}}2^{2j(k-1)s}|\p_{y}\Delta_{j+l}w|^{2}\eta^{2}dtdxdy\Big).
\end{equation}

\textit{(IV): Estimate of $I_{j,4}$.} By the Cauchy--Schwarz inequality, we have
\begin{equation}
\begin{split}
|I_{j,4}|&\lesssim \sum_{|l|\leq5}\left(2^{2jks}\|\p_{y}\Delta_{j}w\cdot \eta(y)\|_{L^{2}(\mathcal{S}_{1})}\cdot\|\p_{y}w\|_{L^{\infty}(\mathcal{S}_{1})}\cdot\|\Delta_{j+l}a\|_{L^{2}(\mathcal{S}_{1})}\right)\\
&\leq \frac{1}{10}I_{j,1}+C\|\p_{y}w\|_{L^{\infty}(\mathcal{S}_{1})}^{2}\sum_{|l|\leq5}2^{2jks}\|\Delta_{j+l}a\|_{L^{2}(\mathcal{S}_{1})}^{2}.
\end{split}
\end{equation}

\textit{(V): Estimate of $I_{j,5}$.} By Bernstein's inequality,
\begin{equation}
\begin{split}
|I_{j,5}|&\lesssim 2^{-j\alpha}[a]_{C^{\alpha}}\Big(\int_{\mathcal{S}_{1}}2^{2jks}|\p_{y}\Delta_{j}w|^{2}\eta^{2}\Big)^{\frac{1}{2}}\Big(\sum_{-5<l\leq 5}\int_{\mathcal{S}_{1}}2^{2jks}|\p_{y}\Delta_{j+l}w|^{2}\eta^{2}dtdxdy\Big)^{\frac{1}{2}}\\
&\leq \frac{1}{10}I_{j,1}+C[a]_{C^{\alpha}}^{2}2^{-2j(\alpha-s)}\Big(\sum_{-5<l\leq 5}\int_{\mathcal{S}_{1}}2^{2j(k-1)s}|\p_{y}\Delta_{j+l}w|^{2}\eta^{2}dtdxdy\Big).
\end{split}
\end{equation}

\textit{(VI): Estimate of $I_{j,6}$.} For $I_{j,6}$, we rewrite this term as
\begin{equation}
I_{j,6}=\sum_{l> j-5,|q-l|\leq 1}\int_{\mathcal{S}_{1}}2^{(j-l)ks}\left(\Delta_{j}\left(2^{lks}\Delta_{l}a\cdot\Delta_{q}\p_{y}w\right)\right)2^{jks}\p_{y}\Delta_{j}w\eta^{2}(y)dtdxdy.
\end{equation}
By the Cauchy--Schwarz inequality,
\begin{equation}
\begin{split}
|I_{j,6}|&\leq\frac{1}{10}I_{j,1}+C\|\p_{y}w\|_{L^{\infty}}^{2}\Big(\sum_{l>j-5}2^{(j-l)ks}\cdot 2^{lks}\|\Delta_{l}a\|_{L^{2}}\Big)^{2}\\
&=\frac{1}{10}I_{j,1}+C\|\p_{y}w\|_{L^{\infty}}^{2}\|(Id-\Delta_{-1})a\|_{\widetilde{H}^{ks}(\mathcal{S}_{1})}^{2}\cdot c_{j},
\end{split}
\end{equation}
where $c_{j}$ is a non--negative sequence with $\sum_{j\geq 10}c_{j}\lesssim 1$. In the above we have applied Young's inequality:
\begin{equation*}
\sum_{j\geq 10}\Big(\sum_{l>j-5}2^{(j-l)ks}\cdot 2^{lks}\|\Delta_{l}a\|_{L^{2}}\Big)^{2}\lesssim \sum_{l\geq 5}2^{2lks}\|\Delta_{l}a\|_{L^{2}}^{2}.
\end{equation*}

We now turn to estimate $J_{j,1}$ and $J_{j,2}$ in \eqref{energyide}. We can bound $J_{j,1}$ straightforwardly as 
\begin{equation}
|J_{j,1}|\leq \left(\int_{\mathcal{S}_{1}}2^{2jks}|\Delta_{j}f|^{2}\eta^{2}(y)dtdxdy\right)^{\frac{1}{2}}\left(\int_{\mathcal{S}_{1}}2^{2jks}|\Delta_{j}w|^{2}\eta^{2}(y)dtdxdy\right)^{\frac{1}{2}}.
\end{equation}

It remains to estimate $J_{j,2}$. We write
\begin{equation}
\begin{split}
J_{j,2}=&\int_{\mathcal{S}_{1}}2^{2jks}\Delta_{j}(a\p_{y}w)\Delta_{j}w\,\eta(y)\eta'(y)dtdxdy\\
=&\int_{\mathcal{S}_{1}}2^{2jks}a(t,x,y)\p_{y}\Delta_{j}w\Delta_{j}w\,\eta(y)\eta'(y)dtdxdy\\
&+\int_{\mathcal{S}_{1}}2^{2jks}[\Delta_{j},a(t,x,y)]\p_{y}w\Delta_{j}w\,\eta(y)\eta'(y)dtdxdy,
\end{split}
\end{equation}
which, by\ \eqref{commutator}, can be further decomposed as
\begin{equation}
   J_{j,2} =:J_{j,2,1}+\cdots+J_{j,2,6},
\end{equation}
where the terms $J_{j,2,l},l\in\Z\cap[1,3]$ are given by
\begin{equation}
\begin{split}
J_{j,2,1}=&\int_{\mathcal{S}_{1}}2^{2jks}a(t,x,y)\p_{y}\Delta_{j}w\Delta_{j}w\,\eta(y)\eta'(y)\,dtdxdy,\\
J_{j,2,2}=&\int_{\mathcal{S}_{1}}2^{2jks}(S_{j-6}-Id)a\,\Delta_{j}\p_{y}w\cdot\Delta_{j}w\,\eta(y)\eta'(y)\,dtdxdy,\\
J_{j,2,3}=&\sum_{|l|\leq 5}\int_{\mathcal{S}_{1}}2^{2jks}[\Delta_{j},S_{j-6}a]\Delta_{j+l}\p_{y}w\cdot\Delta_{j}w\,\eta(y)\eta'(y)\,dtdxdy,
\end{split}
\end{equation}
and the terms $J_{j,2,l},l\in\Z\cap[4,6]$ are given by
\begin{equation}
\begin{split}
J_{j,2,4}=&\sum_{|l|\leq 5}\int_{\mathcal{S}_{1}}2^{2jks}\Delta_{j}(S_{j+l-1}\p_{y}w\Delta_{j+l}a)\cdot\Delta_{j}w\,\eta(y)\eta'(y)\,dtdxdy,\\
J_{j,2,5}=&\sum_{-5<l\leq 5}\sum_{q=-6}^{l-2}\int_{\mathcal{S}_{1}}2^{2jks}\Delta_{j}(\Delta_{j+q}a\Delta_{j+l}\p_{y}w)\cdot \Delta_{j}w\,\eta(y)\eta'(y)\,dtdxdy,\\
J_{j,2,6}=&\sum_{l>j-5,|q-l|\leq1}\int_{\mathcal{S}_{1}}2^{2jks}\Delta_{j}(\Delta_{l}a\Delta_{q}\p_{y}w)\cdot\Delta_{j}w\,\eta(y)\eta'(y)\,dtdxdy.
\end{split}
\end{equation}
For $J_{j,2,1}$, integrating by parts, we get
\begin{equation}
|J_{j,2,1}|\lesssim_{y_{0},y_1}(\|a\|_{L^{\infty}}+\|\p_{y}a\|_{L^{\infty}})\cdot \int_{\mathcal{S}_{y_1}}2^{2jks}|\Delta_{j}w|^{2}dtdxdy.
\end{equation}
The remaining terms $J_{j,2,2}$ to $J_{j,2,6}$ can be estimated in the same fashion as $I_{j,l}$ terms. Indeed, we have 
\begin{equation}
\begin{split}
&|J_{j,2,2}|\leq\frac{1}{10}I_{j,1}+C_{y_{0},y_1}\|2^{jks}\Delta_{j}w\|_{L^{2}(\mathcal{S}_{y_1})}^{2},\\
&|J_{j,2,3}|\lesssim_{y_{0},y_1}\|2^{jks}\Delta_{j}w\|_{L^{2}(\mathcal{S}_{y_1})}^{2}+\|a\|_{C^{\alpha}}^{2}2^{-2j(\alpha-s)}\Big(\sum_{|l|\leq 5}\|2^{j(k-1)s}\p_{y}\Delta_{j+l}w\|_{L^{2}(\mathcal{S}_{y_1})}^{2}\Big),
\end{split}
\end{equation}
as well as
\begin{equation}
\begin{split}
&|J_{j,2,4}|\lesssim_{y_{0},y_1}\|2^{jks}\Delta_{j}w\|_{L^{2}(\mathcal{S}_{y_1})}^{2}+\|\p_{y}w\|_{L^{\infty}}^{2}\Big(\sum_{|l|\leq 5}\|2^{jks}\Delta_{j+l}a\|_{L^{2}(\mathcal{S}_{y_1})}^{2}\Big),\\
&|J_{j,2,5}|\lesssim_{y_{0},y_1}\|2^{jks}\Delta_{j}w\|_{L^{2}(\mathcal{S}_{y_1})}^{2}+[a]_{C^{\alpha}}^{2}2^{-2j(\alpha-s)}\Big(\sum_{-5<l\leq 5}\|2^{j(k-1)s}\p_{y}\Delta_{j+l}w\|_{L^{2}(\mathcal{S}_{y_1})}^{2}\Big),\\
&|J_{j,2,6}|\lesssim_{y_{0},y_1}\|2^{jks}\Delta_{j}w\|_{L^{2}(\mathcal{S}_{y_1})}^{2}+\|\p_{y}w\|_{L^{\infty}}^{2}\|(Id-\Delta_{-1})a\|_{\widetilde{H}^{ks}(\mathcal{S}_{y_1})}^{2}\cdot c_{j}.
\end{split}
\end{equation}

Combining all the above estimates, and summing in $j$, we complete the proof of Lemma \ref{nonlocalenergyesti}.
\end{proof}

\begin{proof}[Proof of Lemma \ref{nonlocaltangential}] The proof of Lemma \ref{nonlocaltangential} is somewhat similar to that of Lemma \ref{tangentialimprove1}, and we will be somewhat brief. We decompose $W$ into two parts, $W_1$ and $W_2$, where $W_{1}$ satisfies \eqref{W1eq} and $W_{2}$ satisfies \eqref{W2eq}.

\textit{$\bullet$ Bounds on $W_{1}$.} By Lemma \ref{h2esti} and \eqref{enhancebound}, we have
\begin{equation}\label{higherw1}
\begin{split}
&\Big\|\p_{y}^{2}|D_{x}|^{ks}W_{1}\Big\|_{L^{2}(\mathbb{H}^-)}^{2}+\Big\||D_{x}|^{\frac{2}{3}}|D_{x}|^{ks}W_{1}\Big\|_{L^{2}(\mathbb{H}^-)}^{2}\\
&\lesssim \left\||D_{x}|^{ks}\left(\eta^{2}f-4\eta(y)\eta'(y)\p_{y}w-\left(2\eta\eta''+2(\eta')^{2}\right)w-2\eta(y)\eta'(y)(a-1)\p_{y}w\right)\right\|_{L^{2}(\mathbb{H}^-)}^{2}\\
&\lesssim \left\|\eta^{2}f(t,x,y)-4\eta(y)\eta'(y)\p_{y}w-\left(2\eta\eta''+2(\eta')^{2}\right)w-2\eta(y)\eta'(y)(a-1)\p_{y}w\right\|_{L^{2}(\mathbb{H}^-)}^{2}\\
&\quad+\sum_{j\geq 10}2^{2jks}\left\|\Delta_{j}\left(\eta^{2}f-4\eta(y)\eta'(y)\p_{y}w-\left(2\eta\eta''+2(\eta')^{2}\right)w-2\eta(y)\eta'(y)(a-1)\p_{y}w\right)\right\|^2_{L^{2}(\mathbb{H}^-)}\\
&\lesssim C_{y_{0},y_1}\big(\|w\|_{\widetilde{H}^{ks}(\mathcal{S}_{y_1})},\|\p_{y}w\|_{\widetilde{H}^{ks}(\mathcal{S}_{y_1})},\|f\|_{\widetilde{H}^{ks}(\mathcal{S}_{y_1})}\big).
\end{split}
\end{equation}
In the last inequality, we used Bony's paraproduct decomposition \eqref{parapro}.

\begin{comment}
By taking Fourier transform with $(t,x)$ in \eqref{W1eq}, we can also obtain 
\begin{equation}
\left\||D_{t}|^{\frac{2}{3}}|D_{t,x}|^{ks}W_{1}\right\|_{L^{2}(\R^{2}_{t,x}\times[0,10])}^{2}\lesssim C_{y_{0}}\Big(\|w\|_{\widetilde{H}^{ks}(\mathcal{S}_{1})},\|\p_{y}w\|_{\widetilde{H}^{ks}(\mathcal{S}_{1})},\|f\|_{\widetilde{H}^{ks}(\mathcal{S}_{1})}\Big).
\end{equation}
\end{comment}

\textit{$\bullet$ Estimates of $W_{2}$.} Write
\begin{equation*}
F(t,x,y)=\left(a(t,x,y)-1\right)\p_{y}w\cdot \eta^{2}.
\end{equation*}
By \eqref{refinebesovesti}, we have for $j\geq 10$,
\begin{equation}\label{higherw2}
    2^{(2ks+\frac{2}{3})j}\|\Delta_{j}W_{2}\|_{L^{2}(\R^{-}\times\R\times[0,10])}^{2}\lesssim\sum_{|l|\leq2}2^{2jks}\|\Delta_{j+l}F\|_{L^{2}(\mathbb{H}^-)}^{2}.
\end{equation}
Applying Bony's paraproduct decomposition \eqref{parapro}, we have
\begin{equation}
\sum_{j\geq 8}2^{2jks}\|\Delta_{j}F\|_{L^{2}(\mathbb{H}^-)}^{2}\lesssim\|\p_{y}w\|_{L^{\infty}\cap\widetilde{H}^{ks}(\mathcal{S}_{y_1})}^{2}\times \Big(\|a\|_{C^{\alpha}(\mathcal{S}_{1})}+\|(Id-\Delta_{-1})a\|_{\widetilde{H}^{ks}(\mathcal{S}_{1})}\Big)^{2}.
\end{equation}
Hence, the conclusion of Lemma \ref{nonlocaltangential} follows by summing \eqref{higherw2} over $j$ and combining with \eqref{higherw1}.
\end{proof}
The analogous statements of Propositions \ref{zerohypo} and \ref{highhypo} for interior estimates are also valid. We present them here for completeness; since the proofs are simpler, we omit them. Here we adopt the notations
\begin{equation}
    \begin{split}
        &\widetilde{S}_{y_{0}}:=(-\infty,0]\times\R\times(-y_{0},y_{0}),\\
        &\|v\|_{\widetilde{H}^{s}_{\rm int}(\widetilde{S}_{y_{0}})}^{2}:=\int_{-y_{0}}^{y_{0}}\|\langle D_{x}\rangle^{s}v(\cdot,\cdot,y)\|_{L^{2}(\R_{t}^{-}\times\R_{x})}^{2}dy.
    \end{split}
\end{equation}
\begin{proposition}\label{interiorhypo1}
    Assume that $c\leq a(t,x,y)\leq c^{-1}$ for some $c\in(0,\frac{1}{2}]$, $f(t,x,y)\in L^{2}(\widetilde{\mathcal{S}}_{1})$, and 
\begin{equation*}
w(t,x,y)\in L^{\infty}(\widetilde{\mathcal{S}}_{1})\cap W^{1,1}(\widetilde{\mathcal{S}}_{1})\cap W^{2,1}_{y}(\widetilde{\mathcal{S}}_{1})
\end{equation*}
is a weak solution to the first equation in $\eqref{maineqsection6}$ in $\widetilde{S}_{1}$. Then for any $y_{0}<1$, we have
\begin{equation}
\|w\|_{\widetilde{H}^{\frac{1}{3}}_{\rm int}(\widetilde{\mathcal{S}}_{y_{0}})}+\|\p_{y}w\|_{L^{2}(\widetilde{\mathcal{S}}_{y_{0}})}\lesssim_{c,y_{0}} \|w\|_{L^{2}(\widetilde{\mathcal{S}}_{1})}+\|f\|_{L^{2}(\widetilde{\mathcal{S}}_{1})}.
\end{equation}
\end{proposition}
\begin{proposition}\label{interiorhypo}
Fix $k\in\Z\cap[1,\infty)$, $\alpha\in(0,1)$, $c\in(0,1/2]$ and a positive $s\leq\min\{1/3,\alpha\}$. Assume that $c\leq a(t,x,y)\leq c^{-1}$, $f(t,x,y)\in \widetilde{H}^{ks}_{\rm int}(\widetilde{\mathcal{S}}_{1})$ and $[a]_{C^{\alpha}(\widetilde{\mathcal{S}}_{1})}\lesssim 1$. Let 
\begin{equation*}
w(t,x,y)\in L^{\infty}(\widetilde{\mathcal{S}}_{1})\cap W^{1,1}(\widetilde{\mathcal{S}}_{1})\cap W^{2,1}_{y}(\widetilde{\mathcal{S}}_{1})
\end{equation*}
be a weak solution to $\eqref{maineqsection6}_{1}$ in $\widetilde{S}_{1}$. Assume in addition that 
\begin{equation}
M_k:=\|\p_{y}a\|_{L^{\infty}(\widetilde{\mathcal{S}}_{1})}+\|(Id-\Delta_{-1})a\|_{\widetilde{H}^{ks}_{\rm int}(\widetilde{\mathcal{S}}_{1})}+[a]_{C^{\alpha}(\widetilde{\mathcal{S}}_{1})}<\infty.
\end{equation}
Then for $0<y_{0}<1$, 
we have
\begin{equation}\label{hohe1.01}
\begin{split}
&\|w\|_{\widetilde{H}^{(k+1)s}_{\rm int}(\widetilde{\mathcal{S}}_{y_{0}})}+\|\p_{y}w\|_{\widetilde{H}^{ks}_{\rm int}(\widetilde{\mathcal{S}}_{y_{0}})}\lesssim_{c,k,s,y_{0},\alpha, M_k} \|w\|_{L^2(\widetilde{\mathcal{S}}_1)}+\|\partial_yw\|_{L^\infty(\widetilde{\mathcal{S}}_1)}+\|f\|_{\widetilde{H}^{ks}_{\rm int}(\widetilde{\mathcal{S}}_{1})}.
\end{split}
\end{equation}
\end{proposition}
\subsection{Smoothness of the solution for \eqref{diverformsectionSIX} with smooth coefficients and forcing term}
Combining the regularity estimates in Sections \ref{continuity}-\ref{higher}, we obtain the following up--to--boundary smoothness result for weak solutions to \eqref{diverformsectionSIX} provided that $a$ and $f$ are smooth. The result can be viewed as a generalization of the classical elliptic/parabolic regularity theory and the celebrated De Giorgi-Nash-Moser estimates to the {\it  ultraparabolic} setting.
%Corollary \ref{hypohilbert19} can be viewed as De Giorgi-Nash-Moser  in the hypoelliptic setting,  which is of independent interest. 
\begin{corollary}\label{hypohilbert19}
Assume that $w\in L^{\infty}\cap W^{1,1}\cap W^{2,1}_{y}$ is a weak solution to \eqref{diverformsectionSIX} in $\mathcal{B}_{1}^+$, and that $a$ and $f$ are smooth in $\mathcal{B}_{1}^{+}$, and $a(t,x,y)$ is uniformly elliptic. Then
\begin{equation*}
    w\in C^{\infty}([-r_{0}^2,0]\times[-r_{0}^3,r_{0}^3]\times[0,r_{0}]),\quad \forall r_{0}<1.
\end{equation*}

\end{corollary}
\begin{proof}[Proof (Sketch)]
By the H\"older estimates in Section \ref{continuity} and Propositions \ref{sobolevesti}-\ref{sob1}, we have
\begin{equation*}
    w\in C^{\alpha}_{\rm loc}(\mathcal{B}_{1}^{+}),\ \text{for\ some}\ \alpha>0,
\end{equation*}
and
\begin{equation*}
    \p_{y}^{2}w\in L^{p}_{\rm loc}(\mathcal{B}_{1}^{+}),\ \text{for\ all}\ 2\leq p<\infty.
\end{equation*}
Applying the anisotropic interpolation (see Lemma \ref{aniinterpolation}), we get that
\begin{equation*}
    \p_{y}w\in L^{\infty}_{\rm loc}(\mathcal{B}_{1}^{+}).
\end{equation*}
To bootstrap to higher regularity, let us recall the enhanced dissipation estimate (see Lemma \ref{h2esti}) and \eqref{enhancebound} for a suitable smooth cutoff function $\varphi\in C_c^\infty(\mathcal{B}_1^+)$,
\begin{equation}
    |D_{x}|^{\frac{2}{3}}(\varphi w)\in L^{2}_{\rm loc}(\mathbb{H}^{-}).
\end{equation}
Therefore, applying Proposition \ref{highhypo} inductively (again with suitable localizations), we obtain that
\begin{equation*}
    \p_{x}^{k}w\in L^{2}_{\rm loc}(\mathcal{B}_{1}^{+}),\quad \forall k\geq 1.
\end{equation*}
Finally, the regularity of $w$ in $t$ and $y$ follows from standard parabolic theory. This completes the outline of the proof.
\end{proof}

%%%%%%%%%%%%%Proof of Theorem 1%%%%%%%%%%%%%%

\subsection{Proof of Theorem \ref{mainone} and \ref{th:smo1}}\label{potmainone}
In this section, we present the proof of Theorems \ref{mainone} and \ref{th:smo1}. First we present the proof of Theorem \ref{th:smo1}: up-to-boundary smoothness in the Crocco coordinates. 

By the assumption of Theorem \ref{th:smo1}, $w\in W^{1,1}_{loc}\cap W^{2,1}_{\eta,loc}((0,T]\times (0,X)\times[0,1))$ satisfies
\begin{equation}
\begin{cases}
\p_{\tau}w+\eta\p_{\xi}w-w^{2}\p_{\eta}^{2}w=0,\ (\tau,\xi,\eta)\in(0,T)\times(0,X)\times (0,1),\\
\p_{\eta}w|_{\eta=0} = 0,
\end{cases}
\end{equation}
and $0<c_{\Omega}\leq w \leq C_{\Omega}$ on any $\Omega \Subset (0,T]\times(0,X)\times [0,1)$.

We need to show that for any $M\in\mathbb{N}$,
\begin{equation}\label{prioricrocco}
    \|w\|_{C^{M}(D^{\ast})}\lesssim 1,\ \text{for}\ D^{\ast}\Subset(0,T]\times(0,X)\times[0,1). 
\end{equation}
In the above, the implied constant depends on $M$, $D^\ast$ and its distance to the inflow and initial boundary, as well as the quantitative bounds in the regularity assumptions of $w$. The proof is divided into two steps:

\textit{Step 1: Boundary Regularity.} To simplify notations, for $t_{0}\in(0,T]$ and $x_{0}\in (0,X)$, we set
\begin{equation}
%\begin{cases}
    z_{0} = (t_{0},x_{0},0),\quad
    d_{0} = \frac{1}{2}\min\Big\{\sqrt{t_{0}},\sqrt[3]{x_{0}},\sqrt[3]{X-x_{0}},1\Big\}.\\
%\end{cases}
\end{equation}
In this step we show that for any $t_{0}\in(0,T]$ and $x_{0}\in (0,X)$, and any $M\in\mathbb{N}$, we have
\begin{equation}\label{pfmaingoal1}
\|w\|_{C^{M}(\mathcal{B}^{+}_{\frac{1}{10}d_{0}}(z_{0}))}\lesssim_{M,d_{0},m_{\mathcal{B}_{d_{0}}^{+}(z_{0})},M_{\mathcal{B}_{d_{0}}^{+}(z_{0})}} 1.
\end{equation}
For ease of notation, below we omit the dependence on $M,d_{0},m_{\mathcal{B}_{d_{0}}^{+}(z_{0})},M_{\mathcal{B}_{d_{0}}^{+}(z_{0})}$ which are fixed.

By direct calculation, $v:=w^{-1}$ satisfies the equation 
\begin{equation}\label{vdivereq}
    \begin{cases}
        \p_{\tau}v+\eta\p_{\xi}v-\p_{\eta}(w^{2}\p_{\eta}v)=0,\ (\tau,\xi,\eta)\in (0,T)\times(0,X)\times(0,1),\\
        \p_{\eta}v|_{\eta=0}=0.
    \end{cases}
\end{equation}
Therefore, by Proposition \ref{holder}, there exists $\alpha>0$ such that
\begin{equation}\label{pfmainholder}
\|v\|_{C^{\alpha}\big(\mathcal{B}_{\frac{1}{2}d_{0}}^{+}(z_{0})\big)}+\|w\|_{C^{\alpha}\big(\mathcal{B}_{\frac{1}{2}d_{0}}^{+}(z_{0})\big)}\lesssim 1.
\end{equation}

It follows from Proposition \ref{sobolevesti}, Proposition \ref{sob1}, and a standard (Gagliardo-Nirenberg) interpolation inequality that
\begin{equation}\label{pfmainw2p}
    \|\p_{\eta}^{2}v\|_{L^{p}\big(\mathcal{B}_{\frac{1}{4}d_{0}}^{+}(z_{0})\big)}+\|\p_{\eta}^{2}w\|_{L^{p}\big(\mathcal{B}_{\frac{1}{4}d_{0}}^{+}(z_{0})\big)}\lesssim 1,\quad \forall 1<p<\infty.
\end{equation}

In view of \eqref{pfmainholder}, \eqref{pfmainw2p}, and the anisotropic embedding Lemma \ref{aniinterpolation} (using a standard density argument), we obtain that
\begin{equation}\label{pfmaininfty}
    \|\p_{\eta}v\|_{L^{\infty}\big(\mathcal{B}_{\frac{1}{4}d_{0}}^{+}(z_{0})\big)}+\|\p_\eta w\|_{L^{\infty}\big(\mathcal{B}_{\frac{1}{4}d_{0}}^{+}(z_{0})\big)}\lesssim 1.
\end{equation}

Thanks to \eqref{vdivereq} and \eqref{pfmaininfty}, we can apply Proposition \ref{highhypo} inductively (with a suitable series of localizations). As a consequence, for any $K\in\Z\cap[1,\infty)$, we obtain that
\begin{equation}
\sum_{\ell\leq K}\|\p_{\xi}^{\ell}v\|_{L^{2}\big(\mathcal{B}_{\frac{1}{5}d_{0}}^{+}(z_{0})\big)}\lesssim_{K} 1.
\end{equation}
Using $w=1/v$ and a standard (Gagliardo-Nirenberg) interpolation inequality, we also get
\begin{equation}\label{infinityx1}
\sum_{\ell\leq K}\|\p_{\xi}^{\ell}w\|_{L^{2}\big(\mathcal{B}_{\frac{1}{5}d_{0}}^{+}(z_{0})\big)}\lesssim_{K} 1,\quad \forall\, K\geq 1.
\end{equation}

It remains to demonstrate that $w$ is also regular in the $(t,y)$ variable. To this end, we bootstrap the regularity using standard tools from parabolic equations. By interpolating \eqref{infinityx1} with \eqref{pfmainholder}--\eqref{pfmainw2p}, for any $K\in\Z\cap[1,\infty)$, $1<p<\infty$ and $0<r<\frac{1}{5}d_{0}$, we get
\begin{equation}\label{infinityx2}
\sum_{\ell\leq K}\Big(\|\p_{\xi}^{\ell}w\|_{L^{p}\big(\mathcal{B}_{r}^{+}(z_{0})\big)}+\|\p_{\xi}^{\ell}\p_{\eta}w\|_{L^{p}\big(\mathcal{B}_{r}^{+}(z_{0})\big)}\Big)\lesssim_{p,r,K} 1.
\end{equation}
As a consequence, differentiating \eqref{pfmainweq} with respect to $\eta$ and applying standard parabolic $W^{2,p}$ estimates inductively, we conclude that for any $K\in\Z\cap[1,\infty)$, $1<p<\infty$ and $0<r<\frac{1}{5}d_{0}$, 
\begin{equation}\label{infinityx3}
\sum_{\ell\leq K}\|\p_{\eta}^{\ell}w\|_{L^{p}\left(\mathcal{B}_{r}^{+}(z_{0})\right)}\lesssim_{p,r,K} 1.
\end{equation}
By interpolation, it follows that for any $K\in\Z\cap[1,\infty)$, $1<p<\infty$ and $0<r<\frac{1}{5}d_{0}$,
\begin{equation}\label{infinityx4}
\sum_{\beta+\gamma\leq K}\|\p_{\xi}^{\beta}\p_{\eta}^{\gamma}w\|_{L^{p}\left(\mathcal{B}_{r}^{+}(z_{0})\right)}\lesssim_{p,r,K} 1.
\end{equation}
Finally, the regularity of $w$ with respect to the temporal variable $t$ follows from \eqref{pfmainweq} and interpolation with \eqref{infinityx4}. The estimate of up--to--boundary smoothing \eqref{pfmaingoal1} is then proved.

\textit{Step 2: Interior Regularity.} Next, we consider the interior smoothing estimates of \eqref{pfmainweq}. Similar to the first step, for $t_{1}\in(0,T]$, $x_{1}\in(0,X)$ and $\eta_{1}\in(0,1)$, let
\begin{equation}
   % \begin{cases}
        z_{1}=(t_{1},x_{1},\eta_{1}),\quad
        d_{1}=\frac{1}{2}\min\{\sqrt{t_{1}},\sqrt[3]{x_{1}},\sqrt[3]{X-x_{1}},\eta_{1},1-\eta_{1}\}.
 %   \end{cases}
\end{equation}
It suffices to show that for any $t_{1}\in(0,T]$, $x_{1}\in (0,X)$, $\eta_{1}\in(0,1)$, and any $M\in\mathbb{N}$, we have
\begin{equation}\label{pfmaingoal2}
\|w\|_{C^{M}(\mathcal{B}_{\frac{1}{10}(d_{1})^{2}}(z_{1}))}\lesssim_{M,d_{1},m_{\mathcal{B}_{d_{1}}(z_{1})},M_{\mathcal{B}_{d_{1}}(z_{1})}} 1.
\end{equation}
Indeed, \eqref{pfmaingoal2} follows from (i) rescaling to shift the center of the cylinder to the origin $(0,0,0)$; (ii) applying the interior estimates from Proposition \ref{interiorholder}, Proposition \ref{interiorsobolevesti}, Proposition \ref{interiorsob1} and Proposition \ref{interiorhypo}. 

We complete the proof of smoothing estimates \eqref{prioricrocco} in the Crocco coordinates by combining \eqref{pfmaingoal1} and \eqref{pfmaingoal2}. This completes the proof of Theorem \ref{th:smo1}.

Now we are ready to complete the proof of Theorem \ref{mainone}.

\begin{proof}[Proof of Theorem \ref{mainone}]
Let $u(t,x,y)$ saitsifies the assumptions in Theorem \ref{mainone}. Fix an arbitrary $D\Subset (0,T]\times(0,X)\times[0,\infty)$, we need to show that $(u,v)$ satisfy \eqref{smooeffesti} for any $m\in\mathbb{N}$. Under the hypotheses in Theorem \ref{mainone}, the the Crocco transform \eqref{croccotrans} is well-defined. Recall that in the Crocco coordinates, $w:=\p_{y}u$ satisfies
\begin{equation}\label{pfmainweq}
    \begin{cases}
        \p_{\tau}w+\eta\p_{\xi}w-w^{2}\p_{\eta}^{2}w=0,\ (\tau,\xi,\eta)\in (0,T)\times(0,X)\times(0,1),\\
        \p_{\eta}w|_{\eta=0}=0.
    \end{cases}
\end{equation}
%Here $\eta_{0}>0$ is chosen such that the image of $D$ in Crocco coordinate is properly contained in $(0,T]\times(0,X)\times[0,\eta_{0})$.
%In above, $U(t,x)$ denotes the outer flow
%$$U(t,x):=\lim_{y\uparrow \infty}u(t,x,y),\ (t,x)\in(0,T)\times(0,X).$$ %Without loss of generality, we assume $U(t,x)\equiv 1$ for clarity.

Notice that by the assumptions in Theorem \ref{mainone}, for any compact subset $D^{\ast}$ of $(0,T]\times(0,X)\times[0,1)$, we have
\begin{equation}\label{eqw5.72}
    \begin{split}
        &0<m_{D^{\ast}}\leq w(\tau,\xi,\eta)\leq M_{D^{\ast}},\ a.e.\ \text{in}\ D^{\ast},\\
        &\|\p_{\tau}w\|_{L^{1}(D^{\ast})}+\|\p_{\xi}w\|_{L^{1}(D^{\ast})}+\|\p_{\eta}w\|_{L^{1}(D^{\ast})}+\|\p_{\eta}^{2}w\|_{L^{1}(D^{\ast})}\lesssim_{D^{\ast}} 1.
    \end{split}
\end{equation}
Recall \eqref{prioricrocco}, for any $M\in\mathbb{N}$,
\begin{equation}\label{prioricrocco1}
    \|w\|_{C^{M}(D^{\ast})}\lesssim 1,\ \text{for}\ D^{\ast}\Subset(0,T]\times(0,X)\times[0,1). 
\end{equation}
In the above, the implied constant depends on $M$, $D^\ast$ and its distance to the inflow and initial boundary, as well as the quantitative bounds in \eqref{eqw5.72}. The proof is divided into two steps:

Recall that for $(t,x,y)\in D$ we can represent $u$ via
\begin{equation}
    y=\int_{0}^{u(t,x,y)}\frac{d\eta}{w(t,x,\eta)},
\end{equation}
where
\begin{equation}
    w(t,x,\eta):=\p_{y}u(t,x,y),\quad \text{for}\ u(t,x,y)=\eta.
\end{equation}
By direct calculation, we obtain that
\begin{equation}
(\p_{t},\p_{x})u(t,x,y)=w(t,x,u(t,x,y))\cdot\int_{0}^{u(t,x,y)}\frac{(\p_{t},\p_{x})w}{w^{2}}d\eta.
\end{equation}
We can also represent the higher--order derivatives of $u$ in terms of $w$ using Fa\`a di Bruno's formula. Moreover, for fixed $M\in\mathbb{N}$ and $D\Subset (0,T]\times(0,X)\times[0,\infty)$, we have
\begin{equation}\label{pfmainbasedesti}
    \|u\|_{C^{M}(\widetilde{D})}\lesssim_{M,m_{D^{\ast}},M_{D^{\ast}},\|w\|_{C^{M}(D^{\ast})}} 1,
\end{equation}
where $D^{\ast}$ is the image of $\widetilde{D}$ under the Crocco transform \eqref{croccotrans}. Hence we conclude the proof of Theorem \ref{mainone} in view of \eqref{prioricrocco1}.
\begin{comment}
Now assume $D^{\star}$ is another compact subset of $D_{T,X}$ such that $\widetilde{D}\Subset D^{\star}\Subset D_{T,X}$, and $D^{\#}$ is the image of $D^{\star}$ under Crocco transform \eqref{croccotrans}. Then by \eqref{pfmainbasedesti} and \eqref{prioricrocco}, we have
\begin{equation}\label{basedesti2}
    \|u\|_{C^{M}(\widetilde{D})}\lesssim_{M,\mathrm{dist}(D^{\ast}, D^{\#}),m_{D^{\star}},M_{D^{\star}}} 1.
\end{equation}
We conclude the proof of Theorem \ref{mainone} using \eqref{basedesti2} and the fact that
\begin{equation}
    m_{D^{\star}}\cdot\mathrm{dist}(\widetilde{D},D^{\star})\leq \mathrm{dist}(D^{\ast},D^{\#})\leq M_{D^{\star}}\cdot\mathrm{dist}(\widetilde{D},D^{\star}).
\end{equation}
\end{comment}
\end{proof}

\section{Global theory of weak solutions and proof of Theorem \ref{mainthree}}\label{basic}
\subsection{Global Weak Solutions in Crocco Coordinates }
Recall the dynamical Prandtl equation after the Crocco transform (here we adopt the coordinate notation $(t,x,y)$ instead of $(\tau,\xi,\eta)$):
\begin{equation}\label{prandtlcrocco}
\begin{cases}
\p_{t}w+y\p_{x}w-w^{2}\p_{y}^{2}w=0,\ (t,x,y)\in (0,T]\times (0,X]\times (0,1),\\
w|_{t=0}=w_{0}(x,y),\ w|_{x=0}=w_{1}(t,y),\\
w|_{y=1}=0,\ \p_{y}w|_{y=0}=0.
\end{cases}
\end{equation}
In this section, we will show the global well-posedness for weak solutions of \eqref{prandtlcrocco}. We first introduce the following definition of weak solutions to \eqref{prandtlcrocco}. For ease of notation, we denote
\begin{equation}\label{nota1}
    R_{T,X}:=(0,T]\times (0,X]\times (0,1), \qquad R_{T,X}^{y_{0}}:=(0,T]\times (0,X]\times (0,y_{0}), \,\,\forall0<y_0<1.
\end{equation}

\begin{definition}\label{defweak}
We call $w(t,x,y)$ a weak solution to \eqref{prandtlcrocco}, provided that
\begin{enumerate}

\item  $w\in L^{\infty}(R_{T,X})\cap W^{1,1}(R_{T,X}),\ \partial_y^2w\in L^1(R_{T,X}^{y_{0}})\,\,{\it for\,\,all\,\,} y_0\in(0,1), \,\,{\it and}\,\, w^{2}\p_{y}^{2}w\in L^{1}(R_{T,X});$

\item  For all $0<y_{0}<1$, $w$ has a positive lower bound in $R_{T,X}^{y_{0}}$;

\item $w(t,x,y)$ satisfies \eqref{prandtlcrocco} in $R_{T,X}$ in the sense of distributions, where the boundary conditions are assumed in the trace sense. 
\end{enumerate}
\end{definition}
We note that by interpolation, a weak solution to \eqref{prandtlcrocco} satisfies $\partial_yw\in L^2(R_{T,X}^{y_{0}})$ for any $y_0\in(0,1)$.

It was proved in \cite{XZZ24} that \eqref{prandtlcrocco} admits a unique global weak solution for any $T,X>0$, provided that for $y\in(0,1)$,
\begin{equation}\label{xzzassume}
w_{0}\sim 1-y,\quad\ w_{1}\sim 1-y,
\end{equation}
and certain regularity conditions are satisfied. Indeed, \eqref{xzzassume} is equivalent to the condition that the initial data and inflow data both converge to the outer flow with an exponential rate towards the interior of the fluid. 

In this section, we significantly extend the existence, uniqueness and regularity theory of weak solutions to \eqref{prandtlcrocco}, by allowing a much wider class of initial and inflow conditions than those considered in the seminal work \cite{XZZ24}. As a result, we can treat boundary layer flows that converge to the outer flow $U(x)\equiv 1$ with three distinct and natural rates in physical variables: {\it Gaussian}, {\it exponential}, and {\it polynomial}.

We first establish a uniqueness result for the weak solution to \eqref{prandtlcrocco} under quite general assumptions.
\begin{proposition}[\bf Uniqueness]\label{uniweaksol}
For any $k\geq 2$, \eqref{prandtlcrocco} admits at most one weak solution $w(t,x,y)\in W^{1,1}(R_{T,X})\cap L^{\infty}(R_{T,X})$ satisfying the additional assumption in $R_{T,X}$ that 
\begin{equation}\label{weakrate}
(1-y)^{k}\lesssim w \lesssim (1-y)\log^{\frac{1}{2}}\Big(\frac{10}{1-y}\Big),
\end{equation}
with given $w_{0}(x,y)$ and $w_{1}(t,y)$.
\end{proposition}

\begin{remark}Compared with the assumptions for the existence result in Proposition \ref{existweaksol}, the uniqueness holds in a wider class of weak solutions. In particular, the uniqueness holds for velocity fields converging to the outer flow at any of the three admissible rates (Gaussian, exponential and arbitrary polynomial) in physical coordinates, which is useful for applications to the study of Navier--Stokes equations and the long time behavior of solutions to \eqref{prandtl}.
\end{remark}

\begin{proof} Assume that $w_{1},w_{2}$ are two weak solutions satisfying the assumptions of Proposition \ref{uniweaksol}. For $i=1,2$, define $v_{i}=w_{i}^{-1}$, and let $v=v_{1}-v_{2}$, $w=w_{1}-w_{2}$. Then in $R_{T,X}$ we have the equation
\begin{equation}\label{reciequ}
\p_{t}v+y\p_{x}v+\p_{y}^{2}w=0.
\end{equation}
Choosing $\lambda=2k+10$, and then testing \eqref{reciequ} by ${\rm sgn}(v)(1-y)^{\lambda}$ and integrating on $y\in[0,1]$, we get that
\begin{equation}\label{uniqueesti}
\begin{split}
&\p_{t}\left(\int_{0}^{1}|v|(1-y)^{\lambda}dy\right)+\p_{x}\left(\int_{0}^{1}|v|y(1-y)^{\lambda}dy\right)-\int_{0}^{1}\p_{y}^{2}w\cdot {\rm sgn}(w)(1-y)^{\lambda}dy=0.
\end{split}
\end{equation}
The above formula can be justified rigorously by replacing ${\rm sgn}(v)$ with $\phi(v/\delta)$ and taking $\delta\to0+$, where the smooth function $\phi(x)$ is odd, equal to $-1$ for $x\leq-1$ and equal to $1$ for $x\ge1$. %See e.g. \cite{XZZ24} for this renormalization argument.

By direct calculations, we obtain
\begin{equation*}
\begin{split}
-\int_{0}^{1}\p_{y}^{2}w\cdot {\rm sgn}(w)(1-y)^{\lambda}dy&\ge-\lambda(\lambda-1)\int_{0}^{1}|w|(1-y)^{\lambda-2}dy.
\end{split}
\end{equation*}
For $t\ge0$, let
\begin{equation*}
F(t):=\iint_{[0,X]\times[0,1]}|v|(1-y)^{\lambda}dxdy.
\end{equation*}
By \eqref{uniqueesti} and the identity $w=\frac{1}{v_1}-\frac{1}{v_2}=-\frac{v_1-v_2}{v_1v_2}=-(v_1-v_2)w_1w_2$, we have
\begin{equation}\label{gia1}
\begin{split}
F'(t)&\leq \lambda(\lambda-1)\int_{0}^{X}\int_{0}^{1}|w|(1-y)^{\lambda-2}dxdy\\
&=\lambda(\lambda-1)\int_{0}^{X}\int_{0}^{1}|v|(1-y)^{\lambda}{\boxed{\frac{w_{1}\cdot w_{2}}{(1-y)^{2}}}}dxdy.
\end{split}
\end{equation}
In view of \eqref{weakrate}, the boxed term in the integrand is not necessarily bounded, and as a consequence we cannot apply Gr\"onwall's inequality directly. To overcome this difficulty, we split the integral on the right-hand side of \eqref{gia1} in $y$ into two parts, one over $(0,y_0)$ and the other over $(y_0,1)$ with $1/2\leq y_0\leq1$, to obtain that
\begin{equation*}
\begin{split}
F'(t)&\leq \lambda(\lambda-1)\int_{0}^{X}\int_{0}^{1}|w|(1-y)^{\lambda-2}dxdy\\
&=\lambda(\lambda-1)\int_{0}^{X}\int_{0}^{y_{0}}|v|(1-y)^{\lambda}\frac{w_{1}\cdot w_{2}}{(1-y)^{2}}dxdy+\lambda(\lambda-1)\int_{0}^{X}\int_{y_{0}}^{1}|w|(1-y)^{\lambda-2}dxdy\\
&\leq C_{1}\log\left(\frac{1}{1-y_{0}}\right)F(t)+C_{2}(1-y_{0})^{\lambda-1}.
\end{split}
\end{equation*}
By Gr\"onwall's inequality, we get
\begin{equation}
\sup_{t\in[0,T^{*}]}F(t)\leq C_{2}\sup_{t\in[0,T^\ast]}\left[t(1-y_{0})^{\lambda-1}\exp\left\{C_{1}t\log\left(\frac{1}{1-y_{0}}\right)\right\}\right]\leq C_{2}T^{*}(1-y_{0})^{\lambda-1-C_{1}T^{*}}.
\end{equation}
By letting $y_{0}\to 1$, we obtain that
\begin{equation}
F(t)\equiv 0,\ 0\leq t\leq T^{*},\ T^{*}:=\frac{\lambda-1}{2C_{1}}.
\end{equation}
Therefore, $F(t)$ vanishes identically by dividing the temporal interval and repeating the above argument, which completes the proof of Proposition \ref{uniweaksol}.
\end{proof}

We now turn to the main result of the section, on the existence and (up--to--boundary) regularity of the unique global weak solution of \eqref{prandtlcrocco}. The global weak solution is constructed in the following space. Define the following functional space
\begin{equation}\label{weaksolspace}
\begin{split}
\mathcal{X}_{T,X}:=&\bigg\{w(t,x,y)\in L^\infty(R_{T,X}): (1-y)^{3/2}\lesssim_{T,X} w(t,x,y)\lesssim_{T,X} (1-y)\log^{\frac{1}{2}}\left(\frac{10}{1-y}\right),\\
&\quad (1-y)^{-1/2}\log^{-\frac{3}{2}}\left(\frac{10}{1-y}\right)\p_{y}w\in L^{2}(R_{T,X}), \frac{(\p_{t},\p_{x})w}{w^{2}}\cdot(1-y) \in L^{\infty}_{t}L^{1}_{x,y}(R_{T,X}),\\
&\quad \p_{yy}w\cdot(1-y)\in L^{\infty}_{t}L^{1}_{x,y}(R_{T,X})\bigg\}.
\end{split}
\end{equation}

\begin{proposition}[\bf Existence and Regularity]\label{existweaksol}
Assume that the initial data $w_{0}(x,y)$ and the inflow data $w_{1}(t,y)$ are bounded functions defined on $D:=[0,\infty)\times[0,1)$, and satisfy the conditions in Assumption \ref{dataassumecrocco4}.
\begin{comment}
Therefore in particular, $w_0, w_1$ satisfy the following pointwise inequality
\begin{equation}\label{initialrate}
\begin{split}
(1-y)^{3/2}\lesssim_{X} w_{0}(x,y)\lesssim_{X} (1-y)\log^{\frac{1}{2}}\Big(\frac{10}{1-y}\Big),\quad \forall (x,y)\in[0,X]\times[0,1),\\
(1-y)^{3/2}\lesssim_{T} w_{1}(t,y)\lesssim_{T} (1-y)\log^{\frac{1}{2}}\Big(\frac{10}{1-y}\Big),\quad \forall (t,y)\in[0,T]\times[0,1),
\end{split}
\end{equation}
as well as the $L^{1}$ bounds for tangential derivatives on the initial and inflow edges
\begin{equation}\label{initialbv}
\begin{split}
&\iint_{[0,X]\times[0,1]}\frac{|\p_{x}w_{0}(x,y)|}{|w_{0}(x,y)|^{2}}\cdot(1-y)dxdy+\iint_{[0,X]\times[0,1]}|\p_{y}^{2}w_{0}(x,y)|\cdot(1-y)dxdy \lesssim_{X} 1,\\
&\iint_{[0,T]\times[0,1]} \frac{|\p_{t}w_{1}(t,y)|}{|w_{1}(t,y)|^{2}}\cdot(1-y)dtdy+\iint_{[0,T]\times[0,1]}|\p_{y}^{2}w_{1}(t,y)|\cdot(1-y)dtdy \lesssim_{T} 1.
\end{split}
\end{equation}
\end{comment}
Then, for \emph{any} $T>0, X>0$, there exists a unique weak solution $w\in\mathcal{X}_{T,X}$ to \eqref{prandtlcrocco}. Moreover, the weak solution $w$ is smooth \emph{up to the boundary}
\begin{equation}\label{smoothwadded}
w(t,x,y)\in C^{\infty}\left((0,\infty)\times (0,\infty)\times [0,1)\right).
\end{equation}
\end{proposition}

\begin{comment}
\begin{remark}Although the assumption \eqref{initialbv} involves singular weight near $y=1$, it still makes sense. For instance, formally speaking, by \eqref{initialrate},
\begin{equation}
\Bigg[\frac{\p_{x}w_{0}}{w_{0}}\log^{-\frac{3}{2}}\left(\frac{10}{1-y}\right)\Bigg]_{y=1}=\Bigg[\p_{x}\left(\log(w_{0})\log^{-\frac{3}{2}}\left(\frac{10}{1-y}\right)\right)\Bigg]_{y=1}=0.
\end{equation}
Hence, 
\begin{equation}
\frac{\p_{x}w_{0}}{w_{0}}\cdot(1-y)^{-1}\log^{-3}\left(\frac{10}{1-y}\right)=o\left(\log^{-\frac{3}{2}}\left(\frac{10}{1-y}\right)\cdot(1-y)^{-1}\right),\ \text{as}\ y\to 1-,
\end{equation}
which is integrable near $y=1$.
\end{remark}
\end{comment}

\begin{remark}The proof of Proposition \ref{existweaksol} consists of three key ingredients: 
\begin{enumerate}
\item Constructing solutions to a family of approximation systems parametrized by $\epsilon>0$;

\item Deriving uniform-in-$\epsilon$ estimates for the  solutions $\{w^{\epsilon}\}$ to the approximate system;

\item Passing to the limit to get the weak solution.
\end{enumerate}

The proof of the uniform bounds is similar to the argument in \cite{XZZ24}. Since the class of  admissible initial data and inflow data considered here is larger than the class in \cite{XZZ24}, suitable modifications are required which we present in detail below.

%It is worth mentioning that, in contrast to \cite{XZZ24}, before passing to the limit, we can obtain the uniform--in--$\epsilon$ smoothing estimates for the approximation solutions $\{w^{\epsilon}\}$ by adopting the argument in Theorem \ref{mainone}. Therefore, we indeed obtain a global classical solution far away from the initial time $\{t=0\}$ and the inflow boundary $\{x=0\}$.
\end{remark}
\begin{comment}
\begin{remark}In Proposition \ref{existweaksol}, we admit the datas vanishes on the boundary $\{y=1\}$ in various possible rates \eqref{initialrate}. In particular, we admit the velocity convergences to the outer flow $U(t,x)\equiv 1$ in three physical settings: Gaussian convergence rate, exponential convergence rate, and polynomial convergence rate.
\end{remark}
\end{comment}
\begin{proof}[Proof of Proposition \ref{existweaksol}]We divide the proof of Proposition \ref{existweaksol} into six steps: 

\textit{Step 1: Approximation.} Fix $1/200>\epsilon>0$. For any $T>0$ and $X>0$, let $w^{\epsilon}\in C^\infty(\overline{R}_{T,X})$ be the unique smooth solution to the following approximation system
\begin{equation}\label{approsys}
\begin{cases}
\p_{t}w^{\epsilon}+(y+\epsilon)\p_{x}w^{\epsilon}-(w^{\epsilon}+\epsilon)^{2}\p_{y}^{2}w^{\epsilon}=0,\ (t,x,y)\in(0,T]\times(0,X]\times(0,1),\\
w^{\epsilon}|_{t=0}=w_{0}^{\epsilon},\ w^{\epsilon}|_{x=0}=w_{1}^{\epsilon},\\
w^{\epsilon}|_{y=1}=\p_{y}w^{\epsilon}|_{y=0}=0,
\end{cases}
\end{equation}
where $w_{0}^{\epsilon},w_{1}^{\epsilon}$ are suitable smooth modifications of $w_{0},w_{1}$ respectively, satisfying the compatibility conditions corresponding to \eqref{approsys} and certain uniform--in--$\epsilon$ bounds, which are presented in Appendix \ref{compatiappendix} (see (WP1), (WP2) and (WP3) in Appendix \ref{requireappendix}). We refer to Proposition \ref{constructappdata2} for the construction of the compatible data $(w_{0}^{\epsilon},w_{1}^{\epsilon})$, and Appendix \ref{proofapproxia} for the fact that the classical solution $w^\epsilon$ of \eqref{approsys} exists globally and is unique.

\textit{Step 2: Pointwise Estimates.} 
We first derive a uniform--in--$\epsilon$ upper bound on $w^{\epsilon}$. More precisely, we shall prove that for any $T>0,\, X>0$ and $(t,x,y)\in R_{T,X}$,
\begin{equation}\label{boundone}
 w^{\epsilon}(t,x,y)\lesssim_{T,X} (1-y)\log^{\frac{1}{2}}\Big(\frac{10}{1-y}\Big).
 \end{equation}
Let
\begin{equation}
h(t,x,y):=w^{\epsilon}(t,x,y)-C(1-y)\log^{\frac{1}{2}}\Big(\frac{10}{1-y}\Big),
\end{equation}
Consider the operator
\begin{equation}
\mathcal{L}_{\epsilon}=\p_{t}+(y+\epsilon)\p_{x}-(w^{\epsilon}+\epsilon)^{2}\p_{y}^{2}.
\end{equation}
By direct calculations, we get that
\begin{equation}\label{signoperator}
\mathcal{L}_{\epsilon}[h]=-C\frac{(w^{\epsilon}+\epsilon)^{2}}{1-y}\Big(\frac{1}{2}\log^{-\frac{1}{2}}\Big(\frac{10}{1-y}\Big)+\frac{1}{4}\log^{-\frac{3}{2}}\Big(\frac{10}{1-y}\Big)\Big)<0.
\end{equation}
Taking $C>1$ sufficiently large (depending on $T$ and $X$), we also have
\begin{equation}\label{signboundary}
\begin{split}
&h|_{t=0}\leq 0,\quad \forall (x,y)\in[0,X]\times[0,1],\\
&h|_{x=0}\leq 0, \quad\forall (t,y)\in[0,T]\times[0,1],\\
&h|_{y=1}=0,\quad \p_{y}h|_{y=0}>0, \quad \forall (t,x)\in[0,T]\times[0,X].
\end{split}
\end{equation}
Hence $h\leq 0$ for any $(t,x,y)\in[0,T]\times[0,X]\times[0,1]$ by \eqref{signoperator}--\eqref{signboundary} and the maximum principle.

Next we turn to the lower bound of $w^{\epsilon}$. Note that $w^{\epsilon}$ is non--negative by the maximum principle. We claim that for $(t,x,y)\in R_{T,X}$,
\begin{equation}\label{boundtwo}
w^{\epsilon}(t,x,y) \gtrsim_{T,X} (1-y)^{3/2}.
\end{equation}
To this end, we define for some $c>0$ and $M>1$ to be determined below,
 \begin{equation}
 g(t,x,y):=w^{\epsilon}(t,x,y)-ce^{-Mt}e^{3y}(1-y)^{3/2}.
 \end{equation}
 It suffices to prove that $g(t,x,y)$ is non--negative in $[0,T]\times [0,X]\times [0,1]$.
By direct calculation,
 \begin{equation}
 \mathcal{L}_{\epsilon}[g]=ce^{-Mt}e^{3y}(1-y)^{-1/2}\left[M(1-y)^{2}+(w^{\epsilon}+\epsilon)^{2}\left(9(1-y)^{2}-9(1-y)+3/4\right)\right].
 \end{equation}
 Hence,
 \begin{equation}
 \mathcal{L}_{\epsilon}[g]>0,\ \text{in}\  (0,T]\times (0,X]\times (0,1),
 \end{equation}
 by taking $M>1$ large enough, depending on $T>0, X>0$.
 
 Notice that by taking $c>0$ small (depending on $T>0, X>0$), we have
 \begin{equation}\label{signcondi}
 \begin{split}
&g|_{t=0}\geq 0,\ \forall (x,y)\in[0,X]\times[0,1],\\
&g|_{x=0}\geq 0,\ \forall (t,y)\in[0,T]\times[0,1]\\
&g|_{y=1}=0,\quad \p_{y}g|_{y=0}=-\frac{3c}{2}e^{-Mt}<0.
 \end{split}
 \end{equation}
Hence, $g$ is non--negative by the maximum principle. 

Combining \eqref{boundone} and \eqref{boundtwo}, we obtain the uniform-in-$\epsilon$ pointwise bound for $(t,x,y)\in R_{T,X}$,
\begin{equation}\label{ptwe1}
    (1-y)^{3/2}\lesssim_{T,X} w^\epsilon(t,x,y)\lesssim_{T,X}(1-y)\log^{\frac{1}{2}}\left(\frac{10}{1-y}\right).
\end{equation}

\textit{Step 3: Energy Estimates.} Testing \eqref{approsys} by
\begin{equation*}
\frac{1}{w^{\epsilon}+\epsilon}(1+\epsilon-y)^{-1}\log^{-3}\left(\frac{10}{1+\epsilon-y}\right),
\end{equation*}
and integrating by parts, we get the following energy identity
\begin{equation}\label{energyidee}
\begin{split}
&\iint_{[0,X]\times [0,1]}(\log (w^{\epsilon}+\epsilon))(1+\epsilon-y)^{-1}\log^{-3}\left(\frac{10}{1+\epsilon-y}\right)dxdy\bigg|^{t=T}_{t=0}\\
&+\iint_{[0,T]\times [0,1]}(\log (w^{\epsilon}+\epsilon)) (y+\epsilon)(1+\epsilon-y)^{-1}\log^{-3}\left(\frac{10}{1+\epsilon-y}\right)dtdy\bigg|^{x=X}_{x=0}\\
&+\iiint_{[0,T]\times[0,X]\times[0,1]}|\p_{y}w^{\epsilon}|^{2}(1+\epsilon-y)^{-1}\log^{-3}\left(\frac{10}{1+\epsilon-y}\right)dtdxdy\\
&+\iiint_{[0,T]\times[0,X]\times[0,1]}(w^{\epsilon}+\epsilon)\p_{y}w^{\epsilon}(1+\epsilon-y)^{-2}\log^{-3}\left(\frac{10}{1+\epsilon-y}\right)dtdxdy\\
&-3\iiint_{[0,T]\times[0,X]\times[0,1]}(w^{\epsilon}+\epsilon)\p_{y}w^{\epsilon}(1+\epsilon-y)^{-2}\log^{-4}\left(\frac{10}{1+\epsilon-y}\right)dtdxdy\\
&-\log^{-3}\left(\frac{10}{\epsilon}\right)\cdot\iint_{[0,T]\times[0,X]}\p_{y}w^{\epsilon}dtdx\bigg|_{y=1}=0.
\end{split}
\end{equation}
The last term in \eqref{energyidee} is non--negative, since $w^{\epsilon}|_{y=1}=0$ and $w^{\epsilon}$ is non--negative, which implies that $\p_{y}w^{\epsilon}|_{y=1}\leq 0$. Therefore, by \eqref{boundone}, \eqref{boundtwo} and \eqref{energyidee}, we have
\begin{equation}\label{boundthree}
\iiint_{[0,T]\times[0,X]\times[0,1]}|\p_{y}w^{\epsilon}|^{2}(1+\epsilon-y)^{-1}\log^{-3}\left(\frac{10}{1+\epsilon-y}\right)dtdxdy\lesssim_{T,X} 1.
\end{equation}

\textit{Step 4: Tangential derivative estimates.} We also need to estimate $\p_{t}w^{\epsilon}$ and $\p_{x}w^{\epsilon}$. Let $v^{\epsilon}=(w^{\epsilon}+\epsilon)^{-1}$. Then $v^{\epsilon}(t,x,y)$ satisfies the following equation for $(t,x,y)\in R_{T,X}$,
\begin{equation}
\p_{t}v^{\epsilon}+(y+\epsilon)\p_{x}v^{\epsilon}+\p_{y}^{2}w^{\epsilon}=0.
\end{equation}
Denote the tangential gradient by $\nabla_{\text{tan}}=(\p_{t},\p_{x})$. Then we have the following equation for the vector field $\nabla_{\text{tan}}v^{\epsilon}$ in $R_{T,X}$:
\begin{equation}\label{recieq}
\p_{t}\nabla_{\text{tan}}v^{\epsilon}+(y+\epsilon)\p_{x}\nabla_{\text{tan}}v^{\epsilon}-\p_{y}^{2}\left(\frac{\nabla_{\text{tan}}v^{\epsilon}}{(v^{\epsilon})^{2}}\right)=0.
\end{equation}
Denoting $\mathrm{sgn}(\nabla_{\text{tan}}v^{\epsilon}):=\left(\mathrm{sgn}(\p_{t}v^{\epsilon}),\mathrm{sgn}(\p_{x}v^{\epsilon})\right)$, testing \eqref{recieq} against 
$\mathrm{sgn}(\nabla_{\text{tan}}v^{\epsilon})\cdot(1-y),
$ and integrating with respect to $y\in [0,1]$, we obtain that
\begin{equation}\label{bvesti}
\begin{split}
&\p_{t}\int_{0}^{1}|\nabla_{\text{tan}}v^{\epsilon}|(1-y)dy+\p_{x}\int_{0}^{1}|\nabla_{\text{tan}}v^{\epsilon}|(y+\epsilon)(1-y)dy\\
&-\int_{0}^{1}\p_{y}^{2}\left(\frac{\nabla_{\text{tan}}v^{\epsilon}}{(v^{\epsilon})^{2}}\right)\mathrm{sgn}(\nabla_{\text{tan}}v^{\epsilon})(1-y)dy:=I_1+I_2+I_3=0.
\end{split}
\end{equation}
In the above, we used the notation 
$$|\nabla_{\text{tan}}h|=|\partial_th|+|\partial_xh|.$$
By direct calculations,
\begin{equation}\label{bvesti3}
\begin{split}
I_3&=-\int_{0}^{1}\p_{y}^{2}\nabla_{\text{tan}}w^{\epsilon} \cdot \mathrm{sgn}(\nabla_{\text{tan}}w^{\epsilon})(1-y)dy\\
&\geq -\int_{0}^{1}\p_{y}\nabla_{\text{tan}}w^{\epsilon}\cdot {\rm sgn}(\nabla_{\text{tan}}w^{\epsilon})dy\\
&=-\int_{0}^{1}\p_{y}|\nabla_{\text{tan}}w^{\epsilon}|dy=|\nabla_{\text{tan}}w^{\epsilon}(t,x,0)|.
\end{split}
\end{equation}
As discussed just below equation \eqref{uniqueesti}, the calculations above can be justified rigorously through replacing the ${\rm sign}$ function by $\phi(\cdot/\delta)$ where $\phi$ is smooth, odd, equal to $-1$ for $x\leq-1$ and equal to $1$ for $x\ge1$ with $\phi'\ge0$, and taking $\delta\to0+$.

Hence by \eqref{bvesti}, we get that
\begin{equation}\label{boundfour}
\begin{split}
&\sup_{t\in[0,T]}\int_{0}^{X}\int_{0}^{1}|\nabla_{\text{tan}}v^{\epsilon}|(1-y)dxdy\\
&\leq\int_{0}^{X}\int_{0}^{1}|\nabla_{\text{tan}}v^{\epsilon}|(1-y)dxdy\bigg|_{t=0}+\int_{0}^{T}\int_{0}^{1}|\nabla_{\text{tan}}v^{\epsilon}|(y+\epsilon)(1-y)dtdy\bigg|_{x=0}\lesssim_{T,X} 1.
\end{split}
\end{equation}
%We emphasize that by virtue of standard tools in real analysis, for fixed $T,X>0$ we can choose suitable mollification sequence $(w_{0}^{\epsilon},w_{1}^{\epsilon})$ (up--to a subsequence if necessary) such that the quantities in the second inequality of \eqref{boundfour} are uniformly bounded independent with $\epsilon$. 

Using the equation \eqref{approsys} and the bound \eqref{boundfour}, we conclude that
\begin{equation}\label{boundfive}
\sup_{t\in[0,T]}\int_{0}^{X}\int_{0}^{1}|\nabla_{\text{tan}}v^{\epsilon}|(1-y)dxdy+\sup_{t\in[0,T]}\int_{0}^X\int_{0}^{1}|\p_{y}^{2}w^{\epsilon}|(1-y)dxdy\lesssim_{T,X}1.
\end{equation}

\textit{Step 5: Higher regularity estimates.} In this step, we establish the uniform--in--$\epsilon$ (up--to--boundary $\{y=0\}$) smoothness of $w^{\epsilon}$. For any fixed $1>\delta>\epsilon$ and $t_{0}>\delta, x_{0}>\delta$, we first prove the uniform--in--$\epsilon$ regularity of $w^{\epsilon}$ in the ball
\begin{equation*}
\left\{(t,x,y):-\delta^{2}<t-t_{0}\leq0, |x-x_{0}|<\delta^{3}, 0\leq y<\delta\right\}.
\end{equation*}
By the translation invariance of the equation \eqref{approsys} with respect to the $(t,x)$-variables, we can assume that $t_{0}=x_{0}=0$. Define
\begin{equation}
\widetilde{w}^{\epsilon}(t,x,y)=w^{\epsilon}\left(\frac{\delta^{2}}{4}t,\frac{\delta^{3}}{8}x+\frac{\epsilon\delta^{2}}{4}t,\frac{\delta}{2}y\right).
\end{equation}
We have
\begin{equation}
\begin{cases}
\p_{t}\widetilde{w}^{\epsilon}+y\p_{x}\widetilde{w}^{\epsilon}-(\widetilde{w}^{\epsilon}+\epsilon)^{2}\p_{y}^{2}\widetilde{w}^{\epsilon}=0,\\
\p_{y}\widetilde{w}^{\epsilon}|_{y=0}=0,\\
\end{cases}
\quad \text{in}\ \left\{-1<t\leq0,|x|<1,0\leq y<1\right\}.
\end{equation}
By the same procedure as in the derivation of \eqref{pfmaingoal1}, we obtain that
\begin{equation}
\|\widetilde{w}^{\epsilon}\|_{C^{k}(-\frac{1}{4}\leq t\leq0,|x|\leq \frac{1}{8},0\leq y\leq \frac{1}{2})}\leq C(k,\tilde{m},\tilde{M}),\quad \forall k\geq 0,
\end{equation}
where
\begin{equation*}
\tilde{m}=\inf_{-1<t<0,|x|<1,0\leq y<1}\widetilde{w}^{\epsilon}(t,x,y),\quad \tilde{M}=\sup_{-1<t<0,|x|<1,0\leq y<1}\widetilde{w}^{\epsilon}(t,x,y).
\end{equation*}
Therefore,
\begin{equation}\label{wepsilonboundary}
\|w^{\epsilon}\|_{C^{k}(-\frac{\delta^{2}}{100}\leq t-t_{0}\leq 0,|x-x_{0}|\leq \frac{\delta^{3}}{100},0\leq y\leq \frac{\delta}{100})}\leq C(k,\delta,m,M),
\end{equation}
where
\begin{equation*}
m=\inf_{-\delta^{2}<t-t_{0}<0,|x-x_{0}|<\delta^{3},0\leq y<\delta}w^{\epsilon}(t,x,y),\quad M=\sup_{-\delta^{2}<t-t_{0}<0,|x-x_{0}|<\delta^{3},0\leq y<\delta}w^{\epsilon}(t,x,y).
\end{equation*}

Next, for any fixed $1>\kappa>\epsilon$ and $t_{0}>\kappa,x_{0}>\kappa,\kappa<y_{0}<1-\kappa$, we prove the uniform--in--$\epsilon$ (interior) regularity of $w^{\epsilon}$ in the ball
\begin{equation*}
\left\{(t,x,y):-\kappa^{2}<t-t_{0}\leq0,|x-x_{0}|<\kappa^{3},|y-y_{0}|<\kappa\right\}.
\end{equation*}
By translation invariance with respect to $(t,x)$--variables, we can assume $t_{0}=x_{0}=0$. Define
\begin{equation}\label{rescale1}
w_{1}^{\epsilon}(t,x,y):=w^{\epsilon}(t,x+(y_{0}+\epsilon)t,y+y_{0}).
\end{equation}
Then $w_{1}^{\epsilon}(t,x,y)$ satisfies
\begin{equation}
    \p_{t}w_{1}^{\epsilon}+y\p_{x}w_{1}^{\epsilon}-(w_{1}^{\epsilon}+\epsilon)^{2}\p_{y}^{2}w_{1}^{\epsilon}=0,\ \text{in}\ \left\{-\frac{\kappa^{2}}{2}<t\leq 0,\ |x|<\frac{\kappa^{2}}{2},\ |y|<\kappa\right\}.
\end{equation}
Next, define
\begin{equation}
    w_{2}^{\epsilon}(t,x,y):=w_{1}^{\epsilon}\left(\frac{\kappa^{2}}{4}t,\frac{\kappa^{3}}{8}x,\frac{\kappa}{2}y\right);
\end{equation}
then $w_{2}^{\epsilon}$ satisfies
\begin{equation}
    \p_{t}w_{2}^{\epsilon}+y\p_{x}w_{2}^{\epsilon}-(w_{2}^{\epsilon}+\epsilon)^{2}\p_{y}^{2}w_{2}^{\epsilon} = 0,\ \text{in}\ \left\{-1<t\leq 0, |x|<1,\ |y|<1\right\}.
\end{equation}

Applying the interior estimates from Proposition \ref{interiorholder}, Proposition \ref{interiorsobolevesti}, Proposition \ref{interiorsob1} and Proposition \ref{interiorhypo}, we get that
\begin{equation}\label{rescale2}
\|w_{2}^{\epsilon}\|_{C^{k}(-\frac{1}{4}\leq t\leq 0,|x|\leq \frac{1}{8},|y|\leq \frac{1}{2})}\leq C(k,\kappa,\bar{m},\bar{M}),
\end{equation}
where
\begin{equation*}
\bar{m}=\inf_{-1<t<0,|x|<1,|y|<1}w_{2}^{\epsilon}(t,x,y),\quad \bar{M}=\sup_{-1<t<0,|x|<1,|y|<1}w_{2}^{\epsilon}(t,x,y).
\end{equation*}
To obtain the estimate on the non-tilted cylinder in \eqref{wepsiloninter}, we repeat the preceding Galilean argument with its center at each point of that cylinder.
Therefore by \eqref{rescale1} and \eqref{rescale2},
\begin{equation}\label{wepsiloninter}
\|w^{\epsilon}\|_{C^{k}(-\frac{\kappa^{2}}{100}\leq t-t_{0}\leq0,|x-x_{0}|\leq\frac{\kappa^{3}}{100},|y-y_{0}|\leq\frac{\kappa}{100})}\leq C(k,\kappa,m,M),
\end{equation}
where
\begin{equation*}
m=\inf_{-\kappa^{2}<t-t_{0}<0,|x-x_{0}|<\kappa^{2},|y-y_{0}|<\kappa}w^{\epsilon}(t,x,y),\quad M=\sup_{-\kappa^{2}<t-t_{0}<0,|x-x_{0}|<\kappa^{2},|y-y_{0}|<\kappa}w^{\epsilon}(t,x,y).
\end{equation*}
Combining \eqref{wepsilonboundary}, \eqref{wepsiloninter}, as well as the uniform--in--$\epsilon$ upper bound \eqref{boundone}, and the lower bound \eqref{boundtwo}, we conclude that for any $\frac{1}{2}\geq\delta>100\epsilon$, and any $k\geq 0$,
\begin{equation}\label{wepsilonlocalsmooth}
\|w^{\epsilon}\|_{C^{k}(\delta\leq t\leq \delta^{-1},\delta\leq x\leq \delta^{-1},0\leq y\leq 1-\delta)}\lesssim_{k,\delta}1.
\end{equation}

\textit{Step 6: Passing to the limit.} By the smoothness estimate \eqref{wepsilonlocalsmooth}, there exist a subsequence of $w^{\epsilon}$ (denoted by $w^{\epsilon_{n}}$), and a function
\begin{equation}
w(t,x,y)\in C^{\infty}\left((0,+\infty)\times(0,+\infty)\times [0,1)\right),
\end{equation}
such that $w^{\epsilon_{n}}$ and its derivatives converge uniformly to $w$ in any compact subset of $(0,+\infty)\times(0,+\infty)\times [0,1)$. Furthermore, by the estimates \eqref{ptwe1}, \eqref{boundthree}, \eqref{boundfour}, \eqref{boundfive} and \eqref{wepsilonlocalsmooth} on $w^{\epsilon}$, we conclude that $w\in\mathcal{X}_{T,X}$, where the functional space $\mathcal{X}_{T,X}$ is defined in \eqref{weaksolspace}. By weak compactness, there also exists a subsequence of $w^{\epsilon}$ (denoted by $w^{\epsilon_n}$), such that for any $T>0,X>0$,  
\begin{equation}\label{eq6.44}
\p_{t}w^{\epsilon_{n}}, \p_{x}w^{\epsilon_{n}} \rightharpoonup \p_{t}w, \p_{x}w,
\end{equation}
in the sense of distributions in $(0,T)\times(0,X)\times (0,1)$.

Before passing to the initial and inflow traces, we record a uniform equi-integrability property near these two faces.  Set \(v^\epsilon=(w^\epsilon+\epsilon)^{-1}\) and define

$$
 A^\epsilon(t,x):=\int_0^1 |\nabla_{\tan}v^\epsilon|(1-y)\,dy,
 \qquad
 B^\epsilon(t,x):=\int_0^1 |\nabla_{\tan}v^\epsilon|(y+\epsilon)(1-y)\,dy.
$$
It follows from \eqref{bvesti}--\eqref{bvesti3} that
$$
 \partial_t A^\epsilon+\partial_x B^\epsilon\leq 0.
$$
Consequently, using the uniform bounds on the compatible initial and inflow data, we obtain
$$
 \sup_{x\in[0,X]}\int_0^T B^\epsilon(t,x)\,dt
 \leq
 \int_0^T B^\epsilon(t,0)\,dt
 +\int_0^X A^\epsilon(0,x)\,dx
 \lesssim_{T,X}1.
$$

Together with \eqref{ptwe1}, \eqref{boundfour}, and
$$
 |\nabla_{\tan}w^\epsilon|
 =(w^\epsilon+\epsilon)^2|\nabla_{\tan}v^\epsilon|,
$$
this implies that, for any compact sets
$$
 K_t\Subset(0,X)\times(0,1),
 \qquad
 K_x\Subset(0,T)\times(0,1),
$$
and any \(0<\delta<1\),
$$
 \sup_{\epsilon}\int_0^\delta\int_{K_t}
 |\partial_t w^\epsilon|\,dx\,dy\,dt
 \lesssim_{K_t,T,X}\delta,
 \qquad
 \sup_{\epsilon}\int_0^\delta\int_{K_x}
 |\partial_x w^\epsilon|\,dt\,dy\,dx
 \lesssim_{K_x,T,X}\delta.
$$

Hence no tangential derivative can concentrate on the faces \(t=0\) or \(x=0\).  By first inserting cutoffs that vanish in a \(\delta\)-neighborhood of these faces, passing to the limit \(\epsilon=\epsilon_n\to0\) using \eqref{eq6.44}, and then letting \(\delta\to0\), we may use test functions that do not vanish at \(t=0\) or \(x=0\). 

Finally, we are ready to show that $w$ satisfies \eqref{prandtlcrocco} as well as the required boundary conditions. Since $w^{\epsilon_{n}}$ and its derivatives converge uniformly to $w$ in any compact subset of $(0,+\infty)\times(0,+\infty)\times [0,1)$, using also \eqref{ptwe1}, we have
\begin{equation}
\begin{cases}
\p_{t}w+y\p_{x}w-w^{2}\p_{y}^{2}w=0,\ (t,x,y)\in(0,+\infty)\times(0,+\infty)\times (0,1),\\
w|_{y=1}=0,\quad \p_{y}w|_{y=0}=0.
\end{cases}
\end{equation}
Therefore, to verify that $w$ satisfies all boundary conditions in \eqref{prandtlcrocco}, it remains to check
\begin{equation}
w|_{t=0}=w_{0}(x,y),\ w|_{x=0}=w_{1}(t,y).
\end{equation}
Here $w|_{t=0}$ and $w|_{x=0}$ are defined as the boundary traces of $W^{1,1}$ functions.

For convenience denote $S:=(0,\infty)\times(0,\infty)\times(0,1)$ and $\p S:=(0,\infty)\times(0,1)$. For any $\phi(x,y)\in C_{c}^{\infty}\left((0,+\infty)\times(0,1)\right)$ and $\eta(t)\in C_{c}^{\infty}([0,+\infty))$ with $\eta(0)=1$, notice that
\begin{equation}
\begin{split}
&\iiint_{S} \p_{t}w(t,x,y)\phi(x,y)\eta(t)dtdxdy\\
&=\lim_{n\to\infty}\iiint_{S} \p_{t}w^{\epsilon_{n}}(t,x,y)\phi(x,y)\eta(t)dtdxdy\\
&=-\lim_{n\to\infty}\bigg(\iiint_{S} w^{\epsilon_{n}}(t,x,y)\phi(x,y)\eta'(t)dtdxdy+\iint_{\p S} w^{\epsilon_{n}}(0,x,y)\phi(x,y)dxdy\bigg)\\
&=-\iiint_{S} w(t,x,y)\phi(x,y)\eta'(t)dtdxdy-\iint_{\p S} w_{0}(x,y)\phi(x,y)dxdy\\
&=\iiint_{S} \p_{t}w(t,x,y)\phi(x,y)\eta(t)dtdxdy+\iint_{\p S} \left[w(0,x,y)-w_{0}(x,y)\right]\phi(x,y)dxdy.
\end{split}
\end{equation}
Hence $w|_{t=0}=w_{0}(x,y)$. Similarly we can check that $w|_{x=0}=w_{1}(t,y)$. Therefore, we finish the proof of Proposition \ref{existweaksol}.
\end{proof}
\begin{remark}
    Our proof in fact implies the following estimate. Assume that the inflow data $w_{1}$ satisfies the stronger uniform--in--$T$ bounds
    \begin{equation}
        \begin{split}
            &(1-y)^{3/2}\lesssim w_{1}(t,y)\lesssim (1-y)\log^{\frac{1}{2}}\left(\frac{10}{1-y}\right),\ \forall(t,y)\in[0,+\infty)\times[0,1),\\
            &\iint_{[0,+\infty)\times[0,1)}\frac{|\p_{t}w_{1}(t,y)|}{|w_{1}(t,y)|^{2}}\cdot(1-y)dtdy
            \lesssim 1.
        \end{split}
    \end{equation}
Then the solution $w(t,x,y)$ satisfies
\begin{equation}\label{twdecay0.1}
\sup_{x\in[0,X]}\iint_{[0,\infty)\times[0,1)}y\cdot\frac{|\p_{t}w(t,x,y)|}{w^{2}(t,x,y)}\cdot(1-y)dtdy\lesssim_{X}1.
\end{equation}
This global--in--time bound is significant since it indicates the decay of $\partial_tw$ and therefore the convergence of the solution to steady states. We expect \eqref{twdecay0.1} to be useful for the study of asymptotic stability for Prandtl equations.
\end{remark}

\subsection{Proof of Theorem \ref{mainthree}}\label{pgf2}
In this subsection, we prove Theorem \ref{mainthree} on the global well-posedness for weak solutions to \eqref{prandtl}.

\begin{proof}[Proof of Theorem \ref{mainthree}]\label{pfg1} By the continuity and (strict) monotonicity in $y$ of $u_{0}$ and $u_{1}$, we can define $w_{0}(\xi,\eta)$ and $w_{1}(\tau,\eta)$ using
\begin{equation}
\begin{split}
&w_{0}(\xi,\eta)=(\p_{y}u_{0})(\xi,y_{0}),\ \text{where}\ \ u_{0}(\xi,y_{0})=\eta,\\
&w_{1}(\tau,\eta)=(\p_{y}u_{1})(\tau,y_{1}),\ \text{where}\ \ u_{1}(\tau,y_{1})=\eta.
\end{split}
\end{equation}
%Under Assumption \ref{dataassumecrocco4}, one can check that $w_{0}(\xi,\eta)$ and $w_{1}(\tau,\eta)$ satisfy the assumptions in Proposition \ref{existweaksol}. Indeed, the bounds \eqref{initialrate} follow from \eqref{matchingrate4}, and the bounds \eqref{initialbv} follow from \eqref{higherassume4}, and the fact that\begin{equation}  \int_{0}^{\infty}(1-u_{i})dy=\int_{0}^{1}\frac{1-\eta}{w_{i}}d\eta.\end{equation}

By Proposition \ref{uniweaksol} and Proposition \ref{existweaksol}, for each positive integer $N$, the transformed Prandtl system admits a unique weak solution $w^{N}(\tau,\xi,\eta)$ in $[0,N]\times [0,N] \times [0,1]$ with the corresponding initial and boundary conditions. Also, $w^{N}\in \mathcal{X}_{N,N}$ for all $N\in\Z\cap[1,\infty)$ (recall the definition of space $\mathcal{X}_{T,X}$ in \eqref{weaksolspace}). Moreover, we have the up--to--boundary regularity
$$w^{N}\in C^{\infty}\left((0,N]\times(0,N]\times[0,1)\right).$$
It follows from Proposition \ref{uniweaksol} that for any $M\geq N \geq 1$, $w^{M}\equiv w^{N}$ in $[0,N]\times[0,N]\times[0,1]$.
Then for any $(t,x,y)\in[0,+\infty)\times[0,+\infty)\times[0,1]$, we can define
\begin{equation}
    w(t,x,y):=\lim_{N\to\infty}w^{N}(t,x,y).
\end{equation}
Clearly, $w\equiv w^{N}$ in $[0,N]\times [0,N]\times[0,1]$. Furthermore, one can check that
\begin{equation}
    \begin{cases}
        w\in\mathcal{X}_{T,X},\ \forall\, T,X>0,\
        w\in C^{\infty}\left((0,\infty)\times(0,\infty)\times[0,1)\right),\\
        w|_{t=0}=w_{0},\ w|_{x=0}=w_{1},\ \p_{y}w|_{y=0}=0,
    \end{cases}
\end{equation}
and that $w(t,x,y)$ satisfies the transformed Prandtl equation \eqref{prandtlcrocco} in $(0,\infty)\times(0,\infty)\times[0,1)$.

To get back to the physical variables, we define the function $u(t,x,y)\ge0$ through the equation
\begin{equation}\label{invercroccoweak}
        y=\displaystyle \int_{0}^{u(t,x,y)}\frac{d\eta}{w(t,x,\eta)}.\\
\end{equation}
Clearly, $u(t,x,y)$ is strictly monotone in the $y$-direction, with
$$u|_{y=0}\equiv 0,\quad \lim_{y\to\infty}u(t,x,y)\equiv 1$$
in view of the pointwise bound $(1-\eta)^{3/2}\lesssim w\lesssim (1-\eta)\log^{\frac{1}{2}}\left(\frac{10}{1-\eta}\right)$. We note that since $w$ is smooth in $(0,\infty)\times(0,\infty)\times[0,1)$, we can differentiate \eqref{invercroccoweak} up to any order. By direct calculations, we get that
\begin{equation}
    \begin{cases}
        \p_{y}u(t,x,y)=w(t,x,u(t,x,y)),\\
        (\p_{t},\p_{x})u(t,x,y)=w(t,x,u(t,x,y))\displaystyle \int_{0}^{u(t,x,y)}\frac{(\p_{t},\p_{x})w}{w^{2}}d\eta. 
    \end{cases}
\end{equation}
Hence, for any $T,X>0$, $u(t,x,y)$ satisfies \eqref{solutionmatching2}. The global $W^{1,1}$ bounds in \eqref{globalbound} follow from 
\begin{equation}
\int_{0}^{\infty}|(\partial_{t},\partial_{x}) u|dy \leq \int_{0}^{1}\frac{|(\p_{t},\p_{x})w|}{w^{2}}\cdot(1-\eta)d\eta
\end{equation}
 and a change of variable argument in the integral.
The second and third estimates in \eqref{globalbound} follow from the bounds \eqref{globalboundcrocco} on $w$ and the Crocco change of variables. The fourth estimate in \eqref{globalbound} follows from integrating the weighted entropy identity
\begin{equation}\label{entropy1}
    \partial_{\eta\eta}(w\log w)=(1+\log w) w_{\eta\eta}+\frac{|w_{\eta}|^2}{w}
\end{equation}
against $(1-\eta)(\log\frac{10}{1-\eta})^{-2}$, and using 
the pointwise matching bounds and the estimate $(1-\eta)w_{\eta\eta}\in L^\infty_\tau L^1_{\xi\eta}$. The fifth estimate then follows from the identity relating $w_{\eta\eta}$ and $u_{yyy}$. 

Furthermore, $u(t,x,y)\in C^{\infty}\left((0,\infty)\times(0,\infty)\times[0,\infty)\right)$ follows from \eqref{invercroccoweak} and the fact that
\begin{equation*}
    w\in C^{\infty}\left((0,\infty)\times(0,\infty)\times[0,1)\right).
\end{equation*}
 By \eqref{invercroccoweak}, we can check that $u|_{y=0} \equiv 0$, as well as
\begin{equation}\label{physicaltrace}
    u|_{t=0} = u_{0},\qquad u|_{x=0} = u_{1}.
\end{equation}
\begin{comment}
We briefly explain \eqref{physicaltrace} for completeness. By the elementary properties on the boundary trace of $W^{1,1}$ functions, We can choosing a sequence $t_{n}\to 0+$, such that
\begin{equation}\label{convergetotrace}
\begin{split}
    &u(t_{n},x,y)\rightarrow (\gamma_{0} u)(x,y),\ \text{almost\ everywhere}\ \&\ L^{1}_{loc},\\
    &w(t_{n},x,\eta)\rightarrow w_{0}(x,\eta),\ \text{almost\ everywhere}\ \&\ L^{1}_{loc}.
    \end{split}
\end{equation}
By virtue of \eqref{convergetotrace} and $(1-\eta)^{m}\lesssim w\lesssim (1-\eta)\log^{\frac{1}{2}}\left(\frac{10}{1-\eta}\right)$, one can check that
\begin{equation}
    \int_{0}^{u(t_{n},x,y)}\frac{d\eta}{w(t_{n},x,\eta)}\rightarrow \int_{0}^{(\gamma_{0} u)(x,y)}\frac{d\eta}{w_{0}(x,\eta)},\ \text{almost\ everywhere}\ \&\ L^{1}_{loc}.
\end{equation}
Hence, 
\begin{equation}
y=\int_{0}^{(\gamma_{0} u)(x,y)}\frac{d \eta}{w_{0}(x,\eta)}.
\end{equation}
It follows that $(\gamma_{0} u)=u_{0}(x,y)$. Also, similar argument implies $(\gamma_{1}u)(t,y)=u_{1}(t,y)$, hence the validity of \eqref{physicaltrace}.\end{comment}

Defining $v$ using the second equation in \eqref{prandtl}, we can then check that $(u,v)$ satisfies the dynamical Prandtl equation \eqref{prandtl}.%We refer to \cite{OS99} or \cite{XZ04} for more detailed computation. 

\begin{comment}where the divergence--free counterpart is constructed by 
\begin{equation}
\begin{split}
    v(t,x,y)& = \p_{\eta}w - \int_{0}^{u(t,x,y)}\frac{\p_{\tau}w+u(t,x,y)\p_{\xi}w}{w^{2}(t,x,\eta)}d\eta\\
    & =   -\frac{-\p_{t}u-u\p_{x}u+\p_{y}^{2}u}{\p_{y}u}.\\
\end{split}
\end{equation}\end{comment}

Finally, we prove the uniqueness statement in Theorem \ref{mainthree}. Indeed, assume that two weak solutions $u_{1}(t,x,y)$ and $u_{2}(t,x,y)$ constructed in Theorem \ref{mainthree} satisfy the Prandtl equation \eqref{prandtl} with the same initial and inflow data. One can check that after the  Crocco transform, $w_{1}(\tau,\xi,\eta)$ and $w_{2}(\tau,\xi,\eta)$ satisfy the assumptions in Proposition \ref{uniweaksol}; hence we have $w_{1}(\tau,\xi,\eta)\equiv w_{2}(\tau,\xi,\eta)$. Notice that
\begin{equation*}
    y\equiv \int_{0}^{u_{i}(t,x,y)}\frac{d\eta}{w_{i}(t,x,\eta)},\ i\in\{1,2\},
\end{equation*}
hence $u_{1}(t,x,y)\equiv u_{2}(t,x,y)$ since $w_{1}\equiv w_{2}$. 

Therefore, we complete the proof of Theorem \ref{mainthree}.
\end{proof}

%section 8
\section{Local existence of classical solutions and proof of Theorem \ref{maintwo}}\label{local}
In this section, we study the local classical solutions to the Prandtl equation \eqref{prandtl} and prove well-posedness results that generalize the pioneering work of Oleinik \cite{OS99} to a larger class of initial and boundary conditions.

Recall that under the Crocco transform, the Prandtl equation becomes
\begin{equation}\label{prandtlsection8}
\begin{cases}
\p_{\tau}w+\eta\p_{\xi}w-w^{2}\p_{\eta}^{2}w=0,\ (\tau,\xi,\eta)\in (0,T]\times(0,X]\times(0,1),\\
w|_{\tau=0}=w_{0}(\xi,\eta),\ w|_{\xi=0}=w_{1}(\tau,\eta),\\
w|_{\eta=1}=\p_{\eta}w|_{\eta=0}=0.
\end{cases}
\end{equation}
In this section, we first present various useful a priori estimates for system \eqref{prandtlsection8} in subsection \ref{priorisection8}; in subsection \ref{localmain}, we then establish the local-in-$\tau$ and local-in-$\xi$ classical solution theory for \eqref{prandtlsection8}, as well as the corresponding local classical theory for the original Prandtl equation \eqref{prandtl}.
\subsection{A priori estimates for \eqref{prandtlsection8}}\label{priorisection8}
Throughout this subsection, we assume
\begin{equation}\label{ptwiseinitial}
\begin{split}
(1-\eta)^{3/2}\lesssim_{X}w_{0}(\xi,\eta)\lesssim_{X}(1-\eta)\log^{\frac{1}{2}}\left(\frac{10}{1-\eta}\right),\ \forall(\xi,\eta)\in[0,X]\times[0,1),\\
(1-\eta)^{3/2}\lesssim_{T}w_{1}(\tau,\eta)\lesssim_{T}(1-\eta)\log^{\frac{1}{2}}\left(\frac{10}{1-\eta}\right),\ \forall(\tau,\eta)\in[0,T]\times[0,1).\\
\end{split}
\end{equation}
First, we prove the pointwise bounds for the solutions of \eqref{prandtlsection8}. 
\begin{lemma}[\bf Pointwise Bounds]\label{pointwisecontrol}
Assume that $w$ is the classical solution to \eqref{prandtlsection8} with initial and boundary data satisfying \eqref{ptwiseinitial}. Then
\begin{equation}\label{potwb2}
(1-\eta)^{3/2}\lesssim_{T} w(\tau,\xi,\eta)\lesssim_{T}(1-\eta)\log^{\frac{1}{2}}\left(\frac{10}{1-\eta}\right),\ \forall (\tau,\xi,\eta)\in[0,T]\times[0,X]\times[0,1),
\end{equation}
provided $X\leq 1$; and also,
\begin{equation}\label{potwb3}
(1-\eta)^{3/2}\lesssim_{X} w(\tau,\xi,\eta)\lesssim_{X}(1-\eta)\log^{\frac{1}{2}}\left(\frac{10}{1-\eta}\right),\ \forall (\tau,\xi,\eta)\in[0,T]\times[0,X]\times[0,1),
\end{equation}
provided $T\leq 1$.
\end{lemma}
\begin{proof}The proof is based on the maximum principle, and follows the same argument as that in the second step of the proof of Proposition \ref{existweaksol}.
\end{proof}

Next, we derive a priori bounds for the higher order derivatives using energy methods.
\begin{lemma}[\bf Energy Estimates]\label{highenergyesti}
Assume that $w$ is the classical solution to \eqref{prandtlsection8}. Then we have
\begin{equation}\label{highenergyid}
\begin{split}
&\sup_{\tau\in[0,T]}\sum_{1\leq\alpha+\beta\leq 2}\int_{0}^{X}\int_{0}^{1}\bigg|\frac{\p_{\tau}^{\alpha}\p_{\xi}^{\beta}w}{w}\bigg|^{2}d\xi d\eta+\sup_{\xi\in[0,X]}\sum_{1\leq\alpha+\beta\leq 2}\int_{0}^{T}\int_{0}^{1}\eta\bigg|\frac{\p_{\tau}^{\alpha}\p_{\xi}^{\beta}w}{w}\bigg|^{2}d\tau d\eta\\
&\quad\quad+\sum_{1\leq\alpha+\beta\leq 2}\iiint_{[0,T]\times[0,X]\times[0,1]}|\p_{\tau}^{\alpha}\p_{\xi}^{\beta}\p_{\eta}w|^{2}d\tau d\eta d\xi\\
&\quad\quad\lesssim\mathcal{N}_{0}+(\mathcal{N}_{1}+ \mathcal{N}_{1}^2)\sum_{1\leq \alpha+\beta\leq 2}\iiint_{[0,T]\times[0,X]\times[0,1]}\bigg|\frac{\p_{\tau}^{\alpha}\p_{\xi}^{\beta}w}{w}\bigg|^{2}d\tau d\eta d\xi.
\end{split}
\end{equation}
In the above $\mathcal{N}_{0}$, $\mathcal{N}_{1}$ denote the quantities
\begin{equation}\label{n0k}
\mathcal{N}_{0}:=\sum_{1\leq\alpha+\beta\leq 2}\int_{0}^{X}\int_{0}^{1}\bigg|\frac{\p_{\tau}^{\alpha}\p_{\xi}^{\beta}w(0,\xi,\eta)}{w_{0}(\xi,\eta)}\bigg|^{2}d\xi d\eta+\sum_{1\leq\alpha+\beta\leq 2}\int_{0}^{T}\int_{0}^{1}\eta\bigg|\frac{\p_{\tau}^{\alpha}\p_{\xi}^{\beta}w(\tau,0,\eta)}{w_{1}(\tau,\eta)}\bigg|^{2}d\tau d\eta,
\end{equation}
as well as 
\begin{equation}\label{n1}
\mathcal{N}_{1}:=\sup_{(\tau,\xi,\eta)\in[0,T]\times[0,X]\times[0,1)}\left(\bigg|\frac{\p_{\tau}w}{w}\bigg|+\bigg|\frac{\p_{\xi}w}{w}\bigg|\right).
\end{equation}
\end{lemma}
\begin{proof} 
%For clarity of presentation, we provide the full details for $k=2$. The computations for $k>2$ are completely analogous although more complicated.

\textit{Step 1: $H^{1}$ Estimates.} Differentiate \eqref{prandtlsection8} with respect to $\tau$. Then $w_{\tau}$ satisfies the following equation
\begin{equation}\label{wteq}
\begin{cases}
\p_{\tau}w_{\tau}+\eta\p_{\xi}w_{\tau}-2w\p_{\eta}^{2}ww_{\tau}-w^{2}\p_{\eta}^{2}w_{\tau}=0,\\
w_{\tau}|_{\eta=1} = \p_{\eta}w_{\tau}|_{\eta=0}=0.
\end{cases}
\end{equation}
Testing \eqref{wteq} by $w_{\tau}w^{-2}$ and integrating in $[0,T]\times [0,X]\times [0,1]$, we obtain the following energy inequality:
{\small\begin{align}\label{h1energyid}
&\sup_{\tau\in[0,T]}\iint_{[0,X]\times [0,1]}\frac{|w_\tau|^{2}}{w^{2}}d\xi d\eta+\sup_{\xi\in[0,X]}\iint_{[0,T]\times [0,1]}\frac{\eta|w_{\tau}|^{2}}{w^{2}}d\tau d\eta+\iiint_{[0,T]\times [0,X]\times [0,1]}|\p_{\eta}w_{\tau}|^{2}d\tau d\xi d\eta\nonumber\\
&\lesssim \iint_{[0,X]\times [0,1]}\frac{|w_\tau(0,\xi,\eta)|^{2}}{w^{2}(0,\xi,\eta)}d\xi d\eta+\iint_{[0,T]\times [0,1]}\frac{\eta|w_{\tau}(\tau,0,\eta)|^{2}}{w^{2}(\tau,0,\eta)}d\tau d\eta\nonumber\\
&\quad+\iiint_{[0,T]\times[0,X]\times[0,1]}\frac{|w_{\tau}|^{2}}{w^{2}}\frac{|w_{\tau}+\eta w_{\xi}|}{w}d\tau d\xi d\eta+\iiint_{[0,T]\times[0,X]\times[0,1]}\frac{|w_{\tau}|^{2}}{w^{2}}|w\p_{\eta}^{2}w|d\tau d\xi d\eta\nonumber\\
&\lesssim  \iint_{[0,X]\times [0,1]}\left|\frac{w_\tau(0,\xi,\eta)}{w_{0}(\xi,\eta)}\right|^{2}d\xi d\eta+\iint_{[0,T]\times [0,1]}\eta\cdot\left|\frac{w_{\tau}(\tau,0,\eta)}{w_{1}(\tau,\eta)}\right|^{2}d\tau d\eta\nonumber\\
&\quad+\left(\left\|\frac{w_{\tau}+\eta w_{\xi}}{w}\right\|_{L^{\infty}_{\tau,\xi,\eta}}\right)\times\iiint_{[0,T]\times[0,X]\times[0,1]}\frac{|w_{\tau}|^{2}}{w^{2}}d\tau d\xi d\eta\nonumber\\
&\lesssim \mathcal{N}_{0}+\mathcal{N}_{1}\iiint_{[0,T]\times[0,X]\times[0,1]}\frac{|w_{\tau}|^{2}}{w^{2}}d\tau d\xi d\eta
\end{align}}
Similarly, differentiating \eqref{prandtlsection8} with respect to $\xi$ and testing against $w_{\xi}w^{-2}$, we can also obtain
\begin{equation}
\begin{split}
&\sup_{\tau\in[0,T]}\iint_{[0,X]\times [0,1]}\frac{|w_\xi|^{2}}{w^{2}}d\xi d\eta+\sup_{\xi\in[0,X]}\iint_{[0,T]\times [0,1]}\frac{\eta|w_{\xi}|^{2}}{w^{2}}d\tau d\eta\\
&+\iiint_{[0,T]\times [0,X]\times [0,1]}|\p_{\eta}w_{\xi}|^{2}d\tau d\xi d\eta\lesssim \mathcal{N}_{0}+\mathcal{N}_{1}\iiint_{[0,T]\times[0,X]\times[0,1]}\frac{|w_{\xi}|^{2}}{w^{2}}d\tau d\xi d\eta.
\end{split}
\end{equation}

\textit{Step 2: $H^{2}$ Estimates.} Next we estimate $w_{\tau\tau},w_{\tau\xi}$ and $w_{\xi\xi}$. We present the calculations for bounding $w_{\tau\tau}$ only as the other cases are similar. To this end, we differentiate \eqref{wteq} with respect to $\tau$; then $w_{\tau\tau}$ satisfies the following equation
\begin{equation}\label{wtteq}
\begin{cases}
\p_{\tau}w_{\tau\tau}+\eta\p_{\xi}w_{\tau\tau}-2w\p_{\eta}^{2}w\cdot w_{\tau\tau}-2\p_{\eta}^{2}w(w_{\tau})^{2}-4ww_{\tau}\p_{\eta}^{2}w_{\tau}-w^{2}\p_{\eta}^{2}w_{\tau\tau}=0,\\
w_{\tau\tau}|_{\eta=1}=\p_{\eta}w_{\tau\tau}|_{\eta=0}=0.
\end{cases}
\end{equation}
Testing \eqref{wtteq} by $w_{\tau\tau}w^{-2}$ and then integrating in $[0,T]\times[0,X]\times[0,1]$, we get the following energy inequality
\begin{equation}\label{h2energyid}
\begin{split}
&\sup_{\tau\in[0,T]}\iint_{[0,X]\times [0,1]}\frac{|w_{\tau\tau}|^{2}}{w^{2}}d\xi d\eta+\sup_{\xi\in[0,X]}\iint_{[0,T]\times [0,1]}\frac{\eta|w_{\tau\tau}|^{2}}{w^{2}}d\tau d\eta\\
&\quad\quad+\iiint_{[0,T]\times [0,X]\times [0,1]}|\p_{\eta}w_{\tau\tau}|^{2}d\tau d\xi d\eta\lesssim\mathcal{N}_{0}+\mathcal{I}_{1}+\mathcal{I}_{2}+\mathcal{I}_{3},
\end{split}
\end{equation}
where
\begin{equation*}
        \mathcal{I}_{1}:= \iiint_{[0,T]\times[0,X]\times[0,1]}\frac{|w_{\tau\tau}|^{2}\cdot|w_{\eta\eta}|}{w}d\tau d\xi d\eta,
\end{equation*}
\begin{equation*}\mathcal{I}_{2}:=\iiint_{[0,T]\times[0,X]\times[0,1]}\frac{|w_{\tau}|^{2}\cdot|w_{\eta\eta}w_{\tau\tau}|}{w^{2}}d\tau d\xi d\eta,
\end{equation*}
and
\begin{equation*}\mathcal{I}_{3}:=\iiint_{[0,T]\times[0,X]\times[0,1]}\frac{|w_{\tau}w_{\tau\tau}\p_{\eta}^{2}w_{\tau}|}{w}d\tau d\xi d\eta.
\end{equation*}
\textit{Estimate of $\mathcal{I}_{1}$.} By \eqref{prandtlsection8},
\begin{equation}
\mathcal{I}_{1}\lesssim \mathcal{N}_{1}\iiint_{[0,T]\times[0,X]\times[0,1]}\frac{|w_{\tau\tau}|^{2}}{w^{2}}d\tau d\xi d\eta.
\end{equation}
\textit{Estimate of $\mathcal{I}_{2}$.} By the Cauchy--Schwarz inequality, we have
\begin{equation}
\begin{split}
\mathcal{I}_{2}&\lesssim \left\|\frac{w_{\tau}}{w}\right\|_{L^{\infty}_{\tau,\xi,\eta}}\times\left\|w\p_{\eta}^{2}w\right\|_{L^{\infty}_{\tau,\xi,\eta}}\times\left\|\frac{w_{\tau}}{w}\right\|_{L^{2}_{\tau,\xi,\eta}}\times \left\|\frac{w_{\tau\tau}}{w}\right\|_{L^{2}_{\tau,\xi,\eta}}\\
&\lesssim (\mathcal{N}_{1})^{2} \left\|\frac{w_{\tau}}{w}\right\|_{L^{2}_{\tau,\xi,\eta}}\times \left\|\frac{w_{\tau\tau}}{w}\right\|_{L^{2}_{\tau,\xi,\eta}}.
\end{split}
\end{equation}
\textit{Estimate of $\mathcal{I}_{3}$.} The estimate of $\mathcal{I}_{3}$ is a bit complicated since more derivatives are involved. Thanks to \eqref{prandtlsection8} and \eqref{wteq}, we have
\begin{equation}
\begin{split}
\p_{\eta}^{2}w_{\tau}&=\frac{w_{\tau\tau}+\eta w_{\xi\tau}-2w\p_{\eta}^{2}w\cdot w_{\tau}}{w^{2}}=\frac{w_{\tau\tau}+\eta w_{\xi\tau}}{w^{2}}-2\frac{w_{\tau}^{2}+\eta w_{\tau}w_{\xi}}{w^{3}}.
\end{split}
\end{equation}
Hence,
\begin{equation}
\begin{split}
\mathcal{I}_{3}&\lesssim\mathcal{N}_{1}\left(\iiint_{[0,T]\times[0,X]\times[0,1]}\frac{|w_{\tau\tau}|^{2}}{w^{2}}d\tau d\xi d\eta+\iiint_{[0,T]\times[0,X]\times[0,1]}\frac{|w_{\xi\tau}|^{2}}{w^{2}}d\tau d\xi d\eta\right)\\
&\quad\quad+(\mathcal{N}_{1})^{2}\left\|\frac{w_{\tau}+\eta w_{\xi}}{w}\right\|_{L^{2}_{\tau,\xi,\eta}}\times \left\|\frac{w_{\tau\tau}}{w}\right\|_{L^{2}_{\tau,\xi,\eta}}.
\end{split}
\end{equation}
Therefore, the conclusion of Lemma \ref{highenergyesti} follows from \eqref{h1energyid}, \eqref{h2energyid}, and bounds on $\mathcal{I}_1, \mathcal{I}_2$, $ \mathcal{I}_3$.
\end{proof}
\begin{remark}By Lemma \ref{highenergyesti}, we can see that in order to control the higher--order energy bounds, we need the pointwise bounds of $\p_{\tau}w$ and $\p_{\xi}w$. However, a priori we only have the pointwise bounds for $w$ itself (see Lemma \ref{pointwisecontrol}). Therefore, the higher--order energy estimates cannot be closed for {\it arbitrarily large} $T>0$ and $X>0$, and we need either $T$ or $X$ to be small to bound the higher order energy.
\end{remark}

Set
\begin{equation}\label{nbdry}
    \mathcal{N}_{\text{bdry}}:=\sup_{\tau\in [0,T],\eta\in[0,1)}\frac{|(\p_{\tau},\p_{\xi})w(\tau,0,\eta)|}{w_{1}(\tau,\eta)}+\sup_{\xi\in[0,X],\eta\in[0,1)}
   \frac{|(\p_{\tau},\p_{\xi})w(0,\xi,\eta)|}{w_0(\xi,\eta)}.
\end{equation}

The next bootstrap lemma is crucial for closing the energy estimates in Lemma \ref{highenergyesti} when either $T$ or $X$ is sufficiently small.
\begin{lemma}[\bf Bootstrap Lemma]\label{bootstrap}
For any $T>0$, there exist $X>0$ and $M>0$ (depending on the initial data $w_{0}$ and the inflow data $w_{1}$, as well as $T$),  such that under the bootstrap assumption
\begin{equation}\label{bootstrap0.1}
\mathcal{N}_{1}:=\sup_{\tau\in[0,T],\xi\in[0,X],\eta\in[0,1)}\left(\left|\frac{w_{\tau}}{w}\right|+\left|\frac{w_{\xi}}{w}\right|\right)\leq 10M,\\
\end{equation}
we have the improved bound
\begin{equation}\label{finerbd}
\mathcal{N}_{1}\leq 5M.\\
\end{equation}
Furthermore, we have
\begin{equation}\label{higherenergyclose}
\begin{split}
&\sup_{\tau\in[0,T]}\sum_{1\leq\alpha+\beta\leq 2}\int_{0}^{X}\int_{0}^{1}\bigg|\frac{\p_{\tau}^{\alpha}\p_{\xi}^{\beta}w}{w}\bigg|^{2}d\xi d\eta+\sup_{\xi\in[0,X]}\sum_{1\leq\alpha+\beta\leq 2}\int_{0}^{T}\int_{0}^{1}\eta\bigg|\frac{\p_{\tau}^{\alpha}\p_{\xi}^{\beta}w}{w}\bigg|^{2}d\tau d\eta\\
&\quad\quad\quad+\sum_{1\leq\alpha+\beta\leq 2}\int_{0}^{T}\int_{0}^{X}\int_{0}^{1}|\p_{\tau}^{\alpha}\p_{\xi}^{\beta}\p_{\eta}w|^{2}d\tau d\eta d\xi\lesssim \mathcal{N}_{0}.
\end{split}
\end{equation}
In the above, we used the definition \eqref{n0k} for the quantity $\mathcal{N}_{0}$.

Similarly, for any $X>0$ there exists $T>0$ (depending on $X$, the initial data $w_{0}$ and the inflow data $w_{1}$) such that \eqref{higherenergyclose} holds.
\end{lemma}

\begin{proof} 
We only give a detailed proof in the case of arbitrarily large $T>0$, and the proof of the case when $X>0$ is arbitrary follows from a similar (indeed, simpler) argument. Since we only need to obtain a local solution in $\xi$, we assume $X\leq 1$ in the following calculations. 

Since $X\leq1$, let $\overline N_0$ and
$\overline N_{\rm bdry}$ denote the quantities in \eqref{n0k} and
\eqref{nbdry}, respectively, with every initial-data interval $[0,X]$
replaced by $[0,1]$. Thus
\[ N_0\leq\overline N_0,
 \qquad
 N_{\rm bdry}\leq\overline N_{\rm bdry},\]
and these majorants depend only on the prescribed data, independently
of the value of $X$ chosen below. Set
\begin{equation}\label{Q_0}
  \mathcal Q_0:=
 \overline{\mathcal N}_0^{1/2}
 +\overline{\mathcal N}_0^{5/2}.
\end{equation}

The proof is divided into several steps. 

\textit{Step 1. Energy bounds.} By the bootstrap assumption \eqref{bootstrap0.1} and Lemma \ref{highenergyesti}, we have
\begin{equation}\label{2thenergy}
\begin{split}
\|w\|_{E_{2}}^2:=&\sup_{\tau\in[0,T]}\sum_{1\leq\alpha+\beta\leq 2}\int_{0}^{X}\int_{0}^{1}\bigg|\frac{\p_{\tau}^{\alpha}\p_{\xi}^{\beta}w}{w}\bigg|^{2}d\xi d\eta+\sup_{\xi\in[0,X]}\sum_{1\leq\alpha+\beta\leq 2}\int_{0}^{T}\int_{0}^{1}\eta\bigg|\frac{\p_{\tau}^{\alpha}\p_{\xi}^{\beta}w}{w}\bigg|^{2}d\tau d\eta\\
&+\sum_{1\leq \alpha+\beta\leq 2}\int_{0}^{T}\int_{0}^{X}\int_{0}^{1}\big|\p_{\tau}^{\alpha}\p_{\xi}^{\beta}\p_{\eta}w\big|^{2}d\tau d\eta d\xi\\
\lesssim& \mathcal{N}_{0}+M^{2}\sum_{1\leq\alpha+\beta\leq 2}\int_{0}^{T}\int_{0}^{X}\int_{0}^{1}\bigg|\frac{\p_{\tau}^{\alpha}\p_{\xi}^{\beta}w}{w}\bigg|^{2}d\tau d\xi d\eta.
\end{split}
\end{equation}
By Lemma \ref{pointwisecontrol}, for any $0<\iota\ll 1$,
\begin{equation}
\begin{split}
&\int_{0}^{T}\int_{0}^{X}\int_{0}^{1}\bigg|\frac{\p_{\tau}^{\alpha}\p_{\xi}^{\beta}w}{w}\bigg|^{2}d\tau d\xi d\eta\\
&=\int_{0}^{T}\int_{0}^{X}\int_{0}^{\iota}\bigg|\frac{\p_{\tau}^{\alpha}\p_{\xi}^{\beta}w}{w}\bigg|^{2}d\tau d\xi d\eta+\int_{0}^{T}\int_{0}^{X}\int_{\iota}^{1}\bigg|\frac{\p_{\tau}^{\alpha}\p_{\xi}^{\beta}w}{w}\bigg|^{2}d\tau d\xi d\eta\\
&\lesssim_{T} \iota\int_{0}^{T}\int_{0}^{X}\|\p_{\tau}^{\alpha}\p_{\xi}^{\beta}w\|_{L^{\infty}_{\eta}[0,1]}^{2}d\tau d\xi+\frac{X}{\iota}\sup_{\xi\in[0,X]}\int_{0}^{T}\int_{0}^{1}\eta\bigg|\frac{\p_{\tau}^{\alpha}\p_{\xi}^{\beta}w}{w}\bigg|^{2}d\tau d\eta\\
&\lesssim_{T} \iota\int_{0}^{T}\int_{0}^{X}\int_{0}^{1}|\p_{\tau}^{\alpha}\p_{\xi}^{\beta}\p_{\eta}w|^{2}d\tau d\xi d\eta+\frac{X}{\iota}\sup_{\xi\in[0,X]}\int_{0}^{T}\int_{0}^{1}\eta\bigg|\frac{\p_{\tau}^{\alpha}\p_{\xi}^{\beta}w}{w}\bigg|^{2}d\tau d\eta.
\end{split}
\end{equation}
Therefore, choosing $\iota=\sqrt{X}$ and setting
\begin{equation}\label{requireone}
X\ll_{T} M^{-4},
\end{equation}
the second term in the RHS of \eqref{2thenergy} can then be absorbed by the LHS of \eqref{2thenergy}. As a consequence,
\begin{equation}\label{2thenergyineq}
\|w\|_{E_{2}}^2\lesssim \mathcal{N}_{0},
\end{equation}
provided that \eqref{requireone} holds.

\textit{Step 2. The maximum principle.} For $f\in\{w_{\tau},w_{\xi}\}$, one can check that $f$ satisfies the following equation
\begin{equation}\label{fequation}
\begin{cases}
\p_{\tau}f+\eta\p_{\xi}f-w^{2}\p_{\eta}^{2}f-2w\p_{\eta}^{2}w\cdot f=0,\\
f|_{\eta=1}=\p_{\eta}f|_{\eta=0}=0.
\end{cases}
\end{equation}
Define the function $F$ and the differential operator $\mathcal{L}$ as
\begin{equation*}
F(\tau,\xi,\eta):=f(\tau,\xi,\eta)-Me^{100M\xi}w,\qquad \mathcal{L}:=\p_{\tau}+\eta\p_{\xi}-w^{2}\p_{\eta}^{2}-2w\p_{\eta}^{2}w.
\end{equation*}
By direct calculations we have
\begin{equation*}
\mathcal{L}[F]=-Me^{100M\xi}w\left(100\eta M-2w\p_{\eta}^{2}w\right).
\end{equation*}
Therefore it follows from \eqref{prandtlsection8} and \eqref{bootstrap0.1} that
\begin{equation}
\mathcal{L}[F](\tau,\xi,\eta)<0,\quad \forall \tau\in[0,T], \xi\in[0,X], \eta\in[1/2,1).
\end{equation}

By \eqref{2thenergyineq}, Lemma \ref{pointwisecontrol} and Lemma \ref{anisotwo},
\[
 \sup_{\substack{\xi\in[0,X],\,\tau\in[0,T]\\
                  \eta\in[1/4,3/4]}}
 \bigl(|w_\tau|+|w_\xi|\bigr)
 \lesssim_T\mathcal Q_0.
\]
Therefore,
\[
 F|_{\eta=1/2}<0,\qquad
 F|_{\tau=0}<0,\qquad
 F|_{\xi=0}<0,
\]
provided that
\begin{equation} \label{requiretwo}
 M\gg_T\mathcal Q_0+\overline N_{\rm bdry}.
\end{equation}

Therefore by the maximum principle (recall $w\p_{\eta}^{2}w$ is bounded under the bootstrap assumption), we can conclude that 
\begin{equation}\label{oneside0.1}
F\leq 0,\ \forall (\tau,\xi,\eta)\in[0,T]\times[0,X]\times[1/2,1].
\end{equation}

By applying the same argument for $\widetilde{F}=f+Me^{100M\xi}w$, we obtain that $\widetilde{F}\ge0$ on $[0,T]\times[0,X]\times[1/2,1]$. Combining this with \eqref{oneside0.1}, we conclude under the condition \eqref{requiretwo},
\begin{equation}\label{farfieldbd}
|w_{\tau},w_{\xi}|\leq M e^{100M\xi}w,\quad\forall (\tau,\xi,\eta)\in[0,T]\times[0,X]\times[1/2,1].
\end{equation}

Thus the improved bound \eqref{finerbd} holds on $[0,T]\times[0,X]\times[1/2,1]$ provided
\begin{equation}\label{requirefour}
X\leq\frac{1}{200M}.
\end{equation}

\textit{Step 3. Closing the bootstrap.} To fully close the bootstrap, it remains to derive the improved bound \eqref{finerbd} in $[0,T]\times[0,X]\times[0,\frac{1}{2}]$.

By \eqref{2thenergyineq} and Lemma \ref{anisotwo}, for every $1<p<\infty$,
\begin{equation} \label{8.31}
 \sup_{\xi\in[0,X],\,\tau\in[0,T]}
 \left(
 \|w_\tau\|_{L^p(0,1/2)}
 +\|w_\xi\|_{L^p(0,1/2)}
 \right)
 \lesssim_{p,T}\mathcal Q_0.
\end{equation}
Let $\alpha>0$ be as in Lemma \ref{aniembed} and set
\[
 \theta:=\frac{1}{1+\alpha p}.
\]
Choose $p$ sufficiently large that $2\theta\leq1/10$. For
$f\in\{w_\tau,w_\xi\}$, interpolation, Lemma \ref{aniembed}, and \eqref{2thenergyineq}
give
\[
\begin{aligned}
 \|f\|_{L^\infty(0,1/2)}
 &\lesssim
 \|f\|_{L^p(0,1/2)}^{1-\theta}
 \|f\|_{C^\alpha(0,1/2)}^\theta                                      \\
 &\lesssim_{p,T}
 \mathcal Q_0^{1-\theta}
 \left(X^{-2}\overline N_0^{1/2}\right)^\theta                         \\
 &\lesssim_T X^{-1/10}\mathcal Q_0.
\end{aligned}
\]
Consequently,
\[
 \sup_{\substack{\tau\in[0,T],\,\xi\in[0,X]\\
                  \eta\in[0,1/2]}}
 \frac{|w_\tau|+|w_\xi|}{w}
 \lesssim_T X^{-1/10}\mathcal Q_0.
 %\tag{8.33}
\]
Thus the improved bootstrap bound in this region holds provided
that
\begin{equation}\label{requirefive}
 X^{-1/10}\mathcal Q_0\ll_T M.
% \tag{8.34}
\end{equation}

It remains to verify that \eqref{requireone}, \eqref{requiretwo}, \eqref{requirefour}, and \eqref{requirefive}
are compatible. Set
\[
 \mathcal D:=1+\mathcal Q_0+\overline N_{\rm bdry}.
\]
Choose $c_T>0$ sufficiently small and $K_T>0$ sufficiently large,
and set
\[
 M:=K_T\mathcal D^{5/3},
 \qquad
 X:=c_TM^{-4}.
\]
Taking $c_T\leq1/200$ sufficiently small ensures \eqref{requireone} and
\eqref{requirefour}, while \eqref{requiretwo} follows from the choice of $K_T$. Moreover,
\[
 \frac{X^{-1/10}\mathcal Q_0}{M}
 =
 c_T^{-1/10}\mathcal Q_0M^{-3/5}
 \leq c_T^{-1/10}K_T^{-3/5},
\]
because
\[
 M^{3/5}=K_T^{3/5}\mathcal D
 \geq K_T^{3/5}\mathcal Q_0.
\]
Thus, after increasing $K_T$ if necessary, \eqref{requirefive} also holds.
This closes the bootstrap.

\end{proof}

%Since we have different controls for tangential derivatives ($\tau,\xi$ direction) and normal derivative ($\eta$ direction), which suggests us to state the following two anisotropic embedding lemmas, which are useful in this section.
We now turn to the proof of Lemma \ref{aniembed} and Lemma \ref{anisotwo}, which are essentially anisotropic Sobolev type inequalities.
%The first anisotropic inequality implies that $w$ is a \emph{classical solution} provided the higher--order energy of $w$ (presented in Lemma \ref{highenergyesti}) is finite (for $k\geq 2$).
\begin{lemma}[\bf Embedding into H\"older Spaces]\label{aniembed}
Assume that $w$ is a classical solution to \eqref{prandtlsection8} on $[0,T]\times[0,X]\times[0,1]$ with $X\in(0,1]$ and $\|w\|_{E_{2}}<\infty$. Then for some $0<\alpha<1$ we have
\begin{equation}
\|(\p_{\tau},\p_{\xi})w\|_{C^{\alpha}([0,T]\times[0,X]\times[0,1])}\lesssim_{T} X^{-2}\cdot\|w\|_{E_{2}}.
\end{equation}
Recall that $\|\cdot\|_{E_{2}}$ is defined in \eqref{2thenergy}.
\end{lemma}
\begin{proof} For notational brevity, we assume $T=1$. By the definition of $\|\cdot\|_{E_{2}}$, we only need to prove that there exists $\alpha>0$ such that (with $D_X:=[0,1]\times[0,X]\times[0,1]$)
\begin{equation}\label{embedtwo}
\begin{split}
&\|f(t,x,y)\|_{C^{\alpha}([0,1]\times[0,X]\times[0,1])}\\
&\lesssim X^{-2}\bigg[\sup_{t\in[0,1]}\Big(\big\||f|+|\p_{t}f|+|\p_{x}f|\big\|_{L^{2}([0,X]\times[0,1])}\Big)+\big\||\p_{y}f|+|\p_{t}\p_{y}f|+|\p_{x}\p_{y}f|\big\|_{L^{2}(D_X)}\bigg].
\end{split}
\end{equation}
In our applications, we take $f\in\{\partial_\tau w, \partial_\xi w\}$. By a standard rescaling argument, we only need to show that
\begin{equation}\label{embedone}
\begin{split}
&\|f(t,x,y)\|_{C^{\alpha}([0,1]\times[0,1]\times[0,1])}\\
&\lesssim \bigg[\sup_{t\in[0,1]}\left(\|f\|_{L^{2}_{x,y}}+\|\p_{t}f\|_{L^{2}_{x,y}}+\|\p_{x}f\|_{L^{2}_{x,y}}\right)+\big\||\p_{y}f|+|\p_{t}\p_{y}f|+|\p_{x}\p_{y}f|\big\|_{L^{2}_{t,x,y}}\bigg].
\end{split}
\end{equation}

\textit{Proof of \eqref{embedone}.} First we show that there exists $\alpha>0$ such that
\begin{equation}\label{holderint}
\begin{split}
&|f(q)-f(\widetilde{q})|\\
&\lesssim  |q-\widetilde{q}|^{\alpha}\times\left[\sup_{t\in[0,1]}\left(\|\p_{t}f\|_{L^{2}_{x,y}}+\|\p_{x}f\|_{L^{2}_{x,y}}\right)+\big\||\p_{y}f|+|\p_{t}\p_{y}f|+|\p_{x}\p_{y}f|\big\|_{L^{2}_{t,x,y}}\right],
\end{split}
\end{equation}
for $q, \widetilde{q}\in [0,1]\times[0,1]\times[0,1]$.

Choose a rectangle $Q$ such that $q, \widetilde{q}\in Q$, and the lengths of the sides parallel to the $t$-, $x$- and $y$-axes are $\kappa,\kappa ,\kappa^{\frac{3}{4}}$, respectively. By H\"older's inequality and Sobolev inequality, we have
\begin{equation}\label{controlonoscil}
\begin{split}
&\left|f(q)-\frac{1}{\kappa^{\frac{11}{4}}}\iiint_{Q}f(z)dz\right| \lesssim \frac{1}{\kappa^{\frac{11}{4}}}\iiint_{Q}|f(q)-f(z)|dz\\
&\quad \lesssim \frac{1}{\kappa^{\frac{11}{4}}}\iiint_{Q}\left|\int_{0}^{1}\frac{d}{ds}f\left(q_{1}+s(z_{1}-q_{1}),q_{2}+s(z_{2}-q_{2}),q_{3}+s^{\frac{3}{4}}(z_{3}-q_{3})\right)ds\right|dz\\
&\quad \lesssim \kappa^{\frac{1}{8}}\|\p_{t}f,\p_{x}f\|_{L^{\infty}_{t}L^{2}_{x,y}}+\kappa^{\frac{1}{8}}\|f_{y}\|_{L^{2}_{y}L^{14}_{t,x}}\\
&\quad \lesssim  \kappa^{\frac{1}{8}}\left(\left\|\p_{t}f,\p_{x}f\right\|_{L^{\infty}_{t}L^{2}_{x,y}}+\left\|f_{y}\right\|_{L^{2}_{y}H^1_{t,x}}\right).
\end{split}
\end{equation}
We can see that \eqref{controlonoscil} implies \eqref{holderint} by the triangle inequality. 

Finally, notice that
\begin{equation}\label{pointwisesec8}
    \sup_{q\in[0,1]^{3}}|f(q)|\lesssim_{\alpha}\|f\|_{L^{2}([0,1]^{3})}+[f]_{C^{\alpha}([0,1]^{3})},
\end{equation}
hence we complete the proof of \eqref{embedone} in view of \eqref{holderint} and \eqref{pointwisesec8}.
\begin{comment}
Do the same way as \eqref{holderint} in $x$ direction and $y$ direction, then we deduce that there exists $\alpha>0$ such that
\begin{equation*}
\|f(t,x,y)\|_{C^{\alpha}([0,1]^{3})}\lesssim \sup_{t\in[0,1]}\left(\|f\|_{L^{2}_{x,y}}+\|\p_{t}f\|_{L^{2}_{x,y}}+\|\p_{x}f\|_{L^{2}_{x,y}}+\|\p_{y}^{2}f\|_{L^{2}_{x,y}}\right).
\end{equation*}\end{comment}
\end{proof}
The second anisotropic inequality is slightly weaker than that in Lemma \ref{aniembed} since we obtain only $L^p$ type control. However, the advantage is that the constant in \eqref{independentx} is independent of $X$ when $0<X\leq1$, which is important for closing the bootstrap assumptions in Lemma \ref{bootstrap}.

\begin{lemma}[\bf Uniform-in-$X$ embedding]\label{anisotwo}
Assume that $w$ satisfies the same assumptions as in Lemma \ref{aniembed}.
For any $1<p<\infty$, we have the inequalities
\begin{equation}\label{independentx}
 \sup_{\xi\in[0,X],\,\tau\in[0,T]}
 \left(
 \|w_\tau\|_{L^p(0,1/2)}
 +\|w_\xi\|_{L^p(0,1/2)}
 \right)
 \lesssim_{p,T}
 \|w\|_{E_2}\bigl(1+\|w\|^{4}_{E_2}\bigr),
\end{equation}
\begin{equation}\label{independentx1}
 \sup_{\substack{\xi\in[0,X],\,\tau\in[0,T]\\
                  \eta\in[1/4,3/4]}}
 \bigl(|w_\tau|+|w_\xi|\bigr)
 \lesssim_T
 \|w\|_{E_2}\bigl(1+\|w\|^{4}_{E_2}\bigr).
\end{equation}
The implied constants are independent of $X\in(0,1]$.
\end{lemma}

\begin{proof}
Let $f\in\{w_\tau,w_\xi\}$. Differentiating \eqref{prandtlsection8} in the
corresponding tangential variable gives
\[
 f_{\eta\eta}
 =
 \frac{f_\tau+\eta f_\xi}{w^2}
 -
 2\frac{f(w_\tau+\eta w_\xi)}{w^3}.
\]
Since $w\sim_T1$ on $0\leq\eta\leq3/4$, the linear term is
controlled by $\|w\|_{E_2}$. 

    To estimate the quadratic term, set
\[
 E:=\|w\|_{E_2},\qquad
 H:=\sup_{\xi\in[0,X]}
 \sum_{f\in\{w_\tau,w_\xi\}}
 \|\sqrt{\eta}\,f_{\eta\eta}\|_{L^2((0,T)\times(0,3/4))}.
\]

Denote 
 $I=(0,3/4),\, \cO=(0,T)\times I$. For every sufficiently regular $h$ on $\cO$, we have
\begin{equation}\label{inter0.1}
\begin{split}
 \|h\|_{L^\infty(0,T;L^4(I))}
 \lesssim_T{}&
 \Big(\|\sqrt y\,h\|_{L^2(\cO)}
      +\|\sqrt y\,\partial_t h\|_{L^2(\cO)}\Big)^{5/8}\\
      &\quad\quad\times \Big(\|\sqrt y\,h\|_{L^2(\cO)}
      +\|\sqrt y\,\partial_y^2h\|_{L^2(\cO)}\Big)^{3/8}.
\end{split}
\end{equation}
\eqref{inter0.1} follows from extending $h$ to $\mathbb{T}_{t}\times\R_{y}$ and employing the interpolation inequality \eqref{ineqfourier} below with $p=4$.

Using \eqref{inter0.1} we obtain that
\[
 \sup_{\xi\in[0,X]}
 \|f\|_{L^\infty_\tau L^4_\eta}
 \lesssim_T E^{5/8}(E+H)^{3/8},
 \qquad f\in\{w_\tau,w_\xi\}.
\]
Consequently,
\[
 \sup_{\xi\in[0,X]}
 \|\sqrt{\eta}\,D_1wD_2w\|_{L^2_{\tau,\eta}}
 \lesssim_T E^{5/4}(E+H)^{3/4},
 \qquad D_1,D_2\in\{\partial_\tau,\partial_\xi\}.
\]
Using the differentiated equation and $w\sim_T1$ on
$0\leq\eta\leq3/4$, we obtain
\[
 H\lesssim_T E+E^{5/4}(E+H)^{3/4}
 \lesssim_T E+E^5+\frac12H,
\]
and hence $H\lesssim_T E+E^5$. 

Therefore, to prove Lemma \ref{anisotwo} it suffices to establish the following general functional inequalities 
\begin{equation}\label{lemma8.7goal1}
\Big(\sup_{t\in[0,T]}\|f\|_{L^{p}[0,\frac{1}{2}]}\Big)^{2}\lesssim_{p,T} \int_{0}^{T}\int_{0}^{\frac{1}{2}}y\left(|f|^{2}+|f_{t}|^{2}+|f_{yy}|^{2}\right)dtdy,
\end{equation}

\begin{equation}\label{lemma8.7goal2}
\Big(\sup_{t\in[0,T],y\in[\frac{1}{4},\frac{3}{4}]}|f(t,y)|\Big)^{2}\lesssim_{T} \int_{0}^{T}\int_{\frac{1}{4}}^{\frac{3}{4}}\left(|f|^{2}+|f_{t}|^{2}+|f_{yy}|^{2}\right)dtdy.
\end{equation}
We only focus on \eqref{lemma8.7goal1} since the proof of \eqref{lemma8.7goal2} follows from a similar but slightly simpler argument. For notational simplicity we assume that $T=1$. We can extend $f(t,y)$ to a function on $\mathbb{T}\times \R^{+}$ satisfying
$$\int_{\T}\int_{\R^+}y\left(|f|^{2}+|f_{t}|^{2}+|f_{yy}|^{2}\right)dtdy\lesssim\int_{0}^{T}\int_{0}^{\frac{1}{2}}y\left(|f|^{2}+|f_{t}|^{2}+|f_{yy}|^{2}\right)dtdy.$$
We can expand the extended $f$ using Fourier series in $t$ as
\begin{equation*}
f(t,y)=\sum_{k\in\mathbb{Z}}\widehat{f}_{k}(y)e^{ik\cdot t}.
\end{equation*}
By Plancherel's identity, it suffices to show that
\begin{equation}\label{ineqfourier}
\|f\|_{L^{\infty}_{t}L^{p}_{y}(\T\times\R^+)}\lesssim_{p}\left(\sum_{k\in\mathbb{Z}}\int_{0}^{\infty}y\langle k\rangle^{2}\big|\widehat{f}_{k}(y)\big|^{2}dy\right)^{\frac{p+1}{4p}}\cdot\left(\sum_{k\in\mathbb{Z}}\int_{0}^{\infty}y\big|\p_{y}^{2}\widehat{f}_{k}(y)\big|^{2}dy\right)^{\frac{p-1}{4p}}.
\end{equation}
Using the interpolation inequality \eqref{interrr} and a standard rescaling argument, we have the following estimates
\begin{equation}\label{ineqf4.0}
\begin{split}
    \|f\|_{L^{\infty}_{t}L^{p}_{y}(\T\times\R^+)}&\leq \sum_{k\in\mathbb{Z}}\Big\|\widehat{f}_{k}(y)\Big\|_{L^{p}_{y}}\lesssim\sum_{k\in\mathbb{Z}}\Big\|\sqrt{y}\widehat{f}_{k}(y)\Big\|_{L^{2}_{y}}^{\frac{p+1}{2p}}\Big\|\sqrt{y}\p_{y}^{2}\widehat{f}_{k}(y)\Big\|_{L^{2}_{y}}^{\frac{p-1}{2p}}\\
    &=\sum_{k\in\mathbb{Z}}\langle k\rangle^{-\frac{p+1}{2p}}\Big\|\langle k\rangle\sqrt{y}\widehat{f}_{k}(y)\Big\|_{L^{2}_{y}}^{\frac{p+1}{2p}}\Big\|\sqrt{y}\p_{y}^{2}\widehat{f}_{k}(y)\Big\|_{L^{2}_{y}}^{\frac{p-1}{2p}}.
\end{split}
\end{equation}
Then \eqref{ineqfourier} follows from \eqref{ineqf4.0} and H\"older's inequality, which completes the proof of Lemma \ref{anisotwo}.
\end{proof}

\subsection{Local classical solution theory} \label{localmain}
First we state the well-posedness of classical solutions for system \eqref{prandtlsection8}.
\begin{theorem}\label{localsolcrocco}
Assume that the initial data $w_{0}$ and inflow data $w_{1}$ satisfy Assumption \ref{dataassumecrocco}. Then we have the following conclusions:
\begin{itemize}
\item[(i)] \textit{The local--in--$\xi$ solution.} For \emph{any} $T>0$, there exists $X=X_{T}>0$ such that \eqref{prandtlsection8} has a unique classical solution $w$ in $H_T:=[0,T]\times[0,X_{T}]\times[0,1)$ satisfying on $H_T$ the pointwise bound
\begin{equation}\label{localxproperty1}
w(\tau,\xi,\eta)\sim F(\eta),\ \forall (\tau,\xi,\eta)\in H_{T},
\end{equation}
and the energy estimates
\begin{equation}\label{localxproperty2}
\begin{split}
&\sup_{H_T}\left(\frac{|\p_{\tau}w|}{w}+\frac{|\p_{\xi}w|}{w}\right)+\sup_{\tau\in[0,T]}\sum_{\alpha+\beta\leq 2}\int_{0}^{X_{T}}\int_{0}^{1}\bigg|\frac{\p_{\tau}^{\alpha}\p_{\xi}^{\beta}w}{w}\bigg|^{2}d\xi d\eta\\
&+\sup_{\xi\in[0,X_{T}]}\sum_{\alpha+\beta\leq 2}\int_{0}^{T}\int_{0}^{1}\eta\bigg|\frac{\p_{\tau}^{\alpha}\p_{\xi}^{\beta}w}{w}\bigg|^{2}d\tau d\eta+\sum_{\alpha+\beta\leq 2}\iiint_{H_{T}}|\p_{\tau}^{\alpha}\p_{\xi}^{\beta}\p_{\eta}w|^{2}d\tau d\xi d\eta\lesssim_{T,w_{0},w_{1}}1.\\
\end{split}
\end{equation}
Moreover, for some $\gamma\in(0,1)$ we have
\begin{equation}\label{localxproperty3}
w,\p_{\tau}w,\p_{\xi}w,w^{2}\p_{\eta}^{2}w\in C^{\gamma}\left([0,T]\times[0,X_{T}]\times[0,1)\right).
\end{equation}
\item[(ii)] \textit{The local--in--$\tau$ solution.} For \emph{any} $X>0$, there exists $T=T_{X}>0$ such that \eqref{prandtlsection8} has a unique classical solution in $K_X:=[0,T_{X}]\times[0,X]\times[0,1)$ satisfying on $K_X$ the pointwise bound 
\begin{equation}\label{localtproperty1}
w(\tau,\xi,\eta)\sim F(\eta),\ \forall (\tau,\xi,\eta)\in [0,T_{X}]\times[0,X]\times[0,1),
\end{equation}
and the energy estimates
\begin{equation}\label{localtproperty2}
\begin{split}
&\sup_{K_X}\left(\frac{|\p_{\tau}w|}{w}+\frac{|\p_{\xi}w|}{w}\right)+\sup_{\tau\in[0,T_{X}]}\sum_{\alpha+\beta\leq 2}\int_{0}^{X}\int_{0}^{1}\bigg|\frac{\p_{\tau}^{\alpha}\p_{\xi}^{\beta}w}{w}\bigg|^{2}d\xi d\eta\\
&+\sup_{\xi\in[0,X]}\sum_{\alpha+\beta\leq 2}\bigg[\int_{0}^{T_{X}}\int_{0}^{1}\eta\bigg|\frac{\p_{\tau}^{\alpha}\p_{\xi}^{\beta}w}{w}\bigg|^{2}d\tau d\eta+\iiint_{K_{X}}|\p_{\tau}^{\alpha}\p_{\xi}^{\beta}\p_{\eta}w|^{2}d\tau d\xi d\eta\bigg]\lesssim_{X,w_{0},w_{1}}1.\\
\end{split}
\end{equation}
Moreover, for some $\gamma\in(0,1)$ we have
\begin{equation}\label{localtproperty3}
w,\p_{\tau}w,\p_{\xi}w,w^{2}\p_{\eta}^{2}w\in C^{\gamma}\left([0,T_{X}]\times[0,X]\times[0,1)\right).
\end{equation}
\end{itemize}
\end{theorem}

\begin{remark}
    We note that the implied constants in \eqref{localxproperty2} and \eqref{localtproperty2} depend, more quantitatively, on the norms $\mathcal{N}_{0}$ and $\mathcal{N}_{\text{bdry}}$ of the data $w_0$ and $w_1$ (see \eqref{n0k}, \eqref{nbdry} in subsection \ref{priorisection8}).  
\end{remark}
\begin{proof}[Proof of Theorem \ref{localsolcrocco}]
The proof of Theorem \ref{localsolcrocco} follows from a priori estimates stated in subsection \ref{priorisection8}, as well as the standard continuation argument. Similarly to the proof of Proposition \ref{existweaksol}, the proof has three main parts: (i) Approximation; (ii) Uniform bounds; (iii) Passing to the limit.

\textit{Step 1: Approximation.} As in \eqref{approsys}, let $w^{\epsilon}$ be the global solution of
\begin{equation}\label{approsyss}
\begin{cases}
\p_{\tau}w^{\epsilon}+(\eta+\epsilon)\p_{\xi}w^{\epsilon}-(w^{\epsilon}+\epsilon)^{2}\p_{\eta}^{2}w^{\epsilon}=0,\ (\tau,\xi,\eta)\in \R^{+}\times\R^{+}\times (0,1),\\
w^{\epsilon}|_{\tau=0}=w_{0}^{\epsilon},\ w^{\epsilon}|_{\xi=0}=w_{1}^{\epsilon},\\
w^{\epsilon}|_{\eta=1}=\p_{\eta}w^{\epsilon}|_{\eta=0}=0,
\end{cases}
\end{equation}
where $w_{0}^{\epsilon},w_{1}^{\epsilon}$ are suitable smooth modifications of $w_{0},w_{1}$ respectively, satisfying the compatibility conditions corresponding to \eqref{approsyss} and certain uniform--in--$\epsilon$ bounds, which are presented in Appendix \ref{compatiappendix} (see (CP1),   (CP2) and (CP3) in Appendix \ref{requireappendix}).  We refer to Proposition \ref{constructappdata1} for the construction of the compatible data $(w_{0}^{\epsilon},w_{1}^{\epsilon})$, and Appendix \ref{proofapproxia} for the global existence and uniqueness of the classical solution $w^\epsilon$ of \eqref{approsyss}.

\textit{Step 2: Uniform bounds.} As in Step 2 of the proof of  Proposition \ref{existweaksol}, for any fixed $T>0$, we have
\begin{equation}\label{approxiaboun}
(1-\eta)^{3/2}\lesssim_{T}w^{\epsilon}(\tau,\xi,\eta)\lesssim_{T} (1-\eta)\log^{\frac{1}{2}}\left(\frac{10}{1-\eta}\right),
\ \forall \tau\in[0,T],\xi\in[0,1],\eta\in [0,1).
\end{equation}
Furthermore, it follows from the same arguments as in the proofs of Lemmas \ref{highenergyesti} and \ref{bootstrap} that for the fixed $T>0$ above there exist $M>0$ and $0<X\leq 1$ depending on both $T$ and the given data $(w_{0},w_{1})$, but \emph{independent} of $\epsilon$, such that
\begin{equation}\label{approxiabound}
\begin{split}
&\sup_{\tau\in[0,T],\xi\in[0,X],\eta\in[0,1]}\left(\left|\frac{\p_{\tau}w^{\epsilon}}{w^{\epsilon}+\epsilon}\right|+\left|\frac{\p_{\xi}w^{\epsilon}}{w^{\epsilon}+\epsilon}\right|\right)+\sup_{\tau\in[0,T]}\sum_{1\leq\alpha+\beta\leq 2}\int_{0}^{X}\int_{0}^{1}\left|\frac{\p_{\tau}^{\alpha}\p_{\xi}^{\beta}w^{\epsilon}}{w^{\epsilon}+\epsilon}\right|^{2}d\xi d\eta\\
&+\sup_{\xi\in[0,X]}\sum_{1\leq\alpha+\beta\leq 2}\int_{0}^{T}\int_{0}^{1}\eta\left|\frac{\p_{\tau}^{\alpha}\p_{\xi}^{\beta}w^{\epsilon}}{w^{\epsilon}+\epsilon}\right|^{2}d\tau d\eta+\sum_{\alpha+\beta\leq 2}\int_{0}^{T}\int_{0}^{X}\int_{0}^{1}|\p_{\tau}^{\alpha}\p_{\xi}^{\beta}\p_{\eta}w^{\epsilon}|^{2}d\tau d\eta d\xi\leq M.
\end{split}
\end{equation}
In \eqref{highenergyid} and \eqref{higherenergyclose}, the sum is over $1\leq\alpha+\beta\leq2$.  The remaining zero-order normal estimate follows from a logarithmic energy estimate by testing \eqref{approsyss} against $(w^\epsilon+\epsilon)^{-1}$. 

\textit{Step 3: Passing to the Limit.} By Lemma \ref{aniembed}, we have
\begin{equation}
\sup_{\epsilon\in(0,\epsilon_0]}\left(\|w^{\epsilon}\|_{C^{\alpha}([0,T]\times[0,X]\times[0,1])}+\|\p_{\tau}w^{\epsilon},\p_{\xi}w^{\epsilon}\|_{C^{\alpha}([0,T]\times[0,X]\times[0,1])}\right)\lesssim M,
\end{equation}
for some generic constant $\alpha>0$. In the above, $\epsilon_0>0$ is a sufficiently small number so that $w^\epsilon$ is defined for $0<\epsilon<\epsilon_0$.  By the Arzel\`a--Ascoli Theorem, there exist a subsequence $w^{\epsilon_{n}}$ and a function $w(\tau,\xi,\eta)\in C^{\alpha}([0,T]\times[0,X]\times[0,1])$ such that
\begin{equation}\label{convergenceunif}
\begin{cases}
w^{\epsilon_{n}}\rightarrow w,\ \text{uniformly\ in}\ [0,T]\times[0,X]\times[0,1],\\
\p_{\tau}w^{\epsilon_{n}}\rightarrow \p_{\tau}w,\ \text{uniformly\ in}\ [0,T]\times[0,X]\times[0,1],\\
\p_{\xi}w^{\epsilon_{n}}\rightarrow \p_{\xi}w,\ \text{uniformly\ in}\ [0,T]\times[0,X]\times[0,1],\\
(w^{\epsilon_{n}}+\epsilon_{n})^{2}\p_{\eta}^{2}w^{\epsilon_{n}}\rightarrow w^{2}\p_{\eta}^{2}w,\ \text{uniformly\ in}\ [0,T]\times[0,X]\times [0,1].
\end{cases}
\end{equation}
By \eqref{approsyss}, we see that $w$ satisfies the original equation \eqref{prandtlsection8}. The desired bounds in Theorem \ref{localsolcrocco} follow from  \eqref{approxiaboun}, \eqref{approxiabound}, \eqref{convergenceunif}, and the lower semi--continuity of norms. 

\textit{Step 4: The refined pointwise bounds.} We also need to show that the classical solution satisfies $w(\tau,\xi,\eta)\sim F(\eta)$. The proof is similar to the second step of Proposition \ref{existweaksol}. 

In the Gaussian case, i.e. $F(\eta)=(1-\eta)\cdot\log^{\frac{1}{2}}\left(\frac{10}{1-\eta}\right)$, we need to show the lower bound $w(\tau,\xi,\eta)\gtrsim F(\eta)$. To this end, we consider the barrier function
\begin{equation}
    c_{T,X}\cdot e^{-M\tau}e^{\eta}(1-\eta)\log^{\frac{1}{2}}\left(\frac{10}{1-\eta}\right),
\end{equation}
for some large constant $M\geq 1$, and then apply the comparison principle as in the proof of Proposition \ref{existweaksol}.

In the exponential and polynomial cases, i.e. $F(\eta)=(1-\eta)^{m}$ for some $m\in[1,\frac{3}{2}]$, we shall first obtain the upper bound $w(\tau,\xi,\eta)\lesssim 1-\eta$ by considering the barrier function $C_{T,X}\cdot(1-\eta)$ as a middle step. Then we can obtain the upper bound and lower bound $w(\tau,\xi,\eta)\sim(1-\eta)^{m}$ by considering the two-sided barrier functions (for some sufficiently large $N\geq 1$)
\begin{equation}
C_{T,X}e^{N\tau}(1-\eta)^{m},\ \text{and}\ c_{T,X}e^{-N\tau}e^{2\eta}(1-\eta)^{m}
\end{equation}
respectively.

Finally, the uniqueness follows from Proposition \ref{uniweaksol}. Therefore, we complete the proof of Theorem \ref{localsolcrocco}.
\end{proof}
By Theorem \ref{localsolcrocco}, using the same arguments as in the proof of Theorem \ref{mainthree}, we can now establish the local existence of classical solutions to \eqref{prandtl} in physical coordinates. The local well-posedness of \eqref{prandtl} has already been studied in several works (see e.g. \cite{OS99,AWXY15,MW15,XZ17}). Nevertheless, it seems that none of them can be applied directly for our purposes. For completeness we provide the following result for \eqref{prandtl} in more general settings.
\begin{theorem}[\bf Local Well-posedness for Classical Solutions]
\label{mainfour}Under Assumption \ref{dataassum}, there exist unique local--in--$t$ and local--in--$x$ classical solutions of \eqref{prandtl}. More precisely, we have the following conclusions.
\begin{itemize}
\item[(i)] \textit{Local--in--$x$ Solution.} For any $T>0$, there exists $X=X_{T}>0$ depending on $T$ and $(u_{0},u_{1})$, such that there exists a unique classical solution to \eqref{prandtl} in $\mathcal{U}_{T}:=[0,T]\times[0, X_T]\times(0,\infty)$, with the following properties:
\begin{enumerate}
\item[(i.1)] $u(t,x,y)$ is continuously differentiable in $\overline{\mathcal{U}}_{T}$, satisfying the initial--boundary conditions in \eqref{prandtl} in the classical sense. Furthermore, $\p_{t}u,\p_{x}u,\p_{y}u,\p_{y}^{2}u$ are H\"older continuous in $\overline{\mathcal{U}}_{T}$ and $u$ satisfies \eqref{prandtl} in the classical sense;
    
\item[(i.2)] $u(t,x,y)$ is strictly monotone in $y$ and satisfies the matching rate
    \begin{equation*}
        \p_{y}u\sim F(u),\ \forall (t,x,y)\in \overline{\mathcal{U}}_{T}.\\
    \end{equation*}
\begin{comment}
\item The Sobolev bounds:
{\small \begin{equation}\label{sobobdlocalx}
\begin{split}
&\sup_{t \in[0,T],x\in[0,X_{T}],y\in \R^{+}}\left(\left|\p_{y}\left(\frac{\p_{t}u}{\p_{y}u}\right)\bigg|+\bigg|\p_{y}\left(\frac{\p_{x}u}{\p_{y}u}\right)\right|\right)+\sup_{t\in[0,T]}\sum_{\alpha+\beta\leq k}\int_{0}^{X_{T}}\int_{0}^{\infty}\frac{|\mathcal{T}_{1}^{\alpha}\mathcal{T}_{2}^{\beta}\p_{y}u|^{2}}{\p_{y}u}dx dy\\
&\quad\quad+\sup_{x\in[0,X_{T}]}\sum_{\alpha+\beta\leq k}\int_{0}^{T}\int_{0}^{\infty}u\cdot\frac{|\mathcal{T}_{1}^{\alpha}\mathcal{T}_{2}^{\beta}\p_{y}u|^{2}}{\p_{y}u}dt dy\lesssim_{T,w_{0},w_{1}}1.\\
\end{split}
\end{equation}}
\end{comment}
\end{enumerate}

\item[(ii)] \textit{Local--in--$t$ Solution.} For any $X>0$, there exists $T=T_{X}>0$ such that there exists a unique classical solution to \eqref{prandtl} in $\mathcal{V}_{X}:=[0,T_{X}]\times[0,X]\times(0,\infty)$, such that
\begin{enumerate}
    \item[(ii.1)] $u(t,x,y)$ is continuously differentiable in $\overline{\mathcal{V}}_{X}$ and satisfies the initial--boundary conditions in \eqref{prandtl} in the classical sense. Furthermore, $\p_{t}u,\p_{x}u,\p_{y}u,\p_{y}^{2}u$ are H\"older continuous in $\overline{\mathcal{V}}_{X}$ and $u$ satisfies \eqref{prandtl} in the classical sense.
    
\item[(ii.2)] $u(t,x,y)$ is strictly monotone in $y$ and satisfies the matching rate
    \begin{equation*}
        \p_{y}u\sim F(u),\ \forall (t,x,y)\in \overline{\mathcal{V}}_{X}.\\
    \end{equation*}
\begin{comment}
\item The Sobolev bounds:
{\small \begin{equation}\label{sobobdlocalt}
\begin{split}
&\sup_{t \in[0,T_{X}],x\in[0,X],y\in \R^{+}}\left(\left|\p_{y}\left(\frac{\p_{t}u}{\p_{y}u}\right)\bigg|+\bigg|\p_{y}\left(\frac{\p_{x}u}{\p_{y}u}\right)\right|\right)+\sup_{t\in[0,T_{X}]}\sum_{\alpha+\beta\leq k}\int_{0}^{X}\int_{0}^{\infty}\frac{|\mathcal{T}_{1}^{\alpha}\mathcal{T}_{2}^{\beta}\p_{y}u|^{2}}{\p_{y}u}dx dy\\
&\quad\quad+\sup_{x\in[0,X]}\sum_{\alpha+\beta\leq k}\int_{0}^{T_{X}}\int_{0}^{\infty}u\cdot\frac{|\mathcal{T}_{1}^{\alpha}\mathcal{T}_{2}^{\beta}\p_{y}u|^{2}}{\p_{y}u}dt dy\lesssim_{X,w_{0},w_{1}}1.\\
\end{split}
\end{equation}}
\end{comment}
\end{enumerate}
\end{itemize}
\end{theorem}

\begin{remark}
Our proof also provides the following weighted bounds in $\mathcal{W}\in\{\overline{\mathcal{U}}_{T}, \overline{\mathcal{V}}_{X}\}$:
\begin{equation}\label{sobobdlocal}
\begin{split}
%\begin{cases}
&\Big|\p_{y}\Big(\frac{\p_{t}u}{\p_{y}u}\Big)\bigg|+\bigg|\p_{y}\Big(\frac{\p_{x}u}{\p_{y}u}\Big)\Big|\in L^{\infty}(\mathcal{W}),\\
&\sum_{\alpha+\beta \leq 2}\frac{|\mathcal{T}_{1}^{\alpha}\mathcal{T}_{2}^{\beta}\p_{y}u|}{(\p_{y}u)^{\frac{1}{2}}}\in L^{\infty}_{t}L^{2}_{x,y}(\mathcal{W}),\\
&\sum_{\alpha+\beta\leq 2}\frac{u^{\frac{1}{2}}\cdot|\mathcal{T}_{1}^{\alpha}\mathcal{T}_{2}^{\beta}\p_{y}u|}{(\p_{y}u)^{\frac{1}{2}}}\in L^{\infty}_{x}L^{2}_{t,y}(\mathcal{W}),
%\end{cases}
\end{split}
\end{equation}
where we recall (with a slight abuse of notation) that
\begin{equation*}
    \mathcal{T}_{1}=\p_{t}-\frac{\p_{t}u}{\p_{y}u}\p_{y},\quad \mathcal{T}_{2}=\p_{x}-\frac{\p_{x}u}{\p_{y}u}\p_{y}.
\end{equation*}

\end{remark}

\begin{comment}
We emphasize that we prove above main results with the help of the transformed Prandtl equation \eqref{prandtlcroccosec1}, and formulations in Assumption \ref{dataassum} are indeed more direct regularity bounds for $(w_{0},w_{1})$ in \eqref{prandtlcroccosec1}. Since Prandtl's equation is originally formulated in physical coordinates, it is more natural to impose assumptions on the initial data within the physical coordinate system. We point that similar conclusions stated in Theorem \ref{mainone} to \ref{mainfour} also hold for the transformed Prandtl system \eqref{prandtlcroccosec1}.

In this subsection, we finish the proof of Theorem \ref{mainthree}: local classical solution theory of \eqref{prandtl}.
\end{comment}
%We now complete the proof of Theorem \ref{mainfour}.
\begin{proof}[Proof of Theorem \ref{mainfour}] 
The proof uses Theorem \ref{localsolcrocco} and otherwise follows from the same argument as in the proof of Theorem \ref{mainthree} (see Section \ref{pgf2}).

\end{proof}

%section 9

\subsection{Proof of Theorem \ref{maintwo}}
In this section, we complete the proof of Theorem \ref{maintwo} on the global well-posedness theory for classical solutions to the dynamical Prandtl equation \eqref{prandtl}.

\begin{proof}[Proof of Theorem \ref{maintwo}]  Theorem \ref{maintwo} essentially follows by combining Theorem \ref{mainthree} on the global well-posedness of weak solutions and Theorem \ref{mainfour} on the local existence of classical solutions. We first note that the uniqueness statement in Theorem \ref{maintwo} follows from Proposition \ref{uniweaksol}. Therefore, in the following we focus on the existence part.

Assume the initial data $u_{0}$ and the inflow data $u_{1}$ satisfy Assumption \ref{dataassum}. By Theorem \ref{mainthree}, there exists a unique global weak solution $u(t,x,y)$ satisfying \eqref{prandtl}, with the properties stated in Theorem \ref{mainthree}. In particular,
\begin{equation}\label{pfmainsmooth}
u\in C^{\infty}((0,\infty)\times(0,\infty)\times[0,\infty)).
\end{equation}

On the other hand, for arbitrary $T,X>0$, by Theorem \ref{mainfour} there exist $X_{T}>0$ and $T_{X}>0$ such that the Prandtl equation \eqref{prandtl} admits a local--in--$x$ classical solution $u^{(1)}(t,x,y)$ in $[0,T]\times[0,X_{T}]\times[0,\infty)$ and a local--in--$t$ classical solution $u^{(2)}(t,x,y)$ in $[0,T_{X}]\times[0,X]\times[0,\infty)$. 

By uniqueness of weak solutions, we have
\begin{equation}\label{weaktoclassical}
    \begin{split}
        u(t,x,y)\equiv u^{(1)}(t,x,y)\quad \text{in}\ [0,T]\times[0,X_{T}]\times[0,\infty),\\
        u(t,x,y)\equiv u^{(2)}(t,x,y)\quad \text{in}\ [0,T_{X}]\times[0,X]\times[0,\infty).
    \end{split}
\end{equation}

Therefore, the continuous differentiability of $u(t,x,y)$ and the H\"older continuity of the associated derivatives $\{\p_{t}u,\p_{x}u,\p_{y}u,\p_{y}^{2}u\}$ for $(t,x,y)\in[0,\infty)\times[0,\infty)\times[0,\infty)$ follow from \eqref{pfmainsmooth}, \eqref{weaktoclassical}, and Theorem \ref{mainfour}. 

Finally, the refined bounds \eqref{solutionmatching}, say $\p_{y}u\sim F(u)$, follow from the same argument in the proof of Theorem \ref{localsolcrocco}. Hence, the proof of Theorem \ref{maintwo} is now complete.
\end{proof}

\begin{remark} By combining Proposition \ref{existweaksol} with Theorem \ref{localsolcrocco}, we also obtain the global well-posedness for classical solutions of the transformed Prandtl system \eqref{prandtlcroccosec1}, which is stated in Theorem \ref{thm:gcs}.
\end{remark}

%acknowledgment
\section*{Acknowledgments}
HJ was supported in part by NSF DMS-2245021 and NSF DMS-2453270; ZL was supported in part by NSFC (No.12494544), NSFC (No.12431007), NSFC Excellent Research Group Project (No.62588101) and New Cornerstone Science Foundation through the XPLORER
PRIZE; CY was supported in part by NSFC (No.123B2008).
%appendix

\appendix
\section{Compatibility conditions and construction of approximating data}\label{compatiappendix}
In this appendix, we present the compatibility conditions and the construction of approximating data $(w_{0}^{\epsilon},w_{1}^{\epsilon})$ for
\begin{equation}\label{approsysappendix}
\begin{cases}
\p_{t}w^{\epsilon}+(y+\epsilon)\p_{x}w^{\epsilon}-(w^{\epsilon}+\epsilon)^{2}\p_{y}^{2}w^{\epsilon}=0,\ (t,x,y)\in(0,T]\times(0,X]\times(0,1),\\
w^{\epsilon}|_{t=0}=w_{0}^{\epsilon},\ w^{\epsilon}|_{x=0}=w_{1}^{\epsilon},\\
w^{\epsilon}|_{y=1}=\p_{y}w^{\epsilon}|_{y=0}=0,
\end{cases}
\end{equation}
which is the approximation system of the transformed Prandtl equation
\begin{equation}\label{prandtlcroccoappendix}
\begin{cases}
\p_{t}w+y\p_{x}w-w^{2}\p_{y}^{2}w=0,\ (t,x,y)\in (0,T]\times (0,X]\times (0,1),\\
w|_{t=0}=w_{0}(x,y),\ w|_{x=0}=w_{1}(t,y),\\
w|_{y=1}=0,\ \p_{y}w|_{y=0}=0.
\end{cases}
\end{equation}

Due to the significance of the approximation system \eqref{approsysappendix} in the study of global weak solutions (Section \ref{basic}) and local classical solutions (Section \ref{local}), we present a description of the compatibility conditions and the construction of approximating data for completeness. We emphasize that these issues arise frequently in initial-boundary-value problems (IBVPs) in PDEs, especially for quasilinear problems where the local well-posedness theory does not follow directly from Duhamel's principle.

\subsection{Formal jets and compatibility conditions for \eqref{approsysappendix}}
First we present the compatibility conditions for \eqref{approsysappendix} with fixed $\epsilon>0$. To prove the global well-posedness of \eqref{approsysappendix} in Appendix \ref{proofapproxia}, the approximating data $w_{0}^{\epsilon},w_{1}^{\epsilon}$ should satisfy certain higher-order compatibility conditions as follows. To make the description clearer, we first introduce \emph{formal boundary jets} applicable to the equation \eqref{approsysappendix}.
\begin{definition} For an initial function $f(x,y)$, define
\begin{equation}\label{initialjets}
\mathcal{T}_{\epsilon}^{0}[f]:=f,\ \mathcal{T}_{\epsilon}^{1}[f]:=(f+\epsilon)^{2}\p_{y}^{2}f-(y+\epsilon)\p_{x}f,
\end{equation}
and recursively
\begin{equation}
\mathcal{T}_{\epsilon}^{k+1}[f]:=\p_{t}^{k}\left((f+\epsilon)^{2}\p_{y}^{2}f-(y+\epsilon)\p_{x}f\right)\bigg|_{t=0},
\end{equation}
where each formal lower derivative of $f$ in $t$ is replaced by the previously defined formal initial jets.

For an inflow function $g(t,y)$, define
\begin{equation}\label{inflowjets}
\mathcal{X}_{\epsilon}^{0}[g]:=g,\ \mathcal{X}_{\epsilon}^{1}[g]:=\frac{(g+\epsilon)^{2}\p_{y}^{2}g-\p_{t}g}{y+\epsilon},
\end{equation}
and recursively
\begin{equation}
\mathcal{X}_{\epsilon}^{k+1}[g]:=\p_{x}^{k}\left(\frac{(g+\epsilon)^{2}\p_{y}^{2}g-\p_{t}g}{y+\epsilon}\right)\bigg|_{x=0},
\end{equation}
where each formal lower derivative of $g$ in $x$ is replaced by the previously defined formal inflow jets.
\end{definition}
Now we present the explicit expression of the compatibility conditions for the approximation system \eqref{approsysappendix}.
\begin{definition}\label{compatibilitydef}
We say the initial data and inflow data $(w_{0}^{\epsilon},w_{1}^{\epsilon})$ satisfy the compatibility conditions up to order $N\in\mathbb{N}$, provided that the following hold for $a+b\leq N$:
\begin{itemize}
\item[(i)] Matching the Dirichlet boundary condition on $y=1$:
\begin{equation}
\p_{x}^{b}\mathcal{T}_{\epsilon}^{a}[w_{0}^{\epsilon}](x,1)\equiv 0,\ \p_{t}^{a}\mathcal{X}_{\epsilon}^{b}[w_{1}^{\epsilon}](t,1)\equiv 0.
\end{equation}
\item[(ii)] Matching the Neumann boundary condition on $y=0$:
\begin{equation}
\p_{y}\p_{x}^{b}\mathcal{T}_{\epsilon}^{a}[w_{0}^{\epsilon}](x,0)\equiv 0,\ \p_{y}\p_{t}^{a}\mathcal{X}_{\epsilon}^{b}[w_{1}^{\epsilon}](t,0)\equiv 0.
\end{equation}
\item[(iii)] Matching at the corner:
\begin{equation}
\p_{x}^{b}\mathcal{T}_{\epsilon}^{a}[w_{0}^{\epsilon}]\big|_{x=0}\equiv \p_{t}^{a}\mathcal{X}_{\epsilon}^{b}[w_{1}^{\epsilon}]\big|_{t=0}\quad \text{as functions of}\ y.
\end{equation}
\end{itemize}
\end{definition}
\begin{remark}\label{descriptioncompati}
Applying the notion of formal jets to the transformed Prandtl system \eqref{prandtlcroccoappendix}, we can similarly define the compatibility conditions for \eqref{prandtlcroccoappendix}. The compatibility conditions for the original Prandtl system \eqref{prandtl} can also be made explicit using Definition \ref{defcrocco} and Lemma \ref{basiccrocco} following an analogous approach.
\end{remark}
The following definition of formal wall jets on $y=0$ is useful in the construction of approximating data.
\begin{definition}
Given a smooth positive wall trace $a=a(t,x)$, put $\Gamma_{-1,\epsilon}[a]=0, \Gamma_{0,\epsilon}[a]=a,\ \Gamma_{1,\epsilon}[a]=0$, and set
\begin{equation*}
A_{0,\epsilon}:=a+\epsilon,\ A_{j,\epsilon}:=\Gamma_{j,\epsilon}[a],\ j\geq 1.
\end{equation*}
For $n\geq 0$, define recursively
\begin{equation}\label{walljets}
\begin{split}
\Gamma_{n+2,\epsilon}[a]:=&\frac{1}{(a+\epsilon)^{2}}\Bigg[\p_{t}\Gamma_{n,\epsilon}[a]+\epsilon\p_{x}\Gamma_{n,\epsilon}[a]+n\p_{x}\Gamma_{n-1,\epsilon}[a]\\
&\quad\quad\quad\quad-\sum_{i+j+k=n,\,k<n}\frac{n!}{i!j!k!}A_{i,\epsilon}A_{j,\epsilon}\Gamma_{k+2,\epsilon}[a]\Bigg];
\end{split}
\end{equation}
the sum is empty when $n=0$.
\end{definition}
\begin{remark} We can see that if $w^{\epsilon}(t,x,y)$ is the classical solution of \eqref{approsysappendix}, then it must hold that
\begin{equation}
\p_{y}^{k}w^\epsilon|_{y=0}=\Gamma_{k,\epsilon}[a^{\epsilon}],\ a^{\epsilon}(t,x)=w^{\epsilon}(t,x,0).
\end{equation}
\end{remark}

\subsection{Approximating data}\label{requireappendix} In this section, we describe the construction of the modified data $(w_{0}^{\epsilon},w_{1}^{\epsilon})$ from $(w_0,w_1)$, which are necessary since the compatibility conditions for \eqref{prandtlcroccoappendix} and \eqref{approsysappendix} differ in subtle but essential ways. Recall that the original data $(w_{0},w_{1})$ satisfy Assumption \ref{dataassumecrocco} (classical solution regime) or Assumption \ref{dataassumecrocco4} 
(weak solution regime). Without loss of generality, we assume that $T=1,X=1$ for simplicity of notation. 

For the theory of weak solutions in Section \ref{basic}, the approximating data $(w_{0}^{\epsilon},w_{1}^{\epsilon})$ are required to satisfy the following properties for $i=0,1$:
\begin{itemize}
    \item[(WP1)] $w_{i}^{\epsilon}$ converges to $w_{i}$ uniformly as $\epsilon\to0+$;
    \item[(WP2)] $w_{i}^{\epsilon}$ is sufficiently regular, and satisfies the compatibility conditions for the approximating system \eqref{approsysappendix} up to the twentieth order;
    \item[(WP3)] The following uniform-in-$\epsilon$ bounds hold:
    \begin{equation}\label{approximatedataassume1}
        \begin{split}
            &(1-y)^{3/2}\lesssim w_{i}^{\epsilon}\lesssim (1-y)\log^{\frac{1}{2}}\left(\frac{10}{1-y}\right),\\
            &\int_{0}^{1}\int_{0}^{1}\frac{|\p_{x}w_{0}^{\epsilon}|}{(w_{0}^{\epsilon}+\epsilon)^{2}}\cdot (1-y)dxdy+\int_{0}^{1}\int_{0}^{1}\frac{|\p_{t}w_{1}^{\epsilon}|}{(w_{1}^{\epsilon}+\epsilon)^{2}}\cdot (1-y)dtdy\lesssim 1,\\
            &\int_{0}^{1}\int_{0}^{1}|\p_{y}^{2}w_{0}^{\epsilon}|\cdot(1-y)dxdy+\int_{0}^{1}\int_{0}^{1}|\p_{y}^{2}w_{1}^{\epsilon}|\cdot(1-y)dtdy\lesssim 1.
        \end{split}
    \end{equation}
\end{itemize}

In the classical solution regime in Section \ref{local}, the approximating data $w_{0}^{\epsilon},w_{1}^{\epsilon}$ are required to satisfy the following properties for $i=0,1$:
\begin{itemize}
    \item[(CP1)] $w_{i}^{\epsilon}$ converges to $w_{i}$ uniformly as $\epsilon\to 0+$.
    \item[(CP2)] $w_{i}^{\epsilon}$ is sufficiently regular, and satisfies the compatibility conditions for the approximating system \eqref{approsysappendix} up to the twentieth order.
    \item[(CP3)] $w_{i}^{\epsilon}$ satisfies the following uniform--in--$\epsilon$ bounds:
    \begin{equation}\label{approximatedataassume2}
        \begin{split}
            &(1-y)^{3/2}\lesssim w_{i}^{\epsilon}\lesssim (1-y)\log^{\frac{1}{2}}\bigg(\frac{10}{1-y}\bigg),\\
            &\mathcal{P}_{\epsilon}[w_{0}^{\epsilon},w_{1}^{\epsilon}]:=\left\|\frac{\p_{x}w_{0}^{\epsilon}}{w_{0}^{\epsilon}+\epsilon}\right\|_{L^{\infty}_{x,y}}+\left\|\frac{\mathcal{T}_{\epsilon}^{1}[w_{0}^{\epsilon}]}{w_{0}^{\epsilon}+\epsilon}\right\|_{L^{\infty}_{x,y}}+\left\|\frac{\p_{t}w_{1}^{\epsilon}}{w_{1}^{\epsilon}+\epsilon}\right\|_{L^{\infty}_{t,y}}+\left\|\frac{\mathcal{X}_{\epsilon}^{1}[w_{1}^{\epsilon}]}{w_{1}^{\epsilon}+\epsilon}\right\|_{L^{\infty}_{t,y}}\lesssim 1
    \end{split}
    \end{equation}
and
\begin{equation}\label{approximatedataassume3}
       \mathcal{E}[w_{0}^{\epsilon},w_{1}^{\epsilon}]:=\sum_{a+b= 2}\bigg(\left\|\frac{\p_{x}^{a}\mathcal{T}_{\epsilon}^{b}[w_{0}^{\epsilon}]}{w_{0}^{\epsilon}+\epsilon}\right\|_{L^{2}_{x,y}}+\left\|\frac{\p_{t}^{b}\mathcal{X}_{\epsilon}^{a}[w_{1}^{\epsilon}]}{w_{1}^{\epsilon}+\epsilon}\right\|_{L^{2}_{t,y}}\bigg)\lesssim 1.
\end{equation}
%\item[(CP4)] Higher--order bounds for $k\geq 3$...
\end{itemize}
%\textit{Construction of approximated data.} 

Our main results in this section are the following.
%We now present the construction of approximating data $(w_{0}^{\epsilon},w_{1}^{\epsilon})$ from the original data $(w_{0},w_{1})$ in the following propositions, and we present the detailed proof for completeness.
\begin{proposition}[\bf Construction of approximating Data (Classical Solution Regime)]
\label{constructappdata1}
Under Assumption \ref{dataassumecrocco}, the approximating initial and boundary data $(w_{0}^{\epsilon},w_{1}^{\epsilon})$ can be constructed from $(w_{0},w_{1})$ such that $(w_{0}^{\epsilon},w_{1}^{\epsilon})$ satisfy (CP1)-(CP3).
\end{proposition}

\begin{proposition}[\bf Construction of approximating Data (Weak Solution Regime)]
\label{constructappdata2}
Under Assumption \ref{dataassumecrocco4}, the approximating initial and boundary data $(w_{0}^{\epsilon},w_{1}^{\epsilon})$ can be constructed from $(w_{0},w_{1})$ such that $(w_{0}^{\epsilon},w_{1}^{\epsilon})$ satisfy (WP1)-(WP3).
\end{proposition}

The proofs of Propositions \ref{constructappdata1} and \ref{constructappdata2} are surprisingly subtle and complicated. We present the detailed proofs in the next subsection.

\subsection{Proof of the approximation propositions}
%In this subsection, we finish the proof of Proposition \ref{constructappdata1} and \ref{constructappdata2}.
 We start with the detailed proof of Proposition \ref{constructappdata1}. 
\begin{proof}[Proof of Proposition \ref{constructappdata1}]
We first consider the Gaussian case 
$$F(y)=(1-y)\log^{\frac{1}{2}}\left(\frac{10}{1-y}\right).$$

\textbf{Step 1: Smooth mollification.} By a standard mollifying procedure (using suitable extensions of $w^{(0)}$ and convolution with an approximate identity that is even in all variables), we can regularize $w^{(0)}(t,x,y)$ to get a smooth function $w^{(0)}_{\epsilon}(t,x,y)$ defined on $[0,1]\times[0,1]\times[0,1-2\epsilon]$, while preserving the Neumann boundary condition up to a second order error $\|\partial_yw^{(0)}_\epsilon|_{y=0}\|_{C^1}\lesssim\epsilon^2$. We assume that $0<\epsilon\ll1$. Here we choose the regularization scale to be exactly $\epsilon$. By direct calculations, we see that the bounds \eqref{matchingrate2}, \eqref{oneorderassume2}, and \eqref{higherassume2} still hold for $w^{(0)}_{\epsilon}(t,x,y)$ in $[0,1]\times[0,1]\times[0,1-2\epsilon]$. The following elementary estimates are useful in the construction below (denoting $R_{\epsilon}(t,x,y)=[t-\epsilon,t+\epsilon]\times[x-\epsilon,x+\epsilon]\times[y-\epsilon,y+\epsilon]\cap[0,1]^3$)
\begin{equation}\label{molliesti}
\begin{split}
&|f_{\epsilon}(t,x,y)-f(t,x,y)|\lesssim \min_{\alpha\in[0,2]}\epsilon^{\alpha}\|f\|_{C^{\alpha}(R_{\epsilon}(t,x,y))},\\
&|\nabla^{k}f_{\epsilon}(t,x,y)|\lesssim \epsilon^{l-k}\|f\|_{C^{l}(R_{\epsilon}(t,x,y))},\,\,\, 0\leq l<k.
\end{split}
\end{equation}

\textbf{Step 2: Linear Cap.} To match the compatibility conditions for the approximation system \eqref{approsysappendix} on the top wall $\{y=1\}$, we extend the smooth modified data $w^{(0)}_{\epsilon}$ from $y\in[0,1-2\epsilon]$ to $y\in[0,1]$ by a linear function near $\{y=1\}$. More precisely, define
\begin{equation}\label{linearcap}
W_{\epsilon}(t,x,y):=(1-y)\log^{\frac{1}{2}}\left(\frac{1}{\epsilon}\right)(1-\chi_{\epsilon}(1-y))+w_{\epsilon}^{(0)}(t,x,y)\chi_{\epsilon}(1-y).
\end{equation}
In the above, we use the refined cut-off function $\chi_{\epsilon}(\cdot)$ to handle the \emph{logarithmic loss} in the Gaussian case. $\chi_{\epsilon}$ is defined as
\begin{equation}
\chi_{\epsilon}(z):=\chi\left(100\cdot\frac{\log(z/\epsilon^{\frac{99}{100}})}{\log(1/\epsilon)}\right),
\end{equation}
where $\chi(\cdot)\in C^{\infty}(\R)$, with $\chi\equiv 0$ in $(-\infty,0]$ and $\chi\equiv 1$ in $[1,\infty)$. We have the following elementary properties of the cut-off function $\chi_{\epsilon}(z)$:
\begin{itemize}
\item [(i)] $\chi_{\epsilon}(z)\equiv 0$ whenever $z\leq \epsilon^{\frac{99}{100}}$, and $\chi_{\epsilon}\equiv 1$ whenever $z\geq \epsilon^{\frac{49}{50}}$;
\item [(ii)] $\chi_{\epsilon}(z)$ is smooth, with the following bounds for derivatives
\begin{equation}\label{boundcutoff}
|\chi_{\epsilon}^{(k)}(z)|\lesssim \frac{1}{z^{k}}\log^{-1}\frac{1}{\epsilon},\ \forall z\in [\epsilon^{\frac{99}{100}},\epsilon^{\frac{49}{50}}],\ k\geq 1.
\end{equation}
\end{itemize}
We can also check that the linear cap modification $W_{\epsilon}(t,x,y)$ satisfies the following properties:
\begin{itemize}
\item[(i)] $W_{\epsilon}(t,x,y)\equiv w_{\epsilon}^{(0)}(t,x,y)$ whenever $1-y\geq \epsilon^{\frac{49}{50}}$;
\item[(ii)] $\big|W_{\epsilon}(t,x,y)-w_{\epsilon}^{(0)}(t,x,y)\big|\lesssim (1-y)\log^{\frac{1}{2}}\left(\frac{1}{1-y}\right)\cdot\mathbf{1}_{1-y\leq\epsilon^{\frac{49}{50}}}$;
\item[(iii)] $W_{\epsilon}(t,x,y)\sim w_{\epsilon}^{(0)}(t,x,y)$ for all $y\in[0,1-2\epsilon]$;
\item[(iv)] $W_{\epsilon}(t,x,y)$ satisfies the bounds \eqref{oneorderassume2} and \eqref{higherassume2} in $[0,1]\times[0,1]\times[0,1]$.
\end{itemize}
Indeed, (i)-(iii) follow from the expression of $W_{\epsilon}(t,x,y)$ in \eqref{linearcap}, and (iv) follows from \eqref{linearcap}, \eqref{boundcutoff} and standard interpolation inequalities.

\textbf{Step 3: Modification Near the Wall.} Denote
\begin{equation}
F_{j,\epsilon}(x):=\Gamma_{j,\epsilon}[a_{\epsilon}](0,x),\ G_{j,\epsilon}(t):=\Gamma_{j,\epsilon}[a_{\epsilon}](t,0),\ a_{\epsilon}(t,x):=W_{\epsilon}(t,x,0),
\end{equation}
where $\Gamma_{j,\epsilon}[\cdot]$ are the wall jets as defined in \eqref{walljets}. Define also the wall defects
\begin{equation}\label{walldefects}
d_{0,j,\epsilon}(x):=F_{j,\epsilon}(x)-\p_{y}^{j}W_{\epsilon}(0,x,0), \quad
d_{1,j,\epsilon}(t):=G_{j,\epsilon}(t)-\p_{y}^{j}W_{\epsilon}(t,0,0).
\end{equation}
%The wall defects $d_{0,j,\epsilon}, \, d_{1,j,\epsilon}$ measure the extent to which $W_\epsilon$ deviates from the compatibility conditions for equation \eqref{approsysappendix} at $y=0$.

Next, define the wall-modified initial and inflow data
\begin{equation}\label{walldata}
\begin{split}
&f_{\epsilon}^{w}(x,y):=W_{\epsilon}(0,x,y)+\mathcal{R}_{1,\epsilon}(x,y):=W_{\epsilon}(0,x,y)+\sum_{j=1}^{200}\frac{y^{j}}{j!}\theta(\frac{y}{\rho_{j,\epsilon}})d_{0,j,\epsilon}(x),\\
&g_{\epsilon}^{w}(t,y):=W_{\epsilon}(t,0,y)+\mathcal{R}_{2,\epsilon}(t,y):=W_{\epsilon}(t,0,y)+\sum_{j=1}^{200}\frac{y^{j}}{j!}\theta(\frac{y}{\rho_{j,\epsilon}})d_{1,j,\epsilon}(t),
\end{split}
\end{equation}
where $\theta\in C_c^\infty(-2,2)$ is a standard even cut-off function; $\theta\equiv 1$ on $[0,1]$ and $\theta\equiv 0$ on $[2,\infty)$. The positive parameters $\rho_{j,\epsilon},\,1\leq j\leq 200$ depend on $\epsilon$, and will be chosen later. It follows from straightforward calculations that
\begin{equation}\label{wallcompati}
\p_{y}^{j}f_{\epsilon}^{w}(x,0)=F_{j,\epsilon}(x),\quad \p_{y}^{j}g_{\epsilon}^{w}(t,0)=G_{j,\epsilon}(t),\quad 0\leq j\leq 200.
\end{equation}
Hence $(f_{\epsilon}^{w},g_{\epsilon}^{w})$ satisfies compatibility conditions up to at least order $80$ on $\{y=0\}$. Furthermore, we can see that 
\begin{equation*}
f_{\epsilon}^{w}(x,y)=W_{\epsilon}^{w}(0,x,y),\quad g_{\epsilon}^{w}(t,y)=W_{\epsilon}^{w}(t,0,y),
\end{equation*}
where
\begin{equation}\label{jointextension}
W_{\epsilon}^{w}(t,x,y)=W_{\epsilon}(t,x,y)+\sum_{j=1}^{200}\frac{y^{j}}{j!}\theta(\frac{y}{\rho_{j,\epsilon}})\left(\Gamma_{j,\epsilon}[a_{\epsilon}](t,x)-\p_{y}^{j}W_{\epsilon}(t,x,0)\right).
\end{equation}

The following basic estimates for the defects $(d_{0,j,\epsilon},d_{1,j,\epsilon})$ are essential in our construction: For $i\in\{0,1\}$ and $j\geq 1$, using the notation that $\lfloor \sigma\rfloor:=\max\{l\in\Z: l\leq \sigma\}$ for $\sigma\in\R$, we have
\begin{itemize}
\item[(i)] $j=1$:
\begin{equation}\label{desti1}
\|d_{i,1,\epsilon}\|_{C^{1}}\lesssim \epsilon^{2},\quad \|d_{i,1,\epsilon}\|_{C^{k}}\lesssim \epsilon^{\frac{7}{2}-k}\ \text{for}\ k\ge 2.
\end{equation}
\item[(ii)] $2\leq j\leq 6$:
\begin{equation}\label{desti2}
\|d_{i,j,\epsilon}\|_{C^{\lfloor \frac{6-j}{2}\rfloor}}\lesssim \epsilon,\quad \|d_{i,j,\epsilon}\|_{C^{k}}\lesssim \epsilon^{4-\frac{j}{2}-k}\ \text{for}\ k>\lfloor \frac{6-j}{2}\rfloor.
\end{equation}
\item[(iii)] $j=7$:
\begin{equation}\label{desti3}
|d_{i,7,\epsilon}|\lesssim \epsilon^{\frac{1}{2}},\quad \|d_{i,7,\epsilon}\|_{C^{k}}\lesssim \epsilon^{-k}\ \text{for}\ k\geq 1.
\end{equation}
\item[(iv)] $j\geq 8$:
\begin{equation}\label{desti4}
\|d_{i,j,\epsilon}\|_{C^{k}}\lesssim \epsilon^{8-k-j},\ \text{for}\ k\geq 0.
\end{equation}
\end{itemize}
Given their importance in our construction, we will give a detailed proof of \eqref{desti1}-\eqref{desti4} below in Section \ref{sec:pofdersti}.
\begin{comment}
Indeed, the above estimates follow from the expression of the wall jets \eqref{walljets}, the elementary mollifier estimates \eqref{molliesti}, anisotropic regularity assumptions \eqref{oneorderassume2}-\eqref{higherassume2} on $w^{(0)}$, and the fact that for $j\leq 7$, $\Gamma_{j,0}[w^{(0)}(t,x,0)]-\p_{y}^{j}w^{(0)}(t,x,0)=0$ implying 
\begin{equation}
\begin{split}
d_{j,\epsilon}(t,x)&:=\Gamma_{j,\epsilon}[a_{\epsilon}](t,x)-\p_{y}^{j}W_{\epsilon}(t,x,0)\\
&=\left(\Gamma_{j,\epsilon}[a_{\epsilon}]-\Gamma_{j,0}[w^{(0)}(t,x,0)]\right)-\left(\p_{y}^{j}W_{\epsilon}(t,x,0)-\p_{y}^{j}w^{(0)}(t,x,0)\right).
\end{split}
\end{equation}
\end{comment}

\textbf{Step 4: Modification Near the Corner.} So far, we have fulfilled the required compatibility conditions (see Definition \ref{compatibilitydef}) on the walls $\{y=0\}$ and $\{y=1\}$. To fulfill the remaining compatibility at the corner $(t,x)=(0,0)$, we define
\begin{equation}
\Theta_{m,\epsilon}(y):=\mathcal{T}_{\epsilon}^{m}[f_{\epsilon}^{w}(x,y)]\big|_{x=0},\quad 1\leq m\leq 50,
\end{equation}
and the corner compatibility defects 
\begin{equation}
c_{m,\epsilon}(y):=\Theta_{m,\epsilon}(y)-\p_{t}^{m}g_{\epsilon}^{w}(0,y),\quad 1\leq m\leq 50.
\end{equation}
Now define the corrected inflow data
\begin{equation}\label{cornermodify}
g_{\epsilon}^{c}:=g_{\epsilon}^{w}(t,y)+\mathcal{R}_{3,\epsilon}(t,y):=g_{\epsilon}^{w}(t,y)+\sum_{m=1}^{50}\frac{t^{m}}{m!}\theta(\frac{t}{\sigma_{m,\epsilon}})c_{m,\epsilon}(y).
\end{equation}
Here the positive parameters $\sigma_{m,\epsilon}$ for the modification scales depend on $\epsilon$, and will be chosen later. 

We have the following identities, which imply that the modified data $(f_{\epsilon}^{w},g_{\epsilon}^{c})$ satisfy the required compatibility conditions in Definition \ref{compatibilitydef}. Indeed, the first statement below implies that $g_{\epsilon}^{c}$ still satisfies the compatibility conditions on the wall $y=0$; and the second statement implies that the compatibility conditions at the corner are fulfilled.
\begin{itemize}
\item[(i)] For integers $j+2m\leq 200$, we have
\begin{equation}
\p_{y}^{j}\Theta_{m,\epsilon}(0)=\p_{t}^{m}\Gamma_{j,\epsilon}[a_{\epsilon}](0,0).
\end{equation}
In particular,
\begin{equation}
\p_{y}^{j}c_{m,\epsilon}(0)=0,\quad j+2m\leq 200.
\end{equation}
\item[(ii)] For integers $\alpha+\beta\leq 30$,
\begin{equation}\label{a.33}
\p_{x}^{\beta}\mathcal{T}_{\epsilon}^{\alpha}[f_{\epsilon}^{w}]\big|_{x=0} = \p_{t}^{\alpha}\mathcal{X}_{\epsilon}^{\beta}[g_{\epsilon}^{c}]\big|_{t=0}.
\end{equation}
\end{itemize}
The above identities follow from the expressions for the boundary jets \eqref{initialjets}, \eqref{inflowjets}, \eqref{walljets}, and inductive calculations.

We use the following harmless restrictions, which are satisfied by the
choices in \eqref{cop0.1} for all sufficiently small $\eps$:
\begin{equation}\label{a.35a}
  0<\rho_{j,\eps}\leq\frac14,
  \qquad 1\leq j\leq200,
  \qquad
  \sum_{j=1}^{200}\rho_{j,\eps}^{j-4}
       |d_{0,j,\eps}(0)|\leq1.
\end{equation}

We also present the basic estimates for $c_{m,\epsilon}(y)$, which are essential for the choice of parameters. The main estimates we need are the quantitative estimates for $c_{1,\epsilon}(y)$ as follows:
\begin{itemize}
\item[(i)] Estimation near $y=1$:
\begin{equation}\label{cesti1}
\begin{cases}
&|c_{1,\epsilon}(y)|\lesssim (1-y)\log^{\frac{1}{2}}\left(\frac{10}{1-y}\right)\mathbf{1}_{\epsilon^{\frac{99}{100}}\leq 1-y \leq \epsilon^{\frac{49}{50}}}+\epsilon^{\frac{199}{200}},\ y\in[\frac{1}{2},1],\\
&\sum_{k=1}^{6}\|(1-y)^{k-1}\p_{y}^{k}c_{1,\epsilon}(y)\|_{L^{\infty}[\frac{1}{2},1]}\lesssim \epsilon^{-\frac{1}{400}},\\
&\|\p_{y}^{k}c_{1,\epsilon}(y)\|_{L^{\infty}[\frac{1}{2},1]}\lesssim \epsilon^{-k},\ k\geq 7.
\end{cases}
\end{equation}
\item[(ii)] Estimation near $y=0$: for $y\in[0,1/2]$ and $0\leq k\leq 6$,
\begin{align}
 |\partial_y^kc_{1,\eps}(y)|
 \lesssim{}&\beta_{k,\eps}
 +\sum_{j=1}^{200}\rho_{j,\eps}^{\,j-k-2}
 \Bigl(
 |d_{0,j,\eps}(0)|
 +\rho_{j,\eps}^2|d'_{1,j,\eps}(0)|
 +\rho_{j,\eps}^2(\rho_{j,\eps}+\eps)
  |d'_{0,j,\eps}(0)|
 \Bigr)
 \mathbf{1}_{\{0\leq y\leq2\rho_{j,\eps}\}},
\label{cesti2}
\end{align}
\begin{comment}
\begin{equation}\label{cesti2}
|\p_{y}^{k}c_{1,\epsilon}|\lesssim c_{k,\epsilon}+\sum_{j=1}^{200}\rho_{j,\epsilon}^{j-k-2}(|d_{0,j,\epsilon}(0)|+\rho_{j,\epsilon}^{2}|d'_{1,j,\epsilon}(0)|)\mathbb{I}_{0\leq y\leq 2\rho_{j,\epsilon}},
\end{equation}\end{comment}
where
\begin{equation*}
\beta_{k,\epsilon}=
\begin{cases}
\epsilon,\ k\leq 4,\\
\epsilon^{\frac{1}{2}},\ k=5,\\
1,\ k=6.
%,\\
%\epsilon^{-(k-6)},\ k>6.
\end{cases}
\end{equation*}
\end{itemize}

In view of the importance of the bounds \eqref{cesti1}-\eqref{cesti2}, we will present a detailed proof below in Section \ref{sec:pcersti1}.

\textbf{Step 5: Choice of Parameters.} Using the above constructions, we obtain the modified data $(f_{\epsilon}^{w},g_{\epsilon}^{c})$, which satisfy the required compatibility conditions as in Definition \ref{compatibilitydef}. We still need to choose the parameters $\{\rho_{j,\epsilon}\}$ and $\{\sigma_{m,\epsilon}\}$ such that the required bounds \eqref{approximatedataassume2} and \eqref{approximatedataassume3} hold.

Substituting the expressions for $(f_{\epsilon}^{w},g_{\epsilon}^{c})$ into \eqref{approximatedataassume2} and \eqref{approximatedataassume3}, we need
\begin{equation}\label{finalrequire}
\begin{split}
&\|\p_{x}^{2}\mathcal{R}_{1,\epsilon}\|_{L^{2}_{x,y}}+\|\p_{y}^{4}\mathcal{R}_{1,\epsilon}\|_{L^{2}_{x,y}}+\|\p_{t}^{2}\p_{y}^{2}\mathcal{R}_{2,\epsilon}\|_{L^{2}_{t,y}}+\|\p_{y}^{6}\mathcal{R}_{2,\epsilon}\|_{L^{2}_{t,y}}\\
&\quad+\|\p_{t}^{2}\p_{y}^{2}\mathcal{R}_{3,\epsilon}\|_{L^{2}_{t,y}([0,1]\times[0,\frac{1}{2}])}+\|\p_{y}^{6}\mathcal{R}_{3,\epsilon}\|_{L^{2}_{t,y}([0,1]\times[0,\frac{1}{2}])}\\
&\quad+\left\|\frac{\p_{t}^{2}\mathcal{R}_{3,\epsilon}}{(1-y)\log^{\frac{1}{2}}\left(\frac{10}{1-y}\right)}\right\|_{L^{2}_{t,y}}+\left\|(1-y)^{3}\log^{\frac{3}{2}}\left(\frac{10}{1-y}\right)\p_{y}^{4}\mathcal{R}_{3,\epsilon}\right\|_{L^{2}_{t,y}}\\
&\quad+\left\|\frac{\p_{t}\mathcal{R}_{3,\epsilon}}{(1-y)\log^{\frac{1}{2}}\left(\frac{10}{1-y}\right)}\right\|_{L^{\infty}_{t,y}}+\left\|(1-y)\log^{\frac{1}{2}}\left(\frac{10}{1-y}\right)\p_{y}^{2}\mathcal{R}_{3,\epsilon}\right\|_{L^{\infty}_{t,y}}\lesssim 1.
\end{split}
\end{equation}
To justify the above reduction, we use, for $s=x$ or $t$, the
anisotropic Sobolev estimate
\[
 \|\partial_s h\|_{L^\infty}
 +\|\partial_y^2h\|_{L^\infty}
 \lesssim
 \|h\|_{L^2}
 +\|\partial_s^2h\|_{L^2}
 +\|\partial_y^4h\|_{L^2},
\]
together with its versions applied to normal derivatives of $h$ and
standard one-dimensional interpolation and trace estimates.  Near
$y=0$, we also use the shifted Hardy inequalities for $q\in\{2,\infty\}$,
\begin{align*}
 &\left\|\frac{h}{y+\epsilon}\right\|_{L^q(0,1/2)}
 \lesssim
 \|\partial_yh\|_{L^q(0,1/2)}
 +\epsilon^{\frac1q-1}|h(0)|,\\
 &\left\|\p_{y}^{k}\left(\frac{h}{y+\epsilon}\right)\right\|_{L^{q}(0,1/2)}\lesssim |b_{k,\epsilon}|\cdot\epsilon^{\frac{1}{q}}+\|\p_{y}^{k+1}h\|_{L^{q}(0,1/2)},\ k\geq 1,\\
 &b_{k,\epsilon}:=\frac{1}{\epsilon^{k+1}}\sum_{i=0}^{k}(-1)^{k-i}\frac{h^{(i)}(0)\epsilon^{i}}{i!}.
\end{align*}

The boundary traces in these estimates and all lower-order terms are
controlled by \eqref{wallcompati}--\eqref{a.33} and the
preceding uniform smallness estimates. Consequently, the $L^2$
quantities in \eqref{finalrequire} control the required $L^\infty$
bounds for $R_{1,\epsilon}$, $R_{2,\epsilon}$, and the near-wall part of
$R_{3,\epsilon}$, while the two weighted $L^\infty$ bounds near $y=1$
are included explicitly in \eqref{finalrequire}.  Hence
\eqref{finalrequire} implies \eqref{approximatedataassume2}--\eqref{approximatedataassume3}.

\begin{comment}
    
In the above, we have used Hardy's inequality to handle the extra denominator $\frac{1}{y+\epsilon}$ appearing in the inflow data.
\end{comment}
In view of \eqref{desti1}-\eqref{desti4} and \eqref{cesti1}-\eqref{cesti2}, we can choose the parameters as follows:
\begin{equation}\label{cop0.1}
\begin{split}
&\rho_{1,\epsilon}=\epsilon^{0.3},\ \rho_{2,\epsilon}=\epsilon^{0.2},\ \rho_{3,\epsilon}=\epsilon^{0.3},\ \rho_{4,\epsilon}=\epsilon^{0.4};\\
&\rho_{5,\epsilon}=\epsilon^{0.3},\ \rho_{6,\epsilon}=\epsilon^{0.3},\ \rho_{7,\epsilon}=\epsilon^{0.4},\ \sigma_{1,\epsilon}=\epsilon^{0.6}.
\end{split}
\end{equation}
For other parameters $\{\rho_{j,\epsilon}\}_{j\geq 8}$ and $\{\sigma_{m,\epsilon}\}_{m\geq 2}$, since they always contribute positive powers in the expression \eqref{finalrequire}, we can choose them to be sufficiently small so that \eqref{finalrequire} holds. Hence we finish the proof of Proposition \ref{constructappdata1} when $F(\eta)=(1-\eta)\log^{\frac{1}{2}}\left(\frac{10}{1-\eta}\right)$.

For the polynomial and exponential cases $F(\eta)=(1-\eta)^{m}(1\leq m\leq \frac{3}{2})$, the construction follows from the same approach. The main modification is that in the linear cap procedure (Step 2 above), we construct $W_{\epsilon}(t,x,y)$ by defining
\begin{equation}
    W_{\epsilon}(t,x,y):=(1-y)r_{\epsilon}^{m-1}\theta\left(\frac{1-y}{r_{\epsilon}}\right)+w_{\epsilon}^{(0)}(t,x,y)\left[1-\theta\left(\frac{1-y}{r_{\epsilon}}\right)\right],\ r_{\epsilon}:=\epsilon^{\frac{1}{m}-\frac{1}{100}}.
\end{equation}
The parameters $\{\rho_{j,\epsilon}\}$ and $\{\sigma_{m,\epsilon}\}$ are determined in the same way as in the above argument. The proof of Proposition \ref{constructappdata1} is now complete.
\end{proof}

Finally we present the proof of Proposition \ref{constructappdata2}.

\begin{proof}[Proof of Proposition \ref{constructappdata2}]
The proof of Proposition \ref{constructappdata2} is similar to that of Proposition \ref{constructappdata1} but simpler since we only need to fulfill the weaker bounds \eqref{approximatedataassume1}. Hence we only sketch the main points.

First we regularize $w^{(0)}(t,x,y)$ to obtain a smooth function $w^{(0)}_{\epsilon}(t,x,y)$ defined on $[0,1]\times[0,1]\times[0,1-\epsilon^{3}]$, and the vanishing rates in \eqref{approximatedataassume1} also hold for $w^{(0)}_{\epsilon}(t,x,y)$. We can also require that $\p_{y}w^{(0)}_{\epsilon}|_{y=0}$ remains zero by using an even extension of $w^{(0)}$ before the regularization. Notice that even extensions keep the $C^{2}_{y}$ regularity, which is sufficient in the weak solution regime, in contrast to the classical solution case. Furthermore, thanks to \eqref{regular4}, we can choose the mollification scale to be sufficiently small such that for every multi-index $(\alpha_1,\alpha_2,\alpha_3)$ with $2\alpha_1+2\alpha_2+\alpha_3\leq 2$,
\begin{equation}
\Big\|\partial_t^{\alpha_1}\partial_x^{\alpha_2}\partial_y^{\alpha_3}\big(w^{(0)}_{\epsilon}(t,x,y)-w^{(0)}(t,x,y)\big)\Big\|_{L^\infty([0,1]\times[0,1]\times[0,1-\epsilon^{2}])}\lesssim\epsilon^{2}.
\end{equation}
In particular, the initial and inflow traces of $w^{(0)}_{\epsilon}(t,x,y)$ satisfy \eqref{approximatedataassume1}.

Next, for the linear cap procedure, define
\begin{equation}
    W_{\epsilon}(t,x,y)=(1-y)\log^{\frac{1}{2}}\left({\frac{1}{\epsilon}}\right)\theta\left(\frac{1-y}{\epsilon^{2}}\right)+w^{(0)}_{\epsilon}(t,x,y)\left[1-\theta\left(\frac{1-y}{\epsilon^{2}}\right)\right].
\end{equation}
A direct calculation shows that the initial and inflow traces of  $W_{\epsilon}(t,x,y)$ still satisfy the bounds \eqref{approximatedataassume1}.

Finally, as in the proof of Proposition \ref{constructappdata1}, we modify the data near the wall and the corner to fulfill the compatibility conditions (see \eqref{walldata} and \eqref{cornermodify}). Notice that the parameters $\{\rho_{j,\epsilon}\}$ and $\{\sigma_{m,\epsilon}\}$ contribute positive powers in \eqref{approximatedataassume1}. We can shrink $\{\rho_{j,\epsilon}\}$ and $\{\sigma_{m,\epsilon}\}$ sufficiently small to preserve the required bounds \eqref{approximatedataassume1}. We conclude the proof of Proposition \ref{constructappdata2}.
\end{proof}

\subsubsection{Proof of the bounds \eqref{desti1}-\eqref{desti4}}\label{sec:pofdersti}

The extensions in \textbf{Step~1} are chosen so that the wall-jet identities
\[
 \Gamma_{j,0}[w^{(0)}(t,x,0)]
   =\partial_y^j w^{(0)}(t,x,0),
 \qquad 1\leq j\leq7,
\]
continue to hold on the support of the tangential mollification. 
Since the cap modification is inactive near \(y=0\), we have
\(W_\epsilon=w_\epsilon^{(0)}\) there.  Define the joint wall defect
\[
 d_{j,\epsilon}(t,x)
 :=
 \Gamma_{j,\epsilon}[a_\epsilon](t,x)
 -\partial_y^jW_\epsilon(t,x,0).
\]
Then
\[
 d_{0,j,\epsilon}(x)=d_{j,\epsilon}(0,x),
 \qquad
 d_{1,j,\epsilon}(t)=d_{j,\epsilon}(t,0).
\]
When applied to a wall trace below,
\(\varphi_\epsilon*_{t,x}\) denotes convolution only in \(t,x\), with $\varphi_\epsilon$ being the mollifier used in the regularization procedure.
Adding and subtracting the tangential mollification of the preceding
wall-jet identity gives
\setcounter{equation}{41}
\begin{align}
 d_{j,\epsilon}(t,x)
 ={}&
 \Big(
   \Gamma_{j,\epsilon}[a_\epsilon](t,x)
   -\varphi_\epsilon*_{t,x}
      \Gamma_{j,0}[w^{(0)}(\mathord\cdot,\mathord\cdot,0)](t,x)
 \Big)
 \notag\\
 &+
 \Big(
   \varphi_\epsilon*_{t,x}
      \partial_y^jw^{(0)}(\mathord\cdot,\mathord\cdot,0)(t,x)
   -\partial_y^jW_\epsilon(t,x,0)
 \Big).
 \label{defect-splitting}
\end{align}

We first consider \(j=1\). Since
\(\Gamma_{1,\epsilon}[a_\epsilon]=0\) and
\(\partial_yw^{(0)}|_{y=0}=0\), Taylor expansion in \(y\), together with
the vanishing first moment of the normal mollifier, gives
\[
 \|d_{i,1,\epsilon}\|_{C^2}\lesssim\epsilon^2.
\]
Moving the remaining tangential derivatives onto the mollifier gives
\[
 \|d_{i,1,\epsilon}\|_{C^k}
 \lesssim\epsilon^{4-k},\qquad k\geq2,
\]
which is stronger than \eqref{desti1}.

For even \(j\in\{2,4,6\}\), the regularity assumption \eqref{regular2}, the standard
mollifier estimates, and induction in the wall-jet recursion \eqref{walljets} give,
for \(i\in\{0,1\}\),
\begin{equation}
  \|d_{i,j,\epsilon}\|_{C^k}
  \lesssim \epsilon+\epsilon^{4-\frac j2-k},
  \qquad k\geq0.
  \label{even-defect-bound}
\end{equation}
Indeed, \(\Gamma_{j,0}\) has tangential differential order \(j/2\).
If \(k\leq(6-j)/2\), then \(j/2+k\leq3\), so \eqref{regular2} leaves one
tangential derivative for the wall-jet commutator; it also leaves the
normal derivative needed to estimate the normal-mollification error.
Both terms in \eqref{defect-splitting} are therefore \(O(\epsilon)\). For larger \(k\), the
remaining tangential derivatives are placed on the mollifier.

Now let \(j\in\{3,5,7\}\) be odd and put \(m=(7-j)/2\). Since
\(2m+j+1=8\), \eqref{regular2} and the normal mollifier estimate imply
\begin{equation}
 \left\|
  \varphi_\epsilon*_{t,x}\partial_y^j
       w^{(0)}(\mathord\cdot,\mathord\cdot,0)
  -\partial_y^jW_\epsilon(\mathord\cdot,\mathord\cdot,0)
 \right\|_{C^m}
 \lesssim\epsilon.
 \label{odd-normal-mollification}
\end{equation}
On the other hand, \(\Gamma_{j,0}\) has tangential differential order
\((j-1)/2\), and
\[
  \frac{j-1}{2}+m=3.
\]
Hence \eqref{regular2} supplies one further tangential derivative of the wall trace.
The product--mollifier commutator estimates and induction in \eqref{walljets}, using
\(a_\epsilon\gtrsim1\) and the explicit factors of \(\epsilon\) in that
recursion, therefore give
\begin{equation}
 \left\|
  \Gamma_{j,\epsilon}[a_\epsilon]
  -\varphi_\epsilon*_{t,x}
     \Gamma_{j,0}[w^{(0)}(\mathord\cdot,\mathord\cdot,0)]
 \right\|_{C^m}
 \lesssim\epsilon.
 \label{odd-wall-jet}
\end{equation}
Combining \eqref{defect-splitting}, \eqref{odd-normal-mollification}, and
\eqref{odd-wall-jet}, and placing any additional tangential derivatives on
the mollifier, yields
\begin{equation}
 \|d_{i,j,\epsilon}\|_{C^{(7-j)/2}}\lesssim\epsilon,
 \qquad
 \|d_{i,j,\epsilon}\|_{C^k}
 \lesssim\epsilon^{\frac{9-j}{2}-k},
 \quad k>\frac{7-j}{2}.
 \label{odd-defect-bound}
\end{equation}
For \(j=3,5\), \eqref{odd-defect-bound} is stronger than \eqref{desti2}, while for
\(j=7\) it is stronger than \eqref{desti3}. Together with
\eqref{even-defect-bound}, this proves \eqref{desti2}--\eqref{desti3}.

Finally, let \(j\geq8\). Directly from the regularity assumption \eqref{regular2} and
the mollifier estimates,
\[
 \left\|
 \partial_s^k\partial_y^jW_\epsilon(\cdot,\cdot,0)
 \right\|_{L^\infty}
 \lesssim\epsilon^{8-j-k},
 \qquad s=t\ \text{or}\ x.
\]

Moreover, \(\Gamma_{j,\epsilon}[a_\epsilon]\) has tangential differential
order at most \(\lfloor j/2\rfloor\). Hence induction in \eqref{walljets} gives
\[
 \left\|
 \partial_s^k\Gamma_{j,\epsilon}[a_\epsilon]
 \right\|_{L^\infty}
 \lesssim_j
 \epsilon^{4-\lfloor j/2\rfloor-k}
 \lesssim_j
 \epsilon^{8-j-k},
 \qquad j\geq8.
\]
Together with the definition of \(d_{i,j,\epsilon}\), this proves
\eqref{desti4}. The constants may depend on \(j\), which is harmless since
only \(j\leq200\) is used. \hfill\(\Box\)

\subsubsection{Proof of the bounds \eqref{cesti1}-\eqref{cesti2}}\label{sec:pcersti1}

By the first condition in \eqref{a.35a}, all the corrections in~(A.23) vanish
for $y\in[1/2,1]$.  Thus $W_\eps^w=W_\eps$ on this interval, and
definitions~(A.29)--(A.30) give
\begin{equation*}
 c_{1,\eps}(y)
 =\left[
 (W_\eps+\eps)^2\partial_y^2W_\eps
 -(y+\eps)\partial_xW_\eps
 -\partial_tW_\eps
 \right]_{t=x=0},
 \qquad \frac12\leq y\leq1.
 %\tag{A.35b}
\end{equation*}
Set $z=1-y$.  Compatibility of the original data implies
\begin{equation}\label{A.35c}
 \left[
 (w^{(0)})^2\partial_y^2w^{(0)}
 -y\partial_xw^{(0)}-\partial_tw^{(0)}
 \right]_{t=x=0}=0.
\end{equation}
Subtracting the mollification of \eqref{A.35c} from the corresponding expression
for $w_\eps^{(0)}$ gives the exact identity
\begin{align}\label{A.35d}
 &(w_\eps^{(0)}+\eps)^2\partial_y^2w_\eps^{(0)}
 -(y+\eps)\partial_xw_\eps^{(0)}
 -\partial_tw_\eps^{(0)}                                     -\varphi_\eps*
 \left((w^{(0)})^2\partial_y^2w^{(0)}
       -y\partial_xw^{(0)}-\partial_tw^{(0)}\right)             \notag\\
 &=(w_\eps^{(0)})^2\partial_y^2w_\eps^{(0)}
 -\varphi_\eps*\left((w^{(0)})^2\partial_y^2w^{(0)}\right)\notag\\
 &\quad\,\,+\varphi_\eps*(y\partial_xw^{(0)})
 -y\partial_xw_\eps^{(0)}
 -\eps\partial_xw_\eps^{(0)}
 +(2\eps w_\eps^{(0)}+\eps^2)\partial_y^2w_\eps^{(0)}.
\end{align}
The centered mollifier estimate, together with \eqref{matchingrate2}, \eqref{oneorderassume2} and \eqref{higherassume2},
therefore gives, whenever $z\geq\eps^{99/100}$,
\begin{align}
 &\left|
 \left[
 (w_\eps^{(0)}+\eps)^2\partial_y^2w_\eps^{(0)}
 -(y+\eps)\partial_xw_\eps^{(0)}
 -\partial_tw_\eps^{(0)}
 \right]_{t=x=0}
 \right|
 \lesssim
 \eps z^{-1/500}\log^{C_0}\!\left(\frac{10}{z}\right),
 \label{A.35e}\\
 &\left|
 \partial_y^k
 \left[
 (w_\eps^{(0)}+\eps)^2\partial_y^2w_\eps^{(0)}
 -(y+\eps)\partial_xw_\eps^{(0)}
 -\partial_tw_\eps^{(0)}
 \right]_{t=x=0}
 \right|
 \lesssim
 z^{1-k-1/500}\log^{C_0}\!\left(\frac{10}{z}\right),
 \quad 1\leq k\leq6.  \label{A.35f}
\end{align}
Here $C_0$ is a fixed constant. Indeed, (1.25) controls
$\partial_tw^{(0)}$, $\partial_xw^{(0)}$, and
$w^{(0)}\partial_y^2w^{(0)}$ without a power loss in $z$.  \eqref{higherassume2} 
controls the higher even normal derivatives, and interpolation gives the
odd derivatives.  In each Leibniz product of total anisotropic order at
most eight, the loss $z^{-1/1000}$ is needed in at most two factors, which
gives \eqref{A.35f}.  \eqref{A.35e} follows by using one additional
first-order mollifier difference in \eqref{A.35d}.

On the transition region
\[
 \eps^{99/100}\leq z\leq\eps^{49/50},
\]
formula~(A.16) reads
\[
 W_\eps
 =z\log^{1/2}\!\frac1\eps
 +\chi_\eps(z)
 \left(w_\eps^{(0)}-z\log^{1/2}\!\frac1\eps\right).
\]
Since the first term is affine in $y$ and independent of $t,x$, direct
differentiation gives
\begin{align}
 &(W_\eps+\eps)^2\partial_y^2W_\eps
 -(y+\eps)\partial_xW_\eps-\partial_tW_\eps                    \notag\\
&=\chi_\eps
 \left[
 (w_\eps^{(0)}+\eps)^2\partial_y^2w_\eps^{(0)}
 -(y+\eps)\partial_xw_\eps^{(0)}
 -\partial_tw_\eps^{(0)}
 \right]                                                        \notag\\
&\quad+\chi_\eps
 \left[(W_\eps+\eps)^2-(w_\eps^{(0)}+\eps)^2\right]
 \partial_y^2w_\eps^{(0)}                                     \notag\\
&\quad+(W_\eps+\eps)^2
 \left[
 2\chi_\eps'(z)\partial_z
 \left(w_\eps^{(0)}-z\log^{1/2}\!\frac1\eps\right)
 +\chi_\eps''(z)
 \left(w_\eps^{(0)}-z\log^{1/2}\!\frac1\eps\right)
 \right].
 \label{A.35g}
\end{align}
On this region,
\[
 w_\eps^{(0)}\sim W_\eps
 \sim z\log^{1/2}\!\left(\frac{10}{z}\right),
 \qquad
 |\chi_\eps^{(k)}(z)|
 \lesssim\frac{z^{-k}}{\log(1/\eps)}.
\]
Consequently, \eqref{oneorderassume2} and \eqref{boundcutoff} bound the last two lines of \eqref{A.35g} by
\[
 z\log^{1/2}\!\left(\frac{10}{z}\right)
 \mathbf{1}_{\{\eps^{99/100}\leq z\leq\eps^{49/50}\}}.
\]
Together with \eqref{A.35e} and
\[
 \eps z^{-1/500}\log^{C_0}\!\left(\frac{10}{z}\right)
 \lesssim\eps^{199/200},
 \qquad z\geq\eps^{99/100},
\]
this proves the first line of \eqref{cesti1}.

Differentiating \eqref{A.35g} and using \eqref{A.35f} gives
\[
 |\partial_y^kc_{1,\eps}(y)|
 \lesssim
 z^{1-k-1/500}\log^{C_0}\!\left(\frac{10}{z}\right),
 \qquad 1\leq k\leq6,\quad z\geq\eps^{99/100}.
\]
If $z\leq\eps^{99/100}$, then
$W_\eps=z\log^{1/2}(1/\eps)$, and hence $c_{1,\eps}=0$ exactly.  Since
\[
 \sup_{z\geq\eps^{99/100}}
 z^{-1/500}\log^{C_0}\!\left(\frac{10}{z}\right)
 \lesssim\eps^{-1/400},
\]
the second line of \eqref{cesti1} follows.

Finally, after six normal derivatives, each additional derivative can be
placed on the scale-$\eps$ mollifier or on $\chi_\eps$, at a cost of at most
$\eps^{-1}$.  Hence, for $k\geq7$ and $z\geq\eps^{99/100}$,
\[
 |\partial_y^kc_{1,\eps}(y)|
 \lesssim_k
 \eps^{-(k-6)}z^{-5-1/500}
 \log^{C_0}\!\left(\frac{10}{z}\right)
 \lesssim_k\eps^{-k}.
\]
Together with the exact vanishing in the cap core, this proves the third
line of \eqref{cesti1}.

We prove the estimate on \(0\leq y\leq1/2\) for
\(0\leq k\leq6\), which is the range stated in \eqref{cesti2}.  In this region the cap is inactive and
$W_\eps=w_\eps^{(0)}$.  At $(t,x)=(0,0)$, direct expansion gives
\begin{align}
 c_{1,\eps}
={}&
 (W_\eps+\eps)^2\partial_y^2W_\eps
 -(y+\eps)\partial_xW_\eps
 -\partial_tW_\eps                                                \notag\\
&+(W_\eps+\eps)^2\partial_y^2(W_\eps^w-W_\eps)
 -(y+\eps)\partial_x(W_\eps^w-W_\eps)
 -\partial_t(W_\eps^w-W_\eps)                                 \notag\\
&+\left[
 2(W_\eps+\eps)(W_\eps^w-W_\eps)
 +(W_\eps^w-W_\eps)^2
 \right]\partial_y^2W_\eps^w.
 \label{A.36a}
\end{align}
Compatibility at the corner, the anisotropic regularity \eqref{regular2}, and the
centered mollifier estimates imply
\begin{equation}
 \left|
 \partial_y^k\left[
 (W_\eps+\eps)^2\partial_y^2W_\eps
 -(y+\eps)\partial_xW_\eps
 -\partial_tW_\eps
 \right]_{t=x=0}
 \right|
 \lesssim \beta_{k,\eps}.
 \label{A.36b}
\end{equation}
Indeed, the left-hand side is \(O(\epsilon)\) for \(0\leq k\leq4\); it is \(O(1)\) for \(k=6\), and interpolation
between these two bounds gives \(O(\epsilon^{1/2})\) for \(k=5\).

For the cutoff monomials in \eqref{jointextension},
\begin{equation}
 \left|
 \partial_y^\ell\left[
 \frac{y^j}{j!}\theta\left(\frac{y}{\rho_{j,\eps}}\right)
 \right]
 \right|
 \lesssim
 \rho_{j,\eps}^{j-\ell}
 \mathbf{1}_{\{0\leq y\leq2\rho_{j,\eps}\}}.
 \label{A.36c}
\end{equation}
Moreover, at the corner,
\begin{equation}
 \begin{aligned}
 d_{j,\eps}(0,0)
 &=d_{0,j,\eps}(0)=d_{1,j,\eps}(0),\\
 \partial_xd_{j,\eps}(0,0)&=d'_{0,j,\eps}(0),\\
 \partial_td_{j,\eps}(0,0)&=d'_{1,j,\eps}(0).
 \end{aligned}
 \label{A.36d}
\end{equation}
Applying \eqref{A.36c}--\eqref{A.36d} to the three linear correction terms in the second line of 
\eqref{A.36a} gives
\begin{align*}
 &\rho_{j,\eps}^{j-k-2}|d_{0,j,\eps}(0)|,\\
 & (\rho_{j,\eps}+\eps)\rho_{j,\eps}^{j-k}
       |d'_{0,j,\eps}(0)|,\\
 &\rho_{j,\eps}^{j-k}|d'_{1,j,\eps}(0)|,
\end{align*}
respectively. The second bound also covers the terms in which a normal derivative falls on
$y+\eps$.

It remains to control the last line of \eqref{A.36a}. We rewrite the last line of \eqref{A.36a} by
\begin{equation}\label{lastlinea36}
\begin{split}
    &\left[
 2(W_\eps+\eps)(W_\eps^w-W_\eps)
 +(W_\eps^w-W_\eps)^2
 \right]\partial_y^2W_\eps^w\\
 &=2(W_{\epsilon}+\epsilon)\p_{y}^{2}W^{\epsilon}(W_{\epsilon}^{w}-W_{\epsilon})+(W_{\epsilon}^{w}-W_{\epsilon})^{2}\p_{y}^{2}W_{\epsilon}\\
 &\quad+\left[2(W_\eps+\eps)(W_\eps^w-W_\eps)
 +(W_\eps^w-W_\eps)^2
 \right]\partial_y^2(W_\eps^w-W_\eps).
\end{split}
\end{equation}
The first term in \eqref{lastlinea36} is linear with respect to $W_{\epsilon}^{w}-W_{\epsilon}$, which can be estimated similarly to the second line of \eqref{A.36a}.

To estimate the nonlinear terms in \eqref{lastlinea36}, from \eqref{jointextension}, \eqref{a.35a}
and \eqref{A.36c}, we obtain that
\[
 \sum_{\ell=0}^4
 \|\partial_y^\ell(W_\eps^w-W_\eps)(0,0,\cdot)\|_{L^\infty}
 \lesssim
 \sum_{j=1}^{200}\rho_{j,\eps}^{j-4}|d_{0,j,\eps}(0)|
 \lesssim1.
\]

For \(0\leq k\leq6\), the one-dimensional tame product estimate
therefore bounds every nonlinear term (with respect to $W^{w}_{\epsilon} - W_{\epsilon}$) in \eqref{lastlinea36} by
\[
 \sum_{j=1}^{200}
 \rho_{j,\epsilon}^{\,j-k-2}
 |d_{0,j,\epsilon}(0)|
 \mathbf{1}_{\{0\leq y\leq2\rho_{j,\epsilon}\}}.
\]
Combining this estimate with \eqref{A.36b}--\eqref{A.36d} gives
\begin{align*}
 |\partial_y^kc_{1,\eps}(y)|
 \lesssim{}&\beta_{k,\eps}
 +\sum_{j=1}^{200}\rho_{j,\eps}^{j-k-2}
 \Big(
 |d_{0,j,\eps}(0)|
 +\rho_{j,\eps}^2|d'_{1,j,\eps}(0)|\\
 &\hspace{40mm}
 +\rho_{j,\eps}^2(\rho_{j,\eps}+\eps)
       |d'_{0,j,\eps}(0)|
 \Big)
 \mathbf{1}_{\{0\leq y\leq2\rho_{j,\eps}\}},
\end{align*}
which is \eqref{cesti2} in the required range.

For completeness, \eqref{desti1}--\eqref{desti4}, the choices in \eqref{cop0.1}, and
$\rho_{j,\eps}=\eps$ for $j\geq8$ give
\begin{align*}
 \sum_{j=1}^{200}\rho_{j,\eps}^{j-4}|d_{0,j,\eps}(0)|
 \lesssim{}&
 \eps^{11/10}+\eps^{3/5}+\eps^{7/10}+\eps
 +\eps^{13/10}+\eps^{8/5}+\eps^{17/10}+\eps^4
 \lesssim\eps^{3/5}.
\end{align*}
Thus the tame condition in \eqref{a.35a} follows automatically for all
sufficiently small $\eps$. \hfill\(\Box\)

\section{Global existence of the approximation system}\label{proofapproxia}

In this appendix, we establish the global--in--$(t,x)$ existence of the approximation system \eqref{approsys} for each \emph{fixed} $\epsilon>0$, which plays an essential role in our construction of weak solutions (in Section \ref{basic}) and local classical solutions (in Section \ref{local}). We present the detailed proofs for the convenience of the reader and for the sake of completeness. For clarity of presentation, we assume $\epsilon=1$ and consider the following system
\begin{equation}\label{equationappendix}
\begin{cases}
\p_{t}w+(y+1)\p_{x}w-(w+1)^{2}\p_{y}^{2}w=0,\ (t,x,y)\in(0,\infty)\times(0,\infty)\times(0,1),\\
w|_{t=0}=w_{0},\ w|_{x=0}=w_{1},\\
w|_{y=1}=\p_{y}w|_{y=0}=0.
\end{cases}
\end{equation}
We assume that $w_{0},w_{1}$ are non--negative and sufficiently regular, and satisfy the compatibility conditions of the equation \eqref{equationappendix} (described in Appendix \ref{compatiappendix}).

\subsection{Local well-posedness of \eqref{equationappendix}}
We first establish the local-in-time existence of the solution to \eqref{equationappendix} for $x\in[0,X]$ and an arbitrary $X>0$. 

Define for $s\in\Z\cap[1,\infty)$ the energy functional for any function $w$,
\begin{equation}\label{energynormappendix}
\E_{s,w}(t)=\sum_{k=0}^{s}\int_{0}^{1}\int_{0}^{X}|\p_{x}^{k}w|^{2}dxdy+\sum_{k=1}^{s+1}\int_{0}^{1}\int_{0}^{X}|\p^k_{y}w|^{2}\chi_{k}^{2}(y)dxdy,
\end{equation}
and the dissipation functional
\begin{equation}\label{dissi}
\mathcal{D}_{s,w}(t):=\sum_{k=0}^{s}\left(\int_{0}^{t}\int_{0}^{1}\int_{0}^{X}\left(|\p_{y}\p_{x}^{k}w|^{2}+|\p_{y}^{k+2}w\cdot\chi_{k}(y)|^{2}\right)dtdxdy\right)
\end{equation}
In the above, $\{\chi_{k}\}_{k=0}^{s+1}\subseteq C^\infty([0,1])$ is a family of cut-off functions such that $\chi_{0}\equiv\chi_{1} \equiv \chi_{2} \equiv 1$, and for $3\leq k \leq s+1$,
\begin{equation}\label{cut-off}
\begin{split}
&k\ \text{odd}:\ \mathrm{supp}\chi_{k}(y)\subset [\frac{1}{10},1],\ \chi_{k}(y)|_{[\frac{1}{5},1]}\equiv 1,\\
&k\ \text{even}:\ \mathrm{supp}\chi_{k}(y)\subset [0,\frac{9}{10}],\ \chi_{k}(y)|_{[0,\frac{4}{5}]}\equiv 1,\\
&\mathrm{supp}\,\chi_{k}\subset \{y\in[0,1]:\chi_{k-2}(y)=1\}.
\end{split}
\end{equation}

We note the following useful bounds for $t\ge0$,
\begin{equation}\label{enerdissi}
\begin{split}
    &\|w(t)\|_{H^s([0,X]\times[0,1])}^2\lesssim \E_{s,w}(t),\qquad \|\partial_yw(\tau)\|_{L^2_{\tau}H^s_{x,y}([0,t]\times[0,X]\times[0,1])}^2\lesssim \mathcal{D}_{s,w}(t).
\end{split}  
\end{equation}
\begin{remark} The main motivation for the selection of cut--off functions is to isolate boundary terms that involve derivative loss.
\end{remark}

\begin{proposition}\label{localsolappendix}
Fix an even integer $s\in\Z\cap[10,\infty)$. For any $X\geq 1$, there exist $0<\delta<1$ and $K>1$ depending on $X$, the initial data $w_{0}$ and inflow data $w_{1}$, such that there exists a unique non--negative classical solution $w(t,x,y)$ to \eqref{equationappendix} defined on
$[0,\delta]\times [0,X]\times[0,1]$,
satisfying
\begin{equation}\label{las1}
\sup_{t\in[0,\delta]}\E_{s,w}(t)+\mathcal{D}_{s,w}(\delta)\leq K.
\end{equation}
In particular,
\begin{equation}
\begin{split}
&\sup_{t\in[0,\delta]}\|w(t)\|_{H^{s}_{x,y}}^{2}+\|\p_{y}w(t)\|_{L^{2}([0,\delta];H^{s}_{x,y})}^{2}\lesssim K,\\
&\sup_{t\in[0,\delta]}\|\p_{t}^{\frac{s}{2}}w(t)\|_{L^{2}_{x,y}}^{2}+\|\p_{y}\p_{t}^{\frac{s}{2}}w(t)\|_{L^{2}_{t,x,y}}^{2}\lesssim K.
\end{split}
\end{equation}
\end{proposition}

\begin{proof} We assume that $X=1$ for notational simplicity, and our arguments below apply in the same way for all $X>0$. We construct the solution of \eqref{equationappendix} in the following steps:

\textit{Step 1: Iterated system.} We use the following iteration scheme for $n\in \Z\cap[1,\infty)$,
\begin{equation}\label{iterappendix}
\begin{cases}
    \p_{t}w^{(n)}+(y+1)\p_{x}w^{(n)}-(w^{(n-1)}+1)^{2}\p_{y}^{2}w^{(n)}=0,\\
    w^{(n)}|_{t=0}=w_{0},\ w^{(n)}|_{x=0}=w_{1},\\
    w^{(n)}|_{y=1}=\p_{y}w^{(n)}|_{y=0}=0.
\end{cases}
\end{equation}
In the above, $w^{(0)}(t,x,y)$ is a suitable Whitney--type boundary extension of $w_{0}$ and $w_{1}$. We assume that $w^{(0)}\ge0$. Indeed, the compatibility conditions on $(w_{0},w_{1})$ ensure the following extension problem ($g_k, h_k$ are computed through $(w_{0},w_{1})$ and the equation \eqref{equationappendix})
\begin{equation}
\p_{t}^{k}w^{(0)}|_{t=0}=g_{k}(x,y),\quad \p_{x}^{k}w^{(0)}|_{x=0}=h_{k}(t,y)
\end{equation}
can be solved for some sufficiently regular extension function $w^{(0)}$ on $[0,\infty)\times[0,\infty)\times[0,1]$, see Whitney \cite{W34} and Lions \cite{L72}. We emphasize that the iterated system \eqref{iterappendix} admits a unique classical solution by standard linear PDE theory. 

\textit{Step 2: Uniform--in--$n$ energy bounds.} By suitable energy estimates and the Sobolev inequalities, we obtain the key energy inequality
\begin{equation}\label{standardenergyappendix}
\sup_{t\in[0,\delta]}\E_{s,w^{(n)}}(t)+\mathcal{D}_{s,w^{(n)}}(\delta)\leq C(w_{0},w_{1})+\int_{0}^{\delta}\mathcal{P}_{s}\left(\E_{s,w^{(n-1)}}(\tau)\right)\cdot\E_{s,w^{(n)}}(\tau)\,d\tau,
\end{equation}
for any $\delta>0$. Here $\mathcal{P}_{s}$ is a monotonically increasing polynomial depending on $s$. Since the proof of \eqref{standardenergyappendix} overlaps substantially with the more complicated proof for \eqref{finalhighenergy0.1} below, to avoid repetitive details, we skip the detailed proof and only make the following remarks.
\begin{itemize}
    \item[(i)] \eqref{standardenergyappendix} is derived by standard energy estimates, through differentiating \eqref{iterappendix} to derive the equation for $\p_{x}^{k}w$ (respectively $\p_{y}^{k}w$) and testing against $\p_{x}^{k}w$ (respectively $\p_{y}^{k}w\cdot\chi^{2}_{k}(y)$).
    \item[(ii)] The terms where derivatives fall on cut--off functions $\chi_{k}$ are bounded by the lower--level diffusion terms.
    \item[(iii)] We use the inequality 
    \begin{equation}\label{embedtohs}
    \|w(t)\|_{H^{s}_{x,y}}^{2}\lesssim \E_{s,w}(t)
    \end{equation}
    to control the lower--order derivatives in cubic (or higher power) terms appearing in the energy estimates. \eqref{embedtohs} follows from \eqref{energynormappendix}, \eqref{cut-off}, and standard elliptic estimates.
    \item[(iv)] In the energy estimates for $\p_{x}^{k}w$, the boundary conditions on $y=0$ and $y=1$ are identical to those for $w$, and therefore do not contribute to the energy estimates. However, in the energy estimates for $\p_{y}^{k}w$, the boundary conditions become nontrivial and require a careful treatment. Our key observation is that at each level of the energy estimates, thanks to our choice of the cutoff functions $\chi_k$, the boundary terms can be expressed through lower--order terms which are then bounded by the energy norm and diffusion terms. We will illustrate this important point in detail when going through the energy estimates in subsection \ref{local-to-global} to avoid repetitive arguments.
\end{itemize}
%\textit{Proof of Proposition \ref{localsolappendix} (continued).} 
Choose $K>1$ and $0<\delta_{1}<1$ such that
\begin{equation}\label{choosepara}
%\begin{cases}
\sup_{t\in[0,1]}\E_{s,w^{(0)}}(t)\leq K,\quad
C(w_{0},w_{1})+\delta_{1}\cdot\mathcal{P}_{s}(K)\cdot K\leq K.
%\end{cases}
\end{equation}
Then, by induction, we obtain from \eqref{standardenergyappendix} that
\begin{equation}\label{unhighn}
    \sup_{n\geq 1}\bigg(\sup_{t\in[0,\delta_{1}]}\E_{s,w^{(n)}}(t)+\mathcal{D}_{s,w^{(n)}}(\delta_{1})\bigg)\leq K.
\end{equation}

\textit{Step 3. Convergence of the iteration sequence.} For $n\geq 2$, define $\overline{w}^{(n)}=w^{(n)}-w^{(n-1)}$. It is clear that $\overline{w}^{(n)}$ satisfies
\begin{equation}
    \begin{cases}
        \p_{t}\overline{w}^{(n)}+(y+1)\p_{x}\overline{w}^{(n)}-(w^{(n-1)}+1)^{2}\p_{y}^{2}\overline{w}^{(n)}-(w^{(n-1)}+w^{(n-2)}+2)\p_{y}^{2}w^{(n-1)}\overline{w}^{(n-1)}=0,\\
        \overline{w}^{(n)}|_{t=0}=\overline{w}^{(n)}|_{x=0}=0,\\
        \overline{w}^{(n)}|_{y=1}=\p_{y}\overline{w}^{(n)}|_{y=0}=0.
    \end{cases}
\end{equation}
By standard energy estimates over $[0,1]\times[0,1]$, we obtain that
\begin{equation}
\frac{d}{dt}\|\overline{w}^{(n)}(t)\|_{L^{2}_{x,y}}^{2}\lesssim (1+\|w^{(n-1)}(t)\|_{W^{2,\infty}_{y}}^{2})\|\overline{w}^{(n)}(t)\|_{L^{2}_{x,y}}^{2}+\|\overline{w}^{(n-1)}(t)\|_{L^{2}_{x,y}}^{2}.
\end{equation} 
By Gr\"onwall's inequality and Sobolev inequalities, there exists $0<\delta\leq \delta_{1}$ such that
\begin{equation}
\sup_{t\in[0,\delta]}\|\overline{w}^{(n)}\|_{L^{2}_{x,y}}\leq \frac{1}{2}\sup_{t\in[0,\delta]}\|\overline{w}^{(n-1)}\|_{L^{2}_{x,y}}.
\end{equation}
Therefore, together with \eqref{unhighn}, we conclude that the sequence $w^{(n)}$ converges strongly to $w(t,x,y)$ in $C^{2}([0,\delta]\times[0,1]\times[0,1])$ with
\begin{equation}\label{solutionspace}
\bigg(\sup_{t\in[0,\delta]}\E_{s,w}(t)+\mathcal{D}_{s,w}(\delta)\bigg)\leq K.
\end{equation}
Letting $n\to\infty$ in \eqref{iterappendix}, we conclude the existence part. The uniqueness follows from a standard maximum principle argument, which completes the proof of Proposition \ref{localsolappendix}.
\end{proof}
\begin{remark}Arguing in the same fashion, we can also prove the local--in--$x$ existence of a classical solution to \eqref{equationappendix}. More precisely, for any $T>0$ there exist $0<\delta<1$ and $K>1$ depending on the sizes of the initial and inflow data and $T$, such that \eqref{equationappendix} admits a unique classical solution $w(t,x,y)$ defined on
$[0,T]\times[0,\delta]\times[0,1]$, such that
\begin{equation}
\begin{split}
&\sup_{x\in[0,\delta]}\bigg[\sum_{k=0}^{s}\int_{0}^{T}\int_{0}^{1}\left(|\p_{t}^{k}w|^{2}+|\p_{y}^{k+1}w\cdot\chi_{k}(y)|^{2}\right)dydt\bigg]\\
&\ +\sum_{k=0}^{s}\iiint_{[0,\delta]\times[0,T]\times[0,1]}\left(|\p_{t}^{k}\p_{y}w|^{2}+|\p_{y}^{k+2}w\cdot\chi_{k}(y)|^{2}\right)dtdxdy\leq K.
\end{split}
\end{equation}
\end{remark}
\subsection{Global well-posedness of \eqref{equationappendix}}\label{local-to-global}
To extend the local solutions obtained in Proposition \ref{localsolappendix} globally in time, by the standard continuation method, we need to show that the norm $\E_{s,w}(t)$ remains finite over arbitrary time intervals.

We first note the pointwise bound
$$0\leq w\leq C_T(1-y),$$
which follows from a simple comparison argument. The pointwise bound is useful in the interpolation estimates below. 

The first main step is to establish the following refined energy estimate for $T>0$,
\begin{equation}\label{finalhighenergy0.1}
    \sup_{t\in[0,T]}\E_{s,w}(t)\lesssim C(w_{0},w_{1})+\int_{0}^{T}(\|\p_{y}w\|_{L^{\infty}}^{2}+1)\cdot (\E_{s,w}(\tau)+1)d\tau.
\end{equation}

The bound \eqref{finalhighenergy0.1} on $w$ directly follows by testing \eqref{equationappendix} against $w$. The bounds \eqref{finalhighenergy0.1} on $\partial_x^kw$, $1\leq k\leq s$, follow from taking derivatives in \eqref{equationappendix} and using energy estimates (against $\partial_x^kw$). The transport term generates a nonnegative outgoing flux at $x=1$. The incoming flux at $x=0$ is controlled by the compatible boundary jets recursively determined from $w_1$. The initial term is controlled by $w_0$.

More precisely, we have
\begin{equation}\label{ener0.2}
    \begin{split}
&\sup_{t\in[0,T]}\left(\int_{0}^{1}\int_{0}^{1}|\p_{x}^kw|^{2}dxdy\right)+\int_{0}^{T}\int_{0}^{1}\int_{0}^{1}|\p_{y}\partial_x^kw|^{2}dtdxdy\\
&\lesssim C(w_{0},w_{1})+\bigg|\int_{0}^{T}\int_{0}^{1}\int_{0}^{1}\Big[\partial_x^k\big((w+1)^2\p_y^2w\big)-(w+1)^2\partial_y^2\partial_x^kw\Big]\partial_x^kw\,dtdxdy\bigg|\\
&\qquad\qquad\quad\,\,\,\,\,+\bigg|\int_{0}^{T}\int_{0}^{1}\int_{0}^{1}(w+1)\partial_yw\partial_y\partial_x^kw\partial_x^kw\,dtdxdy\bigg|.
\end{split}
\end{equation}

The desired bound \eqref{finalhighenergy0.1} for $\partial_x^kw$ then follows from an integration-by-parts argument (to reduce $\partial_x^j\partial_y^2w$ to $\partial_x^j\partial_yw$), and the following (Gagliardo-Nirenberg) interpolation inequality which holds for integers $l, k$ with $1\leq l\leq k$ (for our application, $f\in\{w,\p_{y}w\}$): 
\begin{equation}\label{embedh2tow14}
\|\partial_x^lf\|_{L^{\frac{2k}{l}}}\lesssim_{k,l} \|f\|_{L^{\infty}}^{\frac{k-l}{k}}\|\partial_x^kf\|_{L^{2}}^{\frac{l}{k}}+\|f\|_{L^{\infty}},\quad \forall f\in H^{k}(0,1).
\end{equation}

We now turn to the bounds \eqref{finalhighenergy0.1} on derivatives in $y$, whose derivation is more delicate due to the presence of nontrivial boundary contributions.
Applying $\partial_y^k$ with $1\leq k\leq s+1$ to \eqref{equationappendix} and integrating by parts against $\partial_y^kw\chi_k^2$, we obtain that
\begin{equation}\label{ener0.3}
    \begin{split}
&\sup_{t\in[0,T]}\left(\int_{0}^{1}\int_{0}^{1}|\p_{y}^kw|^{2}\chi_k^2dxdy\right)+\int_{0}^{T}\int_{0}^{1}\int_{0}^{1}|\p_{y}\partial_y^kw|^{2}\chi_k^2dtdxdy\\
&\lesssim C(w_{0},w_{1})+\bigg|\int_{0}^{T}\int_{0}^{1}\int_{0}^{1}\Big[\partial_y^k\big((w+1)^2\p_y^2w\big)-(w+1)^2\partial_y^2\partial_y^kw\Big]\partial_y^kw\chi_k^2\,dtdxdy\bigg|\\
&+k\bigg|\int_{0}^{T}\int_{0}^{1}\int_{0}^{1}\partial_x\partial_y^{k-1}w\partial_y^kw\chi_k^2\,dtdxdy\bigg|+2\bigg|\int_{0}^{T}\int_{0}^{1}\int_{0}^{1}(w+1)^2\partial_y^{k+1}w\partial_y^kw\chi_k'\chi_k\,dtdxdy\bigg|\\
&+\bigg|\int_{0}^{T}\int_{0}^{1}\int_{0}^{1}(w+1)\partial_yw\partial_y^{k+1}w\partial_y^kw\chi_k^2\,dtdxdy\bigg|+\bigg|\int_{0}^{T}\int_{0}^{1}(w+1)^2\partial_y^{k+1}w\partial_y^kw\chi_k^2\bigg|_{y\in\{0,1\}}\,dtdx\bigg|\\
&:=C(w_0,w_1)+I+II+III+IV+V.
\end{split}
\end{equation}
The terms $I$--$IV$ can be bounded by the right hand side of \eqref{finalhighenergy0.1}, using the following weighted interpolation inequalities:
for any integers $(k,l)$ such that $3\leq k\leq s+1$ and $1<l<k$, we have
\begin{equation}\label{weightedinterpolation}
\|\p_{y}^{l}w\cdot\chi_{k}^{\frac{l-1}{k-1}}\|_{L^{\frac{2(k-1)}{l-1}}_{y}}\lesssim\|\p_{y}w\|_{L^{\infty}_{y}}+\|\p_{y}w\|_{L^{\infty}_{y}}^{\frac{k-l}{k-1}}\times\Big(\sum_{2\leq \lambda\leq k}\|\p_{y}^{\lambda} w\cdot\chi_{\lambda}\|_{L^{2}_{y}}\Big)^{\frac{l-1}{k-1}}.
\end{equation}
We assume \eqref{weightedinterpolation} momentarily and continue with our main line of proof. A sketch of proof for the technical inequality \eqref{weightedinterpolation} will be given at the end of the proof. 

It remains to treat the boundary contribution $V$. We first note the boundary conditions
\begin{equation}\label{bd0.1}
\p_{y}w|_{y=0}=\p_{y}^{2}w|_{y=1}=0.
\end{equation}
Then
from the boundary condition $\p_{y}w|_{y=0}=0$ and the differential equation $\p_{t}w+(y+1)\p_{x}w-(w+1)^{2}\p_{y}^{2}w=0$, we obtain that
\begin{equation}\label{3rden2}
    \big(\p_{x}w-(w+1)^{2}\p_{y}^{3}w\big)|_{y=0}=0,
\end{equation}
and hence
\begin{equation}\label{bd0.2}
\partial_y^3w|_{y=0}=\frac{\partial_xw}{(w+1)^2}\bigg|_{y=0}.
\end{equation}
Thanks to the boundary condition for $w$ and the equation \eqref{equationappendix},
\begin{equation}\label{bd0.3}
\p_{y}^{4}w\big|_{y=1}=2\p_{x}\p_{y}w\big|_{y=1}-4\p_{y}w\p_{y}^{3}w\big|_{y=1}.
\end{equation}
By direct calculation, $\p_{y}^{3}w$ satisfies the equation 
\begin{equation}\label{eqpy3}
\begin{split}
&\p_{t}\p_{y}^{3}w+3\p_{x}\p_{y}^{2}w+(y+1)\p_{x}\p_{y}^{3}w-(w+1)^{2}\p_{y}^{5}w\\
&\quad-6(w+1)\p_{y}w\p_{y}^{4}w-8(w+1)\p_{y}^{2}w\p_{y}^{3}w-6\p_{y}w(\p_{y}^{2}w)^{2}-6(\partial_yw)^2\partial_y^3w = 0.
\end{split}
\end{equation}
Using \eqref{equationappendix}, \eqref{eqpy3} and \eqref{bd0.2}, we obtain that
\begin{equation}
\p_{y}^{5}w|_{y=0}=4\frac{\p_{x}\p_{y}^{2}w}{(w+1)^{2}}\bigg|_{y=0}-8\frac{\p_{x}w\cdot\p_{y}^{2}w}{(w+1)^{3}}\bigg|_{y=0}.
\end{equation}
For higher-order normal derivatives, the boundary conditions and
the equation imply the following boundary-jet identities. For every
$k\ge2$,
\begin{equation}\label{bdreduce1}
\begin{split}
\p_{y}^{2k}w|_{y=1}&=a_{k}\p_{y}w|_{y=1}\cdot\p_{y}^{2k-1}w|_{y=1}+\sum_{l\geq 1}\sum_{\alpha\in\mathcal{A}_{k,l}}c_{k,\alpha}\bigg(\prod_{j=1}^{l}\p_{x}^{\alpha_{j,1}}\p_{y}^{\alpha_{j,2}}w\bigg|_{y=1}\bigg)
\end{split}
\end{equation}
and
\begin{equation}\label{bdreduce2}
\p_{y}^{2k+1}w|_{y=0}=\sum_{l\geq 1}\sum_{\beta\in\mathcal{B}_{k,l}}d_{k,\beta}(w+1)^{\lambda_{k,\beta}}|_{y=0}\cdot\bigg(\prod_{j=1}^{l}\p_{x}^{\beta_{j,1}}\p_{y}^{\beta_{j,2}}w\bigg|_{y=0}\bigg),
\end{equation}
where
\begin{equation}\label{admiss1}
\begin{split}
\mathcal{A}_{k,l}&=\bigg\{\alpha=\left((\alpha_{1,1},\alpha_{1,2}),\cdots,(\alpha_{l,1},\alpha_{l,2})\right):\sum_{j=1}^{l}(3\alpha_{j,1}+\alpha_{j,2})=2k,\\
&\qquad\, \min_{1\leq j\leq l}(\alpha_{j,1}+\alpha_{j,2})\ge 1,\max_{1\leq j\leq l}\alpha_{j,2}\leq 2k-3\bigg\},
\end{split}
\end{equation}
and
\begin{equation}\label{admiss2}
    \begin{split}
    \mathcal{B}_{k,l}&=\bigg\{\beta=\left((\beta_{1,1},\beta_{1,2}),\cdots,(\beta_{l,1},\beta_{l,2})\right):\sum_{j=1}^{l}(3\beta_{j,1}+\beta_{j,2})=2k+1,\\
&\qquad\,\, \min_{1\leq j\leq l}(\beta_{j,1}+\beta_{j,2})\geq 1,\max_{1\leq j\leq l}\beta_{j,1}\geq 1,\max_{1\leq j\leq l}\beta_{j,2}\leq 2k-2\bigg\}.
    \end{split}
\end{equation}
\eqref{bdreduce1} and \eqref{bdreduce2} can be deduced from a standard induction argument.
The main point of \eqref{bdreduce1}-\eqref{bdreduce2} is that the order of differentiation is lowered. 

The boundary contribution term $V$ can then be bounded in view of our selection of cutoff functions \eqref{cut-off}, the formulae \eqref{bdreduce1}-\eqref{bdreduce2}, and using the inequality for any $g\in W^{1,1}([0,1])$:
\begin{equation}\label{bd0.4}
\begin{split}
  |g(0)|\lesssim \int_0^{0.01}|g(y)|+|g'(y)| \,dy,\\
  |g(1)|\lesssim \int_{0.99}^1|g(y)|+|g'(y)| \,dy,
  \end{split}
\end{equation}
also employing the Gagliardo--Nirenberg interpolation inequality \eqref{embedh2tow14}.

Lastly, we sketch the proof of \eqref{weightedinterpolation} for completeness. Denote for $2\leq l\leq k$,
\begin{equation}\label{bd0.5}
R_{l,k}:=\|\p_{y}^{l}w\cdot\chi_{k}^{\frac{l-1}{k-1}}\|_{L^{\frac{2(k-1)}{l-1}}_{y}},\qquad R_{1,k}:=\|\partial_yw\|_{L^\infty_y}.
\end{equation}
 Using integration by parts and applying the classical Gagliardo--Nirenberg inequality \eqref{embedh2tow14}, we obtain the following recursive pattern for $2\leq l\leq k-1$ (here we use the fact that $\mathrm{supp}(\chi_{k}')\subset\{y:\chi_{l}=1\}$ for all $l<k$):
\begin{equation}\label{pattern}
    \begin{split}
    R_{l,k}^{\frac{2(k-1)}{l-1}}&\lesssim \|\p_{y}w\|_{L^{\infty}}^{\frac{2(k-1)}{l-1}}+\|\p_{y}w\|_{L^{\infty}}^{\frac{2(k-l)}{l-1}}\times\|\p_{y}^{k} w\cdot\chi_{k}\|_{L^{2}}^{2}\\
    &\quad+R_{l,k}^{\frac{2(k-l)}{l-1}}\cdot R_{l-1,k}\cdot R_{l+1,k}+R_{l,k}^{\frac{k-1}{l-1}}\cdot R_{l-1,k-1}\cdot R_{l,k-1}^{\frac{k-l}{l-1}}.
    \end{split}
\end{equation}
More precisely, set
\[
p=\frac{2(k-1)}{l-1},\qquad
\theta=\frac{l-1}{k-1}.
\]
Then
\[
R_{l,k}^p
=\int_0^1 |\partial_y^l w|^p\chi_k^2\,dy.
\]
Integrating by parts once gives
\[
\begin{aligned}
R_{l,k}^p
={}&
\left[
\partial_y^{l-1}w\,
|\partial_y^lw|^{p-2}\partial_y^lw\,
\chi_k^2
\right]_{0}^{1}-(p-1)\int_0^1
\partial_y^{l-1}w\,
|\partial_y^lw|^{p-2}
\partial_y^{l+1}w\,
\chi_k^2\,dy\\
&-2\int_0^1
\partial_y^{l-1}w\,
|\partial_y^lw|^{p-2}\partial_y^lw\,
\chi_k\chi_k'\,dy.
\end{aligned}
\]
The boundary term is estimated by the one-dimensional trace
Gagliardo--Nirenberg inequality, while the two integral terms are
estimated by Hölder's inequality using the definitions of
$R_{j,m}$ and the fact that $\chi_{k-1}=1$ on
$\operatorname{supp}\chi_k'$.

 We prove (B.23) using \eqref{pattern} by induction on $k$. The induction hypothesis
controls the terms in (B.36) whose second index is $k-1$.
For fixed $k$, the remaining quantities $R_{l,k}$,
$2\le l\le k-1$, are estimated simultaneously by a discrete
maximum argument.

Indeed, let
\[
M=\|\partial_yw\|_{L^\infty_y},\qquad
H_k=M+\sum_{\lambda=2}^k
\|\partial_y^\lambda w\,\chi_\lambda\|_{L^2_y},
\]
and
\[
S_{l,k}
=M^{1-\frac{l-1}{k-1}}
H_k^{\frac{l-1}{k-1}}.
\]
After applying the induction hypothesis to the last term in
(B.36) and using Young's inequality, we obtain
\[
R_{l,k}^{p_l}
\lesssim_k
S_{l,k}^{p_l}
+
R_{l,k}^{p_l-2}R_{l-1,k}R_{l+1,k},
\qquad
p_l=\frac{2(k-1)}{l-1}.
\]
Since $S_{l-1,k}S_{l+1,k}=S_{l,k}^2$, the normalized
quantities $A_l=R_{l,k}/S_{l,k}$ satisfy
\[
A_l^2\lesssim_k 1+A_{l-1}A_{l+1},
\qquad A_1+A_k\lesssim1.
\]
A finite discrete maximum argument then gives
$\max_{1\le l\le k}A_l\lesssim_k1$, proving (B.23).

We see from \eqref{finalhighenergy0.1} that we can continue the solution (with bounded energy) to the system \eqref{equationappendix} as long as
\begin{equation}
\int_{0}^{T^{*}}\|\p_{y}w\|_{L^{\infty}_{x,y}}^{2}dt<\infty.
\end{equation}
Hence, the boundedness of the norm $\E_{s,w}(t)$ on an arbitrary finite time interval (and consequently global existence of \eqref{equationappendix}) follows from the following up--to--boundary smoothing estimate, which implies that we can control the key quantity $\|\p_{y}w\|_{L^{\infty}}$ just in terms of bounds on the initial and boundary data.
\begin{proposition}\label{bdwy}
Fix $T>0$. Assume $w$ solves \eqref{equationappendix} for $t\in[0,T],\,x\in[0,2],\,y\in[0,1]$. Then for any $\delta>0$, we have
\begin{equation}
|\p_{y}w(t,x,y)|\leq C_{\delta},\ \forall t\in[\delta,T],\,\,x\in[\delta,1],\,\, y\in[0,1],
\end{equation}
where $C_{\delta}$ only depends on $\delta>0,\,T$, and
\begin{equation*}
    \sup_{x\in[0,2],\,y\in[0,1)}\frac{w_{0}}{1-y},\quad \sup_{t\in[0,T],\,y\in[0,1)}\frac{w_{1}}{1-y}.
\end{equation*}
\end{proposition}

\begin{proof}[Proof of Proposition \ref{bdwy}.] We can assume that $0<\delta\ll 1$. Using arguments analogous to those in Proposition \ref{existweaksol}, we have
\begin{equation}
|\p_{y}w(t,x,y)|\lesssim_{\delta} 1,\quad\text{provided\,\,that\,\,}\ t\geq \delta,\,x\geq \delta,\,0\leq y\leq 1-\delta.
\end{equation}

It remains to consider the region $1-\delta< y\leq1$.

For fixed $(t_{0},x_{0},y_{0})$ with $t_{0}\geq \delta, x_{0}\geq \delta$ and $1-\delta<y_{0}<1$, we define
\begin{equation}\label{eq:B40}
V(t,x,y)=w\left(t_{0}+r_{0}^{2}t,x_{0}+r_{0}^{3}x+(y_{0}+1)r_{0}^{2}t,y_{0}+r_{0}y\right),\quad r_{0}=1-y_{0}.
\end{equation}
Since the spatial length is arbitrary, we work on \(0\leq x\leq2\) while deriving the estimate for \(0\leq x\leq1\). In the rescaled variables, we use the backward-time cylinder
$$
Q^-:=\{-1<t\leq0,\ |x|<1,\ |y|<1\}.
$$
Indeed, since \(r_0<\delta\ll1\), \(t_0,x_0\geq\delta\), and \(x_0\leq1\), the three arguments of \(w\) in \eqref{eq:B40} remain in \([0,T]\times[0,2]\times[0,1]\) throughout \(Q^-\). Thus \(V\) is well defined on \(Q^-\).

By direct computation, we see that $V$ satisfies for $-1<t\leq 0,\,|x|<1,\,|y|<1$,
\begin{equation}
\p_{t}V+y\p_{x}V-(V+1)^{2}\p_{y}^{2}V=0,
\end{equation}
and
\begin{equation}
\|V\|_{L^{\infty}_{\{-1<t\leq0,|x|<1,|y|<1\}}}\lesssim r_{0}\ll 1.
\end{equation}
In the above, we used the following fact:
\begin{equation}\label{vanishrate}
    0\leq w(t,x,y)\leq A\cdot (1-y),\quad A:=\left(\sup_{x,y}\frac{w_{0}}{1-y}+\sup_{t,y}\frac{w_{1}}{1-y}\right),
\end{equation}
which can be demonstrated by a standard maximum principle argument.

Using arguments from the proof of Lemma \ref{constantcoeff} and the $W^{2,p}$ estimate in the whole space \cite{BC96b}, we have for all $p\in(1,\infty)$,
\begin{equation}\label{w2pinter}
\|(\p_{y}V,\p_{y}^{2}V)\|_{L^{p}_{\{-1/2<t\leq0,|x|<\frac{1}{2},|y|<\frac{1}{2}\}}}\lesssim \|V\|_{L^{\infty}_{\{-1<t\leq0,|x|<1,|y|<1\}}}\lesssim_{p} r_{0}.
\end{equation}

Let $\chi(t,x,y)$ be a cut--off function supported in $\{-1/2<t\leq0,|x|<\frac{1}{2},|y|<\frac{1}{2}\}$ and
\begin{equation*}
\chi(t,x,y)\equiv 1,\ -1/4<t\leq0,|x|<\frac{1}{4},|y|<\frac{1}{4}.
\end{equation*}

Then we have
\begin{equation}
V(t,x,y)\chi(t,x,y)=\int_{-1/2<\tau\leq0,|\xi|<\frac{1}{2},|\eta|<\frac{1}{2}}\widetilde{\Gamma}(t,x,y;\tau,\xi,\eta)\mathcal{W}(\tau,\xi,\eta)\,d\tau d\xi d\eta,
\end{equation}
where $\widetilde{\Gamma}(t,x,y;\tau,\xi,\eta)$ is the fundamental solution for the operator $\mathcal{L}_{0}=\p_{t}+y\p_{x}-\p_{y}^{2}$ in the whole space (see \eqref{wholespace}), and
\begin{equation*}
\mathcal{W}(\tau,\xi,\eta)=(V^{2}\p_{\eta}^{2}V+2V\p_{\eta}^{2}V)\cdot\chi+V(\p_{\tau}\chi+\eta\p_{\xi}\chi-\p_{\eta}^{2}\chi)-2\p_{\eta}V\p_{\eta}\chi.
\end{equation*}

By \eqref{wholespace}, \eqref{w2pinter} and H\"older's inequality, we obtain that
\begin{equation}
|\p_{y}V(0,0,0)|\lesssim r_{0}.
\end{equation}
Therefore,
\begin{equation}
|\p_{y}w(t_{0},x_{0},y_{0})|\leq r_{0}^{-1}|\p_{y}V(0,0,0)|\lesssim 1.
\end{equation}
The proof of Proposition \ref{bdwy} is then complete.
\end{proof}

\begin{remark}As in Sections \ref{fundamental} -- \ref{higher}, an alternative method to prove Proposition \ref{bdwy} is to establish the regularity theory of \eqref{equationappendix} near the boundary $y=1$. The key step of this method is to study the fundamental solution of the operator $\mathcal{L}_{0}$ up to the boundary with a Dirichlet boundary condition. 
\end{remark}

\section{Analysis toolbox}\label{pre}
\subsection{Tools in Real Analysis}
\subsubsection{Schur's Test} We recall the classical Schur's test for integral operators. The proof can be found in Appendix I of \cite{GTM249}.
\begin{lemma}[\bf Schur's Test]
\label{schur}
Let $(X,\mu)$ be a measure space. Define the integral operator $T$ as
\begin{equation}
T(f)(x)=\int_{X}K(x,y)f(y)d\mu (y).
\end{equation}
If
\begin{equation}
A=\sup_{x\in X}\int_{X}|K(x,y)|d\mu (y)+\sup_{y\in X}\int_{X}|K(x,y)|d\mu(x)<\infty,
\end{equation}
then $T$ can be defined as a bounded operator on $L^{p}(X,\mu)$ for any $1\leq p\leq\infty$ with norm less than or equal to $A$.
\end{lemma}
\begin{comment}
\begin{proof}The case of $p=\infty$ is obvious, and $p=1$ is the consequence of Fubini's theorem. For $f\in L^{p}(X,\mu) (1<p<\infty)$, applying H\"older's inequality and Fubini's theorem, we have
\begin{equation*}
\begin{split}
\|Tf\|_{L^{p}(X)}^{p}&\leq\int_{X}\left(\int_{X}|K(x,y)|\cdot|f(y)|d\mu(y)\right)^{p}d\mu(x)\\
&\leq\int_{X}\left(\int_{X} |K(x,y)|d\mu(y)\right)^{\frac{p-1}{p}}\left(\int_{X}|K(x,y)|\cdot|f(y)|^{p}d\mu(y)\right)d\mu(x)\\
&\leq A\int_{X}|f(y)|^{p}d\mu(y).
\end{split}
\end{equation*}
\end{proof}
\end{comment}

\subsubsection{Littlewood--Paley Theory}\label{littlewood} We briefly recall the Littlewood--Paley decompositions, which are important in this paper. Let $\{\Delta_{j}\}_{j\geq -1}$ be the usual Littlewood--Paley projection operators defined on $\mathcal{S}'(\R^{d})$ (for more details we refer to \cite{BCD11}). We have the following Littlewood--Paley decomposition for $f\in\mathcal{S}'(\R^{d})$ as the sum of its dyadic frequency blocks
\begin{equation}
    f=\sum_{j\geq -1}\Delta_{j}f\ \text{in}\ \mathcal{S}'.
\end{equation}
The following equivalent norm holds:
\begin{equation}
\|f\|_{H^{s}(\R^{d})}^{2}\sim \|f\|_{L^2(\R^d)}^2+\sum_{j\geq -1}2^{2js}\|\Delta_{j}f\|_{L^{2}(\R^{d})}^{2},\quad \forall\ s\ge0.
\end{equation}
Here $H^{s}$ is the standard Sobolev space. Therefore, to obtain $H^{s}$ estimates, we only need to obtain the $L^{2}$ estimate for each dyadic block.

The following Bernstein's inequalities are also useful (we adopt the convention that $S_{j}:=\sum_{k< j}\Delta_{k}$):
\begin{equation}\label{bernstein}
\begin{cases}
\|\nabla^{k}\Delta_{j}f\|_{L^{q}(\R^{d})}+\|\nabla^{k}S_{j}f\|_{L^{q}(\R^{d})}\lesssim 2^{jk+jd\left(\frac{1}{p}-\frac{1}{q}\right)}\|f\|_{L^{p}(\R^{d})},\ 1\leq p\leq q\leq \infty,\\
\|\Delta_{j}f\|_{L^{\infty}(\R^{d})}\lesssim 2^{-j\alpha}\|f\|_{C^{\alpha}(\R^{d})},\ 0<\alpha<1.
\end{cases}
\end{equation}

We also need the following paraproduct decomposition of Bony for the product of two functions. For $j\geq 10$, we have
\begin{equation}\label{parapro}
\begin{split}
\Delta_{j}(ab)&=\sum_{|l|\leq 5}\Delta_{j}\left(S_{j+l-1}a\Delta_{j+l}b\right)+\sum_{|l|\leq 5}\Delta_{j}\left(S_{j+l-1}b\Delta_{j+l}a\right)+\sum_{l> j-5,|q-l|\leq 1}\Delta_{j}\left(\Delta_{l}a\Delta_{q}b\right).
\end{split}
\end{equation}
\subsubsection{Interpolation Inequalities}
Next, we state some useful interpolation inequalities in $\R^{+}$.
\begin{lemma}[\bf Interpolation Inequalities]\label{interonline}

\emph{(i)} If $f\in H^{1}(\R^{+})$, then
\begin{equation}\label{spectralgap}
\int_{0}^{\infty}|f|^{2}dy\lesssim \int_{0}^{\infty}|\p_{y}f|^{2}dy+\inf_{c\geq 0}\int_{0}^{\infty}|y-c|\cdot|f|^{2}dy.
\end{equation}

\emph{(ii)} If $f\in H^{2}(\R^{+})$, then for any $1< p\leq \infty$,
\begin{equation}\label{interrr}
\|f\|_{L^{p}(\R^{+})}\lesssim_p \|\sqrt{y}f(y)\|_{L^{2}(\R^{+})}+\|\sqrt{y}f''(y)\|_{L^{2}(\R^{+})}.
\end{equation}

\emph{(iii)} If $f\in W^{2,p}(\R^{+})$ with $1\leq p\leq\infty$, then
\begin{equation}\label{interr}
\|f'(y)\|_{L^{p}(\R^{+})}\lesssim \|f(y)\|_{L^{p}(\R^{+})}^{\frac{1}{2}}\cdot\|f''(y)\|_{L^{p}(\R^{+})}^{\frac{1}{2}},
\end{equation}
\end{lemma}

\begin{proof} 
\textit{Proof of \eqref{spectralgap}:} We need to show that for all $c\geq 0$,
\begin{equation}\label{goalofc7}
    \int_{0}^{\infty}|f|^{2}dy\lesssim \int_{0}^{\infty}|\p_{y}f|^{2}dy+\int_{0}^{\infty}|y-c|\cdot|f|^{2}dy.
\end{equation}
For brevity we prove \eqref{goalofc7} for $c=0$, since for general $c\geq 0$ it is proved in a similar way. Clearly,
\begin{equation}\label{obviousbd}
\int_{1}^{\infty}|f|^{2}dy\leq \int_{1}^{\infty}y|f|^{2}dy.
\end{equation}
By the mean--value theorem, there exists $z\in [1,2]$ such that
\begin{equation*}
|f(z)|^{2}\lesssim \int_{1}^{2}y|f|^{2}dy.
\end{equation*}
Hence, for any $y\in [0,1]$,
\begin{equation}\label{ptcontrol}
|f(y)|\lesssim \int_{0}^{2}|f'(s)|ds +|f(z)|\lesssim \left(\int_{0}^{2}|f'(s)|^{2}ds+\int_{1}^{2}s|f(s)|^{2}ds\right)^{\frac{1}{2}}.
\end{equation}
 \eqref{goalofc7} follows from \eqref{obviousbd} and \eqref{ptcontrol}.

\textit{Proof of \eqref{interrr}:}  We have
\begin{equation*}
\|f\|_{L^{\infty}(1,\infty)}\lesssim \|f\|_{L^{2}(1,\infty)}+\|f'\|_{L^{2}(1,\infty)}.
\end{equation*}
Integrating by parts, we get that
\begin{equation*}
\begin{split}
\int_{1}^{\infty}|f'(y)|^{2}dy\leq \|f\|_{L^{2}(1,\infty)}\|f''\|_{L^{2}(1,\infty)}+|f(1)|\cdot|f'(1)|.
\end{split}
\end{equation*}
By the mean--value theorem, there exist $y_{1}\in[\frac{1}{2},1],y_{2}\in[\frac{3}{2},2]$, such that
\begin{equation*}
|f(y_{1})|^{2}+|f(y_{2})|^{2}\lesssim \int_{\frac{1}{2}}^{2}y|f(y)|^{2}dy.
\end{equation*}
Applying the mean--value theorem again, we can conclude that there exists $y_{3}\in[y_{1},y_{2}]$ such that
\begin{equation*}
|f'(y_{3})|^{2}\lesssim \int_{\frac{1}{2}}^{2}y|f(y)|^{2}dy.
\end{equation*}
Thus, by Newton--Leibniz and Cauchy--Schwarz,
\begin{equation}\label{standard1d}
|f'(1)|^{2}\lesssim  \int_{\frac{1}{2}}^{2}y|f(y)|^{2}dy+ \int_{\frac{1}{2}}^{2}y|f''(y)|^{2}dy.
\end{equation}
Therefore, for $p\geq 2$,
\begin{equation}
\|f\|_{L^{p}(1,\infty)}\leq \|f\|_{L^{2}(1,\infty)}^{\frac{2}{p}}\|f\|_{L^{\infty}(1,\infty)}^{\frac{p-2}{p}}\lesssim \|\sqrt{y}f(y)\|_{L^{2}(\R^{+})}+\|\sqrt{y}f''(y)\|_{L^{2}(\R^{+})}.
\end{equation}
For $1<p<2$, by H\"older's inequality,
\begin{equation}
\|f\|_{L^{p}(1,\infty)}\lesssim \|\sqrt{y}f\|_{L^{2}(1,\infty)}.
\end{equation}

Next, we focus on the interval $y\leq 1$. Integrating by parts and applying \eqref{standard1d}, we have
\begin{equation}
\begin{split}
\int_{0}^{1}|f'(y)|^{2}dy&=|f'(1)|^{2}-2\int_{0}^{1}yf'(y)f''(y)dy.\\
&\lesssim \int_{\frac{1}{2}}^{2}y|f(y)|^{2}dy+ \int_{\frac{1}{2}}^{2}y|f''(y)|^{2}dy+\left(\int_{0}^{1}|f'(y)|^{2}dy\right)^{\frac{1}{2}}\left(\int_{0}^{1}y^{2}|f''(y)|^{2}dy\right)^{\frac{1}{2}}.
\end{split}
\end{equation}
Therefore,
\begin{equation}
\|f\|_{L^{\infty}([0,1])}\leq |f(1)|+\int_{0}^{1}|f'(y)|dy\lesssim \|\sqrt{y}f(y)\|_{L^{2}(\R^{+})}+\|\sqrt{y}f''(y)\|_{L^{2}(\R^{+})}.\end{equation}

\textit{Proof of \eqref{interr}:} Notice that \eqref{interr} is obvious if $f''\equiv 0$; hence we may assume that $f''\not\equiv 0$. 

For $p=\infty$, by the same method as in  \eqref{standard1d}, we can obtain
\begin{equation}\label{standard1dd}
\|f'(y)\|_{L^{\infty}(\R^{+})}\lesssim \|f(y)\|_{L^{\infty}(\R^{+})}+\|f''(y)\|_{L^{\infty}(\R^{+})}.
\end{equation}
Then \eqref{interr} follows from a standard rescaling argument.

For $1\leq p<\infty$, by the same method as for \eqref{standard1d}, we can also obtain for each $k\in\Z\cap[0,\infty)$,
\begin{equation*}
    \int_{k}^{k+1}|f'(y)|^{p}dy\lesssim \int_{k}^{k+1}|f(y)|^{p}dy+\int_{k}^{k+1}|f''(y)|^{p}dy.
\end{equation*}
Hence
\begin{equation}\label{standardoned}
    \|f'\|_{L^{p}(\R^{+})}\lesssim \left(\|f\|_{L^{p}(\R^{+})}^{p}+\|f''\|_{L^{p}(\R^{+})}^{p}\right)^{\frac{1}{p}}\leq \|f\|_{L^{p}(\R^{+})}+\|f''\|_{L^{p}(\R^{+})}.
\end{equation}
\eqref{interr} follows from \eqref{standardoned} and a standard rescaling argument.
\end{proof}

We also need the following anisotropic interpolation, which is stated for the domain $\R^{3}$, but it is clear that the same conclusion holds for a finite cube in $\R^{3}$ by standard Sobolev extensions.
\begin{lemma}\label{aniinterpolation}
  Fix $\alpha\in(0,1)$. Assume that $f(t,x,y)\in C_0^{2}(\R^{3})$. Then for any $p$ with $\frac{2}{\alpha}+1<p<\infty$, 
    \begin{equation}\label{HoInt8}
    \|\p_{y}f\|_{L^{\infty}(\R^{3})}\lesssim_{\alpha,p}\|f\|_{C^{\alpha}(\R^{3})}+\|\p_{y}^{2}f\|_{L^{p}(\R^{3})}.
    \end{equation}
\end{lemma}
\begin{proof}
    Let $\Delta_{j}$ be the Littlewood--Paley projection operator on $(t,x)$. We shall use the interpolation inequality
    \begin{equation}\label{newinterpo5}
    \|g'(y)\|_{L^{\infty}(\R)}\lesssim \|g(y)\|_{L^{\infty}(\R)}^{\frac{p-1}{2p-1}}\cdot\|g''(y)\|_{L^{p}(\R)}^{\frac{p}{2p-1}}.
    \end{equation}
    The proof of \eqref{newinterpo5} follows from the same argument as that for \eqref{interr}.

   In view of \eqref{newinterpo5} and Bernstein's inequality, we get that
\begin{equation}
\begin{split}
\|\p_{y}f\|_{L^{\infty}(\R^{3})}&\leq \sum_{j\ge-1}\|\p_{y}\Delta_{j}f\|_{L^{\infty}_{t,x}L^\infty_y(\R^{3})}\lesssim \sum_{j\ge-1}\|\Delta_{j}f\|_{L^{\infty}_{t,x,y}(\R^3)}^{\frac{p-1}{2p-1}}\|\p_{y}^{2}\Delta_{j}f\|_{L^{\infty}_{t,x}L^{p}_{y}(\R^3)}^{\frac{p}{2p-1}}\\
&\lesssim \sum_{j\ge-1}\left(2^{-j\alpha}\|f\|_{C^{\alpha}(\R^3)}\right)^{\frac{p-1}{2p-1}}\left(2^{\frac{2j}{p}}\|\p_{y}^{2}f\|_{L^{p}(\R^3)}\right)^{\frac{p}{2p-1}}\\
&\lesssim_{p,\alpha} \|f\|_{C^{\alpha}(\R^{3})}+\|\p_{y}^{2}f\|_{L^{p}(\R^{3})},
\end{split}
\end{equation}
provided that $p>\frac{2}{\alpha}+1$ (so that $-\alpha\frac{p-1}{2p-1}+\frac{2}{2p-1}<0$).
\end{proof}
\subsection{Enhanced Dissipation}
In this subsection, we prove an enhanced dissipation type estimate. Let $\xi>0$. Consider the following 1--D parabolic equation on $\R^{+}_{t}\times \R^{+}_{y}$:
\begin{equation}\label{enhanceparabolic}
\begin{cases}
\p_{t}v+iy\xi v-\p_{y}^{2}v=0,\\
\p_{y}v|_{y=0}=0,\ v|_{t=0}=v_{0}(y)\in L^{2}(\R^{+};\mathbb{C}).
\end{cases}
\end{equation}
Our main result is the following. 
\begin{proposition}\label{enhanced} Assume that $v$ is the weak solution of \eqref{enhanceparabolic}. Then
\begin{equation}\label{enhancedrate}
\|v(t,\cdot)\|_{L^{2}_{y}}\lesssim e^{-c\xi^{\frac{2}{3}}t}\|v_{0}\|_{L^{2}_{y}},
\end{equation}
for some universal constant $c>0$.
\end{proposition}
\begin{proof}Let $s=\xi t$ and $v(t,y)=V(s,y)$. Then $V$ satisfies the following parabolic equation
\begin{equation}\label{smallvis}
\begin{cases}
\p_{s}V+iyV-\xi^{-1}\p_{y}^{2}V=0,\\
V|_{s=0}=v_{0}(y),\ \p_{y}V|_{y=0}=0.
\end{cases}
\end{equation}
To establish \eqref{enhancedrate}, we only need to show that (with $\nu=\xi^{-1}$)
\begin{equation}\label{goalenhanced}
\|V\|_{L^{2}_{y}}\lesssim e^{-c\nu^{\frac{1}{3}}s}\|v_{0}\|_{L^{2}_{y}}.
\end{equation}
By the Gearhart--Pr\"uss-type theorem proved by Wei \cite{W21} (see also \cite{HS10}, \cite{HS21}, \cite{J23}), we only need to show that
\begin{equation}\label{opbound}
\left\|(i\lambda+iy-\nu\p_{y}^{2})^{-1}\right\|_{L^{2}_{y}(\R^+)\to L^{2}_{y}(\R^+)}\lesssim \nu^{-\frac{1}{3}},
\end{equation}
for each $\lambda\in\R$.
Equivalently, we need to prove that assuming $f$ satisfies the following elliptic equation for $y\ge0$,
\begin{equation}
\begin{cases}\label{ode}
i(y+\lambda)f-\nu\p_{y}^{2}f=g,\\
\p_{y}f|_{y=0}=0,
\end{cases}
\end{equation}
then
\begin{equation}
\|f\|_{L^{2}(\R^+)}\lesssim \nu^{-\frac{1}{3}}\|g\|_{L^{2}(\R^+)}
\end{equation}
holds for any $\lambda\in\R$. By rescaling, we only need to consider $\nu=1$. By linearity, we assume that $\|g\|_{L^{2}(\R^+)}=1$.

\textit{Case 1. $\lambda\geq 0$.} First, testing $\bar{f}$ in \eqref{ode}, integrating by parts, and taking the real part, we get
\begin{equation}\label{H1control}
\int_{0}^{\infty}|\p_{y}f|^{2}dy\leq \|f\|_{L^{2}(\R^+)}
\end{equation}
Next, testing $\bar{f}$ in \eqref{ode} again, integrating by parts, and taking the imaginary part, we obtain that
\begin{equation}\label{L1control}
\int_{0}^{\infty}(y+\lambda)|f|^2dy\leq  \|f\|_{L^{2}(\R^+)}
\end{equation}
Hence we have $\|f\|_{L^{2}(\R^+)}\lesssim 1$ by \eqref{H1control}, \eqref{L1control}, and \eqref{spectralgap} in Lemma \ref{interonline}.

\textit{Case 2. $\lambda<0$.} This case is more complicated as $y+\lambda$ changes sign on $\R^+$. To overcome this difficulty, define the following cut--off function
\begin{equation}
\Theta(y):=
\begin{cases}
1,\quad y\geq -\lambda+\delta,\\
0,\quad -\lambda-\frac{\delta}{2}\leq y \leq -\lambda+\frac{\delta}{2},\\
-1,\quad y\leq -\lambda-\delta,\\
\end{cases}
\quad |\Theta'(y)|\leq \frac{4}{\delta},
\end{equation}
and $\Theta(y)$ is smooth on transition intervals. Here the parameter $\delta>0$ will be chosen later.

First, testing $\bar{f}$ in \eqref{ode}, we obtain \eqref{H1control} as above. Then testing \eqref{ode} by $\Theta(y)\bar{f}$, and taking the imaginary part,  by \eqref{H1control}  we have
\begin{equation}
\begin{split}
\int_{|y+\lambda|\geq \delta}|y+\lambda||f|^{2}dy&\leq \|f\|_{L^{2}}+4\|\p_{y}f\|_{L^{2}}\|f\|_{L^{2}}\delta^{-1}\leq\|f\|_{L^{2}}+4\delta^{-1}\|f\|_{L^{2}}^{\frac{3}{2}}.
\end{split}
\end{equation}
Applying  \eqref{spectralgap} in Lemma \ref{interonline} we get that
\begin{equation}\label{L2control}
\begin{split}
\int_{0}^{\infty}|f|^{2}dy&\lesssim \int_{0}^{\infty}|\p_{y}f|^{2}dy+\int_{0}^{\infty}|y+\lambda|\cdot|f|^{2}dy\\
&\lesssim\delta\int_{0}^{\infty}|f|^{2}dy+\int_{|y+\lambda|\geq \delta, \,y\in\R^+}|y+\lambda||f|^{2}dy+\int_{0}^{\infty}|\p_{y}f|^{2}dy\\
&\lesssim \|f\|_{L^{2}}+4\delta^{-1}\|f\|_{L^{2}}^{\frac{3}{2}}+2\delta\|f\|_{L^{2}}^{2}.
\end{split}
\end{equation}
Choosing $\delta>0$ sufficiently small, by \eqref{L2control} we obtain that $\|f\|_{L^{2}}\lesssim 1$. The proof of Proposition \ref{enhanced} is complete.
\end{proof}

 % \frenchspacing
 %   \vspace{\fill}
	\bibliographystyle{plain}
	\bibliography{classical_prandtl_reference}	

\end{document}